\documentclass[cjm]{ipart}

\startlocaldefs

 \usepackage{amsmath, amscd, amssymb, mathrsfs, cancel,
  mathtools, fullpage, enumerate, thmtools,
  tabu, mdframed, fancyhdr, wrapfig, xspace, tikz-cd, etoolbox}
 
 \usepackage{stmaryrd}
 \expandafter\def\csname opt@stmaryrd.sty\endcsname
 {only,shortleftarrow,shortrightarrow}
 \usepackage{extpfeil}

\usepackage[pagebackref, hypertexnames=false]{hyperref}
\usepackage[capitalise]{cleveref}

\newenvironment{smatrix}{\left( \begin{smallmatrix} } {\end{smallmatrix} \right) }

\newcommand{\stbt}[4]{\begin{smatrix}#1 & #2 \\ #3 & #4\end{smatrix}}

\theoremstyle{plain}
\newtheorem{theorem}{Theorem}[subsection]
\newtheorem{introtheorem}{Theorem}

\newtheorem{lemma}[theorem]{Lemma}
\newtheorem{proposition}[theorem]{Proposition}
\newtheorem{corollary}[theorem]{Corollary}
\newtheorem{definition}[theorem]{Definition}
\newtheorem{notation}[theorem]{Notation}
\newtheorem{conjecture}[theorem]{Conjecture}
\newtheorem{assumption}[theorem]{Assumption}
\newtheorem*{convention}{Convention}

\theoremstyle{remark}
\declaretheorem[name=Remark,sibling=theorem,qed={\lower-0.3ex\hbox{$\diamond$}}]{remark}
\declaretheorem[name=Note,sibling=theorem,qed={\lower-0.3ex\hbox{$\diamond$}}]{note}
\declaretheorem[name=Remark,unnumbered,qed={\lower-0.3ex\hbox{$\diamond$}}]{remarknonumber}

\DeclareMathOperator{\AJ}{AJ}
\DeclareMathOperator{\alg}{alg}
\DeclareMathOperator{\BGG}{BGG}
\DeclareMathOperator{\br}{br}
\DeclareMathOperator{\codim}{codim}
\DeclareMathOperator{\coh}{coh}
\DeclareMathOperator{\cosp}{cosp}
\DeclareMathOperator{\crit}{crit}
\DeclareMathOperator{\diag}{diag}
\DeclareMathOperator{\DR}{DR}
\DeclareMathOperator{\Fil}{Fil}
\DeclareMathOperator{\fp}{fp}
\DeclareMathOperator{\Gal}{Gal}
\DeclareMathOperator{\gen}{gen}
\DeclareMathOperator{\GL}{GL}
\DeclareMathOperator{\Gr}{Gr}
\DeclareMathOperator{\GSp}{GSp}
\DeclareMathOperator{\HK}{HK}
\DeclareMathOperator{\Hom}{Hom}
\DeclareMathOperator{\id}{id}
\DeclareMathOperator{\Iw}{Iw}
\DeclareMathOperator{\Kl}{Kl}
\DeclareMathOperator{\LE}{\mathcal{LE}}
\DeclareMathOperator{\loc}{loc}
\DeclareMathOperator{\lrigfp}{lrig-fp}
\DeclareMathOperator{\lrigsyn}{lrig-syn}
\DeclareMathOperator{\lrig}{lrig}
\DeclareMathOperator{\MF}{MF}
\DeclareMathOperator{\mom}{mom}
\DeclareMathOperator{\mot}{mot}
\DeclareMathOperator{\myPr}{Pr}\renewcommand{\Pr}{\myPr}
\DeclareMathOperator{\NNfp}{NN-fp}
\DeclareMathOperator{\NNsyn}{NN-syn}
\DeclareMathOperator{\ord}{ord}
\DeclareMathOperator{\pr}{pr}
\DeclareMathOperator{\Pz}{Pz}
\DeclareMathOperator{\rigfp}{rig-fp}
\DeclareMathOperator{\rigsyn}{rig-syn}
\DeclareMathOperator{\Sieg}{Si}
\DeclareMathOperator{\Spec}{Spec}
\DeclareMathOperator{\sph}{sph}
\DeclareMathOperator{\st}{st}
\DeclareMathOperator{\Sym}{Sym}
\DeclareMathOperator{\syn}{syn}
\DeclareMathOperator{\tr}{tr}
\DeclareMathOperator{\vol}{vol}

\DeclarePairedDelimiter{\tb}{{\,]}}{{[\, }} 

\newcommand{\Af}{\AA_{\mathrm{f}}}
\newcommand{\CC}{\mathbf{C}}
\newcommand{\Dcris}{\mathbf{D}_{\mathrm{cris}}}
\newcommand{\DdR}{\mathbf{D}_{\mathrm{dR}}}
\newcommand{\Dpst}{\mathbf{D}_{\mathrm{pst}}}
\newcommand{\Eis}{\mathrm{Eis}}
\newcommand{\Fp}{\mathbf{F}_p}
\newcommand{\HH}{\mathbf{H}}
\newcommand{\Nek}{Nekov\'a\v{r}\xspace}
\newcommand{\Niz}{Nizio\l\xspace}
\newcommand{\Pif}{\Pi_{\mathrm{f}}}
\newcommand{\QQbar}{\overline{\QQ}}
\newcommand{\QQ}{\mathbf{Q}}
\newcommand{\Qp}{\QQ_p}
\newcommand{\RGt}{\widetilde{R\Gamma}}
\newcommand{\RR}{\mathbf{R}}
\newcommand{\ZZ}{\mathbf{Z}}
\newcommand{\Zp}{\mathbf{Z}_p}

\newcommand{\bfj}{\mathbf{j}}
\newcommand{\cD}{\mathcal{D}}
\newcommand{\cE}{\mathcal{E}}
\newcommand{\cF}{\mathcal{F}}
\newcommand{\cG}{\mathcal{G}}
\newcommand{\cH}{\mathcal{H}}
\newcommand{\cL}{\mathcal{L}}
\newcommand{\cM}{\mathcal{M}}
\newcommand{\cN}{\mathcal{N}}
\newcommand{\cO}{\mathcal{O}}
\newcommand{\cP}{\mathcal{P}}
\newcommand{\cQ}{\mathcal{Q}}
\newcommand{\cR}{\mathcal{R}}
\newcommand{\cS}{\mathcal{S}}
\newcommand{\cT}{\mathcal{T}}
\newcommand{\cU}{\mathcal{U}}
\newcommand{\cV}{\mathcal{V}}
\newcommand{\cW}{\mathcal{W}}
\newcommand{\cX}{\mathcal{X}}
\newcommand{\cY}{\mathcal{Y}}
\newcommand{\cZ}{\mathcal{Z}}
\newcommand{\can}{\mathrm{can}}
\newcommand{\ch}{\mathrm{ch}}
\newcommand{\dR}{\mathrm{dR}}
\newcommand{\dep}{\mathrm{dep}}
\newcommand{\et}{\text{\textup{\'et}}}
\newcommand{\fP}{\mathfrak{P}}
\newcommand{\into}{\hookrightarrow}

\newcommand{\onto}{\twoheadrightarrow}
\newcommand{\res}{\mathrm{res}}
\newcommand{\rig}{\mathrm{rig}}
\newcommand{\sC}{\mathscr{C}}
\newcommand{\sFil}{{\mathscr Fil}}
\newcommand{\sF}{\mathscr{F}}
\newcommand{\sG}{\mathscr{G}}
\newcommand{\sH}{\mathscr{H}}
\newcommand{\sL}{\mathscr{L}}
\newcommand{\sgn}{\mathrm{sgn}}
\newcommand{\uGa}{\underline{\Gamma}}
\newcommand{\uPhi}{\underline{\Phi}}
\newcommand{\uPi}{\underline{\Pi}}
\newcommand{\uTh}{\underline{\Theta}}
\newcommand{\uchi}{\underline{\chi}}
\newcommand{\unu}{\underline{\nu}}
\newcommand{\utau}{\underline{\tau}}
\newcommand{\wH}{\widetilde{H}}
\newcommand{\wZ}{\widetilde{Z}}
\renewcommand{\AA}{\mathbf{A}}

\numberwithin{equation}{subsection}
\renewcommand{\le}{\leqslant}
\renewcommand{\leq}{\leqslant}
\renewcommand{\ge}{\geqslant}
\renewcommand{\geq}{\geqslant}
\renewcommand{\emptyset}{\varnothing}

\endlocaldefs

\pubyear{2026}
\volume{0}
\issue{0}
\firstpage{1}
\lastpage{1}
\arxiv{2003.05960}

\begin{document}

\begin{frontmatter}

\title{On the Bloch--Kato conjecture for $\GSp(4)$\protect\thanksref{T1}}

\thankstext{T1}{This research was supported by the following grants: Royal Society University Research Fellowship ``$L$-functions and Iwasawa theory'' and ERC Consolidator Grant ``Shimura varieties and the BSD conjecture'' (Loeffler);  ERC Consolidator Grant ``Euler systems and the Birch--Swinnerton-Dyer conjecture'' (Zerbes); EPSRC grant EP/R014604/1 (both authors).}

\begin{aug}
\author{\fnms{David} \snm{Loeffler} 
\ead[label=e1]{david.loeffler@unidistance.ch}
\ead[label=u1]{foo} }
\address{UniDistance Suisse, Schinerstrasse 18, 3900 Brig, Switzerland \\ 
ORCID: \href{https://orcid.org/0000-0001-9069-1877}{0000-0001-9069-1877} \\
\printead{e1}}

\and

\author{\fnms{Sarah Livia} \snm{Zerbes} 
\ead[label=e2]{sarah.zerbes@math.ethz.ch} }
\address{Department of Mathematics, ETH Z\"urich, R\"amistrasse 101, 8092 Z\"urich, Switzerland \\
ORCID: \href{https://orcid.org/0000-0001-8650-9622}{0000-0001-8650-9622} \\
 \printead{e2}}
\end{aug}

\received{\sday{18} \smonth{2} \syear{2026}}

\begin{abstract}
 We prove an explicit reciprocity law for the Euler system attached to the spin motive of a genus $2$ Siegel modular form. As consequences, we obtain one inclusion of the Iwasawa Main Conjecture for such motives, and the Bloch--Kato conjecture in analytic rank 0 for their critical twists.
\end{abstract}


\setcounter{tocdepth}{1}
\tableofcontents

\end{frontmatter}

\newcommand{\mychapter}[1]{%
  \clearpage
  \addcontentsline{toc}{section}{--- \normalfont \emph{#1} ---}%
  \begin{center}
    \Huge\bfseries #1
  \end{center}
  \vspace{2em}
}
 
 \renewcommand{\crefrangeconjunction}{--} 

\mychapter{Introduction}

\section{Aims of this paper}

 Euler systems are one of the most powerful tools for controlling the cohomology groups of global Galois representations, and hence for proving cases of the two interrelated conjectures linking these groups to values of $L$-functions: the Bloch--Kato conjecture and the Main Conjecture of Iwasawa Theory. More precisely, it follows from work of Kolyvagin, Kato and Rubin that if there exists an Euler system for some Galois representation $V$, and if the bottom class of this Euler system is non-zero, then we obtain a bound on the Selmer group of $V$. So, in order to make progress on the above conjectures, we need to first construct an Euler system for $V$, and then to prove a formula (an \emph{explicit reciprocity law}) relating the localisation of this Euler system at $p$ to the critical values of $L$-functions. The goal of this paper is to carry out this program for the 4-dimensional spin Galois representations arising from Siegel modular forms of genus 2, i.e.~automorphic representations of the group $\GSp_4 / \QQ$.

 This builds on earlier work carried out in the paper \cite{LSZ17} together with Chris Skinner, where we constructed an Euler system for these spin Galois representations. At the time that paper was written, the tools were not available to prove an explicit reciprocity law for the Euler system; so we could not rule out the possibility that the entire Euler system was zero, and the arithmetic applications given in \emph{op.cit.} were conditional on assuming that the Euler system satisfied an explicit reciprocity law of the expected form. The main result of the present paper is a proof of the missing explicit reciprocity law.

 The other main input to this paper is our work \cite{LPSZ1} with Skinner and Vincent Pilloni, in which we constructed a $p$-adic $L$-function interpolating critical values of the spin $L$-functions of an automorphic representation of $\GSp_4$. This uses Piatetski--Shapiro's integral formula for the spin $L$-function \cite{piatetskishapiro97}, Harris' interpretation of this integral in terms of coherent cohomology \cite{harris04}, and Pilloni's results on $p$-adic interpolation of coherent cohomology via higher Hida theory \cite{pilloni20}. Although the explicit reciprocity law can be formulated purely in terms of complex $L$-values (without mentioning $p$-adic $L$-functions), the existence of the $p$-adic $L$-function plays a fundamental role in the proof, since we shall first relate Euler-system classes to \emph{non-critical} values of the $p$-adic $L$-function, and deduce the reciprocity law at critical values by analytic continuation (as in several previous works on reciprocity laws for Euler systems, such as \cite{BDP13} and \cite{KLZ17}).

 As a consequence of the explicit reciprocity law, we obtain one inclusion of the Iwasawa Main Conjecture for the spin Galois representation, and the Bloch--Kato conjecture for the analytic rank 0 twists of this Galois representation. For these arithmetic applications, we assume for simplicitly that our automorphic representation has level 1, and highly regular weight; this is purely to keep the arguments short, and the more general case will be treated in future work. However, these simplifying hypotheses are \emph{not} imposed for the $p$-adic regulator formula which is the main input for these applications; this formula is proved here for arbitrary levels, assuming only a mild condition on the weight ($r_1 - r_2 \ge 3$).

\section{Main results of the paper}

 In order to state the results a little more precisely, we need to introduce some notation. Let $p$ be a prime. As in \cite[\S 2]{LSZ17}, $G$ denotes the symplectic group $\GSp_4$, $P_{\Sieg}$ and $P_{\Kl}$ denote its standard Siegel and Klingen parabolic subgroups, and $H$ denotes the subgroup $\GL_2 \times_{\GL_1} \GL_2$.

 Let $\Pi$ be a non-endoscopic, non-CAP, globally generic automorphic representation of $G(\AA)$, of weights $(k_1,k_2)=(r_1+3,r_2+3)$ with $r_1,r_2\geq 0$, and write $V_\Pi$ for the 4-dimensional $p$-adic spin Galois representation of $\Pi$. Let $(q, r)$ be integers with $0 \le q \le r_2$ and $0 \le r \le r_1 - r_2$; and let $\uchi = (\chi_1, \chi_2)$ be a pair of Dirichlet characters with $\chi_1 \chi_2 = \chi_\Pi$, satisfying the parity constraint
 \[ (-1)^{q + r} = (-1)^{r_1} \chi_1(-1) = (-1)^{r_2} \chi_2(-1).\]

 In \cite{LSZ17} and \cite{LZ-equivar}, we defined a cohomology class
 \[ z^{[\Pi, q, r]}_{\can}(\uchi) \in H^1_{\mathrm{f}}\left(\QQ, V_{\Pi}^*(-q)\right). \]
 using pushforwards of Eisenstein classes from $H.$ Our first main result computes the image of this class under the Bloch--Kato logarithm map, expressing it as a non-critical value of a $p$-adic spin $L$-function:

 \begin{introtheorem}
  \label{thm:main}
  Suppose $\Pi$ is unramified and Klingen-ordinary at $p$, and $r_1 - r_2 \ge 3$. Let $\nu$ be a basis of $\Gr^1 \DdR(V_{\Pi})$, and let $\nu_{\dR}$ be its unique lifting to a vector in $\Fil^1 \DdR(V_{\Pi})^{(\varphi - \alpha)(\varphi-\beta) = 0}$ (\cref{note:1diml}). Then we have
  \[  \left\langle \nu_{\dR},\, \log_{\mathrm{BK}} z^{[\Pi, q, r]}_{\can}(\uchi) \right\rangle_{\Dcris(V_\Pi)}=
  (\star)\times \cL_{p, \nu}(\Pi, \uchi; -1-r_2+q, r). \]
  for an explicit non-zero factor $(\star)$. Here, $\cL_{p}(\Pi, \bfj_1, \bfj_2)$ denotes the 2-variable spin $p$-adic $L$-function constructed in \cite{LPSZ1}.
 \end{introtheorem}

 The proof of this theorem occupies the majority of the paper. Note that we do not require $\Pi$ to have level 1 here.
%

 Our second main result is a considerable strengthening of Theorem A, under far more restrictive hypotheses. We now assume that $\Pi$ satisfies the conditions of Theorem A, and also the following extra conditions:
 \begin{itemize}
  \item $\Pi$ is Borel-ordinary at $p$;
  \item $\Pi$ has level one (i.e.~$\Pi_\ell$ is unramified for all finite primes $\ell$);
  \item $r_1 - r_2 \ge 6$;
  \item for some (and hence every) $G_{\QQ}$-stable lattice $T$ in $V_\Pi^\star$, and every Dirichlet character $\chi$ of prime-to-$p$ conductor, Rubin's ``big image'' condition $\operatorname{Hyp}(\QQ(\mu_{p^\infty}), T(\chi))$ holds (cf.~\cite[Assumption 11.1.2]{LSZ17}).
 \end{itemize}
 The condition $r_1 - r_2 \ge 6$ implies that the $p$-adic $L$-function factors as the product of two copies of a single-variable $p$-adic $L$-function $\cL_p(\Pi, \bfj)$.

 \begin{introtheorem}
  \label{thmB}
  There exists an Euler system for $V_\Pi^*(-1-r_2)$, whose image under the Perrin-Riou cyclotomic regulator map is the $p$-adic $L$-function $\cL_p(\Pi, \bfj)$.
 \end{introtheorem}

 Note that this result relies on Theorem A not only for $\Pi$ itself, but also for all the classical specialisations of a $p$-adic family passing through $\Pi$. From Theorem B we readily obtain the following two arithmetic applications. The first gives one inclusion in the Iwasawa main conjecture for $V_\Pi^*$, up to inverting $p$:

 \begin{introtheorem}
  Let $V=V_\Pi^*(-1-r_2)$, and denote by $\RGt_{\Iw}(\QQ_\infty,V)$ the \Nek \hspace{0.5ex} Selmer complex, with the unramified local conditions at $\ell\neq p$ and the Greenberg-type local condition at $p$ determined by the Klingen-ordinarity of $\Pi$. Assume that the above conditions are satisfied. Then the module $\wH^2_{\Iw}(\QQ(\mu_{p^\infty}), V)$ is torsion over the Iwasawa algebra, and its characteristic ideal divides the $p$-adic $L$-function $\cL_p(\Pi, \bfj)$.
 \end{introtheorem}

 Note that this is a divisibility of ideals in $\Lambda_L(\Zp^\times)$ where $\Gamma \cong \Zp^\times$ and $L$ is a finite extension of $\Qp$. The module $\wH^2_{\Iw}$ can also be interpreted more classically as the base-extension to $L$ of the Pontryagin dual of a Selmer group attached to a representation of cofinite type over $\Zp$, linking up with more classical formulations of an Iwasawa main conjecture; see \cref{prop:nekovarduality} below.

 Our second application is to the Bloch--Kato conjecture:

 \begin{introtheorem}
  Assume that the above conditions are satisfied. Let $0\leq j\leq r_1-r_2$, and let $\rho$ be a finite-order character of $\Zp^\times$. If $L\left(\Pi\otimes \rho,\frac{1-r_1+r_2}{2}+j\right)\neq 0$, then $H^1_\mathrm{f}(\QQ,V(-j-\rho))=0$.
 \end{introtheorem}

 \subsection*{Relations to other work}

  In sequels to this paper, \cite{LZ20-yoshida} and \cite{LZ21-BSD}, we relax the conditions on the weight and tame level of $\Pi$, and consider applications to the Iwasawa main conjecture for quadratic Hilbert modular forms, and to the Birch--Swinnerton-Dyer conjecture for modular abelian surfaces. The methods of this paper can also be applied to Euler systems for $\GSp_4 \times \GL_2$ and $\GSp_4 \times \GL_2 \times \GL_2$; see \cite{LZ20b-regulator} and \cite{LZ21-erl}.

  More generally, the strategy that we developed for the proof of the explicit reciprocity law should be applicable to many other cases where an Euler system has been constructed, but where the relevant $L$-values cannot be expressed purely in terms of degree zero coherent cohomology, as they can in previously-studied cases such as $\GL_2 \times \GL_2$. For instance, this applies to the Asai representations of quadratic Hilbert modular forms. See \cite{grossiloefflerzerbes-GO4} for a proof of the $p$-adic regulator formula in this setting, and applications to the Bloch--Kato and Iwasawa main conjectures for Asai motives; and \cite{LZ-adjoint} for an application to the adjoint of an elliptic modular form. It should also be possible to prove an explicit reciprocity law for the $\operatorname{GU}(2, 1)$ Euler system of \cite{LSZ-unitary} via similar methods, and we hope to pursue this in future work.

\section{Strategy}

 We outline the overall strategy used in the proofs of Theorems A and B.

 \subsection{Strategy for Theorem A}

  \begin{enumerate}
   \item Using equivariance properties of the Lemma--Flach classes as the test data $(w, \uPhi)$ vary, we show that it suffices to prove the theorem for $(w, \uPhi)$ which have a certain specific type at $p$ (``Klingen-type test data''). For these Klingen test data at $p$, the left-hand side of \cref{thm:main} can be expressed as a pairing \eqref{eq:goal2} between a de Rham cohomology class $\eta_{\dR}$ of Klingen parahoric level which is an ordinary eigenvector for the Hecke operator $U_{2, \Kl}'$, and the logarithm of an \'etale class which is the pushforward of a pair of $\GL_2$ Eisenstein classes along a certain ``twisted'' embedding $\iota_\Delta$ of Shimura varieties $Y_{H,\Delta}\hookrightarrow Y_{G,\Kl}$. (This embedding is also used in the definition of the $p$-adic $L$-function $\cL_p(\Pi)$ in \cite{LPSZ1}.)

   \item We express the pairing \eqref{eq:goal2} using the ``\Nek--\Niz finite-polynomial cohomology'' of \cite{besserloefflerzerbes16} (a variant of the syntomic cohomology introduced in \cite{nekovarniziol16}). This gives a formalism of Abel--Jacobi maps, allowing us to write \eqref{eq:goal2} as a cup-product between the pushforward of the syntomic $\GL_2 \times \GL_2$ Eisenstein class and a class $\eta_{\NNfp,-D}$ which is a lifting of $\eta_{\dR}$ to \Nek--\Niz fp-cohomology; see \eqref{eq:step1}. By a new comparison result due to Ertl--Yamada \cite{ertlyamada19}, this is equivalent to a pairing in log--rigid finite-polynomial cohomology (c.f. Proposition \ref{prop:lrigreduction}). 

   \item In Section \ref{sect:redstep2}, we show that the pairing factors through a pairing in the rigid fp-cohomology of the p-rank $m$ locus $Y_{G, \Kl}^m$, which only ``sees'' the restriction of the Eisenstein class to the ordinary locus $Y_{H,\Delta}^{2,m} \subseteq Y_{H, \Delta}$ (Theorem \ref{thm:redtoord}). This allows us to use the explicit description, due to Bannai--Kings, of the syntomic Eisenstein classes for $\GL_2$ over the ordinary locus, in terms of non-classical $p$-adic Eisenstein series.

   \item To actually compute the pairing of Theorem \ref{thm:redtoord} and relate it to $p$-adic $L$-functions, we need an explicit description of the lifting $\eta_{\NNfp,-D}$ in terms of classes in the coherent cohomology groups studied in Pilloni's higher Hida theory. This is the most novel part of the construction, and relies on two new ingredients:
   \begin{itemize}
    \item A theory of rigid and coherent cohomology with \emph{partial compact support} (see Section \ref{section:prelimrigcohom}), i.e.~with compact support towards some of the closed strata of the special fibre but not towards others. This allows us to bypass the lack of a Frobenius lifting over $Y_{G, \Kl}^m$, by instead working in the cohomology of the ordinary locus $Y_{G, \Kl}^{2,m}$ with an appropriate partial support condition; see \cref{prop:restricttoc0} for this reduction.

    \item A new  spectral sequence, the \emph{Pozna\'n spectral sequence} 
    (Proposition \ref{prop:Poznan}), relating syntomic (or finite-polynomial) cohomology to the mapping fibre of a polynomial in Frobenius over coherent cohomology. This spectral sequence can be seen as a syntomic analogue of the Fr\"olicher spectral sequence relating de Rham and coherent cohomology.
   \end{itemize}

   \item We now use an identity relating Hecke operators  on $G$ and on $H$ (Proposition \ref{prop:weirdcorresp}) to simplify the coherent cohomology pairing until we are left with only two terms. Both can be recognised as special values at $\bfj = 0$ of $p$-adic measures $\sL_1(\bfj)$ and $\sL_2(\bfj)$, which are very similar, but \emph{a priori} not quite identical, to the $p$-adic $L$-function of \cite{LPSZ1} -- the difference lies in the choice of local data at $p$. By a local zeta-integral computation, we show that at critical values the measure $\sL_1$ has the same interpolating property as the $p$-adic $L$-function, while the measure $\sL_2$ is identically 0. So the regulator is given by the value of $\sL_1$ at $\bfj = 0$, and this corresponds to a non-critical value of the $p$-adic $L$-function. This completes the proof of Theorem A.
  \end{enumerate}

  \begin{remarknonumber}\
   \begin{itemize}
   	\item In a previous version of the paper, steps (3) and (5) relied on the so-called \textit{eigenspace vanishing conjecture}. This is no longer the case: in step (3), we use instead an argument suggested to us by George Boxer and Vincent Pilloni (Section \ref{sect:BPproof}), and step (5) relies on a weaker result concerning Hecke eigenspaces in the rigid cohomology of the cuspidal boundary (c.f. Appendix \ref{chap:appendix}).

    \item The first glimpse of the Pozna\'n spectral sequence is \cite[Proposition A.16]{bannaikings10}, which represents elements of the first syntomic cohomology group of a smooth pair in terms of classes in coherent cohomology.
    \item The Hecke operator identity of \cref{prop:weirdcorresp} is an analogue in the present setting of an identity of Hecke operators for $\GL_2 \times \GL_2$ which occurs in the proofs of regulator formulae for Rankin--Selberg convolutions; see the proof of \cite[Lemma 6.4.6]{KLZ20}.
    \item The idea of (coherent) cohomology with partial compact support was discovered independently by Pilloni (\cite{pilloni20}).\qedhere
   \end{itemize}
  \end{remarknonumber}

 \subsection{Strategy for Theorem B}

  In order to deduce Theorem B from Theorem A, we use variation in a $p$-adic family. We use $p$-adic families of ``Siegel type'' -- one-parameter families in which we vary $(r_1, r_2)$ $p$-adically while keeping the difference $r_1 - r_2$ fixed.

  If we knew that the $p$-adic $L$-function of \cite{LPSZ1} extended to Siegel-type families, and that there existed a $p$-adic Eichler--Shimura isomorphism for such families, interpolating the period isomorphisms for the middle steps of the Hodge filtration at each classical specialisation (analogous to the results of Ohta \cite{ohta95} and Andreatta--Iovita--Stevens for $\GL_2$ \cite{andreattaiovitastevens}), then Theorem B would be a virtually immediate consequence of Theorem A (we sketch the argument in \cref{sect:ESconj}). However, these ingredients do not seem to be available yet for $\GSp_4$; both statements seem to be accessible for Klingen-type families (with $r_1$ varying but $r_2$ fixed), but the case of Siegel-type families is less clear.

  Instead, we use an alternative argument, relying on the existence of a $p$-adic $L$-function for functorial liftings to $\GL_4$ of Siegel-type families due to Barrera Salazar et al \cite{BDGJW25}. A careful analysis of the relation between this new ``Betti'' $p$-adic $L$-function for the family, and the ``coherent'' $p$-adic $L$-function of \cite{LPSZ1} for its classical specialisations, leads to the conclusion that the image of the Euler system for $\Pi$ under the Perrin-Riou regulator must be a scalar multiple of the $p$-adic $L$-function.

  What remains to be proven is that this scalar factor is not zero. We show that if the ratio of periods giving this scalar factor degenerates to 0, then this happens not only for the Euler system class over the cyclotomic extension $\QQ(\mu_{p^\infty})$, but simultaneously for the classes over $\QQ(\mu_{Mp^\infty})$ for all auxiliary conductors $M$. This gives an Euler system satisfying a stronger-than-expected local condition at $p$, and an result due to Mazur and Rubin shows that in fact no such Euler system can exist, contradicting our assumption. This completes the proof of Theorem B.

\section{Acknowledgements}

 We would like to express our sincere gratitude to John Coates, Henri Darmon and Gert Schneider for their interest and constant encouragement in us writing this paper. We would also like to thank Fabrizio Andreatta, Antonio Cauchi, John Coates, Veronika Ertl, Elmar Gro{\ss}e-Kl\"onne, Kiran Kedlaya, Kai-Wen Lan, Chris Lazda, Bernhard Le Stum, Vincent Pilloni, Joaquin Rodrigues, Christopher Skinner, Nina Wawrow, Chris Williams and Kazuki Yamada for many helpful discussions, and the referees for their valuable comments.

 We are very grateful to George Boxer and Vincent Pilloni for explaining to us the argument in Section \ref{sect:BPproof} and allowing us to include it here.

 One of the key ingredients for proving the main results of this paper is Pilloni's Higher Hida Theory, which we learnt about during the workshop \emph{Motives, Galois Representations and Cohomology Around the Langlands Program} at the IAS in November 2017. We would like to express our gratitude to the organizers for the invitation.

 We discovered a new spectral sequence which is crucial for the proof of the explicit reciprocity law while attending Gregorz Banaszak's birthday conference in Pozna\'n in September 2018. We are very grateful to the organizers for the invitation to such an inspiring event.

 Part of the work on this paper was carried out during visits to the Centre Bernoulli at EPFL, Princeton University, the Morningside Centre for Mathematics, the Isaac Newton Institute, and the Simons Laufer Mathematical Sciences Institute (formerly MSRI). We would like to thank all of these institutions for their hospitality.


\section{Notation and conventions}


\begin{itemize}

 \item Let $J$ be the skew-symmetric $4\times 4$ matrix over $\ZZ$ given by $\begin{smatrix} &&& 1\\ 	&&1&\\&-1&&\\-1&&&\end{smatrix}$. Let $G=\GSp_4$ be the group scheme over $\ZZ$ defined by
 \[ G(R)=\GSp_4(R)=\left\{ g \in \GL_4(R)\times \GL_1(R)\, : \, g^t\cdot J\cdot g= \mu J \text{ for some $\mu \in R^\times$} \right\}\]
 for any unital commutative ring $R$.

 \item Define the standard Borel subgroup $B\subseteq G$ as the intersection of $G$ with the upper-triangular matrices in $\GL_4$.

 \item Denote by $P_{\Sieg}$ and $P_{\Kl}$ the Siegel, resp.~Klingen, parabolic subgroups of $G$ given by
 \begin{align*}
  P_{\Sieg}&=\begin{smatrix} \star&\star&\star&\star\\ \star&\star&\star&\star\\&&\star&\star\\&&\star&\star\end{smatrix},&
  P_{\Kl} &= \begin{smatrix} \star&\star&\star&\star\\ &\star&\star&\star\\
  &\star&\star&\star\\ &&&\star\end{smatrix}.
 \end{align*}
Write $M_{\Sieg}$ and $M_{\Kl}$ for the standard (block-diagonal) Levi subgroups of $P_{\Sieg}$ and $P_{\Kl}$, and $T$ for the diagonal maximal torus.

 \item For a prime $p$ and $n \ge 1$, let $\Kl(p^n)$ denote the open compact subgroup $\{g \in G(\Zp): g \bmod {p^n} \in P_{\Kl}(\ZZ/p^n)\}$, and similarly for $\Sieg(p^n)$ (although the latter will play a relatively minor role in this paper).

 \item Let $H = \{ (h_1, h_2) \in \GL_2 \times \GL_2: \det(h_1) = \det(h_2)\}$, and let $\iota$ denote the embedding $H\hookrightarrow G$ given by
 \[ \left(\begin{pmatrix} a & b\\ c & d\end{pmatrix}, \begin{pmatrix} a'& b'\\ c' & d'\end{pmatrix}\right) \mapsto
 \begin{smatrix} a&&& b\\ & a' & b' & \\ & c' & d' & \\ c &&& d\end{smatrix}.\]

 \item $\|\cdot\|$ denotes the ad\`ele norm map $\AA \to \RR_{\ge 0}$.
 
 \item We shall identify a Dirichlet character $\chi : (\ZZ / N)^\times \to \CC^\times$ with the  unique continuous character of $\AA^\times / \QQ^\times$ that is unramified outside $N$ and maps $\varpi_\ell$ to $\chi(\ell)$ for $\ell \nmid N$, where $\varpi_\ell$ is any uniformizer at $\ell$. Note that the restriction of this adelic $\chi$ to $\widehat{\ZZ}^\times \subset \AA^\times$ is the composite of the projection $\widehat{\ZZ}^\times \to (\ZZ / N)^\times$ with the \emph{inverse} of $\chi$.

 \item In a slight conflict with the previous notation, if $j \in \ZZ$, and $\chi$ is a Dirichlet character conductor $p^m$ for some $m$ (valued in some $p$-adic field $L$), we write ``$j + \chi$'' for the continuous character $\Zp^\times \to L$ given by $x \mapsto x^j \cdot \chi(x \bmod p^m)$.

 \item For $r_1, r_2, c \in \ZZ$ such that $r_1+r_2\equiv c\pmod 2$, let $\lambda(r_1,r_2;c)$ denote the unique character of the diagonal torus $T$ of $G$ such that
 \begin{equation}
  \label{eq:defweight}
  \begin{smatrix} st_1&&&\\&st_2 &&\\&&st_2^{-1} &\\  	&&& st_1^{-1}\end{smatrix}\mapsto t_1^{r_1}t_2^{r_2}s^c.
 \end{equation}
 If $r_1 \ge r_2 \ge 0$, this character is dominant with respect to $B_G$, and we write $V_G(r_1,r_2;c)$ for the corresponding irreducible representation\footnote{In \cite{LSZ17} we used a slightly different parametrisation of the irreducible representations by pairs of integers $a, b \ge 0$. The representation denoted $V^{ab}$ of \emph{op.cit.} is $V(a+b, a; 2a+b)$ in our present notations.} of $G$. Similarly, if $r_1, r_2 \ge 0$ then $\lambda(r_1, r_2; c)$ is dominant for $H$ and we write $V_H(r_1, r_2; c)$ for the analogous highest-weight representation of $H$. If the group is clear from context we omit the subscript $G$.

 \item We shall use Roman letters $X, Y, \dots$ for schemes, Fraktur letters $\mathfrak{X}, \mathfrak{Y}, \dots$ for $p$-adic formal schemes, and calligraphic letters $\cX, \cY, \dots$ for rigid-analytic dagger spaces.

\end{itemize}


\mychapter{Step 1: The problem, and a first reduction}



\section{Euler systems for Siegel automorphic representations}\label{sect:ES}

 Here we briefly recall the Galois cohomology classes constructed in \cite{LSZ17, LZ-equivar} and formulate the problem we are trying to solve, which is to evaluate the images of these classes under the Bloch--Kato logarithm at $p$. We then explain a reduction step (the first of many), expressing these quantities as cup-products in the variant of finite-polynomial cohomology for $\Qp$-varieties introduced in \cite{nekovarniziol16} and \cite{besserloefflerzerbes16}.

 \subsection{Automorphic representations}

  We recall the hypotheses on the automorphic representations we shall consider, following \cite[\S 10.1]{LSZ17} and \cite[\S 3.1]{LZ-equivar}.

  \begin{notation}
   We let $\Pi$ be a cuspidal automorphic representation of $G$, with finite-order central character $\chi_{\Pi}$, which is regular algebraic at $\infty$. We shall also suppose that $\Pi$ is \emph{of general type} in Arthur's classification (cf.~\cite{arthur04}), i.e.~its functorial lift to $\GL_4$ is cuspidal; and that $\Pi$ is \emph{globally generic} (has a non-vanishing Whittaker coefficient).
  \end{notation}

  \begin{remark}
   For our applications to the Bloch--Kato conjecture and Iwasawa main conjecture, the restriction to general-type automorphic representations is no loss. The non-general-type representations corrspond to Galois representations which are direct sums of automorphic Galois representations arising from $\GL_1$ or $\GL_2$, so the Bloch--Kato conjecture for these representations can be attacked using the methods of \cite{huberkings03} and \cite{kato04}. Moreover, the general-type automorphic representations may be partitioned into global ``packets'' in such a way that each packet contains a (unique) globally-generic representation. Since any two representations in the same global packet have the same Galois representation and the same $L$-function, we also lose no generality by supposing that $\Pi$ is globally generic. (For more details see \cite{arthur04} and \cite{geetaibi18}.)
  \end{remark}

  Since $\Pi$ is assumed regular algebraic, there exists a unique pair $(r_1, r_2)$ of integers with $r_1 \ge r_2 \ge 0$ such that we have
  \[ H^*( \mathfrak{g}, K_\infty; \Pi_\infty' \otimes V(r_1, r_2; r_1 + r_2) ) \ne 0, \]
  where $\mathfrak{g} = \operatorname{Lie} G$, $K_\infty$ is the maximal compact-mod-centre subgroup of $G(\mathbf{R})$, and $\Pi'$ denotes the ``arithmetically normalised'' twist $\Pi \otimes \|\cdot\|^{-(r_1 + r_2)/2}$.

  \begin{remark} \
   \begin{enumerate}[(i)]
    \item Note that the central character of $\Pi'$ is $\chi_{\Pi} \,\|\cdot\|^{-(r_1 + r_2)}$, mapping a uniformiser $\varpi_\ell$ of $\QQ_\ell$ to $\ell^{(r_1 + r_2)} \chi_{\Pi}(\ell)$ for almost all primes $\ell$; and $\Pi'_\infty$ has the same infinitesimal character as the algebraic representation $V(r_1, r_2; -r_1-r_2)$.

    \item The representation $\Pi'$ is always \emph{C-algebraic} in the sense of \cite{buzzardgee14}; in particular its finite part is the base-extension of an $E[G(\Af)]$-module, for some number field $E$.

    \item The representation $\Pi$ is not generated by holomorphic automorphic forms, but there exists a unique representation $\Pi^H$ in the same global packet as $\Pi$ such that $(\Pi^H)_\mathrm{f} = \Pif$ but $\Pi^H_\infty$ is a holomorphic discrete series. This representation is generated by classical Siegel modular forms of weight $(k_1, k_2) = (r_1 + 3, r_2 + 3)$ (i.e.~valued in the representation $\Sym^{k_1 - k_2} \otimes \det^j$ of $K_\infty$). These are vector-valued forms if $k_1 > k_2$.\qedhere
   \end{enumerate}
  \end{remark}

 \subsection{Hecke parameters at $p$}
  \label{sect:heckeparams}

  Let $p$ be a prime such that $\Pi$ is unramified at $p$, and write $w \coloneqq r_1 + r_2 + 3$.

  \begin{definition} \
   \begin{itemize}
    \item We define the \emph{Hecke polynomial} at $p$ to be the degree 4 polynomial $P_p(X)$ such that
    \[ L(\Pi_p', s - \tfrac{3}{2}) = L(\Pi_p, s - \tfrac{w}{2}) = P_p(p^{-s})^{-1}. \]
    \item The \emph{Hecke parameters} of $\Pi'$ at $p$ are the complex numbers $\alpha, \beta, \gamma, \delta$ such that
    \[
     P_p(X) = (1 - \alpha X) (1 - \beta X) (1 - \gamma X) (1 - \delta X),\qquad
     \alpha \delta = \beta \gamma = p^w \chi_{\Pi}(p).
    \]
   \end{itemize}
  \end{definition}

  If $E$ is any number field over which $\Pif$ is definable, then the coefficients of $P_p(X)$ lie in $\cO_E$; the Hecke parameters are algebraic integers in $\bar{E}$, and are well-defined up to the action of the Weyl group. Extending $E$ if necessary, we may assume that they lie in $\cO_E$ itself. All of the Hecke parameters have complex absolute value $p^{w/2}$ (see \cite[Theorem 1]{weissauer05}).

  \begin{note}
   Our notations here for Hecke polynomials and Hecke parameters are consistent with the notations of \cite{LSZ17} (see Theorem 10.1.3 of \emph{op.cit.} in particular). It is also consistent with \S 10 of \cite{LPSZ1}, where the main theorems of that paper are given. Note, however, that the Hecke parameters here are \emph{not} the same as the $(\alpha, \beta, \gamma, \delta)$ in \cite{LPSZ1} Proposition 3.2, which are the Hecke parameters of a different twist of $\Pi_p$. We apologise to readers of \cite{LPSZ1} for shifting normalisations in the middle of the paper.
  \end{note}

  We shall fix an embedding $E \into L \subset \QQbar_p$, where $L$ is a finite extension of $\Qp$, and let $v_p$ be the valuation on $L$ such that $v_p(p) = 1$. If we order $(\alpha, \beta, \gamma, \delta)$ in such a way that $v_p(\alpha) \le \dots \le v_p(\delta)$ (which is always possible using the action of the Weyl group), then we have the valuation estimates
  \begin{equation}
   \label{eq:valuations}
   v_p(\alpha) \ge 0,\qquad v_p(\alpha \beta) \ge r_2 + 1.
  \end{equation}

  \begin{remark}\label{rem:Dcrisevals}
   These inequalities correspond to the fact that the Newton polygon of the $p$-adic Galois representation associated to $\Pi$ lies on or above the Hodge polygon; see \cref{prop:Dcrisevals} below.
  \end{remark}

  \begin{definition}
   We say $\Pi$ is \emph{Siegel ordinary} at $p$ if $v_p(\alpha) = 0$, and \emph{Klingen ordinary} at $p$ if $v_p(\alpha\beta) = r_2 + 1$ (and \emph{Borel ordinary} if it is both Siegel and Klingen ordinary).
  \end{definition}

  \begin{lemma}
   \label{lem:trivzero}
   If $\Pi$ is Klingen-ordinary at $p$, then none of $(\alpha, \beta, \gamma, \delta)$ has the form $p^n \zeta$ with $n \in \ZZ$ and $\zeta$ a root of unity. (In other words, Assumption 11.1.1 of \cite{LSZ17} is satisfied.)
  \end{lemma}

  \begin{proof}
   Since all of the Hecke parameters are Weil numbers of weight $w$, it follows that if one of the parameters has this form, then $w$ must be even and $n = w/2$. In particular, this parameter has $p$-adic valuation $w/2$. However, if $\Pi$ is Klingen-ordinary then $\alpha, \beta$ have valuations at most  $r_2 + 1 \le \tfrac{w-1}{2}$, and $\gamma, \delta$ have valuations at least $r_1 + 2 \ge \tfrac{w+1}{2}$, so none can have valuation $w/2$.
  \end{proof}

 \subsection{Shimura varieties}\label{ss:Shimvarandcoeff}

  \begin{definition}
   For $U \subset G(\Af)$ a sufficiently small level, and $K$ a field of characteristic 0, let $Y_{G}(U)_{K}$ denote the base-extension to $K$ of the canonical $\QQ$-model of the level $U$ Shimura variety for $G$. We denote by $Y_{G, K}$ the pro-variety $\varprojlim_U Y_G(U)_{K}$.
  \end{definition}

  \begin{definition}
   For each algebraic representation $V$ of $G$, let $\cV$ denote the
   $G(\Af)$-equivariant relative Chow motive over $Y_{G, \QQ}$ associated to $V$ via Ancona's functor, as in \cite[\S 6.2]{LSZ17}.
  \end{definition}

  \begin{remark}
   Our conventions are such that the 4-dimensional defining representation $V(1, 0; 1)$ of $G$ corresponds to the relative motive $h^1(A)$, where $A$ is the universal abelian surface over $Y_{G, \QQ}$; and the 1-dimensional symplectic multiplier representation $V(0, 0; 2)$ maps to $\QQ(-1)[-1]$, where the square brackets $[-1]$ signify twisting the $G(\Af)$-action by the character $\|\cdot\|^{-1}$.
  \end{remark}

  Any relative Chow motive over $Y_G(U)_{\QQ}$ gives rise to an object of Voevodsky's triangulated category of geometrical motives over $\QQ$ (via pushforward along the structure map $Y_G(U)_{\QQ} \to \Spec \QQ$). Hence we can make sense of motivic cohomology $H^*_{\mot}(Y_G(U)_{\QQ}, \cV)$. We use the same symbol $\cV$ for the $p$-adic \'etale realisation of this motive, which is a locally constant \'etale sheaf of $\Qp$-vector spaces on $Y_G(U)_{\QQ}$ (with a natural extension to the canonical integral model $Y_G(U)_{\ZZ[1/N]}$ if $U$ is unramified outside $N$).

  \begin{remark}[``Liebermann's trick'']
   \label{rem:Liebermann}
   Explicitly, suppose that $V$ is a direct factor of $W^{\otimes n}(m)$, where $W$ is the defining representation of $G$. Then $H^i_{\mot}(Y_G(U)_{\Qp}, \cV)$ is a direct summand of $H^{i+n}(A^n,\QQ(m))$, where $A$ is the universal abelian scheme over $Y_{\sG}$. We have
   \[ H^\star_{\mot}(Y_\sG, \cV_{\mot})=e_{\cV} \cdot H_{\mot}^{\star+n}(A^n,\QQ(m))\]
   for some projector $e_{\cV}$.
  \end{remark}

 \subsection{Galois representations and Euler system classes}
  \label{ss:modularparam}

  Taking $V$ to be the representation $V_G(r_1, r_2; r_1 + r_2)$ (with weights parametrised as in ``Conventions'' above), the $\Pif'$-isotypical part of $H^3_{\et,c}(Y_G(U)_{\QQbar}, \cV_G) \otimes_{\Qp} L$ is isomorphic to the sum of $\dim\left( \Pif^U\right)$ copies of a 4-dimensional $L$-linear Galois representation $V_{\Pi}$ (uniquely determined up to isomorphism). As in \cite[\S 3.3]{LZ-equivar}, we shall \emph{fix} a choice of representation $V_{\Pi}$ in this isomorphism class, as follows. We have assumed that $\Pi$ is globally generic, so it has a Whittaker model with respect to the character of $N(\AA)$ given by
  \[ \psi_N\left(\begin{smatrix} 1 & x \\ & 1 & y \\ && 1 & -x \\ &&& \phantom{-}1 \end{smatrix}\right) = \psi(x + y),\]
  where $\psi$ denotes the additive character of $\AA / \QQ$ of conductor 1 which restricts to $x \mapsto e^{-2\pi i x}$ on $\mathbf{R}$.We denote this space by $\mathcal{W}(\Pif')$, and $\mathcal{W}(\Pif')_E$ the subspace of Whittaker functions which are \emph{defined over $E$} in the sense of \cite[Definition 10.2]{LPSZ1}. This gives a canonical model of $\Pif'$ as an $E$-linear representation, so we can define $\cW(\Pif')_F$ for any extension $F$ of $E$ by base-extension.

  \begin{definition}
   With the above notations, we set
   \[ V_\Pi = \Hom_{L[G(\Af)]}\Big(\mathcal{W}(\Pif')_L, H^3_{\et, c}( Y_{G, \QQbar}, \cV_G)_L \Big).\]
  \end{definition}

  This is a canonically-defined 4-dimensional $L$-linear representation of $\Gal(\QQbar/\QQ)$, which is a distinguished representative of the isomorphism class of representations above. It is characterised up to semisimplification by the relation
  \[ \det\left(1 - X \rho_{\Pi, p}(\operatorname{Frob}_\ell^{-1})\right) = P_\ell(X)\]
  for primes $\ell \ne p$ at which $\Pi$ is unramified, where $\operatorname{Frob}_\ell$ is an arithmetic Frobenius at $\ell$. We expect that $V_\Pi$ is always irreducible; this is true for $p > 2w + 1$ by a theorem of Ramakrishnan \cite{ramakrishnan13}.

  \begin{definition}
   \label{def:ESclass}
   Let $q, r$ be integers with $0 \le q \le r_2$ and $0 \le r \le r_1 - r_2$, and $\uchi = (\chi_1, \chi_2)$ a a pair of Dirichlet characters with $\chi_1\chi_2 = \chi_{\Pi}$. We let
   \[ z^{[\Pi, q, r]}(\uchi) \in H^1\left(\QQ, V_{\Pi}^*(-q)\right) \]
   denote the cohomology class constructed in \cite[Theorem A]{LZ-equivar}.
  \end{definition}

  Note that we must have $(-1)^{r_1} \chi_1(-1) = (-1)^{r_2}\chi_2(-1)$, since $\chi_1\chi_2 = \chi_{\Pi}$ has sign $(-1)^{r_1 + r_2}$. As shown in \cite[Corollary 3.6.2]{LZ-equivar}, the class $z^{[\Pi, q, r]}(\uchi)$ is zero unless the following parity constraint is satisfied:
  \begin{equation}
   \label{eq:parity}
   (-1)^{q + r} = (-1)^{r_1} \chi_1(-1) = (-1)^{r_2} \chi_2(-1).
  \end{equation}
  We shall assume this holds henceforth.

  The theorem \emph{loc.cit.}~shows that the cohomology classes constructed in \cite{LSZ17}, which depend on various choices of auxiliary local data at the finite places, are in fact linear combinations of the $z^{[\Pi, q, r]}(\uchi)$, with the coefficients of the linear combinations given by explicit products of local zeta integrals.

 \subsection{Exponential maps and regulators}

  \begin{convention}
  	The representation $\Qp(1)$ of $\Gal(\QQbar_p / \Qp)$ has Hodge--Tate weight $1$, and crystalline Frobenius $\varphi$ acts on $\Dcris(\Qp(1))$ as multiplication by $1/p$.
  \end{convention}

  We recall the following properties of $V_{\Pi} |_{G_{\Qp}}$ (see \cite[\S 6.1]{LZ-equivar}):

  \begin{proposition}\label{lemma:h1g}
   \label{prop:Dcrisevals}
   The representation $V_{\Pi} |_{G_{\Qp}}$ is crystalline. The eigenvalues of $\varphi$ on $\Dcris(V_{\Pi})$ are the Hecke parameters $\{ \alpha, \beta, \gamma, \delta\}$ of \cref{sect:heckeparams}, and its Hodge--Tate weights are $\{0, -r_2-1, -r_1-2, -r_1-r_2-3\}$. Moreover, for each integer $q$ with $0 \le q \le r_2$, we have the following:
   \begin{enumerate}
    \item[(a)] The operators $1 - \varphi$ and $1 - p\varphi$ are bijective on $\Dcris(V_\Pi^*(-q))$.
    \item[(b)] The Bloch--Kato $H^1_{\mathrm{e}}$, $H^1_{\mathrm{f}}$ and $H^1_{\mathrm{g}}$ subspaces of $H^1(\Qp, V_\Pi^*(-q))$ coincide.
    \item[(c)] The Bloch--Kato exponential map
    \[   \exp: \frac{\DdR(V_\Pi^*)}{\Fil^{-q} \DdR(V_\Pi^*)} \to H^1_{\mathrm{e}}(\Qp, V_\Pi^*(-q))\]
    is an isomorphism.\qed
   \end{enumerate}
  \end{proposition}

  Since the localisation at $p$ of the class $ z^{[\Pi, q, r]}(\uchi)$ is in $H^1_{\mathrm{g}}$ (by \cite[Theorem B]{nekovarniziol16}), it is also in $H^1_{\mathrm{e}}$. Letting $\log$ denote the inverse of the Bloch--Kato exponential, we may define
  \[ \log\left(  z^{[\Pi, q, r]}(\uchi) \right) \in \frac{\DdR(V_\Pi^*)}{\Fil^{-q} \DdR(V_\Pi^*)} = \left(\Fil^1 \DdR(V_\Pi) \right)^*.\]
  Note that the target of this map is 3-dimensional (and independent of $q$ in this range).

  \begin{assumption}
   We assume henceforth that $\Pi$ is Klingen-ordinary at $p$.
  \end{assumption}

  It follows that there is a distinguished pair of Hecke parameters $(\alpha, \beta)$ of minimal valuation, and hence a distinguished 2-dimensional subspace
  \begin{equation} \label{eq:defcurlyP}
   \Dcris(V_\Pi)^{\cQ(\varphi) = 0}, \qquad \cQ(t) = (1 - \tfrac{t}{\alpha})(1 - \tfrac{t}{\beta}).
  \end{equation}

  \begin{note}\label{note:1diml}
   From weak admissibility, we see that $\Dcris(V_\Pi)^{\cQ(\varphi) = 0} \cap \Fil^1$ must have dimension exactly 1, and that it surjects onto the 1-dimensional graded piece $\Fil^1 / \Fil^{r_2 + 2}$.
  \end{note}

  \begin{definition}
   \label{def:nudR}
   Let $\nu$ be a basis of the 1-dimensional $L$-vector space $\Gr^{r_2 + 1} \DdR(V_\Pi)$, and let $\nu_{\dR}$ denote its unique lifting to $\Dcris(V_\Pi)^{\cQ(\varphi) = 0} \cap \Fil^{r_2 + 1}$.
  \end{definition}

  We can now formulate the key problem treated in this paper:
  \bigskip
  \begin{mdframed}
   \textbf{Problem}: Compute the quantity
   \begin{equation}
    \label{eq:goal0}
    \operatorname{Reg}_{\nu}^{[\Pi, q, r]}(\uchi) \coloneqq \left\langle \nu_{\dR}, \log\left( z^{[\Pi, q, r]}(\uchi) \right)\right\rangle_{\DdR(V_\Pi)} \in L.
   \end{equation}
  \end{mdframed}
  \bigskip

  \begin{note}[Compatibility with base-extension]
   \label{note:baseext}
   If we let $K$ be a finite extension of $\Qp$, then the conclusions of Lemma \ref{lemma:h1g} also apply to the restriction $V_\Pi^* |_{\Gal(\overline{K} / K)}$, so we can also consider the logarithm of the restricted class $\operatorname{res}_{\Qp}^K\left(z^{[\Pi, q, r]}\right)$ as an element of $\left(\Fil^1 \DdR(V_\Pi |_{\Gal(\overline{K} / K)}) \right)^*$. Then we obtain an element
   \[
    \left\langle \nu_{\dR}, \log\left(\operatorname{res}_{\Qp}^K z^{[\Pi, q, r]}(\uchi) \right)\right\rangle_{\DdR(V_\Pi |_{\Gal(\overline{K} / K)})} \in K \otimes_{\Qp} L.
   \]
   One checks that this is simply the image of $\operatorname{Reg}_{\nu}^{[\Pi, q, r]}(\uchi)$ via the natural map $L \into K \otimes_{\Qp} L$. Since this inclusion map is injective, it suffices to evaluate the regulator after restricting to any finite extension. This will be useful later, since we will need to work over ramified field extensions in order to find semistable models for our Shimura varieties.
  \end{note}


 \subsection{Periods and $p$-adic $L$-functions}
 \label{sect:periods}

  We assume now that the characters $\uchi$ are both unramified at $p$.

  \begin{definition}
   Let $\cL_{p, \nu}(\Pi, \uchi)$ denote the 2-variable $p$-adic $L$-function defined in \cite[Theorem 6.2.5]{LZ-equivar}.
  \end{definition}

  This is a measure on $\Zp^\times \times \Zp^\times$, supported on the open-and-closed subset parametrising pairs of characters $(\lambda_1, \lambda_2)$ of $\Zp^\times$ satisfying $\lambda_1(-1) \lambda_2(-1) = -\chi_2(-1)$. It generalises the construction of \cite[Proposition 10.4]{LPSZ1}, which is the special case $\uchi = (\chi_{\Pi}, \id)$.

  We briefly recall the interpolating property of this $p$-adic $L$-function, as it will be important for the sequel. Recall that $\nu$ was a basis of the 1-dimensional $L$-vector space $\frac{\Fil^1 \DdR(V_\Pi)}{\Fil^{r_2 + 2} \DdR(V_\Pi)}$. This space is canonically the base-extension to $L$ of an $E$-vector space, namely
  \[ \Hom_{E[G(\Af)]}\Big( \cW(\Pif'), H^2(\Pif)_E\Big),\]
  where $H^2(\Pif)_E$ denotes the unique copy of $\Pif'$ inside a coherent $H^2$ of a toroidal compactification of $Y_{G, E}$, as in \cite[\S 5.2]{LPSZ1}. If we let $\nu^{\alg}$ be a basis of this $E$-vector space, we obtain a $p$-adic period $\Omega_p(\Pi, \nu, \nu^{\alg}) \in L^\times$ by comparing $\nu^{\alg}$ with $\nu$, and and an Archimedean period $\Omega_\infty(\Pi, \nu^{\alg}) \in \CC^\times$ by comparing $\nu^{\alg}$ with the natural rational structure on Whittaker functions (the period denoted $\Omega^W(\Pi)$ in \S 10.2 of \cite{LPSZ1}).

  \begin{remark}
   The quantities $\Omega_p(\Pi, \nu, \nu^{\alg})$ and $\Omega_\infty(\Pi, \nu^{\alg})$ each depend on the choice of $\nu^{\mathrm{alg}}$, but the ratio
   \[ \Omega_p(\Pi, \nu, \nu^{\alg})^{-1} \otimes \Omega_\infty(\Pi, \nu^{\alg}) \in L \otimes_{E} \CC\] depends only on $\nu$; the dependency on $\nu^{\alg}$ cancels out.
  \end{remark}

  \begin{theorem}[{cf.~\cite[Theorem 6.2.5]{LZ-equivar}}]
   \label{thm:padicLfcn}
   For $a_i$ integers with $0 \le a_1, a_2 \le r_1 - r_2$, and $\rho_i$ Dirichlet characters of $p$-power conductors satisfying $(-1)^{a_1 + a_2} \rho_1(-1) \rho_2(-1) = -\chi_2(-1)$, we have
   \begin{multline*}
    \frac{\cL_{p, \nu}(\Pi, a_1 + \rho_1, a_2 + \rho_2)}{\Omega_p(\Pi, \nu, \nu^{\alg})} \\
    = R_p(\Pi, \rho_1, a_1) R_p(\Pi \times \chi_2^{-1}, \rho_2, a_2) \cdot
    \frac{\Lambda(\Pi \times \rho_1^{-1}, \tfrac{1-r_1+r_2}{2} + a_1)\Lambda(\Pi \times \rho_2^{-1}\chi_2^{-1}, \tfrac{1-r_1+r_2}{2} + a_2)}{\Omega_{\infty}(\Pi, \nu^{\alg})},
   \end{multline*}
   for some (and hence every) choice of $\nu^{\alg}$ as above.
  \end{theorem}

  For the definition of $R_p(\Pi, \rho, j)$ when $\rho$ is nontrivial see \cite{LPSZ1}. For $\rho = \id$ we have the relation
  $R_p(\Pi, \id, a) = \cE_p(\Pi, r_2 + 1 + a)$, where $\cE_p$ is the Euler factor defined by
  \[ \cE_p(\Pi, n) \coloneqq \left(1- \tfrac{p^n}{\alpha}\right) \left(1- \tfrac{p^n}{\beta}\right) \left(1- \tfrac{\gamma}{p^{n+1}}\right) \left(1- \tfrac{\delta}{p^{n+1}}\right).\]
  (This is nonzero for all $n \in \ZZ$ under our present assumptions, by Lemma \ref{lem:trivzero}). We can now give a precise statement of the theorem we shall prove:
  \bigskip

  \begin{mdframed}
   \begin{theorem}[{\cref{thm:main}}]
    \label{thm:mainthm}
    Assume $r_1 - r_2 \ge 3$. For any $q, r$ with $0 \le q \le r_2$, $0 \le r \le r_1 - r_2$, and $q+r = r_2 \bmod 2$, we have
    \[
     \operatorname{Reg}^{[\Pi, q, r]}_{\nu}(\uchi) = \frac{(-2)^q (-1)^{r_2 - q+1} (r_2 - q)!}{\cE_p(\Pi, q) \cE_p(\Pi, r_2 + 1 + r)} \cdot \cL_{p, \nu}(\Pi,\uchi, -r_2-1 + q, r).
    \]
   \end{theorem}
  \end{mdframed}

  \begin{note} As shown in \S 10 of \cite{LPSZ1}, the hypothesis $r_1 - r_2 \ge 3$ implies that $\cL_{p, \nu}(\Pi, \bfj_1, \bfj_2)$ factors as a product of a function of $j_1$ and a function of $j_2$. However, our proof of the theorem will not directly ``see'' this finer decomposition.
  \end{note}

 \subsection{Explicit formulation}
  \label{ssec:testdataatp}

  We now give an alternative, more concrete reformulation of \cref{thm:main} which is more convenient for the proof. First we define a suitable map of Shimura varieties.

  \begin{notation}
   Let $u_{\Kl} \in G(\Zp)$ denote any element with first column $\begin{smatrix} 1 \\ 1 \\ 0 \\ 0 \end{smatrix}$; and let $\Kl(p)$ denote the Klingen parahoric in $G(\Zp)$. Define $K_{p, \Delta} \subset H(\Zp)$ by
   \[
     K_{p,\Delta}=\left\{ h\in H(\Zp):\, h=\left(\begin{pmatrix} x & \star\\  & y\end{pmatrix},\begin{pmatrix} x & \star\\  & y\end{pmatrix}\right)\pmod{p}\quad\text{for some $x, y \in \Zp^\times$}\right\}.
   \]
  \end{notation}

  \begin{definition}
   For $U^p \subset G(\Af^p)$ an open compact subgroup, write $Y_{G, \Kl,\QQ}$ for the $G$-Shimura variety of level $U^p \Kl(p)$, and $Y_{H, \Delta,\QQ}$ for the $H$-Shimura variety of level $V^p K_{p, \Delta}$, where $V^p = U^p \cap H$.

  \end{definition}

  \begin{remark}
   We will define integral models of these Shimura varieties in \cref{sect:intmodels}.
  \end{remark}

  We have $u_{\Kl}^{-1} \cdot K_{p, \Delta}\cdot u_{\Kl} \subset \Kl(p)$, so as in \cite[\S 4.1]{LPSZ1}, $u_{\Kl}$ gives a finite morphism of Shimura varieties
  \begin{equation}
   \label{eq:iotaDelta}
   \iota_{\Delta}: Y_{H, \Delta,\QQ} \to Y_{G, \Kl,\QQ}.
  \end{equation}

  \subsubsection*{Coefficient sheaves} For $(q, r)$ as in \cref{thm:main}, let us define
  \begin{equation}
   \label{eq:qrtetc}
   (t_1, t_2) \coloneqq (r_1 - q - r, r_2 - q + r) \qquad \text{(so $t_1, t_2 \ge 0$).}
  \end{equation}
  Let $V_H = V_H(t_1, t_2; t_1 + t_2)$ and $V_G = V_G(r_1, r_2; r_1 + r_2)$. Then there is a non-zero map of $H$-representations $V_G \to V_H \otimes \det{}^q$, or dually $V_H^\vee \to (V_G^\vee \otimes \mu^q)|_{H}$. (We describe an explicit choice of a map in this Hom-space in \cite[eq. (5)]{LSZ17}). This gives a pushforward map on motivic cohomology, or on \'etale cohomology\footnote{More precisely, we should either take continuous \'etale cohomology in the sense of Jannsen \cite{jannsen88}, or \'etale cohomology of $\ZZ[1/\Sigma]$-models for a sufficiently large finite set of primes $\Sigma$.}
  \[
   \iota_{\Delta, \star}^{[t_1, t_2]}:
   H^2_{\et}\left(Y_{H, \Delta, \QQ}, \cV_H^\vee(2)\right) \to H^4_{\et}\left(Y_{G, \Kl, \QQ}, \cV_G^\vee(3-q)\right).
  \]

  The right-hand side is related to Galois cohomology of \'etale cohomology over $\QQbar$ by the Hochs\-child--Serre spectral sequence. Since $\Pif'$ does not contribute to \'etale cohomology outside degree 3, the natural projection map onto the $\Pif'^\vee$-eigenspace lifts to an ``Abel--Jacobi'' map
  \begin{equation}
   \label{eq:def_prPi}
   \AJ^{[\Pi, q]} : H^4_{\et}\left(Y_{G,\QQ}, \cV_G^\vee(3-q)\right) \to H^1\left(\QQ, H^3_{\et}(Y_{G,\QQbar}, \cV_G^\vee(3-q))[\Pif'^\vee]\right),
  \end{equation}
  characterised as the unique Hecke-equivariant map agreeing with the $\Pif'^\vee$-projection of the \'etale Abel--Jacobi map on homologically trivial classes.


  \subsubsection*{Schwartz functions}

   \begin{notation}
    We let $\cS_{(0)}(\Af^2 \times \Af^2)$ denote the space of $E$-valued Schwartz functions on $\Af^2 \times \Af^2$ satisfying the following vanishing property: if $t_1 = 0$, then $\Phi( (0, 0) \times -)$ vanishes identically, and if $t_2 = 0$, then $\Phi(- \times (0, 0))$ vanishes identically.
   \end{notation}

   Beilinson's Eisenstein symbol (cf.~\cite[\S 7.2]{LSZ17}) gives a map
   \[
    \Eis_{\et}^{[t_1, t_2]} : \cS_{(0)}(\Af^2 \times \Af^2) \to  H^2_{\et}\left(Y_{H, \Delta, \QQ}, \cV_H^\vee(2)\right).
   \]

   \begin{notation}
    Let $\Phi'_{\crit} = \ch(\Zp \times \Zp^\times) \in \cS(\Qp^2)$, and let $\Phi_{\crit}$ denote unique the Schwartz function whose Fourier transform \emph{in the second variable only} is $\Phi'_{\crit}$. Write $\uPhi_{\Kl} = \Phi_{\crit} \boxtimes \Phi_{\crit} \in \cS(\Qp^2 \times \Qp^2)$.
   \end{notation}

   Thus, given any $\uPhi^p$ invariant under $V^p = H \cap U^p$, we can make sense of the class
   \[
    \left(\log \circ \AJ^{[\Pi, q]} \circ \iota^{[t_1,t_2]}_{\Delta, \star}\right)(\Eis^{[t_1, t_2]}_{\et, \uPhi^p\uPhi_{\Kl}}) \in \left(\Fil^{1 + q} H^3_{\dR, c}(Y_{G, \Kl, \QQ}, \cV_G)_L[\Pif']\right)^\vee.
   \]
   The group $\Fil^{1 + q} H^3_{\dR, c}(Y_{G, \Kl, \QQ}, \cV_G)_L[\Pif'] = \Fil^1 H^3_{\dR, c}(Y_{G, \Kl, \QQ}, \cV_G)_L[\Pif']$ is independent of $q$ in the given range. Moreover, given any $\eta \in H^2(\Pif)^{U^p \Kl(p)}$, we can construct a canonical element $\eta_{\dR} \in \Fil^{1 + q} H^3_{\dR, c}(Y_{G, \Kl, \QQ}, \cV_G)_L[\Pif']$ as the unique lifting of $\eta$ to the kernel of $\cQ(\varphi)$, as in \cref{note:1diml}.

   We shall restrict to classes $\eta$ lying in the ordinary eigenspace for $U'_{2, \Kl}$ (i.e.~the $U_{2, \Kl}' = \frac{\alpha\beta}{p^{r_2 + 1}}$ eigenspace). By a zeta-integral computation carried out in \cite{LZ-equivar}, the values
   \[ \left\langle\left(\log \circ \AJ^{[\Pi, q]} \circ \iota^{[t_1,t_2]}_{\Delta, \star}\right)(\Eis^{[t_1, t_2]}_{\et, \uPhi}), \eta_{\dR}\right\rangle_{\dR, Y_{G, \Kl, L}} \]
   for $\eta$ lying in this eigenspace, and $\uPhi$ of the form $\uPhi^p\uPhi_{\Kl}$ for $\uPhi^p \in \cS_{(0)}((\Af^p)^2 \times (\Af^p)^2)$, uniquely determine $\langle \nu_{\dR}, \log z^{[\Pi, q, r]}(\uchi)\rangle$ (for all pairs of characters $\uchi$ satisfying our conditions).

  \subsubsection*{Coherent side} We now derive a corresponding formula for the right-hand side of \cref{thm:main}. As we shall recall in \cref{sect:intmodels} below, we can consider a toroidal compactification $X_{G, \Kl, \QQ}$ of $Y_{G, \Kl \QQ}$, for a suitable choice of boundary data; this has a canonical $\Zp$-model $X_{G,\Kl}$, and we let $\mathfrak{X}_{G,\Kl}$ denote its $p$-adic completion, as a formal scheme over $\Zp$.

  Given $\uPhi^p$, the construction of \cite{LPSZ1} \S 7.4 gives a 2-parameter $p$-adic family of Eisenstein series on $H$, which we denote simply by $\cE(\uPhi^p)$, interpolating classical nearly-holomorphic Eisenstein series which are $p$-depleted (i.e.~lie in the kernel of $U_p$). Then the $p$-adic interpolation theory of \emph{op.cit.} allows us to make sense of $\iota_{\Delta, \star}\left(\cE(\uPhi^p)\right)$ as a class in $H^1$ of the multiplicative locus\footnote{This was denoted $X_{G, \Kl}^{\ge 1}$ in \emph{op.cit.}, but this notation is somewhat misleading since this space is only one component of the $p$-rank $\ge 1$ locus at Klingen level, so we shall use the above notations here. We shall introduce these spaces in detail in \S \ref{sect:intmodels} below.} $\mathfrak{X}_{G, \Kl}^m \subset \mathfrak{X}_{G, \Kl}$. This class takes values in a sheaf of $\Lambda_L(\Zp^\times \times \Zp^\times)$-modules, and hence allows us to define a measure
  \[
   \left\langle \iota_{\Delta, \star}\left(\cE(\uPhi^p)\right), \eta \right\rangle_{\mathfrak{X}_{G, \Kl}^{m}} \in \Lambda_L(\Zp^\times \times \Zp^\times),
  \]
  for any $\eta \in H^2(\Pif)^{U^p\Kl(p)}[U_{2, \Kl}' = \frac{\alpha\beta}{p^{r_2 + 1}}]$. The $p$-adic $L$-function $\cL_{p, \nu}(\Pi, \uchi)$ is (roughly) the ``greatest common divisor'' of these measures as the test data away from $p$ vary (with $\uPhi$ assumed to lie in the $\uchi^{-1}$-eigenspace for the centre). Since the pushforward map is compatible with specialisation in the coefficient ring, for a pair of $L$-valued characters $(\lambda_1, \lambda_2)$ of $\Zp^\times$ (giving a homomorphism $\Lambda_L(\Zp^\times \times \Zp^\times) \to L$), we have
  \[ \left\langle \iota_{\Delta, \star}\left(\cE(\uPhi^p)\right), \eta \middle\rangle_{\mathfrak{X}_{G, \Kl}^{m}} \right|_{(\lambda_1, \lambda_2)} = \left\langle \iota_{\Delta, \star}\left(\cE(\uPhi^p)\middle|_{(\lambda_1, \lambda_2)}\right), \eta \right\rangle_{\mathfrak{X}_{G, \Kl}^{m}}.\]

  \begin{proposition}[{Proposition 6.4.4 of \cite{LZ-equivar}}]
   \label{prop:weprove}
   \cref{thm:mainthm} is equivalent to the following assertion: for all prime-to-$p$ levels $U^p$, all $\uPhi^p$ stable under $U^p \cap H$, and all $\eta \in H^2(\Pif)^{U^p \Kl(p)}[U_{2, \Kl}' = \frac{\alpha\beta}{p^{r_2 + 1}}]$, we have
   \begin{multline}
    \label{eq:goal2}
    \left\langle \left(\log \circ \AJ^{[\Pi, q]} \mathop{\circ} \iota^{[t_1,t_2]}_{\Delta, \star}\right)(\Eis^{[t_1, t_2]}_{\et, \uPhi^p \uPhi_{\Kl}}),\eta_{\dR}\right\rangle_{\dR, Y_{G, \Kl, \Qp}}
    \\ = \frac{(-2)^q (-1)^{r_2 - q+1}(r_2 - q)!}
    {\left(1 - \frac{\gamma}{p^{1+q}}\right)\left(1 - \frac{\delta}{p^{1+q}}\right)}
    \cdot \Big\langle \iota_{\Delta, \star}\left(\cE(\uPhi^p)\big|_{(-1-r_2 + q, r)}\right), \eta \Big\rangle_{\mathfrak{X}_{G, \Kl}^{m}}.
   \end{multline}
  \end{proposition}

  It is this formula we shall actually prove.


\section{Finite-polynomial cohomology and Abel--Jacobi maps}
 \label{sect:abeljacobi}

 We briefly recall some geometric formalism from \cite{nekovarniziol16} and \cite{besserloefflerzerbes16}, which we shall use to give formulae for the Abel--Jacobi map of \'etale cohomology. In this section we shall only consider varieties over finite extensions $K/\Qp$; integral models (over $\cO_K$) will enter the picture later, when we start to make computations.

 \subsection{P-adic Hodge theory}

  We recall some constructions from $p$-adic Hodge theory and Galois cohomology; see \cite[\S 2D]{nekovarniziol16} and \cite[\S 1]{besserloefflerzerbes16} for further details. In this section $G_K$, for $K/\Qp$ finite, denotes $\Gal(\QQbar_p / K)$.

  \subsubsection*{Filtered modules}

   Let $\Qp^{\mathrm{nr}}$ denote the maximal unramified extension of $\Qp$; and let $K$ be an arbitrary finite extension of $\Qp$.

   \begin{definition}
    A \emph{filtered $(\varphi, N, G_K)$-module} is a finite-dimensional $\Qp^{\mathrm{nr}}$-vector space $D$ equipped with the following structures:
    \begin{itemize}
     \item an $\Qp^{\mathrm{nr}}$-semilinear Frobenius $\varphi$;
     \item an $\Qp^{\mathrm{nr}}$-linear monodromy operator $N$ satisfying $N\varphi = p \varphi N$;
     \item an $\Qp^{\mathrm{nr}}$-semilinear action of $G_K$ commuting with $\varphi$ and $N$, such that every $v \in D$ is fixed by some open subgroup;
     \item a decreasing $K$-linear filtration $\Fil^\bullet$ on
     \[
      D_{\dR} \coloneqq \left(D \otimes_{\Qp^{\mathrm{nr}}} \QQbar_p\right)^{G_K}.
     \]
    \end{itemize}
    We write $D_{\st} \coloneqq D^{G_K}$ and $D_{\mathrm{cris}} \coloneqq D^{(G_K, N = 0)}$, both of which are vector spaces over $K_0 = K \cap \Qp^{\mathrm{nr}}$.
   \end{definition}

   Fontaine's functor $\Dpst$ gives an equivalence of categories between potentially semistable $p$-adic representations of $G_K$ and the subcategory of \emph{weakly admissible} filtered $(\varphi, N, G_K)$-modules. If $D = \Dpst(V)$, then we have $D_{\st} = \mathbf{D}_{\st}(V)$, $D_{\mathrm{cris}} = \Dcris(V)$, and $D_{\dR} = \DdR(V)$ (hence the notation).

   \begin{notation}
    For $n \in \ZZ$, let $\Qp^{\mathrm{nr}}(n)$ denote the filtered $(\varphi, N, G_K)$-module whose underlying vector space is $\Qp^{\mathrm{nr}}$, with $N=0$ and the $G_K$-action being the obvious one, but taking $\varphi = p^{-n} \sigma$ where $\sigma$ is the native arithmetic Frobenius of $\Qp^{\mathrm{nr}}$, and the filtration concentrated in degree $-n$.
   \end{notation}

   Clearly we have $\Qp^{\mathrm{nr}}(n) = \Dpst(\Qp(n))$, by identifying $1 \in \Qp^{\mathrm{nr}}$ with the basis vector $t^{-n} \otimes e_n \in \mathbf{B}_{\mathrm{cris}} \otimes \Qp(n)$.

  \subsubsection*{The semistable $P$-complex}

   \begin{definition}
    Let $P \in \Qp[t]$ be a polynomial with constant term 1, and $D$ a filtered $(\varphi, N, G_K)$-module. Define $H^i_{\mathrm{st}, P}(K, D)$ to be the $i$-th cohomology group of the complex
    \[
     C_{\st, P}(D) \coloneqq \left[ D_{\st} \longrightarrow  D_{\st}\oplus D_{\st}\oplus \tfrac{D_{\dR}}{\Fil^0} \longrightarrow D_{\st}\right],
    \]
    where the maps are given by
    \[
     x \mapsto \left( P(\varphi)x,\, Nx,\, x \bmod \Fil^0 \right)\qquad \text{and}\qquad (u,v,w)\mapsto Nu-P(p\varphi)v.
    \]
    If $P(t) = 1 - t$, then we omit it and write simply $H^i_{\st}(K, D)$ etc.
   \end{definition}

   \begin{note}
    More generally, it will sometimes be convenient to extend the definitions to the case when $P$ is a polynomial in $\mathcal{R}[t]$, where $\mathcal{R}$ is a commutative $\Qp$-subalgebra of the endomorphism algebra of $D$. (We shall apply this with $\mathcal{R}$ a Hecke algebra.)
   \end{note}

   If $P \mid Q$ then we have a natural map of complexes $C_{\st, P}(D) \to C_{\st, Q}(D)$ which is the identity in degree 0. There are also products
   \[ C^i_{\st, P}(K, D) \otimes C^j_{\st, Q}(K, E)
   \to C^{i+j}_{\st, P\star Q}(K, D \otimes E),\]
   well-defined up to homotopy, where $P \star Q$ is the convolution product (the polynomial whose roots are the pairwise products of those of $P$ and $Q$). We also have base-extension maps
   \[ C_{\st, P}(D) \to C_{\st, P}\left(D|_{G_{K'}}\right)\]
   for $K'$ a finite extension of $K$.

  \subsubsection*{Galois cohomology}

   If $V$ is a potentially semistable $G_K$-representation, then $C_{\st}(\Dpst(V))$ is the $G_K$-invariants of a complex of $G_K$-modules that is quasi-isomorphic to $V$. This gives rise to boundary maps
   \begin{equation}
    \label{eq:genexpmap}
    H^i_{\st}(K, \Dpst(V)) \to H^i(K, V),
   \end{equation}
   which are isomorphisms for $i = 0$ and injective for $i = 1$.

   \begin{definition}
    The \emph{semistable Bloch--Kato exponential} is the map $\exp_{\st, V}: H^1_{\st}(K, \Dpst(V)) \into H^1(K, V)$ given by \eqref{eq:genexpmap} for $i = 1$. Its image is the Bloch--Kato subspace $H^1_{\mathrm{g}}(K, V)$.
   \end{definition}

   This terminology is justified by the fact that the composition
   \[
    \frac{\DdR(V)}{\Fil^0 \DdR(V)} \to H^1_{\st}(K, \Dpst(V)) \xrightarrow[\cong]{\exp_{\st, V}} H^1_{\st}(K, V)
   \]
   is the usual Bloch--Kato exponential map $\exp_V$, with image $H^1_{\mathrm{e}}(K, V) \subseteq H^1_{\mathrm{g}}(K, V)$.

   \begin{notation}
    We say a filtered $(\varphi, N, G_K)$-module $D$ is \emph{convenient} if it is crystalline (i.e.~$D_{\dR} = D_{\mathrm{cris}} \otimes_{K_0} K$) and $1-\varphi$ and $1 - p\varphi$ are bijective on $D_{\mathrm{cris}}$. We say a crystalline $G_K$-representation $V$ is convenient if $D = \Dpst(V)$ is convenient.
   \end{notation}

   \begin{note}
    If $D$ is a filtered $(\varphi, N, G_K)$-module, then $D$ is convenient if and only if $D^*(1)$ is.
   \end{note}

   If $D$ is convenient, then $H^i_{\st}(K, D) = 0$ for $i \ne 1$, and the natural map $D_{\dR} / \Fil^0 D_{\dR} \to H^1_{\mathrm{st}}(K, D)$ is an isomorphism. In particular, for a convenient Galois representation $V$ we have $H^1_{\mathrm{e}}(K, V) = H^1_{\st}(K, V)$ and $\exp_{\st, V}$ is identified with $\exp_V$.

  \subsubsection*{Traces and duality}

  \begin{definition}
   Say $P \in 1 + X \Qp[X]$ has \emph{no bad roots} if $P(\zeta) \ne 0$ and $P(\zeta/p) \ne 0$ for every $\zeta \in \mu_{[K_0 : \Qp]}$.
  \end{definition}

   \begin{lemma}
    \label{lemma:abstracttrace}
    Suppose $P$ is a polynomial with no bad roots. Then $P(\sigma)$ and $P(\sigma / p)$ are bijective as $\Qp$-linear endomorphisms of $K_0$, where $\sigma$ is the arithmetic Frobenius of $K_0$; and there is a canonical ``trace'' isomorphism
    \[ \tr_{\st, P}: H^1_{\st, P}(K, \Qp^{\mathrm{nr}}(1)) \cong K \]
    given by mapping $(x, y, z) \in Z^1(C_{\st, P}(\Qp^{\mathrm{nr}}(1)))$ to $z - P(\sigma/p)^{-1} x$.

    If $P \mid Q$ are two polynomials, both having no bad roots, then the trace maps for $P$ and $Q$ are compatible with the change-of-polynomial maps. These pairings are compatible with base-change for extensions $K' / K$.
   \end{lemma}

   \begin{proof}
    Immediate from the definitions.
   \end{proof}

   \begin{corollary}
    \label{cor:convenientpair}
    Suppose $D$ is convenient. Then, for any $P$ having no bad roots, the pairing given by
    \[
     H^0_{\st, P}(K, D)\times H^1_{\st}(K, D^*(1)) \longrightarrow H^1_{\st, P}(K, \Qp^{\mathrm{nr}}(1)) \xrightarrow{\tr_{\st, P}} K
    \]
    is the restriction to $H^0_{\st, P}(K, D) = \left(D_{\mathrm{cris}}^{P(\varphi) = 0} \cap \Fil^0 D_{\dR}\right)$ of the natural duality pairing $\left(\Fil^0 D_{\dR}\right) \times \left(\frac{D_{\dR}^*(1)}{\Fil^0 D_{\dR}^*(1)}\right) \to K$.\qed
   \end{corollary}

   \begin{remark}
    The assumption that $D$ be convenient implies that any class in $D_{\mathrm{cris}}$ is killed by $P(\varphi)$ for some $\varphi$ having no bad roots.
   \end{remark}

   The pairing defining the regulator in \cref{eq:goal0} is of this type, with $D = \Dpst\left( V_{\Pi}(1+q)\right)$ and $P$ the polynomial $(1 - p^{1+q} t / \alpha)(1 - p^{1+q} t / \beta)$. The ``convenient'' condition on this $D$ is satisfied by \cref{prop:Dcrisevals} and \cref{lemma:h1g}(a). This will allow us to use the formalism of semistable $P$-complexes to evaluate the regulator.


 \subsection{\Nek--\Niz cohomology}

  Let $X$ be any $K$-variety, and let $n \in \ZZ$. Then \Nek--\Niz \cite{nekovarniziol16} define $R\Gamma_{\NNsyn}(X, n)$ and $R\Gamma_{\NNsyn,c}(X, n)$. This cohomology theory is a Bloch--Ogus theory (Appendix B in \emph{op. cit.}), so it has all of the good functorial properties one expects, such as cup-products, pullbacks, pushforward maps, etc. More generally, we can define groups  $R\Gamma_{\NNfp}(X, n,P)$ and $R\Gamma_{\NNfp,c}(X, n,P)$ for any polynomial $P$ as above, with the case $P(t)=1-t$ recovering the theory of \cite{nekovarniziol16}; see \cite{besserloefflerzerbes16} for this generalisation.

  By construction, these cohomology theories satisfy the following \emph{descent spectral sequence}. Let us write $D^i(X, n) = \Dpst(H^i_{\et}(X_{\overline{K}}, \Qp(n)))$, viewed as a filtered $(\varphi, N, G_K)$-module, and similarly $D^i_c$ for compactly-supported cohomology.

  \begin{proposition}
   There exists a spectral sequence
   \begin{equation}
    \label{eq:NNsynSS}
    {}^{\mathrm{NN}}E_2^{ij}=H^i_{\st, P}\Big(K, D^j(X, n)\Big) \Rightarrow H^{i+j}_{\NNfp}(X, n, P),
   \end{equation}
   compatible with cup-products and change-of-polynomial maps (and similarly for the compactly-sup\-port\-ed variant).
  \end{proposition}

  If $X$ is smooth of pure dimension $d$, then the \'etale cohomology of $X_{\QQbar_p}$ vanishes in degrees $> 2d$, and there is a $(\varphi, N, G_K)$-equivariant trace map $D^{2d}_c(X, d+1) \to \Qp^{\mathrm{nr}}(1)$; so the edge map of this spectral sequence, combined with \cref{lemma:abstracttrace}, gives a canonical trace map
  \begin{equation}\label{eq:NNtrace}
   \tr_{\NNfp, X, P}:  H^{2d+1}_{\NNfp, c}(X, d+1, P) \to H^1_{\st, P}(K, \Qp^{\mathrm{nr}}(1))\cong  K, \
  \end{equation}
  and hence a pairing
  \[
   \langle\ , \ \rangle_{\NNfp, X, P}: H^i_{\NNsyn}(X, r) \times H^{2d+1-i}_{\NNfp, c}(X, d+1-r; P) \to K,
  \]
  for any polynomial $P$ with no bad roots. These pairings are compatible with the change-of-$P$ maps (and so we shall generally omit the subscript $P$). They are also compatible with base-extension in $K$.

  \begin{theorem}
   \label{thm:NNetcomp}
   For all $i \ge 0$, there is a natural map
   \[ \mathrm{comp}: H^i_{\NNsyn}(X,n)\to H^i_{\et}(X,\Qp(n))\]
   which is functorial in $X$ and fits into the commutative diagram
   \[
    \begin{tikzcd}[row sep=large]
     & H^i_{\mot}(X,n) \ar[rd, "r_{\et}"] \ar[ld, "r_{\syn}" above] & \\
     H^i_{\NNsyn}(X,n) \ar[rr, "\operatorname{comp}"] & & H^i_{\et}(X,\Qp(n));
    \end{tikzcd}
   \]
   and there is a morphism of spectral sequences ${}^{\mathrm{NN}}E_r^{ij}\rightarrow {}^{\et}E_r^{ij}$, compatible with $\mathrm{comp}$ on the abutment, which is given on the $E_2$ page by the maps \eqref{eq:genexpmap}. Here, ${}^{\et}E_r^{ij}$ denotes the Hochschild--Serre spectral sequence
   \[ H^i(K, H^j_{\et}(X_{\QQbar_p},\Qp(n)))\Rightarrow H^{i+j}_{\et}(X,\Qp(n)).\]
  \end{theorem}

  \begin{proof}
   This is Theorem A of \cite{nekovarniziol16}.
  \end{proof}


 \subsection{Formalism of Abel--Jacobi maps}

  Let $X$ be a smooth equidimensional $K$-variety of dimension $d$, as before. Recall the following definition:

  \begin{definition}
   A class in $H^i_{\mot}(X, n)$ is said to be \emph{homologically trivial} if it is in the kernel of the edge map
   \[ H^i_{\mot}(X, n) \to H^0(K, H^i_{\et}(X_{\QQbar_p}, \Qp(n))) \]
   induced by the Hochschild--Serre spectral sequence. We denote this kernel by $H^i_{\mot}(X, n)_0$.
  \end{definition}

  Since $G_K$ has cohomological dimension 2, the spectral sequence gives a natural map, the \emph{\'etale Abel--Jacobi map},
  \[
   \AJ_{\et}: H^i_{\mot}(X, n)_0 \to H^1\left(K, H^{i-1}_{\et}(X_{\QQbar_p}, \Qp(n))\right) .
  \]

  \begin{note}
   From Theorem \ref{thm:NNetcomp}, we have
   \[
    \AJ_{\et} = \exp_{\st} \circ \AJ_{\syn},
   \]
   where
   \[ \AJ_{\syn}: H^i_{\mot}(X, n)_0 \to H^1_{\st}\left(K, D^{i-1}(X, n)\right)\]
   is the map given by the spectral sequence ${}^{\mathrm{NN}}E_r^{ij}$. In particular, the map $\AJ_{\et}$ takes values in the subspace $H^1_{\mathrm{g}}\left(K, H^{i-1}_{\et}(X_{\QQbar_p}, \Qp(n))\right)$ (c.f.  \cite[Theorem B]{nekovarniziol16}).
  \end{note}

  We shall use \Nek--\Niz cohomology to describe the values of the map $\AJ_{\syn}$ after projecting to a convenient quotient. More precisely, let $W$ be a $G_K$-invariant subspace of 
  \[\left[H^{i-1}_{\et}(X_{\QQbar_p}, \Qp(n))\right]^*(1) = H^{2d+1-i}_{\et, c}(X_{\QQbar_p}, \Qp(d+1-n)),\]
  and suppose $W$ is convenient. Then $W^*(1)$ is naturally a quotient of $H^{i-1}_{\et}(X_{\QQbar_p}, \Qp(n))$, so we have a projection map $\pr_{W^*(1)}$ onto this quotient. Moreover, the natural map $\frac{\DdR(W^*(1))}{\Fil^0} \to H^1_{\st}(K, W^*(1))$ is an isomorphism; we write $\log_{W^*(1)}$ for its inverse.

  \begin{notation}
   Write $\AJ_{W^*(1)}$ for the morphism
   \[
    \AJ_{W^*(1)} \coloneqq \log_{W^*(1)} \circ \operatorname{pr}_{W^*(1)} \circ \AJ_{\et}:\quad H^i_{\mot}(X, n)_0 \to \frac{\DdR({W^*(1)})}{\Fil^0 \DdR({W^*(1)})} = \left[ \Fil^0 \DdR(W)\right]^*.
   \]
  \end{notation}

  The canonical pairing
  \[ \langle\ , \ \rangle_{\dR, W}: \DdR(W) \times \DdR(W^*(1)) \to \DdR(\Qp(1)) \cong K \]
  identifies the target of $\AJ_{W^*(1)}$ with the dual of $\Fil^0 \DdR(W)$. If we have some $\eta_{\dR} \in \Fil^0 \DdR(W) \cap \Dcris(W)$, and $P$ is a polynomial with no bad roots such that $P(\varphi)(\eta_{\dR}) = 0$ (which exists, since $W$ is convenient), then we can interpret the above pairing as a duality pairing $H^0_{\st, P} \times H^1_{\st} \to K$ via \cref{cor:convenientpair}.

  The spectral sequence \eqref{eq:NNsynSS} gives a boundary map
  \[ H^{2d+1-i}_{\NNfp, c}(X, d+1-n, P) \longrightarrow H^0_{\st, P}(K, D^{2d+1-i}_c(X, d+1-n)) \supseteq H^0_{\st, P}(K, W), \]
  So it makes sense to ask if $\eta_{\dR} \in H^0_{\st, P}(K, W)$ lifts to $H^{2d+1-i}_{\NNfp,c}(X,  d+1-n, P)$ (which is equivalent to asking that $\eta_{\dR}$ map to zero in $H^2_{\st, P}(K, D^{2d-i}_c(X, d+1-n))$, since the syntomic descent spectral sequence degenerates at $E_3$.)

  \begin{proposition}
   \label{prop:computeAJwithNN}
   Let $\eta_{\dR} \in \Fil^0 \DdR(W) \cap \Dcris(W)$, and let $P$ be a polynomial with no bad roots such that $P(\varphi)(\eta_{\dR}) = 0$. Suppose that $\eta_{\dR}$ lifts to some $\eta_{\NNfp} \in H^{2d+1-i}_{\NNfp,c}(X,  d+1-n, P)$. Then, for any $x \in H^i_{\mot}(X,\Qp(n))_0$, we have
   \[
    \big\langle \AJ_{W^*(1)}(x) , \eta_{\dR} \big\rangle_{\dR, W} = \big\langle r_{\syn}(x),\, \eta_{\NNfp} \big\rangle_{\NNfp, X, P}.
   \]
  \end{proposition}

  \begin{proof}
   Since the syntomic descent spectral sequence is compatible with products, we have
   \[
    \langle r_{\syn}(x),\, \eta_{\NNfp}\rangle_{\NNfp, X, P}  = \tr_{\st, P}\left( \AJ_{\syn}(x) \cup \eta_{\dR} \right).
   \]
   Since $\eta_{\dR} \in H^0_{\st, P}(K, W)$, this pairing depends only on the projection of the class $\AJ_{\syn}(x)$ to $H^1_{\st}(K, \Dpst(W^*(1)))$, which is by construction $\AJ_W(x)$. By Corollary \ref{cor:convenientpair}, the pairing between $H^1_{\st}(K, \Dpst(W^*(1)))$ and $H^0_{\st, P}(K, \Dpst(W))$ is simply the de Rham duality pairing.
  \end{proof}

  \begin{remark}[Homological triviality]
   \label{rem:homtrivl}
   Note that the hypothesis that $x$ be homologically trivial is necessary here; otherwise, the pairing $\big\langle r_{\syn}(x),\, \eta_{\NNfp} \big\rangle_{\NNfp, X, P}$ would not be independent of the choice of lift $\eta_{\NNfp}$ of $\eta_{\dR}$. This presents a complication in our application to Euler systems, since there is no obvious reason why the image of the motivic Lemma--Eisenstein map $\LE^{[q,r]}$ should take values in the homologically-trivial subspace. A second complication is that we need to show that $\eta_{\dR}$ is in the kernel of the somewhat inscrutable ``knight's move'' differential on the $E_2$ page of the spectral sequence.

   We shall work around this as follows: the class $\eta_{\dR}$ we shall consider lies in the image of a projector $e$ in the Hecke algebra (corresponding to the generalised eigenspace of $\Pif'$) which annihilates the de Rham cohomology outside degree $d = 3$. So the image of $\eta_{\dR}$ under the ``knight's move'' map is automatically zero. Moreover, $e^\vee x$ is homologically trivial, where $e^\vee$ is the transpose of $e$. So we can apply the above formalism under the additional assumption that $e \, \eta_{\NNfp} = \eta_{\NNfp}$, i.e.~that our lift of $\eta$ is in the $\Pif'$ generalised eigenspace.
  \end{remark}

 \subsection{Pushforward and pullback}

  We can now use the functorial properties of $\NNfp$ cohomology to compute the right-hand side of the formula of \cref{prop:computeAJwithNN}. More precisely, let $\iota: Z \into X$ be a finite morphism of smooth $K$-varieties, of codimension $c$. Then there are pushforward maps
  \[ H^{i-2c}_{\mot}(Z, r-c) \to H^i_{\mot}(X, r)\]
  and similarly for $H^*_{\NNsyn}$ and $H^*_{\et}$; and these are compatible with the maps $r_{\syn}$, $r_{\et}$, and $\operatorname{comp}$ appearing in the diagram of Theorem \ref{thm:NNetcomp}.

  (For the existence of syntomic pushforwards compatible with motivic cohomology, see Proposition B.4 in the appendix by D\'eglise to \cite{nekovarniziol16}. The compatibility with \'etale cohomology is not explicitly stated in \emph{op.cit.}, but it can be extracted from the construction; alternatively, it is an immediate corollary of the naturality of the comparison maps between \'etale and syntomic realisation functors on Voevodsky's category of geometrical motives proved in \cite{degliseniziol18}, see Remark 4.22.2 of \emph{op.cit.}.)

  \begin{proposition}
   \label{prop:pushpull}
   For $z \in H^{i-2c}_{\mot}(Z, r-c)$, we have
   \[
    \operatorname{tr}_{\NNfp, X, P}\left(\iota_\star(z) \cup \tilde{w}\right) =
    \operatorname{tr}_{\NNfp, Z, P}\left(z \cup \iota^\star(\tilde{w})\right).
   \]
  \end{proposition}

  \begin{proof}
   This follows from the adjunction formula relating pushforward and pullback.
  \end{proof}


 \subsection{Coefficients}

   If $X=Y_G(U)$ as in Section \ref{ss:Shimvarandcoeff}, then we can use Liebermann's trick (\cref{rem:Liebermann}) to define cohomology with coefficients in algebraic representations $V$ and to obtain versions of the spectral sequence \eqref{eq:NNsynSS} and of Theorem \ref{thm:NNetcomp} with coefficients. In particular, the composition of the cup product and \eqref{eq:NNtrace} defines a pairing
   \[
    \langle \quad,\quad\rangle_{\NNfp,X,P}:\ H^{i}_{\NNsyn}(X, \cV,r)\ \times\ H^{2d+1-i}_{\NNfp, c}(X, \cV^\vee,d+1-r, P)\longrightarrow K
   \]
   for any $P$ such that $\Qp^{\mathrm{nr}}(1)$ is $P$-convenient. (With this formalism, adding 1 to $r$ corresponds to twisting $\cV$ by the inverse of the symplectic multiplier, so we could assume $r = 0$ if we wish, but it will be more convenient to allow general $r$.)

    For cohomology with coefficients, the formalism of pushforward and pullback maps works as follows: suppose that we have a closed immersion of PEL Shimura varieties $\iota: Y_H(U')\hookrightarrow Y_G(U)$ of codimension $c$, for some reductive group $H$ and  $U'=U\cap H(\Af)$. Assume that the closed immersion extends to the toroidal compactifications.  Let $W$ be a direct summand of $V|_H$. We then obtain
    \begin{align}
     (\iota_U^{\cW})_\star:\, H^\star_{\NNsyn}(Y_H(U'),\cW,r) & \to H^{\star+2c}_{\NNsyn}(Y_G(U),\cV,r+c), \label{eq:Wpushforward}  \\
     (\iota_U^{\cW})^\star: \, H^{\star}_{\NNsyn,c}(Y_G(U),\cV,r) & \to  H^\star_{\NNsyn,c}(Y_H(U'),\cW,r) \label{eq:Wpullback}
    \end{align}
    for all $r\in\ZZ$; and similarly for finite-polynomial cohomology with any polynomial $P$.

   Moreover, these maps  are adjoint with respect to this pairing: identifying $W^\vee$ with a direct summand of $(V^\vee)|_H$, and letting $d = \dim Y_G$, then for all $x\in H^{i-2c}_{\NNsyn}(Y_H(U'),\cW,r-c)$ and $y\in  H^{2d+1-i}_{\NNfp, c}(X, \cV^\vee,d+1-r, P)$, we have
   \[ \left\langle (\iota_U^{\cW})_\star(x),\, y\right\rangle_{\NNfp,Y_G(U),P} = \left\langle x,\, (\iota_U^{\cW^\vee})^*(y)\right\rangle_{\NNfp,Y_H(U'),P}.\qedhere\]


 \subsection{The regulator as a pairing in NN-fp cohomology}\label{sect:reduction1}

  Returning to the specific case $G = \GSp_4$ and $H = \GL_2 \times_{\GL_1} \GL_2$, we can now put together all of the above pieces to express the regulator pairing \eqref{eq:goal2} in terms of NN-fp cohomology.

%

  \begin{lemma}
   Let $V_G$ and $V_H$ be as in \cref{eq:qrtetc}, and $\cV_G$, $\cV_H$ the corresponding coefficient sheaves. Then we have pullback and pushforward maps
   \begin{align}
     (\iota^{[t_1,t_2]}_\Delta)_\star:&\, H^\star_{\NNsyn}(Y_{H,\Delta}, \cV_H^\vee, 2)  \to H^{\star+2}_{\NNsyn}(Y_{G,\Kl},\cV_G^\vee, 3-q), \label{eq:s1s2pushforward}  \\
      (\iota^{[t_1,t_2]}_\Delta)^*:& \, H^{\star}_{\NNfp,c}(Y_{G,\Kl},\cV_G, 1 + q, P)  \to  H^\star_{\NNfp,c}(Y_{H,\Delta}, \cV_H, 1, P) \label{eq:s1s2pullback}.
   \end{align}
   for any polynomial $P$.
  \end{lemma}
  \begin{proof}
   This is an instance of \eqref{eq:Wpushforward} and \eqref{eq:Wpullback}.
  \end{proof}

  \begin{notation}
   Write $\Eis^{[t_1,t_2]}_{\syn,\underline\Phi}$ for the image of $\Eis^{[t_1,t_2]}_{\mot,\underline\Phi}$ under $r_{\NNsyn}$.
  \end{notation}

  As in Section \ref{sect:periods}, let $(w,\uPhi)$ be the product of  some arbitrary test data $(w^p, \Phi^p)$ away from $p$ and the Klingen test data at $p$. Shrinking our tame level $U^p$ if necessary, we may assume that $U^p$ fixes $w^p$, and $V^p = U^p \cap H(\Af^p)$ fixes $\Phi^p$.

  Let $\eta_{\dR}\in \Fil^1\DdR(V_\Pi)\subset  \Fil^1 H^3_{\dR, c}(Y_{G, \Kl}, \cV_G)[\Pif']$ be as in Section \ref{ssec:testdataatp} (see also  \cref{def:nudR}).

  \begin{note}
   Observe that $\eta_{\dR}$ lies in $\Fil^{1+q} H^3_{\dR, c}(\dots)$, for any $0\leq q\leq r_2$.
  \end{note}

  \begin{lemma}
   \label{lem:etaNNfp}
   Let $0\leq q\leq r_2$; and let $\cQ_{1+q}(T)=\cQ(p^{1+q} T) = \left(1-\frac{p^{1+q} T}{\alpha}\right)\left(1-\frac{p^{1+q} T}{\beta}\right)$, where $\cQ$ is as in \eqref{eq:defcurlyP}. Then there exists a unique class
   \[ \eta_{\NNfp} \in H^3_{\NNfp,c}\Big(Y_{G,\Kl, \Qp}, \cV_G, 1+q, \cQ_{1+q}\Big)[\Pif'] \]
   lifting $\eta_{\dR}$.
  \end{lemma}

  \begin{note}
   The group in which this class lies depends on $q$; but the natural maps between these groups for differing values of $q$ are isomorphisms on the $\Pif'$-eigenspace, so the class is ``independent of $q$'' in a certain sense.
  \end{note}

  \begin{proof}
   By the definition of $\eta_{\dR}$ we have $\cQ(\varphi)\eta_{\dR} = 0$, and hence $\cQ_{1+q}(\varphi)\,(\eta_{\dR} \otimes e_{1+q}) =\cQ(\varphi)(\eta) \otimes e_{1 + q} = 0$, where $e_{1+q}$ is the canonical basis of $\DdR(\Qp(1))^{\otimes (1+q)}$. Thus $\eta_{\dR}$ defines a class in the group $H^0_{\st, P}\left(\Qp, H^3_{\et, c}\left(Y_{G, \Kl, \Qp}, \cV_G\right)(-q)\right)$.

   Since the $\Pif'$-generalised eigenspace in $H^j_{\dR, c}\left(Y_{G, \Kl, \Qp}, \cV_G \right)$ is zero for $j \ne 3$, the class $\eta_{\dR}$ is in the image of the edge map of the spectral sequence \eqref{eq:NNsynSS} for $P = \cQ_{1+q}$ (cf.~\cref{rem:homtrivl} above). So it can be lifted to a class in $H^3_{\NNfp, c}\left(Y_{G, \Kl, \Qp}, \cV_G, 1+q; \cQ_{1+q}\right)$; and, again using the fact that $\Pif'$ does not contribute to cohomology in degrees $\ne 3$, there is a unique lift which lies in the $\Pif'$-eigenspace.
  \end{proof}

  \begin{remark}\label{rem:Quniqueness}
   For the existence and uniqueness of the lift to fp-cohomology, it suffices to assume that $\eta_{\NNfp}$ lies in the $\Pif'$ \emph{generalised} eigenspace for the spherical Hecke algebra (generated by Hecke operators away from $p$ and the level of $U$).
  \end{remark}

  Observe that $\cQ_{1+q}$ has no bad roots (for any $q$ in the relevant range $0 \le q \le r_2$), by \cref{lem:trivzero}. So we can use \cref{prop:computeAJwithNN} (and \cref{rem:homtrivl}) to compute the Bloch--Kato logarithm map as a pairing in $\NNfp$ cohomology of $Y_G$; and \cref{prop:pushpull} to relate this to a pairing on $Y_H$. This gives our first step towards \eqref{eq:goal2}:
  \vspace{1ex}

  \begin{mdframed}
   \textbf{1st reduction}: The left-hand side of \eqref{eq:goal2} can be rewritten as
   \begin{equation}
   \label{eq:step1}
   \begin{aligned}
   \left\langle \left(\log \circ \AJ^{[\Pi, q]} \circ \iota^{[t_1,t_2]}_{\Delta, \star}\right)(\Eis^{[t_1, t_2]}_{\et, \uPhi}),\eta_{\dR}\right\rangle_{\dR, Y_{G, \Kl}}
   &=
   \left\langle \iota_{\Delta, \star}^{[t_1,t_2]}\left( \Eis^{[t_1,t_2]}_{\syn,\uPhi}\right),\, \eta_{\NNfp}\right\rangle_{\NNfp,Y_{G, \Kl}} \\
   &= \left\langle \Eis^{[t_1,t_2]}_{\syn,\uPhi},\, (\iota_{\Delta}^{[t_1,t_2]} )^\star \left(\eta_{\NNfp}\right)\right\rangle_{\NNfp,Y_{H, \Delta}}.
   \end{aligned}
   \end{equation}
   \end{mdframed}

\mychapter{Step 2: Reduction to a pairing on the multiplicative-ordinary locus}

 Our next goal will be to re-express the pairing of \eqref{eq:step1} in a fashion which is more amenable to computation, as follows:
 \begin{itemize}
 \item We will replace \Nek-\Niz cohomology, which has very powerful functorial properties but is rather inexplicit in its definition, with a more ``low-tech'' cohomology theory (the \emph{rigid syntomic cohomology} of Besser).

 \item We will show that the computation can be carried out after passing to certain open subsets of toroidal compactifications of the Shimura varieties for $G$ and for $H$ (the multiplicative components of the ordinary loci), using an appropriate formalism of ``partially compactly supported'' cohomology.
 \end{itemize}
 These two reductions will allow us to link up with Boxer and Pilloni's \emph{higher Coleman theory}, which will be the next major step later in this paper.

 These two reductions are interrelated, since Besser's cohomology applies to \emph{smooth} $\Zp$-schemes, while the natural $\Zp$-models of our Shimura varieties are not smooth -- their special fibres are singular. However, the loci that we want to study are the tubes of subvarieties of the special fibre; and these subvarieties are contained in the smooth locus. So we shall use a generalisation of Besser's rigid cohomology to semistable schemes -- the \emph{log-rigid syntomic cohomology} of Ertl and Yamada \cite{ertlyamada18} -- as a bridge between \Nek--\Niz cohomology and Besser's theory. However, the situation is complicated by the fact that although $Y_{G, \Kl}$ is semistable, $Y_{H, \Delta}$ is not. We shall bypass this by working with auxiliary Shimura varieties for $G$ and $H$, with deeper level structures at $p$, which admit compatible semistable models over the ramified extension $\Zp[\zeta_p]$. We shall work with semistable models over this extension, and descend to $\Qp$ after we have restricted to a smooth open subvariety.


\section{Log-rigid syntomic and fp-cohomology}
\label{sect:logsyntomic}
%

 \subsection{Log structures}

  Let $\pi$ be a uniformizer of $K$, and write $\cO_K^\pi$ for the scheme $\Spec \cO_K$ with the canonical log structure, given by the chart $1 \mapsto \pi$. 
  Denote by $k$ the residue field of $\cO_K$, and write $k^0$ for the scheme $\Spec k$, with the log structure given by $1\mapsto 0$.

  \begin{definition}\label{def:SSschemeboundary}
   A \emph{strictly semistable $\cO_K$-scheme with boundary} is a pair $(X, D)$, where $X$ is a finite-type $\cO_K$-scheme and $D$ a closed subscheme (both flat over $\cO_K$), with the following properties:
   \begin{enumerate}[(a)]
   \item the union of $D$ and the special fibre $X_0$ is a strict normal crossing divisor;
   \item each point of $X$ has a Zariski-open neighbourhood which is smooth over
   \[ \Spec \cO_K[t_1, \dots, t_m, s_1, \dots, s_n] / (t_1\dots t_m - \pi)\]
   for some $m$ and $n$, with $D$ corresponding to $s_1\dots s_n = 0$.
   \end{enumerate}
  \end{definition}

  \begin{remark}
   We do not assume here that $X$ be proper over $\cO_K$.
  \end{remark}

  By the same argument as in \cite[\S 2.1]{grossekloenne05}, since we have assumed our divisors to be strictly normal-crossing, it is equivalent to suppose that charts as in (b) exist \'etale-locally on $X$ (rather than Zariski-locally). Note that any strictly semistable $\cO_K$-scheme satisfies the conditions if we let $D = \varnothing$.

 \subsection{Log-rigid syntomic cohomology}

  Given a strictly semistable $\cO_K$-scheme with boundary $(X, D)$, we may equip $X$ with the log-structure associated to the divisor $D \cup X_0$. Then $X$ (with this log-structure) is a \emph{strictly semistable log-scheme with boundary over $\cO_K^{\pi}$} in the sense of \cite[Definition 3.3]{ertlyamada18}. We can then consider three complexes associated to $X$ and $D$:
  \begin{itemize}
   \item the \emph{rigid Hyodo-Kato cohomology} $R\Gamma^{\HK}_{\rig}(X_0\langle D_0\rangle)$ (see Section 1.3.3. in \emph{op.cit.}), which is a complex of $F$-vector spaces with an $F$-semilinear Frobenius $\varphi$ and an $F$-linear monodromy operator $N$, satisfying $N\varphi = p \varphi N$, where $F$ is the maximal unramified subfield of $K$. (Its definition involves a rather intricate limiting process over collections of liftings of open subsets of $X_0$ to characteristic zero, since it is not generally possible to find a global lifting of $X_0$ compatible with Frobenius.)

   \item the \emph{log-rigid cohomology} $R\Gamma_{\lrig}(X_0 \langle D_0\rangle / \cO_K^\pi)$ (see Section 1.3 in \emph{op.cit.}), which is a complex of $K$-vector spaces, quasi-isomorphic to the de Rham cohomology of the dagger space $\cX =\  ]X_0[\,_{X_K^{\mathrm{an}}}^\dag$ with log poles along $D_K$.

   \item the \emph{Deligne-de Rham cohomology} $R\Gamma_{\dR}^D(U_K)$ (see \cite{deligne74}), which is a complex of $K$-vector spaces with a filtration $\Fil^r$ (the Hodge filtration). If $X_K$ is proper, it is quasi-isomorphic to $R\Gamma(X_K, \Omega^\bullet_{X_K}\langle D_K \rangle)$ with the filtration defined by truncation.
  \end{itemize}

  \begin{note}
   These complexes are related by morphisms in the derived category of $K$-vector spaces (cf.~Equation 3.11 of op.cit.)
   \[
    R\Gamma^{\HK}_{\rig}(X_0\langle D_0\rangle) \otimes_F K\xrightarrow{\iota_{\pi}^{\rig}} R\Gamma_{\lrig}(X_0\langle D_0\rangle / \cO_K^\pi) \xleftarrow{\,\mathrm{sp}\,} R\Gamma^D_{\dR}(U_K).
   \]
   The morphism $\iota_{\pi}^{\rig}$ is a quasi-isomorphism, and $\mathrm{sp}$ is also a quasi-isomorphism if $X$ is proper.
  \end{note}

  For $r\geq 0$, Ertl--Yamada \cite[Definition 3.4]{ertlyamada18} define a \emph{log-rigid syntomic cohomology}:

  \begin{definition}
   \label{def:lrigsyn}
   Define $R\Gamma_{\lrigsyn}\left(X\langle D\rangle, r\right)$ to be the homotopy limit of the diagram
   \[
    \begin{tikzcd}[row sep=large, column sep = tiny]
    \Fil^r R\Gamma_{\dR}^D(U_K)
      \arrow[dr, "{\operatorname{sp}}"]
    &
    {} &
    R\Gamma_{\rig}^{\HK}(X_0\langle D_0\rangle)
      \arrow[dl, "{\iota_\pi^{\rig}}"']
      \arrow[d, "N"]
      \arrow[rr, "{1-\varphi_r}"]
    &{\hspace{1.5em}}&
    R\Gamma_{\rig}^{\HK}(X_0\langle D_0\rangle)
      \arrow[d, "N"]
    \\
    {} &
    R\Gamma_{\lrig}(X_0\langle D_0 \rangle / \cO_K^\pi)
    &
    R\Gamma_{\rig}^{\HK}(X_0\langle D_0\rangle)
      \arrow[rr, "{1-\varphi_{r-1}}"]
    &&
    R\Gamma_{\rig}^{\HK}(X_0\langle D_0\rangle)
    \end{tikzcd}
   \]
   Here $\varphi_r \coloneqq p^{-r} \varphi$.
  \end{definition}

  \begin{note}
   There is no properness assumption on $X$, so if we start with a strictly semistable log-scheme $U$, we can always simply take $X = U$ and $D = \varnothing$ in the above construction. However, it is important to allow more general $X$ in order to prove the following theorem, showing that for proper $X$, log-syntomic cohomology coincides with \Nek--\Niz cohomology of the complement of $D$.
  \end{note}

  \begin{theorem}\label{thm:comparison}
   Let $U = X-D$, where $(X, D)$ is a strictly semistable $\cO_K$-scheme with boundary, and suppose that $X$ is proper. Then for all $r\geq 0$, there exist canonical quasi-isomorphisms
   \[ R\Gamma_{\lrigsyn}\left(U, r\right) \cong R\Gamma_{\lrigsyn}(X\langle D\rangle, r)\cong R\Gamma_{\NNsyn}(U_K,r).\]
  \end{theorem}

  \begin{proof}
   This is \cite[Corollary 4.2]{ertlyamada18}.
  \end{proof}

  \begin{note}
   We may define \emph{log-rigid finite-polynomial cohomology} similarly, replacing $1-\varphi_r$ with more general polynomials $P(\varphi_r)$; and we obtain a comparison with the finite-polynomial variant of \Nek--\Niz cohomology considered in \cite{besserloefflerzerbes16}, extending \cref{thm:comparison}. However, we shall not actually use this directly here; instead, we shall use its compactly-supported analogue developed in the next section.
  \end{note}


 \subsection{Compactly supported log-rigid syntomic and fp-cohomology}

  Let $(X, D)$ be a strictly semistable $\cO_K$-scheme with boundary, as before; we \emph{do} now assume that $X$ is proper. In \cite{ertlyamada19}, Ertl and Yamada define
  \begin{itemize}
   \item \emph{rigid Hyodo-Kato cohomology with compact support}, $R\Gamma^{\HK}_{\rig}(X_0 \langle -D_0\rangle)$,
   \item \emph{log-rigid cohomology with compact support},  $R\Gamma_{\lrig}(X_0 \langle -D_0\rangle / \cO_K^\pi)$.
  \end{itemize}
  Again, the former is $F$-linear and equipped with Frobenius and monodromy operators, and the latter is $K$-linear.

  \begin{proposition}
   Let $d=\dim(X)$. Then there exist canonical isomorphisms
   \begin{align}
    R\Gamma^{\HK}_{\rig}(X_0 \langle - D_0\rangle) & \xrightarrow{\cong} R\Gamma^{\HK}_{\rig}(X_0 \langle D_0\rangle)^*[-2d]\label{eq:HKduality}\\
    R\Gamma_{\lrig}(X_0 \langle -D_0\rangle / \cO_K^\pi) & \xrightarrow{\cong} R\Gamma_{\lrig}(X_0 \langle D_0\rangle / \cO_K^\pi)^*[-2d].
   \end{align}
  \end{proposition}

  \begin{proof}
   See \cite[Theorem 6.1]{ertlyamada19}.
  \end{proof}

  \begin{note}\label{note:Frobholo}
   The morphism \eqref{eq:HKduality} is compatible with $\varphi$ and $N$, if we define the Frobenius $\Phi$ and monodromy on the right-hand side as $p^d \cdot (\varphi^{-1})^\vee$ and $-N^\vee$.
  \end{note}

  \begin{remark}
    The Frobenius on $R\Gamma^{\HK}_{\rig}(X_0 \langle D_0\rangle)$ admits an inverse in the derived category, although it is not necessarily invertible at the level of complexes; explicitly, we can replace the complex computing $R\Gamma^{\HK}_{\rig}(X_0 \langle D_0\rangle)$ with its ``perfection'', as in \cite[\S 4]{besser12}.
   \end{remark}

  We also have the complex $R\Gamma^D_{\dR,c}(U_K) \coloneqq R\Gamma(X_K, \Omega^\bullet_{X_K/K}\langle -D_K\rangle)$ computing compactly-suppor\-ted de Rham cohomology of $U_K$, with its truncation filtration; and there are maps in the derived category relating these three complexes, as before.

 We define log-rigid syntomic cohomology with compact support as follows:

 \begin{definition}\label{def:lrigsyncompact}
  We define $R\Gamma_{\lrigsyn}(X\langle -D\rangle,r)$ as the homotopy limit of the diagram analogous to \cref{def:lrigsyn} with the three complexes replaced by their $\langle-D\rangle$ versions.
 \end{definition}

 \begin{definition}
  Replacing $1-\varphi_r$ by $P(\varphi_r)$, for some $P \in \Qp[t]$ with constant coefficient $1$, we obtain log-rigid fp-cohomology with compact support, which we denote by $R\Gamma_{\lrigfp}(X\langle -D\rangle, r; P)$.
 \end{definition}

 The following result is a consequence of the results in \emph{op.cit.}.

 \begin{theorem}\label{thm:comparisonc}
  For all $r\geq 0$, there exists a canonical isomorphism
  \[ R\Gamma_{\lrigfp}(X\langle -D \rangle, r; P)\cong R\Gamma_{\NNfp,c}(U_K, r; P).\]
  Moreover, this isomorphism is compatible with pullback.
 \end{theorem}

 The following result (c.f. \cite[\S 4]{besser12}) will be useful for the constructon of an `extension-by-0' map (c.f. Proposition \ref{prop:extbyzero}):

 \begin{proposition}\label{prop:lrigqisom}
  The complex defined in \ref{def:lrigsyncompact} is quasi-isomorphic to the homotopy limit of the following diagram, shifted by $[-2d]$:
  \[
  \begin{tikzcd}[row sep=large, column sep=0.2em]
  \left(\Fil^{d-r} R\Gamma_{\dR}^D(U_K)\right)^*
    \arrow[dr, "{(\operatorname{sp}^\vee)^{-1}}"]
  &
  {} &
  \left(R\Gamma_{\rig}^{\HK}(X_0\langle D_0\rangle)\right)^*
    \arrow[dl, "{\left[(\iota_\pi^{\rig})^\vee\right]^{-1}}"']
    \arrow[d, "{-N^\vee}"]
    \arrow[rr, "{1-\Phi_r}"]
  &
  \hspace{2em}
  &
  \left(R\Gamma_{\rig}^{\HK}(X_0\langle D_0\rangle)\right)^*
    \arrow[d, "{-N^\vee}"]
  \\
  {} &
  \left(R\Gamma_{\lrig}(X_0\langle D_0 \rangle / \cO_K^\pi)\right)^*
  &
  \left(R\Gamma_{\rig}^{\HK}(X_0\langle D_0\rangle)\right)^*
    \arrow[rr, "{1-\Phi_{r-1}}"]
  &
  &
  \left(R\Gamma_{\rig}^{\HK}(X_0\langle D_0\rangle)\right)^*
  \end{tikzcd}
  \]
%
 \end{proposition}
 \begin{proof}
  Immediate from Note \ref{note:Frobholo}.
 \end{proof}


  Let $(X, D)$ be as above, and $U_K$ the open variety $X_K - D_K$.

  \begin{proposition}
   We have cup products
   \begin{align}
    R\Gamma_{\NNsyn}(U_K,r) \times R\Gamma_{\NNfp,c}(U_K,s;P) &
    \to R\Gamma_{\NNfp,c}(U_K,r+s;P) ,\label{eq:pairingNN}\\
    R\Gamma_{\lrigsyn}(X\langle D \rangle,r) \times R\Gamma_{\lrigfp}(X\langle -D \rangle,s;P) & \to R\Gamma_{\lrigsyn}(X\langle-D\rangle,r+s; P)\label{eq:pairinglrig}
   \end{align}
   which are compatible under the isomorphisms in Theorems \ref{thm:comparison} and \ref{thm:comparisonc}.
  \end{proposition}
  \begin{proof}
   The proof for the \Nek--\Niz cohomology is given in \cite{besserloefflerzerbes16}. The proof for  log-rigid fp-cohomology is analogous. The compatibility follows from \cite{ertlyamada19}.
  \end{proof}

  \begin{corollary}
   When $i+j=2d+1$, $r+s=d+1$, and $P$ is a polynomial with no bad roots, then we get $K$-valued pairings, denoted $\langle\quad,\quad\rangle_{\NNfp, U_K}$ and $\langle\quad,\quad\rangle_{\lrigfp,X}$, respectively; and these are compatible under the maps of Theorems \ref{thm:comparison} and \ref{thm:comparisonc}.\qed
  \end{corollary}


\section{Rigid syntomic and fp-cohomology for smooth schemes}

 We now explain a simpler counterpart of the constructions of the previous section applying to smooth, rather than semistable, $\cO_K$-schemes: Besser's rigid syntomic and finite-polynomial cohomology.

 \subsection{Rigid syntomic/fp-cohomology}\label{sect:rigsyn}

  Let $X$ be a smooth $\cO_K$-scheme with generic fibre $X_K$ and special fibre $X_0$. We assume that $X_K$ is proper (but $X$ itself may not be). Denote by $\cX$ the dagger space tube of $X_0$ in $X_K^{\mathrm{an}}$. Let $D$ be a divisor in $X$ which intersects transversely with the special fibre; and let $\cD$ be the divisor $D_K^{\mathrm{an}} \cap \cX$.

  \begin{definition}
   An \emph{overconvergent filtered $F$-isocrystal} on $(X, D)$ consists of the following data:
   \begin{itemize}
    \item an overconvergent $F$-isocrystal $\sF_{\rig}$ on $X_0$, with logarithmic poles along $D_0$;
    \item an algebraic vector bundle $\sF_{\dR}$ on the variety $X_K$, endowed with a connection with log poles along $D_K$, and with a filtration satisfying Griffiths transversality;
    \item an isomorphism of rigid-analytic vector bundles over the dagger space $\cX$, compatible with connections,
    \[ \sF_{\dR}|_{\cX} \cong \sF_{\rig, X}, \]
    where $\sF_{\rig, X}$ is the realisation of $\sF_{\rig}$ corresponding to the lifting $X$ of $X_0$.
   \end{itemize}
  \end{definition}

  \begin{definition}
   Define
   \[ R\Gamma_{\rig}(X_0\langle D_0\rangle,\sF_{\rig})= R\Gamma(\cX,\sF_{\rig, X} \otimes\Omega^\bullet_{X_K}\langle D_K\rangle).\]
 \end{definition}

 This depends functorially on $(X_0, \sF_{\rig})$, and in particular is equipped with a Frobenius (even though this may not lift to $X$).

 \begin{note}\label{note:proprig}~
  \begin{enumerate}
   \item In the case of trivial coefficients, we recover the complex $R\Gamma_{\dR}(\cX\langle \cD \rangle)$ (c.f. Notation \ref{not:logdR}).
   \item The rigid cohomology used here coincides with Hyodo--Kato cohomology; that is, if $X$ is smooth, then $R\Gamma_{\rig}(X_0\langle D_0\rangle,\sF_{\rig})=R\Gamma^{\HK}_{\rig}(X_0\langle D_0\rangle,\sF_{\rig})$ (with monodromy acting as 0).
   \item There exists a specialisation map (see \cite{baldassarriberthelot04})
  \[  \operatorname{sp}: R\Gamma_{\dR}(X_K\langle D_K\rangle,\sF_{\dR})\to R\Gamma_{\rig}(X_0\langle D_0\rangle,\sF_{\rig})_K.\qedhere \]
  \end{enumerate}
 \end{note}


  \begin{definition}
    Let $r \in \ZZ$, and let $P \in \Qp[t]$ have constant coefficient $1$. Following Besser \cite{besser12}, we define the \emph{rigid fp-cohomology} of $X\langle D\rangle$ with coefficients $\sF$, twist $r$ and polynomial $P$ as the homotopy limit of the diagram
    \begin{equation}
     \label{eq:rigsyn}
     \begin{tikzcd}[row sep=large, column sep=-1em]
      \Fil^r R\Gamma_{\dR}(X_K\langle D_K\rangle, \sF_{\dR})
        \arrow[dr, "{\scriptstyle \operatorname{sp}}"]
      &&
      R\Gamma_{\rig}(X_0\langle D_0\rangle,\sF_{\rig})
        \arrow[dl]
        \arrow[dr, "{\scriptstyle P(\varphi_r)}"]
      \\
      &
      R\Gamma_{\rig}(X_0\langle D_0\rangle,\sF_{\rig})_K
      &
      &
      R\Gamma_{\rig}(X_0\langle D_0\rangle,\sF_{\rig})
     \end{tikzcd}
    \end{equation}
    where the unlabelled arrow is base-extension to $K$. We denote it by $R\Gamma_{\rigfp}(X\langle D\rangle,\sF,r;P)$. When $P(t)=1-t$, then we call it \emph{rigid syntomic cohomology}, denoted by $R\Gamma_{\rigsyn}(X\langle D\rangle, \sF, r)$.
  \end{definition}

  \begin{notation}
   We shall write $R\Gamma_{\rigfp}(X\langle D\rangle,r;P)$ if $\sF$ is the trivial isocrystal.
  \end{notation}

  \begin{note}\label{note:propfigfp}
   \begin{enumerate}
    \item We have $R\Gamma_{\rigfp}(X\langle D\rangle,\sF,r;P) = R\Gamma_{\rigfp}(X\langle D\rangle,\sF(r),0; P)$ where $\sF(r)$ is the $r$-th Tate twist of $\sF$.
    \item If $F = K$, the middle arrow is the identity, and the zigzag diagram collapses to the mapping fibre of the map
    \[  \Fil^r R\Gamma_{\dR}(X_K\langle D_K\rangle, \sF_{\dR}) \xrightarrow{\ P(\varphi_r) \circ \operatorname{sp}\ } R\Gamma_{\rig}(X_0\langle D_0\rangle,\sF_{\rig}).\]
    \item If $X$ is equipped with the log structure associated to $X_0 \cup D$, then by Note \ref{note:proprig} (2) we have a natural map
    \[
     \delta: R\Gamma_{\lrigfp}(X\langle D\rangle,\sF,r;P) \to R\Gamma_{\rigfp}(X\langle D\rangle,\sF,r;P).
    \]
    However, this is not an isomorphism in general (even when $X$ is smooth and proper and $D = \varnothing$); the difference between the two is essentially the mapping fibre of $P(\varphi_{r-1})$.\qedhere
   \end{enumerate}
  \end{note}

  \begin{lemma}\label{lem:reslogfptofp}
   Let $X$ be a strictly semistable proper log scheme over $\cO_K^\pi$, and let $D\subset X$ be a closed subscheme with complement $U$. Suppose that $(U,X)$ is a strictly semistable log scheme, and let $Z$ be a smooth open subscheme of $X$. We then have a restriction map
   \[ \res_Z: R\Gamma_{\lrigfp}(X\langle D\rangle,\sF,r;P)\to R\Gamma_{\rigfp}(Z\langle D\rangle,\sF,r;P).\]
  \end{lemma}
  \begin{proof}
   Consequence of Note \ref{note:propfigfp}, together with the restriction map on lrig-fp cohomology.
  \end{proof}

  \subsubsection*{Rigid fp-cohomology with compact support}

   We now consider a compactly-supported variant.

   \begin{notation}
    Write $\cosp: R\Gamma_{\rig,c}(X_0\langle -D_0\rangle,\sF_{\rig})\rightarrow R\Gamma^D_{\dR}(X_K,\sF_{\dR})$ for the cospecialisation map (see \cite{baldassarrietal04}).
   \end{notation}

   \begin{definition}
    Let $r\geq 0$, and let $Q\in\Qp[t]$ have constant coefficient $1$. Define the rigid fp-cohomology with compact support of $X$ with coefficients $\sF$,  twist $r$ and polynomial $Q$, as the homotopy limit of the zigzag diagram
    \begin{equation}
    \label{eq:crigsyn}
    \begin{tikzcd}[row sep=large, column sep=-1em]
     \Fil^r R\Gamma^D_{\dR}(X_K, \sF_{\dR}) 
     \arrow[rd] & 
     {} & 
     R\Gamma_{\rig,c}(X_0\langle -D_0\rangle, \sF_{\rig})
     \arrow[ld, "\cosp"]
     \arrow[rd, "Q(\varphi_r)"] \\
     {} &
     R\Gamma^D_{\dR}(X_K,\sF_{\dR}) & 
     {} & 
     R\Gamma_{\rig,c}(X_0\langle -D_0\rangle,\sF_{\rig}).
    \end{tikzcd}
    \end{equation}
     We denote it by  $R\Gamma_{\rigfp,c}(X\langle -D\rangle,\sF,r;Q)$.
   \end{definition}

   We have the following analogue of Proposition \ref{prop:lrigqisom} for rigid fp-cohomology with compact support:

   \begin{lemma}\label{lem:holorigfp}
    The complex \eqref{eq:crigsyn} is quasi-isomorphic to the homotopy limit of the following diagram, shifted by $[-2d]$:
    \[
    \begin{tikzcd}[row sep=large, column sep = -1em]
     \Fil^{d-r} \left(R\Gamma_{\dR}^D(X_K),\sF_{\dR}\right)^* \arrow[rd] &
     {} &
     R\Gamma_{\rig}(X_0\langle D_0\rangle,\sF_{\rig})^*
     \arrow[ld, "{\mathrm{sp}}^\vee"]
     \arrow[rd, "Q(\Phi_r)"] \\
     {} & 
     R\Gamma_{\dR}^D(X_K,\sF_{\dR})^*&
     {} & 
     R\Gamma_{\rig}(X_0\langle D_0\rangle,\sF_{\rig})^*.
   \end{tikzcd}   
   \]
   Here  $\Phi=(\varphi^\vee)^{-1}$.
   \end{lemma}
   \begin{proof}
    See \cite[\S 4]{besser12}.
   \end{proof}

   \begin{proposition}\label{prop:extbyzero}
    Let $X$ be a strictly semistable proper log scheme over $\cO_K^\pi$, and let $D\subset X$ be a closed subscheme with complement $U$. Suppose that $(U,X)$ is a strictly semistable log scheme, and let $Z$ be a smooth open subscheme of $X$. Then in the derived category, we have an extension-by-0 morphism (omitting the coefficient sheaf for clarity):
    \[ R\Gamma_{\rigfp,c}(Z\langle -D \cap Z\rangle,r;Q)\to R\Gamma_{\lrigfp}(X\langle -D\rangle,r;Q).\]
   \end{proposition}
   \begin{proof}
    Clear from  Proposition \ref{prop:lrigqisom}  and Lemma \ref{lem:holorigfp}. Here, the morphism
    \[ R\Gamma_{\rig}(Z_0\langle D_0\cap Z_0\rangle)^*\to R\Gamma^{\HK}_{\rig}(X_0\langle D_0\rangle)^*\]
    is given by the composition of $(\iota_\pi^{\rig})^\vee$ with the dual of the natural restriction map
    \[ R\Gamma_{\lrig}(X_0\langle D_0\rangle)  \to R\Gamma_{\rig}(Z_0\langle D_0\cap Z_0\rangle).\qedhere\]
   \end{proof}


  \begin{proposition}\label{prop:rigfpcup}
   For $r,s\geq 0$, we have a cup product
   \[
    R\Gamma^i_{\rigfp}(X\langle D\rangle,\sF,r;P) \times R\Gamma^j_{\rigfp,c}(X\langle -D\rangle,\sG,s;Q)  \longrightarrow R\Gamma_{\rigfp,c}^{i+j}(X,\sF\otimes\sG,r+s;P\star Q).\label{eq:pairingrig}
   \]
  \end{proposition}
  \begin{proof}
   See \cite[\S 2]{besser12}.
  \end{proof}

  \begin{lemma}
   If $X$ is connected of dimension $d$ and $Q$ has no bad roots, then there is a canonical isomorphism
   \[ \tr_{\mathrm{fp},X}: H^{2d+1}_{\rigfp,c}(X,d+1;Q) \cong K. \]
   It is given explicitly by mapping $(x, y) \in H^{2d}_{\dR, c}(X_K) \oplus H^{2d}_{\rig, c}(X_0)$ to $\tr_{\dR, X_K}(x) - Q(\tfrac{\varphi}{p})^{-1} \tr_{\rig, X_0}(y)$.
  \end{lemma}

  \begin{remark}
   The factor $Q(\varphi/p)$ is included to make the isomorphism compatible with change of $Q$.
  \end{remark}

%
%

  \begin{corollary}\label{cor:rigfppairing}
   Assume that $(P\star Q)$ has no bad roots. When $i+j=2d+1$, $\sG=\sF^\vee$ and $r+s=d+1$, then we get a pairing denoted  $\langle\quad,\quad\rangle_{\rigfp,X}$. In the setting of \cref{prop:extbyzero}, the extension-by-0 map and the restriction map of \ref{lem:reslogfptofp} are transposes of one another with respect to the duality pairings on rigid and log-rigid cohomology.
  \end{corollary}

  \begin{proof}
   It suffices to check that the restriction and extension-by-zero maps are transposes of each other on each term in the defining diagrams, which is clear by construction.
  \end{proof}



\subsection{Gros fp-cohomology}

 In Section \ref{section:Poznan}, we will need a variant of rigid fp-cohomology which is less refined, but more convenient for computations; see \S 9 of \cite{besser00}, in particular Definition 9.3.

  Recall that we have
  \[ R\Gamma_{\rig}(X\langle \pm D_0\rangle,\sF_{\rig})_K = R\Gamma( \cX, \sF_{\rig, X} \otimes \Omega^\bullet_{\cX/K}\langle \pm \cD_K\rangle) = R\Gamma( \cX, \sF_{\dR}|_{\cX} \otimes \Omega^\bullet_{\cX/K}\langle \pm \cD_K\rangle), \]
  where $\cX$ is the tube of $X_0$ in $X_K^{\mathrm{an}}$.

  \begin{definition}
   For $r \ge 0$, we define \emph{truncated rigid cohomology}, denoted by $\RGt_{\dR}(\cX\langle \pm \cD\rangle,\sF, r)$, to be the cohomology of the subcomplex $\left(\Fil^{r-\bullet} \sF_{\dR}\right)|_{\cX} \otimes \Omega^\bullet_{\cX/K}\langle \pm \cD\rangle$, and similarly with compact support.   \end{definition}

   \begin{note}\label{note:filtspcosp}
    We obtain ``filtered'' specialisation and cospecialisation maps
    \begin{gather*}
     \operatorname{sp}: \Fil^r R\Gamma_{\dR}(X_K\langle \pm D_K\rangle, \sF_{\dR})\to  \RGt_{\dR}(\cX\langle \pm \cD\rangle,\sF, r)_K,\\
     \cosp: \RGt_{\dR, c}(\cX\langle \mp \cD\rangle, \sF, r) \to \Fil^r R\Gamma_{\dR, c}(X_K\langle \mp D_K\rangle, \sF_{\dR}),
    \end{gather*}
    compatible with the usual specialisation and cospecialisation maps on the non-filtered complexes.
   \end{note}
      
  \begin{remark}
   The inclusion of $\left(\Fil^{r-\bullet} \sF_{\dR}\right)|_{\cX} \otimes \Omega^\bullet_{\cX/K}\langle \pm \cD\rangle$ into the full de Rham complex gives maps
   \[ \iota: \RGt_{\dR}(\cX\langle \pm \cD\rangle,\sF, r) \to R\Gamma_{\rig}(X_0\langle \pm D_0\rangle,\sF),\]
   and similarly with compact support; but the maps induced by $\iota$ on cohomology are not necessarily either injective or surjective, and the groups $\wH^i_{\dR}$ and $\wH^i_{\dR,c}$ may not even be finite-dimensional over $K$.
  \end{remark}
      
  \begin{definition}\label{def:grosfp} \
   \begin{enumerate}[(a)]
    \item Define the \emph{Gros fp-cohomology} of $X\langle \pm D\rangle$ with coefficients $\sF$, twist $r$ and polynomial $P$ to be the cohomology of the complex $\RGt_{\rigfp}(X\langle \pm D\rangle,\sF,r;P)$ which is the homotopy limit of the diagram
    \[
     \begin{tikzcd}[row sep=large, column sep=-1em]
      \RGt_{\dR}(\cX\langle \pm \cD\rangle,\sF, r)_K 
       \arrow[rd, "\iota"] &
      {} & 
      R\Gamma_{\rig}(X_0\langle \pm D_0\rangle,\sF_{\rig}) 
       \arrow[rd, "P(\varphi_r)"] 
       \arrow[ld] \\
      {} & 
      R\Gamma_{\rig}(X_0\langle \pm D_0\rangle,\sF_{\rig})_K & 
      {} & 
      R\Gamma_{\rig}(X_0\langle \pm D_0\rangle,\sF_{\rig})
     \end{tikzcd}
    \]
    where the unlabelled arrow is base-extension.

    \item Similarly, define the \emph{Gros fp-cohomology with compact support} of $X\langle \mp D\rangle$ with coefficients $\sG$, twist $s$ and polynomial $Q$ to be the homotopy limit $\RGt_{\rigfp,c}(X\langle \mp D\rangle,\sG_{\rig},s;P)$ of the diagram
    \[
     \begin{tikzcd}[row sep=large, column sep=-1em]
      \RGt_{\dR,c}(\cX\langle \mp \cD\rangle,\sG,s)_K 
       \arrow[rd, "\iota"] &
      {} & 
      R\Gamma_{\rig,c}(X_0\langle \mp D_0\rangle, \sG_{\rig} 
       \arrow[rd, "Q(\varphi_s)"] 
       \arrow[ld] \\
      {} & 
      R\Gamma_{\rig,c}(X_0\langle \mp D_0\rangle,\sG_{\rig})_K & 
      {} & 
      R\Gamma_{\rig, c}(X_0\langle \mp D_0\rangle,\sG_{\rig}).
     \end{tikzcd}
    \]
   \end{enumerate}
  \end{definition}

  \begin{note}\label{note:proptildefp}
    As before, if $F = K$ then the middle arrow is the identity map and both diagrams can be simplified to mapping fibres: in this case we have
    \begin{align*}
     \RGt_{\rigfp}(X\langle \pm D\rangle,\sF,r;P) & = MF\Big[ \RGt_{\dR}(\cX\langle \pm \cD\rangle,\sF_{\rig},r)_K \xrightarrow{\  P(\varphi_r)\circ\iota \ } R\Gamma_{\rig}(X_0\langle \pm D_0\rangle,\sF_{\rig})_K\Big],\\
     \RGt_{\rigfp,c}(X\langle \mp D\rangle,\sG,s;Q) & = MF\Big[ \RGt_{\dR,c}(\cX\langle \mp \cD\rangle,\sG_{\rig},s)_K\xrightarrow{\ Q(\varphi_s) \circ \iota\ } R\Gamma_{\rig}(X_0\langle \mp D_0\rangle,\sG_{\rig})_K\Big].\qedhere
    \end{align*}
  \end{note}

  \begin{remark}
   \label{rem:Grosspcosp}
   Comparing the diagrams of \cref{def:grosfp} with \eqref{eq:rigsyn} and \eqref{eq:crigsyn}, we see that the filtered specialisation map (c.f. Note \ref{note:filtspcosp}) on the de Rham cohomology gives a map
   \[ \gamma^\star: R\Gamma_{\rigfp}(X,\sF\langle \pm D\rangle,r;P) \to \RGt_{\rigfp}(X\langle \pm D\rangle,\sF,r;P) .\]
   Similarly, the filtered cospecialisation induces a map
   \[ \gamma_\star: \RGt_{\rigfp,c}(X\langle \mp D\rangle,\sG,s; Q) \to R\Gamma_{\rigfp, c}(X\langle \mp D\rangle,\sG,s; Q). \]

   We also have cup-products
   \[ \RGt_{\rigfp}(X\langle \pm D\rangle,\sF,r;P) \times \RGt_{\rigfp,c}(X\langle \mp D\rangle,\sG,s; Q) \to \RGt_{\rigfp, c}(X,\sF\otimes\sG,r+s; P\star Q), \]
   related to those in the (non-Gros) rigid fp-cohomology (c.f. Proposition \ref{prop:rigfpcup}) by the adunction formula
   \[ \gamma_\star( \gamma^\star(x) \cup y) = x \cup \gamma_\star(y). \]
   In particular, $\gamma^\star$ and $\gamma_\star$ are transposes of each other with respect to the pairing induced by the trace map on the degree $2d+1$ cohomology. Moreover, the pairing is compatible with the maps in Note \ref{note:proptildefp}.
  \end{remark}

  \begin{notation}
   We denote the pairing by $\langle\quad,\quad\rangle_{\widetilde{\rigfp},X}$.
  \end{notation}



\section{Integral models of Siegel threefolds}
 \label{sect:intmodels}

 We now discuss integral models of various Shimura varieties (with non-trivial level at $p$), and their compactifications. We shall fix a level structure $K^p_G$ away from $p$, which we assume to be neat. The Shimura variety $Y_{G, \QQ}$ of level $K^p_G \cdot G(\Zp)$ therefore has a smooth integral model over $\Zp$, which we denote by $Y_G$. This scheme has an interpretation as a moduli space for abelian surfaces $A$ with a polarisation and prime-to-$p$ level structure (depending on $K^p_G$).

 \subsection{The Klingen-level Siegel threefold}

  We now consider the case of Klingen level structure at $p$.

  \begin{definition}
   Let $Y_{G,\Kl}$, or just $Y_{\Kl}$ for short, be the canonical $\Zp$-model of the Siegel $3$-fold of level $K^{p} \times \Kl(p)$, which parametrises choices of order $p$ finite flat subgroup-scheme $C \subset A[p]$.
  \end{definition}

  \begin{remark}
   Note that $Y_{G, \Kl} \to Y_G$ is a proper morphism, but it is not finite: above certain points in the supersingular locus of $Y_{G, 0}$ (superspecial points), the fibre of $Y_{G, \Kl,0}$ is a $\mathbf{P}^1$.
  \end{remark}

  We now define a stratification of the special fibre $Y_{\Kl, 0}$ which will be of fundamental importance in the remainder of this paper.

  \begin{definition}
   Let $i \in \{0, 1, 2\}$, and let $T$ denote one of the symbols $\{ m, e, \alpha\}$, signifying a group scheme that is either multiplicative, \'etale, or $\alpha_p$. We write $Y_{\Kl, 0}^{i, T}$ for the locus in $Y_{\Kl, 0}$ parametrising $(A, C)$ such that $A$ has $p$-rank $i$, and $C$ is \'etale-locally of type $T$.
  \end{definition}

  Note that only six of the possible combinations $(i, T)$ correspond to non-empty strata; if $i = 2$ then $T$ has to be either $m$ or $e$, and if $i = 0$ then $T$ must be $\alpha$.

  \begin{theorem}\label{thm:Klingenstrata1}
   The loci $Y_{\Kl, 0}^{i, T}$ for varying $i$ and $T$ are locally-closed subvarieties forming a stratification of $Y_{\Kl, 0}$, with the closure relation given by the diagram\footnote{Here we follow the notation of \cite{shenyuzhang21}, that a chain of arrows from stratum $A$ to stratum $B$ indicates that $A$ is contained in the closure of $B$.}
   \[
    \begin{tikzcd}
                 & (1, m)     \rar & (2, m)\\
     (0, \alpha) \ar[r]\ar[ru]\ar[rd] & (1, \alpha) \ar[ru] \ar[rd]\\
                 & (1, e)     \rar & (2, e).
    \end{tikzcd}
   \]
   The dimension of $Y^{i, T}_{\Kl, 0}$ is $1 + i$; and all of the strata are smooth except the $(0, \alpha)$ stratum (which is a union of $\mathbf{P}^1$'s intersecting transversely).
  \end{theorem}

  See \cref{fig:strata} for a visual representation of the strata and their intersections.

  \begin{proof}
   Almost all of these statements can be extracted from the analysis of the EKOR stratification for parahoric subgroups of $\GSp_4$ carried out in \cite[\S 6.3]{shenyuzhang21}; case (3) of \emph{op.cit.} is the Klingen parahoric. The EKOR stratification in this setting is slightly more refined than our stratification above, since the $(0, \alpha)$ stratum (i.e.~the supersingular locus) is decomposed into three strata labelled $\tau$, $s_0 \tau$, and $s_1 \tau$ in \emph{op.cit.}. However, since the union of these is closed and is contained in the closure of every other stratum (according to diagram (6.3.4) of \emph{op.cit.}), amalgamating these together gives a stratification with the above closure relation.
  \end{proof}

  We can consider the six-element set $\mathcal{S}$ of indices $(i, T)$ as a finite topological space, with the topology defined by the diagram of Theorem \ref{thm:Klingenstrata1}, so that the natural map $Y_{\Kl, 0} \to \mathcal{S}$ sending $Y^{(i, T)}_{\Kl, 0}$ to $(i, T)$ is continuous. For $J \subseteq \cS$, let $Y_{\Kl, 0}^J$ be the corresponding union of strata in $Y_{\Kl, 0}$.

  For each $T \in \cT = \{m, e, \alpha\}$, we write $Y^T_{\Kl, 0} = \bigcup_i Y^{i, T}_{\Kl, 0}$. These also form a (coarser) stratification of $Y_{\Kl, 0}$, indexed by a quotient $\cS'$ of $\cS$, with a unique closed stratum $T = \alpha$ and two open strata; as explained in \cite{shenyuzhang21}, it is a truncated form of the Kottwitz--Rapoport stratification, with the supersingular strata amalgamated.

  \begin{theorem}\label{thm:Klingenstrata}
   The scheme $Y_{\Kl}$ is strictly semistable over $\Zp$. Its special fibre is the union of two closed, smooth subvarieties, $\overline{Y_{\Kl, 0}^m}$ and $\overline{Y_{\Kl, 0}^e}$, intersecting transversely along $Y_{\Kl, 0}^\alpha$, which is a smooth surface.
  \end{theorem}

  \begin{proof}
   See Theorem 3 of \cite{tilouine06}. (Note that our notations are slightly different from Tilouine's: he uses the notation $X_P(p)^m$ for the \emph{closure} $\overline{Y_{\Kl, 0}^m}$ of the multiplicative locus, rather than the multiplicative locus alone, and similarly for $X_P(p)^e$.)
  \end{proof}

  \begin{remark} \
   \begin{enumerate}[(i)]
    \item The above theorem shows that for three of the possible $(i, T) \in \cS$, namely $(i, T) = (2, m), (2, e)$, and $(1, \alpha)$, the closure of $Y^{(i, T)}_{\Kl, 0}$ in $Y_{\Kl, 0}$ is smooth. We do not know if $\overline{Y^{1, m}_{\Kl, 0}}$ or $\overline{Y^{1, e}_{\Kl, 0}}$ is smooth (and $\overline{Y^{0, \alpha}_{\Kl, 0}} = Y^{0, \alpha}_{\Kl, 0}$ never is).
    \item The EKOR strata $Y_{\Kl,0}^{(i, T)}$ are connected, and all except $Y_{\Kl, 0}^{(0, \alpha)}$ are irreducible, by \cite[Proposition 6.3.5]{shenyuzhang21}. In particular $Y_{\Kl, 0}$ has precisely two irreducible components, $\overline{Y_{\Kl,0}^{m}}$ and $\overline{Y_{\Kl,0}^{e}}$ (the closures of the open EKOR strata). \qedhere
   \end{enumerate}
  \end{remark}

  \begin{figure}[ht]
   \caption{Visual representation of the strata in $Y_{\Kl, 0}$. (Since the $\GSp_4$ Shimura variety is 3-dimensional, and we are attempting to draw it on a 2-dimensional page, we have shrunk all the dimensions by one; hence the single point marked ``$0, \alpha$'' actually stands for a curve, etc.)}
   \label{fig:strata}
   \medskip

  \tikzset{every picture/.style={line width=0.75pt}} 

  \begin{tikzpicture}[x=0.75pt,y=0.75pt,yscale=-1,xscale=1]

  \draw   (380,116) -- (380,270) -- (200,204) -- (200,50) -- cycle ;
  \draw   (20,116) -- (20,270) -- (200,204) -- (200,50) -- cycle ;
  \draw    (140,30) -- (200,50) ;
  \draw    (260,30) -- (200,50) ;
  \draw    (140,30) -- (140,70) ;
  \draw    (260,30) -- (260,70) ;
  \draw    (20,170) .. controls (102,105.5) and (122,195.5) .. (200,120) ;
  \draw  [dash pattern={on 4.5pt off 4.5pt}]  (50,70) -- (59.75,148.02) ;
  \draw [shift={(60,150)}, rotate = 262.87] [color={rgb, 255:red, 0; green, 0; blue, 0 }  ][line width=0.75]    (10.93,-3.29) .. controls (6.95,-1.4) and (3.31,-0.3) .. (0,0) .. controls (3.31,0.3) and (6.95,1.4) .. (10.93,3.29)   ;
  \draw  [dash pattern={on 4.5pt off 4.5pt}]  (290,50) -- (201.83,89.19) ;
  \draw [shift={(200,90)}, rotate = 336.04] [color={rgb, 255:red, 0; green, 0; blue, 0 }  ][line width=0.75]    (10.93,-3.29) .. controls (6.95,-1.4) and (3.31,-0.3) .. (0,0) .. controls (3.31,0.3) and (6.95,1.4) .. (10.93,3.29)   ;
  \draw    (200,120) .. controls (254,183) and (351,131) .. (380,180) ;
  \draw  [dash pattern={on 4.5pt off 4.5pt}]  (360,70) -- (291.32,148.49) ;
  \draw [shift={(290,150)}, rotate = 311.19] [color={rgb, 255:red, 0; green, 0; blue, 0 }  ][line width=0.75]    (10.93,-3.29) .. controls (6.95,-1.4) and (3.31,-0.3) .. (0,0) .. controls (3.31,0.3) and (6.95,1.4) .. (10.93,3.29)   ;
  \draw  [fill={rgb, 255:red, 0; green, 0; blue, 0 }  ,fill opacity=1 ] (196,120) .. controls (196,117.79) and (197.79,116) .. (200,116) .. controls (202.21,116) and (204,117.79) .. (204,120) .. controls (204,122.21) and (202.21,124) .. (200,124) .. controls (197.79,124) and (196,122.21) .. (196,120) -- cycle ;
  \draw  [dash pattern={on 4.5pt off 4.5pt}]  (210,244) .. controls (177.49,200.17) and (180.89,164.58) .. (199.16,125.78) ;
  \draw [shift={(200,124)}, rotate = 115.69] [color={rgb, 255:red, 0; green, 0; blue, 0 }  ][line width=0.75]    (10.93,-3.29) .. controls (6.95,-1.4) and (3.31,-0.3) .. (0,0) .. controls (3.31,0.3) and (6.95,1.4) .. (10.93,3.29)   ;

  \draw    (288,28) -- (322,28) -- (322,54) -- (288,54) -- cycle  ;
  \draw (291,32) node [anchor=north west][inner sep=0.75pt]   [align=left] {$\displaystyle 1,\alpha $};
  \draw    (348,44) -- (379,44) -- (379,70) -- (348,70) -- cycle  ;
  \draw (351,48) node [anchor=north west][inner sep=0.75pt]   [align=left] {$\displaystyle 1,e$};
  \draw    (208,238) -- (242,238) -- (242,264) -- (208,264) -- cycle  ;
  \draw (211,242) node [anchor=north west][inner sep=0.75pt]   [align=left] {$\displaystyle 0,\alpha $};
  \draw    (35,46) -- (72,46) -- (72,72) -- (35,72) -- cycle  ;
  \draw (38,50) node [anchor=north west][inner sep=0.75pt]   [align=left] {$\displaystyle 1,m$};
  \draw    (48,208) -- (85,208) -- (85,232) -- (48,232) -- cycle  ;
  \draw (51,212.4) node [anchor=north west][inner sep=0.75pt]    {$2,m$};
  \draw    (318,208) -- (349,208) -- (349,232) -- (318,232) -- cycle  ;
  \draw (321,212.4) node [anchor=north west][inner sep=0.75pt]    {$2,e$};

  \end{tikzpicture}\end{figure}

 \subsection{Compactifications}

  Let $X_{\Kl}$ be a toroidal compactification of $Y_{\Kl}$ (for some polyhedral cone decomposition $\Sigma$, which we suppose to be ``good'' in the sense of \cite[\S 6.1.5]{pilloni20}). Write $D$ for the boundary divisor of the toroidal compactification. The moduli interpretation of $X_{\Kl}$ parametrises semiabelian schemes with a Klingen level structure and some appropriate degeneration data at the boundary (depending on $\Sigma$). We shall now extend the above stratifications to $X_{\Kl, 0}$, using the general theory developed in \cite{lanstroh} and \cite{mao-preprint}. (We are grateful for Kai-Wen Lan for his assistance with this section.)

  \begin{theorem}[Lan--Stroh, Mao]
   There exists a stratification
   \[ X_{\Kl, 0} = \bigsqcup_{(i, T) \in \mathcal{S}} X_{\Kl, 0}^{(i, T)} \]
   with the following property: if $J \subset \cS$ is closed, then $X^J_{\Kl, 0} = \bigsqcup_{(i, T) \in J} X_{\Kl, 0}^{(i, T)}$ is the closure of $Y^J_{\Kl, 0}$ in $X_{\Kl, 0}$. Moreover, if $(i, T) \ne (0, \alpha)$, then $X_{\Kl, 0}^{(i, T)}$ is smooth.
  \end{theorem}

  \begin{proof}
   The existence of an extension of the stratification to $X_{\Kl}$ follows from the general theory of compactification of \emph{well-positioned subvarieties} in Shimura varities developed in \cite{lanstroh}, together with a more recent result of Mao \cite[Theorem 1.1]{mao-preprint} showing that the EKOR strata in $X_{\Kl, 0}$ are indeed well-positioned subvarieties. We shall recall these general results and make explicit what they give for the $\GSp_4$ Klingen in the appendix below; see \cref{sect:explcitintersection}. The smoothness follows from \cite[Proposition 2.3.13]{lanstroh}.
  \end{proof}

  Ignoring $i$ gives a coarser stratification
  \[ X_{\Kl, 0} = \bigsqcup_{T \in \mathcal{T}} X_{\Kl, 0}^{T} \]
  with the property that the closure of each $X_{\Kl, 0}^{T}$ is smooth (again by Proposition 2.3.13 of \cite{lanstroh}).

  \begin{proposition}
   The pair $\left(X_{\Kl}, Y_{\Kl}\right)$ is \emph{strictly semistable with boundary} in the sense of \cref{def:SSschemeboundary}.
  \end{proposition}

  \begin{proof}
   We need to check that the union of the toroidal boundary $X_{\Kl}-Y_{\Kl}$ and the special fibre $X_{\Kl,0}$ is a strict normal crossing divisor. By Stacks Project Tag 0BIA, it suffices to prove that each boundary component of $X_{\Kl, 0}$ intersects each of $X_{\Kl, 0}^m$, $X_{\Kl, 0}^e$, and $X_{\Kl, 0}^{\alpha}$ in a smooth subvariety of the appropriate codimension. As shown in  \cref{sect:explcitintersection}, each of these intersections is either empty, or is equal to the preimage in the universal elliptic curve of an intersection of irreducible components in a modular curve (of prime-to-$p$ or $\Gamma_0(p)$ level) and from this explicit description one can readily verify that the intersections are indeed smooth of the expected dimension.
  \end{proof}

 \subsection{Refined Klingen level structures}

  In this section we shall define an integral model of the Shimura variety with level at $p$ given by
  \[ \breve{\Kl}(p) = \{ g \in G(\Zp): g =
  \begin{smatrix} 1  & \star  & \star  & \star \\ & \star  & \star  & \star \\ & \star  & \star  & \star  \\ &&&1 \end{smatrix} \bmod p\}. \]
  This is naturally a variety over $\QQ(\zeta_p)$, and we shall work over the ring $R = \Zp[\zeta_p]$.

  \begin{remark}
   The refined level structures introduced for $G$ here, and for $H$ in \cref{ss:iotaproperties} below (and their integral models over $\Zp[\zeta_p]$) -- all the objects marked with a ``breve'' accent $\breve{X}, \breve{Y}$ etc -- will only play a role in step 2 of our argument; they will not reappear after the end of \cref{sect:redstep2}. On the other hand, the ``usual'' Klingen level structures will play a major role throughout the argument.
  \end{remark}

  \begin{definition}
   We denote by $\breve{Y}_{G, \Kl}$, or simply $\breve{Y}_{\Kl}$, the moduli space of quadruples $(A, C, P, Q)$ over $R$, where $(A, C)$ is as described for $Y_{\Kl}$ above, $P$ is a generator of $C$, and $Q$ is a generator of $A[p] / C^\perp \cong C^\vee$, subject to the condition that $\langle P, Q\rangle = \zeta_p$.
  \end{definition}

  This is the analogue for $\GSp_4$ of the ``balanced level $\Gamma_1(p)$ structures'' considered in \cite[\S 3.3]{katzmazur85} for $\GL_2$. Exactly as in the $\GL_2$ case treated in \cite[\S V.2]{delignerapoport73} (or the analogous result for unitary Shimura varieties at split primes in \cite[\S 3.3]{harristaylor02}), a computation using Tate--Oort theory shows that this moduli problem is indeed represented by an $R$-scheme, and that this scheme is strictly semistable over $R$; and its special fibre has a stratification with three strata
  \[ \breve{Y}_{\Kl,0} = \breve{Y}_{\Kl,0}^e \sqcup \breve{Y}_{\Kl,0}^m \sqcup \breve{Y}_{\Kl,0}^{\alpha}\]
  depending on the type of $C$, exactly as for $Y_{\Kl}$ above.


  For a suitable choice of cone decomposition $\breve\Sigma$, we obtain a toroidal compactification $\breve{X}_{\Kl}$ of $\breve{Y}_{\Kl}$, which is a proper $R$-scheme; and the same analysis as before shows that $\breve{X}_{\Kl}$ is strictly semistable over $R$, and $(\breve{X}_{\Kl}, \breve{Y}_{\Kl})$ is strictly semistable with boundary. We may suppose that $\breve\Sigma$ is a refinement of the cone-decomposition $\Sigma$ we used to define $X_{\Kl}$, so we obtain a natural map $\breve{X}_{\Kl} \to X_{\Kl}$.

  \begin{remark}
   The $p$-adic completion of the open subscheme of $\breve{X}_{G, \Kl}$ on which $C$ is multiplicative coincides with the base-extension to $R$ of the first layer in the Igusa tower $\mathfrak{IG}_G$ considered in \cite{pilloni20,LPSZ1}.
  \end{remark}

\section{Integral models for \texorpdfstring{$H$}{H}}

 Let us now fix a level $K^p_H$ away from $p$ (which we again suppose to be neat), so we have a $\Zp$-model $Y_H$ of the Shimura variety for $H$ of level $K^p_H \cdot H(\Zp)$, which represents pairs $(E_1, E_2)$ of elliptic curves with appropriate prime-to-$p$ level structures.

 \subsection{An integral model of $Y_{H, \Delta}$}

  Let $Y_{H}' \to Y_H$ denote the moduli space of pairs $(C_1, C_2)$, where each $C_i$ is a $\Gamma_0(p)$ level structure on $E_i$ (i.e.~a finite flat subgroup-scheme of rank $p$). Via Tate--Oort theory, each $C_i$ corresponds to a triple $(L_i, a_i, b_i)$, with $L_i$ a line bundle and $a_i \in L_i^{\otimes (p-1)}$, $b_i \in L_i^{\otimes (1-p)}$ sections such that $a_i b_i = w_p$ (where $w_p \in \Zp$ is a certain element of valuation 1); the locus where $C_i$ is \'etale (resp.~multiplicative) is given by $a_i \ne 0$ (resp.~$b_i \ne 0$).

  The choices of rank $p$ finite flat subschemes $C \subseteq C_1 \times C_2$ correspond to rank 1 direct summands $L \subseteq L_1 \oplus L_2$. The condition for such a subscheme to be a subgroup is given by a compatibility with the $a_i$ and $b_i$; if we choose bases of the $L_i$ locally on $X_{H}'$, so $L$ corresponds to some $(f_1: f_2) \in \mathbf{P}^1$, then we can write these equations in the form
  \[
   \{ (f_1 : f_2) \in \mathbf{P}^1\ |\ a_1 f_1 f_2^{p} = a_2 f_1^{p} f_2,\ b_1 f_1^p f_2 = b_2 f_1 f_2^p \}.
  \]
  This is the union (not disjoint!) of three closed subschemes, defined by $\{ f_1 = 0\}$, $\{ f_2 = 0\}$, and a third subscheme $\{ a_1 f_2^{(p-1)} = a_2 f_1^{(p-1)}, b_1 f_1^{(p-1)} = b_2 f_2^{(p-1)} \}$. Away from the special fibre, these three subschemes are disjoint, parametrising subgroups which are respectively equal to $C_2$, equal to $C_1$, and mapping isomorphically to both factors.

  \begin{definition}
   We let $Y_{H, \Delta}$ denote the subscheme of $\mathbf{P}(L_1 \oplus L_2)$ cut out by the equations $\{ a_1 f_2^{(p-1)} = a_2 f_1^{(p-1)}, b_1 f_1^{(p-1)} = b_2 f_2^{(p-1)} \}$.
  \end{definition}

  By construction this is a $\Zp$-model of the variety $Y_{H, \Delta, \Qp}$, and the local equations above show that it is regular (but its special fibre is not reduced, so it is in particular not semistable). If our levels are chosen compatibly, i.e.~$K^p_H \subseteq \iota^{-1}(K^p_G)$, then the map sending $(E_i, C_i, C)$ to $(E_1 \oplus E_2, C)$ defines a morphism of $\Zp$-schemes
  \[ \iota_{\Delta}: Y_{H, \Delta} \to Y_{G, \Kl}\]
  extending the map of \eqref{eq:iotaDelta} on the generic fibre. Since $Y_H'$ is finite over $Y_H$, and $Y_{H, \Delta}$ is by construction a closed subscheme of $Y_H' \times_{Y_G} Y_{G, \Kl}$, the map $\iota_{\Delta}$ is finite. (Note that $Y_{H, \Delta}$ is not finite over $Y_H$ or $Y_H'$.)

  \begin{definition}
   We let $X_{H, \Delta}$ be the normalisation of $X_{G, \Kl}$ in $Y_{H, \Delta}$, so that $\iota_{\Delta}$ extends to a finite morphism $X_{H, \Delta} \to X_{G, \Kl}$.
  \end{definition}

  This is an integral model of the toroidal compactification $X_{H, \Delta, \QQ}$ given by the cone-decomposition $\Sigma_H= \iota^{-1}(\Sigma_G)$.

  \begin{proposition}\label{prop:preimagemult}
   The preimage under $\iota_{\Delta}$ of the open subset $X_{G, \Kl, 0}^m \subset X_{G, \Kl, 0}$ where $C$ is multiplicative is the open subset $X_{H, \Delta, 0}^{(m, m)}$ where both of the $C_i$ are multiplicative. Moreover, $\iota_{\Delta}^{-1}\left(X_{G, \Kl, 0}^{(1, m)} \right) = \varnothing$.
  \end{proposition}

  \begin{proof}
   This is closely related to \cite[Prop. 4.6]{LPSZ1}. We give a slightly different proof using the notations introduced above: if the Tate--Oort parameters of $C$ are $(L, a, b)$, then $b$ is given in the local coordinates above by the relation $b = f_1^{(p-1)} b_1 = f_2^{(p-1)} b_2$. So if $b \ne 0$, then we must have both $b_1 \ne 0$ and $b_2 \ne 0$. In particular, the $p$-rank of $E_1 \oplus E_2$ cannot be 1.
  \end{proof}

 \subsection{Compatible semistable models and the map \texorpdfstring{$\iota$}{iota}}\label{ss:iotaproperties}

  We shall also need a related Shimura variety for $H$, which has a natural semistable model over $R = \Zp[\zeta_p]$.

  \begin{definition}
   We let $\breve{Y}_{H, \Delta}$ denote the moduli space of data
   \[ (E_1, P_1, Q_1; E_2, P_2, Q_2; C)\]
   over $R$, where $(E_i, P_i, Q_i)$ are elliptic curves with balanced level $\Gamma_1(p)$ structures, and $C$ is a choice of order $p$ subgroup-scheme inside $C_1 \times C_2$ such that $(P_1, P_2)$ generates $C$ and $(Q_1, -Q_2)$ annihilates $C$.
  \end{definition}

  One can obtain local equations for $\breve{Y}_{H, \Delta}$ using Tate--Oort theory, starting with the description of the moduli space $Y_1(p)^{\mathrm{bal}}$ for balanced $\Gamma_1(p)$ structures given in \cite{delignerapoport73}. The computation is similar to the previous one, but simpler: locally the $(C_i, P_i, Q_i)$ are parametrised by triples $(L_i, u_i, v_i)$ with $L_i$ a line bundle and $u_i \in L_i$, $v_i \in L_i^\vee$ sections satisfying $u_i v_i = w_\zeta$ (where $w_\zeta$ is a certain uniformizer of $R$ such that $(w_\zeta)^{p-1} = w_{p-1}$). If we choose local bases so $C$ corresponds to some $(f_1, f_2) \in \mathbf{P}^1$, then the conditions on $C$ translate into the equations
  \[ u_i v_i = w_{\zeta}, \qquad  u_1 f_2 = u_2 f_1, \quad v_1 f_1 = v_2 f_2. \]
  Over the open subset $f_1 \ne 0$ we can set $t = f_2 / f_1$ and obtain the local equation $u_1 v_2 t = w_{\zeta}$, where $t = f_2 / f_1$, and similarly on the open $f_2 \ne 0$; thus $\breve{Y}_{H, \Delta}$ is semistable.

  Our definitions of moduli-space structures are evidently chosen so that for compatible tame levels $K^p_H$ and $K^p_G$, we obtain a natural map of $R$-schemes
  \[ \breve{\iota}_{\Delta}: \breve{Y}_{H, \Delta} \to \breve{Y}_{G, \Kl}, \]
  extending the map $\iota$ above; this is given by mapping $(E_1, \dots)$ to $(A, C, P, Q)$ where $A = E_1 \oplus E_2$, $P = (P_1, P_2)$, and $Q$ the image of $Q_1$ (or $Q_2$) mod $C^\perp$. Since we can reconstruct $\breve{Y}_{H, \Delta}$ as a subscheme of the fibre product $Y_H \times_{Y_G} Y_{G, \Kl}$, we see that this morphism $\breve{\iota}_{\Delta}$ is finite. As in the treatment of $X_{H, \Delta}$ above, we may compactify this to obtain a finite map of $R$-schemes
  \[ \breve\iota_{\Delta}: \breve{X}_{H, \Delta} \to \breve{X}_{ \Kl} \]
  where $\breve{X}_{H, \Delta}$ is an integral model of the toroidal compactification $\breve{X}_{H, \Delta, \QQ}$ corresponding to the cone-decomposition $\breve{\Sigma}_H = \iota^{-1}\left(\breve{\Sigma}_G\right)$.

\section{Cohomology classes from \texorpdfstring{$\Pi$}{Pi}}
 \emph{In this section the group $H$ does not appear, so we shall omit subscripts $G$.}

 \subsection{Rigid classes on the multiplicative locus}
  \label{sect:rigidfromPi}

  Let $\eta_{\dR}$ be the de Rham cohomology classx described in \cref{ssec:testdataatp} above (determined by $\Pi$, $p$, the $\nu$ of \cref{def:nudR}, and some choice of prime-to-$p$ Whittaker function $w^p$). Then we have $\eta_{\dR} \in \Fil^1 H^3_{\dR,c}(Y_{ \Kl, \Qp}, \cV)$. Forgetting the filtration information, we can consider it as a class
  \[
   \eta_{\lrig,-D} \in H^3_{\lrig,c}(Y_{\Kl, \Qp}, \cV) \cong H^3_{\dR}\left( \cX_{\Kl}\langle -D \rangle, \cV\right).
  \]
  Moreover, if we equip this with the Frobenius transported from log-rigid cohomology, then we have the identity $\cQ(\varphi)(\eta_{\lrig,-D}) = 0$, since the comparison isomorphisms between log-rigid cohomology and $\mathbf{D}_{\st}$ of \'etale cohomology are compatible with the Frobenius action.

  \begin{proposition}
   \label{prop:BPrigclass}
   There exists a unique class
   \[ \eta_{\rig, -D}^m \in H^3_{\dR, c}\left( \cX_{\Kl}^m\langle -D \rangle, \cV\right)\]
   with the following properties:
   \begin{enumerate}[(i)]
    \item The image of $\eta_{\rig,-D}^m$ under the extension-by-0 map is $\eta_{\lrig, -D}$.
    \item The class $\eta_{\rig,-D}^m$ is an eigenvector for the operators $U_{1, \Kl}'$ and $U'_{2, \Kl}$, with eigenvalues $\alpha + \beta$ and $p^{-(r_2+1)} \alpha\beta$ respectively.
    \item We have $\cQ(\varphi) \eta_{\rig,-D}^m = 0$, where $\varphi$ is the Frobenius of rigid cohomology.
    \item The class $\eta_{\rig,-D}^m$ lies in the $\Pif'$-eigenspace for the spherical Hecke operators.
   \end{enumerate}
   Moreover, if $\xi$ is any lifting of $\eta_{\lrig, -D}$ to $H^3_{\dR, c}\left( \cX_{\Kl}^m\langle -D \rangle, \cV\right)$ which is a \emph{generalised} eigenvector for the operators $U_{1, \Kl}'$ and $U_{2, \Kl}'$ with eigenvalues as in (ii), then we must have $\xi = \eta_{\rig,-D}^m$.
  \end{proposition}

  The proof of this proposition will be deferred until \cref{sect:BPproof} below, since it uses the formalism of partially-compactly-supported cohomology which we shall develop in \cref{sect:partialsupport}. We also need a corresponding statement at level $\breve{\Kl}(p)$: the pullback of $\eta_{\rig}^m$ defines a class
  \[ \breve\eta_{\rig,-D}^m \in H^3_{\rig, c}\left( \breve{X}_{\Kl, 0}^m\langle -D \rangle, \cV\right)^{\cQ(\varphi) = 0}, \]
  whose image under cospecialisation is $\breve{\eta}_{\dR}$, the pullback of $\eta_{\dR}$ to level $\breve{\Kl}(p)$. Since cospecialisation for the full variety $\breve{X}_{\Kl, 0}\langle -D \rangle$ is an isomorphism, it follows that the image of $\breve\eta_{\rig,-D}^m$ under the pushforward map of \cref{prop:extbyzero} for the inclusion $\breve{X}_{\Kl, 0}^m \into \breve{X}_{\Kl, 0}$ is $\breve\eta_{\lrig,-D}$.

  \begin{remark}
   The proof of the above statement requires a little care, as we cannot directly compare the log-rigid cohomology of $X_{\Kl, 0}$ and $\breve{X}_{\Kl, 0}$: they are not both semistable over the same ring. This can probably be circumvented by the use of coverings by $\Zp[[T]]$-schemes, as in Proposition 3.11 of \cite{nekovarniziol16}; but since the classes of interest are pushforwards from a smooth open subscheme, we can avoid this and instead appeal to (much easier) base-extension results for rigid cohomology of smooth schemes.
  \end{remark}

 \subsection{Log-rigid fp-cohomology classes}
  \label{sect:logrigfromPi}
  Recall the class $\eta_{\NNfp}$ defined in \cref{lem:etaNNfp} above. As we have seen, the variety $Y_{\Kl, \Qp}$ has a model as a strictly semistable scheme with boundary over $\Zp$, so we may compute \Nek--\Niz cohomology using this model.

  \begin{notation} Write $\eta_{\lrigfp,-D}$ for the image of $\eta_{\NNfp}$ in $H^3_{\lrigfp}(X_{\Kl}\langle -D\rangle,\cV, 1+q; \cQ_{1+q})$.
  \end{notation}

  We write $\breve\eta_{\NNfp}$ for the pullback of $\eta_{\NNfp}$ to the $K$-variety $\breve{Y}_{\Kl, K}$, where $K = \Qp(\zeta_p)$. Repeating the above constructions using the semistable model of $\breve{Y}_{\Kl, K}$ over $\cO_K$, we obtain similarly log-rigid-fp classes
  \[ \breve\eta_{\lrigfp, -D} \in H^3_{\lrigfp, c}(\breve{X}_{\Kl}\langle -D\rangle,\cV,1+q;\cQ_{1+q}).\]

 \subsection{Partial integral models}
  \label{sect:partialmodels}

  Recall the stratification of $X_{\Kl, 0}$ as the union of multiplicative, \'etale, and $\alpha_p$ strata. We define
  \[ X_{\Kl}^m = X_{\Kl} - \left( X_{\Kl,0}^e \cup X_{\Kl,0}^\alpha\right) = X_{\Kl, \Qp} \cup X_{\Kl, 0}^m.\]
  This is an open subscheme of $X_{\Kl}$ which is smooth over $\Zp$, and its generic fibre is the same as that of $X_{\Kl}$. (It is, of course, not proper over $\Zp$.) We regard it as a \emph{partial $\Zp$-model} of $X_{\Kl, \Qp}$. The same remarks apply to $\breve{X}^m_{\Kl}$, which we interpret as a partial $\cO_K$-model of $\breve{X}_{\Kl, K}$.

  \begin{remark}
   We shall consider the rigid fp-cohomology $R\Gamma_{\rigsyn}(X_{\Kl}^m\langle D \rangle, \cV, 1+q; \cQ_{1+q})$, where $D$ is the toroidal boundary divisor. Note that the ``de Rham'' term in the mapping fibre defining this cohomology is just the usual de Rham cohomology for the smooth proper $\Qp$-variety $X_{\Qp}$, but the ``rigid'' terms only detect the multiplicative locus. However, the inclusion of $X_{\Kl, 0}^m$ into $X_{\Kl}$ gives a smooth proper frame for $X_{\Kl, 0}^m$, so these rigid-cohomology terms are computed by the de Rham cohomology of the dagger space tube $\cX_{\Kl}^m$.

   Hence the specialisation map ``$\mathrm{sp}$'' for this scheme corresponds to the pullback map in de Rham cohomology (with log poles along $D$) from $\cX_{\Kl}$ to its open dagger subspace $\cX_{\Kl}^m$; in particular we should \textbf{not} expect this map to be a quasi-isomorphism. Similar remarks apply to the cospecialisation map for cohomology with compact supports.
  \end{remark}

 \subsection{Rigid fp-cohomology classes on the multiplicative locus}

  We now define classes in the fp-cohomology of these partial integral models, combining the results of \cref{sect:rigidfromPi,sect:logrigfromPi}.

  Using \cref{prop:extbyzero} in the case where $X = X_{\Kl}$, $D$ is the toroidal boundary (so $U=Y_{\Kl}$) and $Z=X_{\Kl}^{m}$, for any $Q(T)\in 1+T\Qp[T]$ we obtain an extension-by-zero map
  \[
   H^3_{\rigfp, c}(X_{\Kl}^{m}\langle -D\rangle, \cV, n;Q)\to  H^3_{\lrigfp}(X_{\Kl}\langle -D\rangle, \cV, n;Q).
  \]

  \begin{proposition}\label{prop:eta-rigsyn}
   Let $\eta_{\rig,-D}^m$ be as in \cref{prop:BPrigclass}. If we take $Q(t) = \cQ_{1+q}(t)$, then we may find a class
   \[ \eta_{\rigfp,-D}^{m} \in H^3_{\rigfp, c}(X_{\Kl}^{m}\langle -D\rangle, \cV, 1+q;\cQ_{1+q})[\Pif'] \]
   whose image in rigid cohomology is $\eta_{\rig,-D}^{m}$, and whose image under the extension-by-zero map above is $\eta_{\lrigfp,-D}$.
  \end{proposition}

  \begin{proof}
   The pair $(\eta_{\dR}, \eta^m_{\rig,-D})$ satisfies $\cQ_{1+q}(\varphi_{1+q})(\eta^m_{\rig}) = 0$ and $\cosp(\eta^m_{\rig}) = \eta_{\dR}$, so we can lift this pair to a class in fp-cohomology. Moreover, we may assume this lifting lies in the $\Pif'$-generalised eigenspace for the spherical Hecke operators (since this is true of $\eta_{\dR}$ and $\eta^m_{\rig,-D}$). By construction, its image under the extension-by-zero map is a class in the $\Pif'$ generalised eigenspace, whose image in de Rham cohomology is $\eta_{\dR}$, and whose image in log-rigid cohomology is $\eta_{\lrig, -D}$. These proerties uniquely characterise $\eta_{\lrigfp,-D}$.
  \end{proof}

  \begin{remark}
   Note that we are not asserting that $\eta_{\rigfp,-D}^{m}$ is uniquely determined by the stated conditions; only that such a class exists. The conditions above determine it modulo an element of the group
   \[ H^2_{\rig, c}\left(X_{\Kl,0}^m\langle -D\rangle, \cV\right)[\Pif'] / \cQ(\varphi).\]
   We expect that in fact $H^2_{\rig, c}(X_{\Kl,0}^m\langle -D\rangle, \cV)[\Pif'] = 0$; this was part of the ``eigenspace vanishing conjecture'' assumed in previous iterations of the present work. However, in the present account we do not need to know if this holds.
  \end{remark}

  \begin{notation}\label{not:breve-eta-rig}
   We let $\breve\eta_{\rigfp,-D}^m$ be the pullback of $\eta_{\rigfp,-D}^{m}$ to $\breve{X}_{\Kl}^m$.
  \end{notation}

\section{Reduction step 2: Regulators via rigid syntomic cohomology}
 \label{sect:redstep2}

 \subsection{Step 2a: relating \Nek--\Niz and log-rigid pairings}

  \begin{proposition}\label{prop:lrigreduction}
    The pairing \eqref{eq:step1}  is equal to
    \[ \tfrac{1}{(p+1)} \left\langle \Eis^{[t_1,t_2]}_{\lrigsyn,\underline\Phi},\, (\breve\iota_\Delta^{[t_1,t_2]})^\star(\breve\eta_{\lrigfp,-D})\right\rangle_{\lrigfp,\breve{X}_{H,\Delta}}.\]
  \end{proposition}

  \begin{proof}
   The pairing on the right-hand side of \eqref{eq:step1} is defined using the $\Qp$-variety $Y_{G, \Kl, \Qp}$; but the pairings for $Y_{G, \Kl, \Qp}$ and $Y_{G, \Kl, K}$, where $K = \Qp(\zeta_p)$, are compatible via the natural embedding $L \into L \otimes_{\Qp} K$ (see \cref{note:baseext} above). Moreover, the duality pairings for $Y_{G, \Kl, K}$ and $\breve{Y}_{G, \Kl, K}$ are compatible up to the factor $(p+1)$, as this is the degree of the finite map $\breve{Y}_{G, \Kl, K} \to Y_{G, \Kl, K}$. Thus we have
   \begin{align*}
    \left\langle \iota_{\Delta, \star}^{[t_1,t_2]}( \Eis^{[t_1,t_2]}_{\syn,\underline\Phi}),\,  \eta_{\NNfp}\right\rangle_{\NNfp,Y_{G,\Kl, \Qp}}
    &= \tfrac{1}{(p+1)} \left \langle\breve{\iota}_{\Delta, \star}^{[t_1,t_2]}( \Eis^{[t_1,t_2]}_{\syn,\underline\Phi}),\,  \breve\eta_{\NNfp, -D} \right\rangle_{\NNfp,\breve{Y}_{G,\Kl, K}} \\
    &= \tfrac{1}{(p+1)} \left\langle \Eis^{[t_1,t_2]}_{\syn,\underline\Phi},\,  (\breve\iota_{\Delta}^{[t_1,t_2]})^\star \left( \breve\eta_{\NNfp, -D}\right) \right\rangle_{\NNfp,\breve{Y}_{H,\Delta, K}} \\
    &= \tfrac{1}{(p+1)} \left\langle \Eis^{[t_1,t_2]}_{\syn,\underline\Phi},\,  (\breve\iota_{\Delta}^{[t_1,t_2]})^\star \left( \breve\eta_{\lrigfp, -D}\right) \right\rangle_{\lrigfp,\breve{X}_{H,\Delta}},
   \end{align*}
   where the second equality uses uses adjunction between pushforward and pullback, and the final one the compatibility of pullback and cup-product maps with the isomorphisms between \Nek--\Niz and log-rigid fp-cohomology.
  \end{proof}

 \subsection{Step 2b: restriction to the multiplicative locus}

  \label{sect:reduction2}
  We now apply \cref{cor:rigfppairing} to relate pairings over $\breve{X}_{G,\Kl}$ and over its multiplicative locus.

  \begin{notation}
   Write $\Eis^{[t_1,t_2],(m,m)}_{\rigsyn,\underline\Phi}$ for the image of $\Eis^{[t_1,t_2]}_{\rigsyn,\underline\Phi}$ in $H^2_{\rigfp}\left(\breve{X}_{H,\Delta}^{(m, m)}\langle D \rangle,\cV_H^\vee ,2\right)$ under the restriction map.
  \end{notation}

  \begin{proposition}
   Let $\breve\eta^{ m}_{\rigfp,-D}$ be as in \cref{not:breve-eta-rig}. Then we have
   \[
    \left\langle \Eis^{[t_1,t_2]}_{\lrigsyn,\underline\Phi}, \,
    (\breve\iota_\Delta^{[t_1,t_2]})^\star(\breve\eta_{\lrigfp,-D})
    \right\rangle_{\lrigfp,\breve{X}_{H,\Delta}}
    = \left\langle
    \Eis^{[t_1,t_2],(m,m)}_{\rigsyn, \underline\Phi},\,
    (\breve\iota_\Delta^{[t_1,t_2]})^\star(\breve\eta^{m}_{\rigfp,-D})
    \right\rangle_{\rigfp,\breve{X}_{H,\Delta}^{(m, m)}}.
   \]
  \end{proposition}

  \begin{proof}
   This is precisely the result of \cref{cor:rigfppairing} in our specific case.
  \end{proof}

  We can now complete the second main step of our argument:\bigskip

  \begin{mdframed}{\bf 2nd reduction}
  \begin{theorem}\label{thm:redtoord}
   Let
   \[
    \eta_{\rigfp, -D}^{m} \in H^3_{\rigfp, c}\left(X_{G,\Kl}^{m}\langle -D\rangle, \cV_G, 1+q;\cQ\right)[\Pif']
   \]
   be the class of \cref{prop:eta-rigsyn}. Then the pairing \eqref{eq:step1} is equal to
   \begin{align*}
    \left\langle \Eis^{[t_1,t_2],(m,m)}_{\rigsyn,\underline\Phi},\, (\iota_\Delta^{[t_1,t_2]})^\star( \eta_{\rigfp,-D}^{m})\right\rangle_{\rigfp, X^{(m, m)}_{H,\Delta}}.
   \end{align*}
  \end{theorem}
  \end{mdframed}

  \begin{proof}
   This follows from the previous proposition, together with the observation that $\breve\eta^{m}_{\rigfp,-D}$ is the pullback of $\eta_{\rigfp,-D}^{m}$.
  \end{proof}

  \begin{remark}
   At this point we may wave goodbye to the ``breve'' objects $\breve{\eta}$ etc; they will not be used again in this paper.
  \end{remark}

\mychapter{Interlude: Partial compact support}


\section{Cohomology with partial support}
\label{section:prelimrigcohom}

 We recall some basic formalism regarding cohomology of sheaves on rigid spaces, following \cite{lestum07} and \cite{grossekloenne00}, and define variants with ``partial compact support''. Let $K$ be a finite extension of $\Qp$, with residue field $k$ and ring of integers $\cO_K$. We shall use Roman letters $X, Y, Z$ etc for $\cO_K$-schemes (including $k$-varieties), Fraktur letters $\mathfrak{X}, \mathfrak{Y}, \dots$ for formal schemes over $\cO_K$, and calligraphic letters $\cX, \cY, \dots$ for rigid-analytic dagger spaces over $K$. The subscript $X_0$ denotes the special fibre of $X$, and similarly for formal schemes.

 \subsection{A motivational remark}

  As a guide to the reader, we point out that for a smooth proper curve $C / \cO_K$, and a section $P \in C(\cO_K)$, we can consider four different complexes of rigid-analytic differential forms associated to $(C, P)$:
  \begin{enumerate}[(i)]
  \item the complex of overconvergent differential forms on $C_K^{\mathrm{an}} - \tb{P_0}$ (the complement of the residue disc of $P$);
  \item the complex $\Omega^\bullet_{C_K^{\mathrm{an}} / K} \langle P \rangle$ of differential forms with log poles at $P$, i.e.~$[\cO_{C_K^{\mathrm{an}}} \longrightarrow \Omega^1_{C_K^{\mathrm{an}}}(P)]$;
  \item the complex of differential forms vanishing logarithmically along $P$, i.e.~$[\cO_{C_K^{\mathrm{an}}}(-P) \longrightarrow \Omega^1_{C_K^{\mathrm{an}}}]$'
  \item the complex of differential forms which vanish identically on $\tb{P_0}$.
  \end{enumerate}
  We have ordered these from ``worst'' to ``best'' behaviour along $P$ in some sense. The complexes (i) and (ii) are quasi-isomorphic, and compute rigid (resp.~de Rham) cohomology of $C - P$; meanwhile, (iii) and (iv) are quasi-isomorphic and compute de Rham (resp.~rigid) cohomology with compact supports.

  As noted in \cite{LSZ-asai}, given a curve $C$ and a finite set of disjoint sections $P, P', P'', \dots$, one can mix and match the above support conditions to define cohomology groups with compact support towards some of the $P$'s but not others. The aim of this section is to describe analogous ``mixed support'' cohomology groups in the more general setting where the ambient space can have dimension $> 1$ and the boundary components are not assumed to be disjoint.

  \begin{remark}
   Our treatment is strongly motivated by \cite[\S 4.2]{deligneillusie87}, where such a theory is developed for de Rham cohomology in characteristic 0. See also \cite[\S III]{faltings89} for \'etale cohomology, \cite[\S 2]{mieda09} for Hyodo--Kato cohomology, and \cite{browndupont} for Hodge cohomology of varieties over $\CC$.
  \end{remark}
 \subsection{Frames and tubes}

  Recall that a \emph{frame} denotes the data of a triple $(X \into Y \into \fP)$, where $X$ and $Y$ are $k$-varieties, $\fP$ is a formal $\cO_K$-scheme, $X \into Y$ is an open immersion, and $Y \into \fP$ is a closed immersion of $Y$ into $\fP$, necessarily factoring through the special fibre $\fP_0$ \cite[Definition 3.1.5]{lestum07}.

  \begin{note}
   We shall always assume $\fP$ is an admissible formal scheme, and thus in particular quasi-compact (this is automatically satisfied if $\fP$ is the $p$-adic completion of a finite-type flat $\cO_K$-scheme).
  \end{note}

   \begin{definition}
   The frame $(X \into Y \into \fP)$ is said to be \emph{smooth} if $\fP$ is smooth over $\cO_K$ in a neighbourhood of $X$ (Definition 3.3.5 of \emph{op.cit.}); it is said to be \emph{proper} if $Y$ is proper over $k$ (Definition 3.3.10).
  \end{definition}

  The theory is typically only well-behaved for smooth proper frames; note that this does not imply that $Y$ is smooth, or that $X$ is either smooth or proper.


  If $(X\into Y \into \fP)$ is a frame, then the tube $\tb{X}_{\fP}$ is an open rigid-analytic subvariety of the analytic generic fibre $\fP_K$. We shall henceforth omit the subscript $\fP$ if it is clear from context. If $X$ is affine and open in $\fP_0$, then $\tb{X}$ is affinoid; it follows that if $X$ is any open subvariety of $\fP_0$, then $\tb{X}$ is quasi-compact.

  If $X$ is not assumed to be open in $\fP_0$, then $\tb{X}$ is no longer quasi-compact. However, it can be written as an increasing union of quasi-compact subsets, the closed tubes $[X]_{\lambda}$ of radius $\lambda < 1$ (which are well-defined if $\lambda > |\varpi_K|$).

 \subsection{Sections with support}

  Let $V$ be any rigid analytic space over $K$, and $T$ an admissible open subset of $V$. Then in the notation of \cite[\S 5.1 and 5.2]{lestum07}, we have a short exact sequence of exact functors on the category of abelian sheaves on $V$,
  \[ 0 \to \uGa^\dag_T  \to \id \to j^\dag_{V-T} \to 0,\]
  where $j^\dag_{V-T}$ is interpreted as ``overconvergent sections on $V-T$''; and there is a left-exact sequence of left-exact functors
  \[ 0 \to \uGa_{V-T} \to \id \to h_\star h^{-1}, \]
  where $\uGa_{V-T}$ denotes sections supported on $V-T$, and $h$ is the inclusion $T \into V$. The second sequence is also exact on the right on injective sheaves, and thus gives an exact triangle of right-derived functors.

  The functors $\uGa$ and $\uGa^\dag$ are related by the following formula. Let us say that $T' \subseteq T$ is a \emph{interior subset} if $\{ T, V-T' \}$ is an admissible covering of $V$ (i.e.~$V-T'$ is a strict neighbourhood of $V-T$). Then we have
  \[ \uGa^\dag_T(\cF) = \varinjlim_{T'}\ \uGa_{T'}(\cF), \]
  where the limit is over interior subsets $T' \subseteq T$. In particular, if $T$ and $T'$ are both admissible open, then there is a natural inclusion $\uGa^\dag_T(\cF) \subseteq \uGa_T(\cF)$ as subsheaves of $\cF$, but this is not an equality (except in the trivial case  when $\{ T, V-T\}$ is an admissible covering that disconnects $V$). It seems reasonable to describe $\uGa^\dag_T(\cF)$ as the sections \emph{strictly} supported in $T$.

  \begin{remark}
   Recall that if $X \into Y \into \fP$ is a proper smooth frame, then the rigid cohomology of $X$ (with and without compact supports) is defined by
   \[ R\Gamma_{\rig}(X) \coloneqq R\Gamma\left(\tb{Y}, j^\dag_{\tb{X}}\, \Omega^{\bullet}_{\tb{Y}}\right), \qquad R\Gamma_{\rig, c}(X) \coloneqq R\Gamma\left(\tb{Y}, R\uGa_{\tb{X}}\, \Omega^{\bullet}_{\tb{Y}}\right),\]
   while the functor $\uGa^\dag$ is used to define rigid cohomology with support in a closed subvariety. See e.g.~\cite{berthelot97} or \cite[Chapter 5]{lestum07} for further details.
  \end{remark}

  \begin{proposition}
   \label{prop:supportfunctors}
   We have $\uGa_{V-T} \circ j^\dag_{V-T} = j^\dag_{V-T}$ and $j^\dag_{V-T} \circ \uGa_{V-T} =  \uGa_{V-T}$.
  \end{proposition}

  \begin{proof}
   By definition of $j^\dag_{V-T}$, we have $h^{-1} \circ j^\dag_{V-T} = 0$ and hence $h_\star h^{-1} \circ j^\dag_{V-T} = 0$. Similarly, $h^{-1} \circ \uGa_{V-T} = 0$ and hence $\uGa^\dag_T \circ \uGa_{V-T} = h_! h^{-1} \circ \uGa_{V-T} = 0$. The results now follow from the above exact sequences.
  \end{proof}

  It is important to note that if $Z \into \fP$ is a formal embedding with $\fP$ proper, and we take $V = \fP_K$ and $T =\tb{Z}$, then the closed tubes $[Z]_\lambda$ of radius $\lambda < 1$ are cofinal among interior subsets of $T$, and also among quasi-compact subsets of $T$. So $\uGa^\dag_{\tb{Z}}(\cF)$ is precisely the sections of $\cF$ supported in a quasicompact subset of $\tb{Z}$.

 \subsection{Partial compact supports}
  \label{sect:partialsupport}

  \begin{wrapfigure}{r}{0.25\textwidth}
   \centering
   \begin{tabular}{|l|l|}
    \hline
    \multicolumn{2}{|c|}{V} \\ \hline
    W          & \hspace{1cm}\rule[-1cm]{0pt}{2cm} U  \hspace{1cm}  \\ \hline
   \end{tabular} \\[0.5ex]
   Figure 1
  \end{wrapfigure}
  We shall consider the following setting. We suppose we are given a closed embedding $Y \into \fP$, with $Y$ and $\fP$ both  proper, and a smooth open subvariety $U \subseteq Y$ with complement $Z = Y - U$. We suppose $\fP$ is smooth in a neighbourhood of $U$, so that $U \into Y \into \fP$ is a smooth proper frame for $U$.

  Let $V \subseteq Z$ be a closed subvariety, and set $W = Z - V$, as in Figure 1.  We want to attach a meaning to cohomology of $U$ with compact support ``towards $V$'' or ``towards $W$''.

\WFclear

  \begin{proposition}
   \label{prop:fiddlysupports}
   Let $\cF$ be an abelian sheaf on $\tb{Y}$, and let $V'$ be any closed subvariety of $Y$ such that $Z = V \cup V'$.
   \begin{enumerate}[(a)]
    \item We have canonical isomorphisms
    \begin{align*}
     j^{\dag}_{\tb{Y - V'}}\ \uGa_{\tb{Y - V}}\, \cF &= j^{\dag}_{\tb{Y - Z}}\ \uGa_{\tb{Y - V}}\, \cF \quad\text{and}\\
     \uGa_{\tb{Y - V'}}\ j^{\dag}_{\tb{Y - V}}\, \cF &=  \uGa_{\tb{Y - Z}}\ j^{\dag}_{\tb{Y - V}}\, \cF.
    \end{align*}

    \item There is a natural map
    \[  j^{\dag}_{\tb{Y - V'}}\ \uGa_{\tb{Y - V}}\, \cF \to \uGa_{\tb{Y-V}}\ j^{\dag}_{\tb{Y-V'}}\,\cF,\]
    which is an isomorphism away from $]V \cap V'[$.
   \end{enumerate}
  \end{proposition}

  \begin{proof}
   (a) As in Prop 5.1.11 of \cite{lestum07}, since $\tb{V}$ and $]V'[$ admissibly cover $\tb{Z}$ we have $j^{\dag}_{\tb{Y - V'}}\circ j^{\dag}_{\tb{Y - V}} = j^{\dag}_{\tb{Y - Z}}$. So, using \cref{prop:supportfunctors}, we have
   \[ j^{\dag}_{\tb{Y - V'}}\circ \uGa_{\tb{Y - V}} = j^{\dag}_{\tb{Y - V'}}\circ j^{\dag}_{\tb{Y - V}}\circ \uGa_{\tb{Y - V}} = j^{\dag}_{\tb{Y - Z}}\circ \uGa_{\tb{Y - V}}.\]
   Similarly, we have $\uGa_{\tb{Y - V'}} \circ \uGa_{\tb{Y - V}} = \uGa_{\tb{Y - Z}}$ and hence
   \[ \uGa_{\tb{Y - V'}}\circ j^{\dag}_{\tb{Y - V}} = \uGa_{\tb{Y - V'}}\circ \uGa_{\tb{Y - V}}\circ j^{\dag}_{\tb{Y - V}} =  \uGa_{\tb{Y - Z}}\circ j^{\dag}_{\tb{Y - V}}. \]

   (b) It suffices to show that the composite map 
   \[j^\dag_{\tb{Y - V'}}\, \uGa_{\tb{Y-V}}\, \cF \to j^\dag_{\tb{Y-V'}}\, \cF \to h_\star h^{-1}(j^\dag_{\tb{Y-V'}}\, \cF) = h_\star j^\dag_{\tb{V' - (V \cap V')}} (h^{-1}\cF)\]
   is zero. However, this map factors through $h_\star h^{-1}(\uGa_{\tb{Y-V}} \cF)$ which is the zero sheaf.

   The pair $\{ Y - V, Y - V'\}$ is an open covering of $Y - (V \cap V')$ as a $k$-variety, so their tubes admissibly cover $]Y - (V \cap V')[$. It is clear that the above map is an isomorphism after restriction to either $]Y - V[$ or $]Y - V'[$, so we obtain an isomorphism of sheaves over $]Y- (V \cap V')[$.
  \end{proof}

  \begin{definition}
   \label{def:partialsupport}
   Let $\cF$ be an abelian sheaf on $\tb{Y}$.
   \begin{itemize}
    \item We define \emph{cohomology with compact support towards $V$} (recall $V$ is closed in $Z$) by
    \[ R\Gamma_{cV}(\tb{U}, \cF) =
    R\Gamma\left(\tb{Y},
    j^{\dagger}_{\tb{Y-Z}}\ R\uGa_{\tb{Y-V}}\ \cF\right).
    \]
    \item We define  \emph{cohomology with compact support towards $W$} (recall $W$ is open in $Z$) by
    \[ R\Gamma_{cW}(\tb{U}, \cF) =
    R\Gamma\left(\tb{Y},
    R\uGa_{\tb{Y-Z}}\
    j^\dagger_{\tb{Y-V}}\  \cF\right).
    \]
   \end{itemize}
  \end{definition}

  Note that this notation is \emph{a priori} ambiguous, since if both $V$ and $W$ are closed in $Y$, we have two candidate definitions of $R\Gamma_{cV}(-)$; but in fact the two candidate definitions agree, since if we start from the first definition we have
  \begin{align*}
  R\Gamma_{cV}(\tb{U}, \cF) \coloneqq&
  R\Gamma\left(\tb{Y},
  j^{\dagger}_{\tb{Y-Z}}\ R\uGa_{\tb{Y-V}}\ \cF\right) \\
  =& R\Gamma\left(\tb{Y}, j^{\dagger}_{\tb{Y-W}}\ R\uGa_{\tb{Y-V}}\ \cF\right) \quad\text{ (by part (a) of the proposition)}\\
  =& R\Gamma\left(\tb{Y}, R\uGa_{\tb{Y-V}}\ j^{\dagger}_{\tb{Y-W}}\ \cF\right) \quad\text{ (by part (b) of the proposition)} \\
  =& R\Gamma\left(\tb{Y}, R\uGa_{\tb{Y-Z}}\ j^{\dagger}_{\tb{Y-W}}\ \cF\right) \quad\text{ (by part (a) of the proposition)}
  \end{align*}
  which is the second definition of $R\Gamma_{cV}(\tb{U}, \cF)$.

  \begin{remark}
   In particular, this applies when one of $V$ and $W$ is empty, and we conclude that cohomology with compact support towards $\varnothing$, or towards all of $Z$, has the expected meaning. We also see that if $Y$ is a smooth proper curve, and $V$ and $W$ disjoint sets of points, then we recover the notion of ``partial compact support'' considered in \cite{LSZ-asai}.
  \end{remark}

  \begin{proposition}
   \label{prop:exactseqs}
   We have exact triangles
   \[ R\Gamma_{cW}(\tb{U}, \cF) \to R\Gamma\left(\tb{Y},\ j^\dagger_{\tb{Y-V}}\, \cF\right) \to R\Gamma\left(\tb{Z},\  j_{\tb{W}}^\dag(\cF|_{\tb{Z}})\right) \to [+1] \]
   and
   \[
    R\Gamma\left(\tb{Y},\ \uGa_{\tb{Z}}^\dag\  R\uGa_{\tb{Y-V}}\,\cF\right) \to
    R\Gamma\left(\tb{Y},\ R\uGa_{\tb{Y-V}}\,\cF\right) \to
    R\Gamma_{cV}\left(\tb{U},\, \cF\right) \to [+1].
   \]
  \end{proposition}

  \begin{proof}
   By definition we have an exact triangle of complexes of sheaves on $\tb{Y}$
   \[
   R\uGa_{\tb{Y-Z}}\ j^\dagger_{\tb{Y-V}}\, \cF \to
   j^\dagger_{\tb{Y-V}}\, \cF \to
   Rh_\star h^{-1}\left( j^\dagger_{\tb{Y-V}}\, \cF\right) \to [+1],
   \]
   where $h: Z \into Y$ is the inclusion map. However, since $(Y - V) \cap Z = W$, we have $h^{-1}\left( j^\dagger_{\tb{Y-V}}\, \cF\right) = j^\dagger_{\tb{Z-V}}(h^{-1} \cF)$, by Corollary 5.1.15 of \cite{lestum07}. Applying the (triangulated) functor $R\Gamma(\tb{Y}, -)$ gives the first triangle. The second is obtained similarly.
  \end{proof}

  Let us note some ``naturality'' properties of the construction. Firstly, if we fix $Y$ and $Z$, and let $J \supseteq J'$ be two subvarieties of $Z$, then we have natural maps
  $R\Gamma_{cJ}(U, \cF) \to R\Gamma_{cJ'} (U, \cF)$ if:
  \begin{itemize}
   \item $J$ and $J'$ are both open,
   \item $J$ and $J'$ are both closed,
   \item $J$ is closed and $J'$ is open (using \cref{prop:fiddlysupports}(b) with $V = J$ and $V' = Z - J'$).
  \end{itemize}
  (We do not consider the case when $J$ is open and $J'$ is closed, since we will not need it here.) Secondly, for coherent sheaves we have cup-products
  \[  R\Gamma_{cV}(\tb{U}, \cF) \otimes R\Gamma_{cW}(\tb{U}, \mathcal{G}) \to R\Gamma_{cZ}(\tb{U}, \cF \otimes \mathcal{G}),\]
  and these are compatible with the exact triangles of \cref{prop:exactseqs}. Finally, we have the following compatibility with respect to morphisms of frames:

  \begin{proposition}
   \label{prop:pullback}
   Let $u: \fP' \to \fP$ be a morphism of formal schemes over $\cO_K$, and define $Y', V', W'$ as the preimages of $Y, V, W$ etc. Then pullback along $u$ induces canonical maps
   \[ u^\star: R\Gamma_{cV}(\tb{U}, \cF) \to R\Gamma_{cV'}(\tb{U'}, \cF') \]
   and
   \[ u^\star: R\Gamma_{cW}(\tb{U}, \cF) \to R\Gamma_{cW'}(\tb{U'}, \cF'),\]
   compatible with the exact triangles of \cref{prop:exactseqs}.
  \end{proposition}

  \begin{proof}
   This is immediate from the compatibility of $j^{\dagger}$ and $R\uGa$ with pullback.
  \end{proof}

  \begin{remark}
   If we start with a variety $Y$ and two closed subvarieties $A, B$, and put $Z = A \cup B$, then we are interpreting $R\Gamma(\tb{Y}, j^{\dag}_{\tb{Y - B}}\ R\uGa_{\tb{Y - A}}\, \cF)$ as cohomology with compact support towards the closed subvariety $A$ of $Z$, and the subtly different group $R\Gamma(\tb{Y}, R\uGa_{\tb{Y - A}}\ j^{\dag}_{\tb{Y - B}}\, \cF)$ (with the order of the functors interchanged) as cohomology with compact support towards the open subvariety $A - (A\cap B)$ of $Z$. These agree if $A \cap B = \varnothing$, but they are genuinely different otherwise (as the special case $A = B$ shows). We shall show in \cref{sect:transverse} below that they give the same result for the cohomology of the de Rham complex when $A$ and $B$ intersect transversely.
  \end{remark}

 \subsection{The transversal case}
  \label{sect:transverse}

  Although we shall not use it in the remainder of the paper, it would be remiss not to point out the following consistency property of the above constructions. For simplicity, we suppose that $\fP$ is smooth and proper over $\cO_K$, and $Y = \fP_0$. Let $A, B$ be two closed subvarieties of $Y$, let $U = Y - A - B$, and let $A^\circ = A - (A \cap B)$ and $B^\circ = B - (A \cap B)$.

  \begin{proposition}
   If $A$, $B$, and $A \cap B$ are smooth, and $\codim_Y(A\cap B) = \codim_Y(A) + \codim_Y(B)$, then there are isomorphisms
   \[ H^i_{\dR, cA}(\tb{U}) \cong H^i_{\dR, cA^\circ}(\tb{U}) \]
   for all $i$.
  \end{proposition}

  \begin{proof}
   Consider the following $3 \times 3$ grid, in which each row and column is an exact triangle:
   \[\begin{tikzcd}
    ? \arrow[r] \arrow[d]                       & {R\Gamma_{\rig, B}(Y)} \ar[d]\ar[r]  & {R\Gamma_{\rig, A \cap B}(A)} \arrow[d] \\
    {R\Gamma_{\rig, c}(Y-A)} \arrow[d] \arrow[r] & R\Gamma_\rig(Y) \arrow[d] \arrow[r] & R\Gamma_{\rig}(A) \arrow[d]             \\
    {R\Gamma_{\dR,cA^\circ}(\tb{U})} \arrow[r]      & R\Gamma_\rig(Y-B) \arrow[r]          & R\Gamma_{\rig}(A^\circ),
   \end{tikzcd}\]
   where the term marked `?' is $R\Gamma\left(\fP_K, R\uGa_{\tb{Y-A}}\ \uGa^\dag_{\tb{B}}\, \Omega^\bullet\right)$. Our smoothness assumptions imply that there is a Gysin isomorphism $R\Gamma_{\rig, B}(Y) = R\Gamma_{\rig}(B)[-2c]$ where $c = \codim_Y(B)$, and similarly that $R\Gamma_{\rig, A \cap B}(A) = R\Gamma_{\rig}(A\cap B)[-2c]$. Moreover, the map $R\Gamma_{\rig, B}(Y) \to R\Gamma_{\rig, A \cap B}(A)$ is identified, via the Gysin isomorphisms, with the obvious restriction map $R\Gamma_{\rig}(B) \to R\Gamma_{\rig}(A \cap B)$ (shifted by $-2c$). Note that this compatibility of Gysin morphisms is far from being merely formal, but rather is a basic case of the ``excess intersection formula'' of D\'eglise, see \cite[Proposition 4.10]{deglise08}. So the group `$?$' has to be isomorphic to the mapping fibre of this map, which is simply $R\Gamma_{\rig, c}(B^\circ)[-2c]$.

   We claim that applying the functor $\operatorname{RHom}(-, K[-2d])$ to this diagram, and then reflecting in the off-diagonal, gives the corresponding diagram with the roles of $A$ and $B$ interchanged, except possibly for the bottom left corner. That is, if $\cD_{ij}(A, B)$ denotes the object at the $(i, j)$ position in this diagram, for $1 \le i,j \le 3$, we claim that
   \[ \operatorname{RHom}\Big(\cD_{ij}(A, B), K[-2d]\Big) = \cD_{4-j, 4-i}(B, A),\]
   for all $(i, j)$ except possibly $(3, 1)$. For instance, letting $c' = \codim_Y(A)$, we have
   \begin{align*}
    \operatorname{RHom}(\cD_{1, 3}(A, B), K[-2d])&= \operatorname{RHom}(R\Gamma_{\rig, A \cap B}(A), K[-2d])\\
     &= \operatorname{RHom}(R\Gamma_{\rig}(A \cap B)[-2c], K[-2d]) \\
    &= \operatorname{RHom}(R\Gamma_{\rig}(A \cap B), K[-2(d-c)]) \\
    &= R\Gamma_{\rig}(A \cap B)[-2c'] &&\text{(Poincar\'e duality for $A \cap B$)}\\
    &= R\Gamma_{\rig, A \cap B}(B) &&\text{(since $\codim_{B}(A \cap B) = c'$)}\\
    &= \cD_{1, 3}(B, A).
   \end{align*}
   (The remaining cases are similar, and indeed rather more straightforward.) After a little book-keeping, one also sees that these isomorphisms are compatible with the arrows in the two diagrams. Hence we deduce an isomorphism in the remaining corner also, namely
   \[ \operatorname{RHom}(R\Gamma_{\dR, cA^\circ}(\tb{U}), K[-2d]) \cong R\Gamma_{\dR, cB^\circ}(\tb{U}).\]
   But we have seen in \cref{cor:partialpoincare} that the dual of $R\Gamma_{\dR, cB^\circ}(\tb{U})[-2d]$ is $R\Gamma_{\dR, cA}(\tb{U})$. Putting these together, we finally arrive at the required isomorphism
   \[ R\Gamma_{\dR, cA^\circ}(\tb{U}) \cong R\Gamma_{\dR, cA}(\tb{U}).\qedhere\]
  \end{proof}

 \subsection{Interpretation via dagger spaces}

  We recall from \cite{grossekloenne00} the category of \emph{dagger spaces} over $K$. Note that if $\fP$ is a proper (admissible) $\cO_K$-scheme, and $X$ is a locally closed subvariety of $\fP_0$, then there is a natural structure of a dagger space on the tube $\tb{X}$; we denote this dagger space by $\tb{X}^\dag$, and similarly $[X]^\dag_{\lambda}$ for the tubes of radius $\lambda < 1$.

  \subsubsection{Non-compact support}

   Essentially by definition, if $X \into Y \into \fP$ is a proper smooth frame, and $V$ any strict neighbourhood of $\tb{X}$ in $\tb{Y}$, then any coherent sheaf $\cF$ on $V$ defines a coherent sheaf on $\tb{X}^\dagger$, and we have
   \[ R\Gamma(\tb{X}^\dag, \cF) = R\Gamma\left(V, j_{\tb{X}}^\dag\, \cF\right) \]
   (and similarly for hypercohomology of complexes of coherent sheaves).

  \subsubsection{Compact support}

   There is also a concept of \emph{compactly-supported cohomology} for coherent sheaves on dagger spaces: see \cite[\S 4.3]{grossekloenne00}. We will need the following computation:

   \begin{proposition}
    \label{prop:daggerspace}
    Let $\fP$ be a proper admissible formal $\cO_K$-scheme, and $W$ a locally closed subvariety of $\fP_0$. Write $W = X \cap Z$ with $X$ open and $Z$ closed. Then we have
    \[ R\Gamma_c(\tb{W}^\dagger, \cF) = R\Gamma\left(\fP_K, \uGa^\dag_{\tb{Z}}\ R\uGa_{\tb{X}}\, \cF\right).\]
   \end{proposition}

   \begin{proof}
    We have $\uGa^\dag_{\tb{Z}}\ \uGa_{\tb{X}}\, \cF = \varinjlim_{\lambda} \Gamma_{[Z]_\lambda}\ \Gamma_{\tb{X}}\, \cF = \varinjlim_{\lambda}\Gamma_{[W]_\lambda}\, \cF$, since $[W]_\lambda = [Z]_\lambda \cap \tb{X}$. Applying this to an injective resolution of $\cF$ gives the result, since $R\Gamma_c(\tb{W}, \cF) = \varinjlim_\lambda R\Gamma\left(\fP_K, R\uGa_{[W]_\lambda}\, \cF\right)$.
   \end{proof}

   These results allow the triangles of \cref{prop:exactseqs} to be written in the following more convenient form. Let $(U, V, W)$ be as above, and denote the dagger space tubes of these by $\cU, \cV, \cW$. Then there are exact triangles
   \[
    R\Gamma_c\left(\cW,\cF\right) \to
    R\Gamma_c\left(\cU \cup \cW,\cF\right) \to
    R\Gamma_{cV}\left(\cU, \cF\right) \to [+1]
   \]
   and
   \[ R\Gamma_{cW}(\cU, \cF) \to R\Gamma\left(\cU \cup \cW, \cF\right) \to R\Gamma\left(\cW, \cF\right) \to [+1]. \]

  \subsection{Duality}
  \label{sect:serredual}

   Theorem 4.4 of \cite{grossekloenne00} is a form of Serre duality for smooth \emph{affinoid} dagger spaces $\cX$, giving a perfect duality of Hausdorff topological vector spaces (for any $i \ge 0$ and any coherent sheaf $\cF$ on $\cX$)
   \[ H^i_c(\cX, \cF) \times \operatorname{Ext}^{d-i}_{\cO_X}\left(\cF, \omega_\cX\right) \to K, \]
   where $d = \operatorname{dim} \cX$ and $\omega_{\cX}$ is the line bundle $\Omega^d_{\cX / K}$. (The proof is only sketched in \emph{op.cit.}; a fuller account, in German, can be found in  Grosse-Kl\"onne's thesis \cite[\S 7.1]{grossekloenne-thesis}.)

   For our purposes it is convenient to extend this to non-smooth affinoids $\cX$. We choose a closed embedding $\iota: \cX \into \cP$ where $\cP$ is a smooth affinoid. Then we define
   \[ \underline{\omega}_{\cX} = \iota^*\left( \underline{\operatorname{RHom}}_{\cO_\cP}\left(\iota_\star \cO_\cX, \underline{\omega}_{\cP}\right)\right),\qquad \underline{\omega}_{\cP} = \omega_{\cP}[\dim \cP]. \]
   This is an object of the derived category of bounded complexes of coherent sheaves\footnote{It is in general not a perfect complex without additional smoothness conditions on $\cX$; we thank the referee for stressing this remark.} on $\cX$, and it follows from Serre duality for $\cP$ that for any coherent sheaf $\cF$ on $\cX$ there is a perfect duality
   \[ 
    H^i(\cX, \cF) \times \HH^{-i}_c\left(\cX, \underline{\operatorname{RHom}}_{\cO_{\cX}}(\cF, \underline{\omega}_{\cX})\right) \to K, 
   \]
   with both sides zero if $i \ne 0$. 
   
   \begin{remark}
    This is essentially the same argument used in to deduce Serre duality for general smooth affinoids from the case of affinoid discs in Satz 7.2 of \cite{grossekloenne-thesis}. As noted in \emph{op.~cit.}, if $\cX$ is smooth, then the cohomology of $\underline{\operatorname{RHom}}_{\cO_\cP}\left(\pi_\star \cO_\cX, \underline{\omega}_{\cP}\right)$ is concentrated in degree $d$ and in that degree is just $\omega_{\cX}[d]$, giving the familiar statement of Serre duality in the smooth case. However, inspecting the proof, one sees that if we do not impose any smoothness assumptions on $\cX$ the same argument gives the above more general theorem. 
   \end{remark}
   
   In general, a dualizing complex on $\cX$ may not be unique; but it is unique up to tensoring with a line bundle (by the same argument as in the algebraic-variety case, for which see \cite[\S V.2]{hartshorne66}). If $\cX$ is normal, then this line bundle must be trivial, since the restriction of $\underline{\omega}_{\cX}$ to the smooth locus must coincide with the usual sheaf of top-degree differentials, and the non-smooth locus has codimension $\ge 2$, so a line bundle whose restriction to the smooth locus is trivial must be trivial on $\cX$. So for normal $\cX$ the dualizing complex is unique, and thus independent of the embedding.

   The case of interest for us is when $\cX$ is an open affinoid in the analytification of a normal $K$-variety $X$. We claim that $\underline{\omega}_{\cX}$ coincides with the restriction to $\cX$ of $(\underline{\omega}_X)^{\mathrm{an}}$, where $\underline{\omega}_X$ is the algebraic dualizing complex of $X$, compatibly with the isomorphisms between both complexes and $\omega_{\cX/K}[d]$ on the smooth locus. Suppose $\iota : X \into P$ is an embedding into a smooth $K$-variety. If there exists an open affinoid $\cP \subseteq P^{\mathrm{an}}$ with $\cX = \iota^{-1}(\cP)$, then we may use this $\cP$ to construct $\underline{\omega}_{\cX}$; since analytification of coherent sheaves is compatible with pullback and $\underline{\operatorname{RHom}}$ the compatibility is clear. If $\cX$ is arbitrary, we can cover it by smaller affinoids $\cU_i$ which do have this form; this gives isomorphisms between $(\underline{\omega}_X)^{\mathrm{an}}|_{\cU_i}$ and $\underline{\omega}_{\cX}|_{\cU_i}$, and these must agree on the overlaps since they are compatible with the isomorphisms to $\omega_{\cX/K}[d]$ on the smooth locus.
%

   \begin{remark}
    This Serre duality does \textbf{not} seem to extend straightforwardly to non-affinoid dagger spaces (even smooth ones). If $\cX$ is smooth and quasi-compact, and $\{\cX_i\}_{i \in I}$ is a finite affinoid covering, then we can form Cech complexes representing $R\Gamma(\cX, \cF)$ and $R\Gamma_c(\cX, \cF^\vee \otimes \omega_{\cX})$ with respect to this covering. These are complexes of complete locally-convex $K$-vector spaces which are term-wise dual to one another, so we obtain natural pairings between the cohomology groups. However, it is not clear if the differentials in these complexes are strict; so one does not know if these pairings are perfect dualities of topological vector spaces (or even if the induced topologies on the cohomology groups are Hausdorff).
   \end{remark}

 \subsection{Finiteness and Poincar\'e duality}

  We now consider the special case of the hypercohomology of the de Rham complex.


   \begin{theorem}[Grosse-Kl\"onne]
    \label{thm:GK-finiteness}
    If $\cX$ is a dagger space of the form $\cU- \cV$, where $\cU$ is smooth and quasicompact, and $\cV \subseteq \cU$ is a quasicompact open subset, then the cohomology groups $H^i_{\dR}(\cX) \coloneqq \HH^i(\cX, \Omega^\bullet_{\cX/K})$ are finite-dimensional over $K$ for all $i$.
   \end{theorem}

   \begin{proof}
    This is (a special case of) the main theorem of  \cite{grossekloenne02}.
   \end{proof}

   \begin{note}
    This implies finite-dimensionality of rigid cohomology, since for a proper smooth frame $X \into Y \into \fP$, the dagger space $\cX =\ \tb{X}^\dag_{\fP}$ satisfies the hypotheses of \cref{thm:GK-finiteness}, and we have $H^i_{\rig}(X) = H^i_{\dR}(\cX)$.
   \end{note}

   There is also a compactly-supported analogue of this result, and a Poincar\'e duality theorem; these are straighforward consequence of results of Grosse-Kl\"onne, but curiously do not seem to be explicitly written down in the literature:

   \begin{theorem}
    \label{thm:poincare}
    Let $\cX$ be a smooth dagger space of the form $\cU - \cV$ with $\cU, \cV$ quasicompact, as in \cref{thm:GK-finiteness}, of pure dimension $d$. Then $H^i_{\dR, c}(\cX)$ is also finite-dimensional for all $i$, and we have perfect pairings of finite-dimensional vector spaces
    \[ H^i_{\dR}(\cX) \times H^{2d-i}_{\dR, c}(\cX) \to K \quad \text{for $0 \le i \le 2d$}.\]
   \end{theorem}

   \begin{proof}
    The case of affinoid $\cX$ is treated in Theorem 4.9 and remark 4.10 of \cite{grossekloenne00}. The case of $\cX$ quasicompact follows readily from this, using the Cech spectral sequence associated to a finite covering of $\cX$ by affinoids (since we know that the Cech complex consists of finite-dimensional vector spaces, there are no topological issues to worry about).

    We now consider the general case. We can write $\cX$ as a countable increasing union $\{\cX_n\}_{n \in \mathbb{N}}$ of quasi\-compact subsets. Then we have
    \[ H^i_{\dR}(\cX) = \varprojlim_n H^i_{\dR}(\cX_n), \qquad H^{2d-i}_{\dR, c}(\cX) = \varinjlim_n H^{2d-i}_{\dR, c}(\cX_n).\]
    Since the terms in the two limits are dual to each other, and we know that $H^i_{\dR}(\cX)$ is finite-dimensional, it follows that $H^{2d-i}_{\dR, c}(\cX)$ is also finite-dimensional and that Poincar\'e duality holds for $\cX$.
   \end{proof}

   \begin{remark}
    \label{rmk:rigidcpctsupp}
    We caution the reader that if $X$ is a $k$-variety and $X \into Y \into \mathfrak{P}$ is a smooth proper frame, the compactly-supported de Rham cohomology of the dagger tube $\cX = \tb{X}_{\fP} \subset \cP$ is \textbf{not} automatically equal to the compactly-supported rigid cohomology of $X$; indeed this is impossible to reconcile with Poincar\'e duality, since $\cX$ could have much larger dimension than $X$ and hence the dualities would land in different degrees. However, the two do coincide if $X$ is open in $\fP_0$, since in this case $\left[ X \into \fP_0 \into \fP \right]$ is also a smooth proper frame for $X$.
   \end{remark}


   \begin{corollary}
    \label{cor:partialpoincare}
    If we are given varieties $U, V, W \subseteq \fP_0$ as in \cref{sect:partialsupport}, with $U$ open in $\fP_0$ and $\fP_K$ smooth, then the cohomology groups of the complexes
    \[ R\Gamma_{\dR, cV}(\cU) \coloneqq R\Gamma_{cV}(\cU, \Omega^\bullet) \quad \text{and}\quad R\Gamma_{\dR, cW}(\cU)\coloneqq R\Gamma_{cW}(\cU, \Omega^\bullet) \]
    are finite-dimensional for all $i$, and there are perfect pairings
    \[ H^i_{\dR, cV}(\cU) \times H^{2d-i}_{\dR, cW}(\cU) \to K. \]
   \end{corollary}

   \begin{proof}
    Rewriting the exact triangles of \cref{prop:exactseqs} in terms of dagger spaces using \cref{prop:daggerspace}, as explained above, and taking $\cF$ to be the rigid-analytic de Rham complex, we have long exact sequences
    \[ \dots \to H^i_{\dR, cW}(\cU) \to H^i_{\dR}(\cU \cup \cW) \to H^i_{\dR}(\cW) \to \dots\]
    and
    \[\dots \leftarrow H^{2d-i}_{\dR, cV}(\cU) \leftarrow H^{2d-i}_{\dR, c}(\cU \cup \cW) \leftarrow H^{2d-i}_{\dR, c}(\cW) \leftarrow \dots.\]
    Moreover, there are compatible pairings between the groups in the first row and their neighbours in the second row. By \cref{thm:GK-finiteness,thm:poincare}, the middle and right groups on each row are finite-dimensional and the pairings between them are perfect. By induction on $i$ we deduce that the groups in the left-hand column are also finite-dimensional and in perfect duality, as required.
   \end{proof}

  \subsection{A ``logarithmic'' variant}

   Sadly the above setting is still not quite general enough, and we shall need to consider yet another possibility. Suppose we have an proper admissible formal $\cO_K$-scheme $\fP$, a proper closed subvariety $Y \into \fP_0$, and a decomposition $Y = U \cup V \cup W$ as above. We also suppose that $\mathfrak{D} \subseteq \fP$ is a simple normal crossing divisor relative to $\operatorname{Spf} \cO_K$, which intersects transversely with $U$ and $W$. We write $\cP$ for the dagger space generic fibre of $\fP$, and $\cU, \cV, \cW$ for the dagger tubes of $U, V, W$ respectively.

   \begin{notation}\label{not:logdR}
    Write $R\Gamma_{\dR, cV}(\cU\langle \cD \rangle)$, resp.~ $R\Gamma_{\dR, cW}(\cU\langle \cD \rangle)$, for the hypercohomology of $\cU$ with compact support towards $V$ (resp.~$W$) of the \emph{logarithmic} de Rham complex $\Omega^\bullet_{\cP}\langle \cD\rangle$. Similarly, we write $R\Gamma_{\dR, cV}(\cU\langle -\cD \rangle)$ for the hypercohomology of the ``minus-log'' complex $\Omega^\bullet_{\cP}\langle - \cD \rangle \coloneqq \Omega_{\cP}^\bullet\langle \cD\rangle(-\cD)$.
   \end{notation}

   \begin{proposition}
    \label{prop:logpartialpoincare}
    We have perfect pairings of finite-dimensional $K$-vector spaces
    \[ H^i_{\dR, cV}(\cU, \langle -\cD \rangle) \times H^{2d-i}_{\dR, cW}(\cU, \langle \cD \rangle) \to K. \]
    and
    \[ H^i_{\dR, cV}(\cU, \langle \cD \rangle) \times H^{2d-i}_{\dR, cW}(\cU, \langle -\cD \rangle) \to K. \]
   \end{proposition}

   \begin{proof}
    By the same long exact sequence argument as above, it suffices to prove the proposition in the special case $\cW = \varnothing$, i.e.~that
    \[ H^i_{\dR, c}(\cU, \langle -\cD \rangle) \times H^{2d-i}_{\dR}(\cU, \langle \cD \rangle) \to K\]
    and
    \[ H^i_{\dR}(\cU, \langle -\cD \rangle) \times H^{2d-i}_{\dR, c}(\cU, \langle \cD \rangle) \to K\]
    are perfect pairings of finite-dimensional spaces. We prove the former; the argument for the latter is identical with the role of compact and non-compact support interchanged.

    Let $\cD^{(n)}$ denote the disjoint union of the $n$-fold intersections of components of $\cD$, and $\iota^{(n)}: \cD^{(n)} \to \cP$ the natural map. The logarithmic de Rham complex $\Omega^\bullet_{\cP}\langle \cD\rangle$ has an increasing filtration, whose $n$-th graded piece is $\iota^{(n)}_\star\left( \Omega^\bullet_{\cD^{(n)}}\right)$. Similarly, the complex $\Omega^\bullet\langle -\cD\rangle$ has a decreasing filtration, with the same graded pieces; and the logarithmic duality pairing
    \begin{equation}
     \label{eq:logdual}
     \Omega^i_{\cP}\langle \cD\rangle \otimes \Omega^{d-i}_{\cP}\langle -\cD\rangle \to \Omega^d_{\cP}\langle -\cD\rangle = \omega_{\cP},
    \end{equation}
    where $\omega_{\cP}$ is the dualizing sheaf, is compatible with these filtrations, and the pairing it induces on the $n$-th graded piece is the usual (non-logarithmic) duality pairing on each of the $n$-fold intersections.

    So we have spectral sequences
    \[ E_1^{ij} = \HH^j_c\left(\cU, \iota^{(i)}_\star\left( \Omega^\bullet_{\cD^{(i)}}\right)\right) \Rightarrow H^{i+j}_{\dR, c}\left(\cU\langle -\cD\rangle\right) \]
    and
    \[ 'E_1^{-i,j} = \HH^{j-2i}\left(\cU, \iota^{(i)}_\star\left( \Omega^\bullet_{\cD^{(i)}}\right)\right) \Rightarrow H^{-i+j}_{\dR}\left(\cU\langle \cD\rangle\right).
    \]
    We have $\HH^j_c\left(\cU, \iota^{(i)}_\star\left( \Omega^\bullet_{\cD^{(i)}}\right)\right) = H^j_{\dR, c}\left(\cU^{(i)}\right)$, where $\cU^{(i)} = (\iota^{(i)})^{-1}(\cU)$, and similarly without compact support. So by Theorem \ref{thm:poincare}, the spaces $E_1^{ij}$ and $'E_1^{ij}$ are finite-dimensional for all $i, j$ (and zero outside a bounded region), and the pairing $E_1^{i, j} \times {}'E_1^{-i, 2d-j} \to E_1^{0, 2d} \cong K$ induced by \eqref{eq:logdual} is perfect. Hence the limits of the two spectral sequences are also finite-dimensional and in perfect duality, as required.
   \end{proof}

  \subsection{Gros fp-cohomology with partial support}

   Since Gros fp-cohomology is defined using complexes of overconvergent differential forms, we can use the above formalism to define variants of this cohomology with partial compact support.

   More precisely, as in \cref{sect:partialsupport}, let $\fP$ be a formal $\cO_K$-scheme with special fibre $\fP_0$, and suppose $U, V, W \subseteq \fP_0$ are subvarieties as before, with $U$ open in $\fP_0$ and $\fP$ smooth in a neighbourhood of $U$. For $s\geq 0$, denote by $\RGt_{\dR,cW}(\tb{U},\sG,s)$ the hypercohomology of the $s$-th filtration subcomplex of the de Rham complex of $\sG$.

   \begin{definition}
    Let $Q\in\Qp[t]$ be a polynomial with constant coefficient $1$. Define the Gros fp-cohomology of $\tb{U}$ with compact support towards $W$, coefficients $\sG$, twist $s$ and polynomial $Q$ as the mapping fibre
    \[   \RGt^j_{\rigfp,\cW}(\tb{U},\sG,s;Q)=\MF\Big[ \RGt_{\dR,cW}(\tb{U},\sG,s)\xrightarrow{\ Q(\varphi_s) \circ \iota\ } R\Gamma_{\dR,cW}(\tb{U},\sG)\Big],\]
    where $\iota$ denotes the natural map
    \[ \RGt_{\dR,cW}(\tb{U},\sG,s)\to R\Gamma_{\dR,cW}(\tb{U},\sG).\]
   \end{definition}

 \section{Application to \texorpdfstring{$\GSp_4$}{GSp(4)} Shimura varieties}

  Having developed the above general formalism, we now specify to which varieties it will be applied.

  \subsection{Shimura varieties for $G$} We consider the Klingen-level Siegel threefold $Y_{G, \Kl}$ (over $\Zp$), and a choice of arithmetic toroidal compactification $X_{G, \Kl}$. Then we may consider the following spaces:

  \paragraph*{The multiplicative locus} (This case was already treated in \cref{sect:partialmodels}; we recall it here for completeness.) Since $X_{G, \Kl, 0}^m$ is smooth and open in $X_{G, \Kl, 0}$, the sequence of inclusions
   \[ \left[ X_{G, \Kl, 0}^m \into \overline{X_{G, \Kl, 0}^m} \into \mathfrak{X}_{G, \Kl} \right] \]
   defines a proper smooth frame for $X_{G, \Kl, 0}^m$, where $\mathfrak{X}_{G, \Kl}$ is the $p$-adic completion of $X_{G, \Kl}$. Since $X_{G, \Kl}$ is proper, the analytification of $X_{G, \Kl, \Qp}$ concides with the rigid-analytic generic fibre of $\mathfrak{X}_{G, \Kl}$.

   Hence we may compute rigid cohomology of $X_{G, \Kl, 0}^m$ as the de Rham cohomology of the dagger space $\cX_{G, \Kl}^m$; and similarly for the cohomology with compact supports, using Remark \ref{rmk:rigidcpctsupp}.

  \paragraph*{The $(2, m)$ locus} We now consider the decomposition of the multiplicative locus
  \[ X_{G, \Kl, 0}^{m} = X_{G, \Kl, 0}^{(1, m)} \cup X_{G, \Kl, 0}^{(2, m)} \]
  as the union of a closed and an open subvariety. We can then apply the formalism of \cref{sect:partialsupport} with the following choices:
  \begin{align*}
   Y &= X_{G, \Kl, 0}^{m} \sqcup X_{G, \Kl, 0}^{\alpha}, &
   U &= X_{G, \Kl, 0}^{(2, m)}, \\
   V &= X_{G, \Kl, 0}^{\alpha}, &
   W &= X_{G, \Kl, 0}^{(1, m)}.
  \end{align*}

  \begin{notation}
   \label{notation:defc0support}
   For $\cF$ an abelian sheaf (or complex of abelian sheaves) on $\cX_{G, \Kl}$, we use the notation $R\Gamma_{c0}\left(\cX_{G, \Kl}^{(2, m)}, \cF\right)$ for the space $R\Gamma_{c\cX_{G, \Kl}^{\alpha}}\left(\cX_{G, \Kl, 0}^{(2, m)}, \cF\right)$, i.e.~cohomology with compact support towards the closed subvariety $V = X_{G, \Kl,0}^{\alpha}$ of the boundary.

   We write $R\Gamma_{c1}\left(\cX_{G, \Kl, 0}^{(2, m)}, \cF\right)$ for the space $R\Gamma_{c\cX_{G, \Kl}^{(1, m)}}\left(\cX_{G, \Kl}^{(2, m)}, \cF\right)$, i.e.~cohomology with compact support towards the open subvariety $W = X_{G, \Kl, 0}^{(1, m)}$ of the boundary.
  \end{notation}

  We therefore have exact triangles
  \[
   R\Gamma_{c}(\cX_{G, \Kl}^{(1, m)}, \cF) \longrightarrow R\Gamma_{c}(\cX_{G, \Kl}^{m}, \cF) \longrightarrow R\Gamma_{c0}(\cX_{G, \Kl}^{(2, m)}, \cF) \to [+1]
  \]
  and
  \[ R\Gamma_{c1}(\cX_{G, \Kl}^{(2, m)}, \cF) \longrightarrow R\Gamma(\cX_{G, \Kl}^{m}, \cF) \longrightarrow R\Gamma(\cX_{G, \Kl}^{(1, m)}, \cF) \to [+1].
  \]
  If we take $\cF$ to be the rigid-analytic de Rham complex, then the $c0$ and $c1$ support cohomology groups are independent of the choice of frame, and hence functorial in the special fibre; in particular, they have actions of the Frobenius map (compatible with the above exact triangles).

  \begin{remark}
   In the above situation, the boundary $Z = \overline{X_{G, \Kl, 0}^m} - X_{G, \Kl, 0}^{(2, m)}$ is in fact the union of two 2-dimensional subvarieties, namely $X_{G, \Kl, 0}^{\alpha}$ and $\overline{X_{G, \Kl, 0}^{(1, m)}}$, intersecting along the 1-dimensional subvariety $X_{G, \Kl, 0}^{(0, \alpha)}$ (the supersingular locus). However, we cannot interchange the roles of these two subvarieties in our construction, since the formal scheme $\mathfrak{X}_{G, \Kl}$ is not smooth along $X_{G, \Kl, 0}^{\alpha}$.
  \end{remark}

 \subsection{Restricting to \texorpdfstring{$c0$}{c0} support}

  As we showed in \cref{thm:redtoord} above, the regulator that we are trying to compute is given by
  \[ \left\langle \Eis^{[t_1,t_2],(m,m)}_{\rigsyn,\underline\Phi},\, (\iota_\Delta^{[t_1,t_2]})^\star( \eta_{\rigfp,-D}^{m})\right\rangle_{\rigfp, X^{(m, m)}_{H,\Delta}} \]
  where $\eta_{\rigfp,-D}^{m}$ lies in $H^3_{\rigfp, c}\left(X_{G,\Kl}^{m}\langle -D\rangle, \cV_G, 1+q; \cQ_{1+q}\right)$ (the compactly-supported rigid fp-co\-hom\-ology of the multiplicative locus). Using our formalism of partially-compactly-supported cohomology, we can express this in a more convenient form as follows:

  \begin{proposition}\label{prop:restricttoc0}
   The map $(\iota_\Delta^{[t_1,t_2]})^\star$ factors through the restriction map
   \[
    H^3_{\rigfp, c}\left(X_{G,\Kl}^{m}\langle -D\rangle, \cV_G, 1+q; \cQ_{1+q}\right) \to
    H^3_{\rigfp, c0}\left(X_{G,\Kl}^{(2, m)}\langle -D\rangle, \cV_G, 1+q; \cQ_{1+q}\right).
   \]
  \end{proposition}

  \begin{proof}
   This follows from the fact that the image of $X_{H, \Delta}$ in $X_{G, \Kl}$ does not intersect the $(1, m)$ locus. Hence, by \cref{prop:pullback}, we obtain a pullback map from $H^3_{\rigfp, c0}\left(X_{G,\Kl}^{(2, m)}\langle -D\rangle, \cV_G, 1+q; \cQ_{1+q} \right)$ to $H^3_{\rigfp, c}\left(X_{H,\Delta}^{(m, m)}\langle -D\rangle, \cV_H, 1; \cQ_{1+q} \right)$, and this is compatible with the pullback from $X_{G,\Kl}^{m}$.
  \end{proof}

  \begin{remark}
   The advantage of working with $c0$-support cohomology of $X_{G, \Kl}^{2, m}$, rather than fully compactly-supported cohomology of $X_{G, \Kl}^{m}$, is that over $X_{G, \Kl}^{2, m}$ we can find a lift of the Frobenius map. It is this which will allow us to express rigid fp-cohomology (somewhat) concretely in terms of coherent sheaves and make the link to higher Coleman theory.
  \end{remark}

 \section{Proof of \texorpdfstring{\cref*{prop:BPrigclass}}{Theorem 11.1.1}}
  \label{sect:BPproof}

  Using the formalism of partially-compactly-supported cohomology developed above, we can now prove \cref{prop:BPrigclass}. The argument below is due to George Boxer and Vincent Pilloni (pers.~comm.); we are very grateful to them for explaining the argument to us. (Again, in this section the group $H$ plays no role, and we omit subscripts $(\dots)_G$.)

  \begin{remark}
   In an earlier draft of this paper, we gave a different argument for the existence of such a unique lifting, depending on an assertion (the ``Eigenspace Vanishing Conjecture'') describing the prime-to-$p$ Hecke eigenspaces appearing in the rigid cohomology of each stratum of $X_{\Kl, 0}$. The argument below replaces this ``prime-to-$p$'' information with (unconditional) information about the Frobenius action at $p$.
  \end{remark}

 \subsection{An exact sequence} We start with the following result:

  \begin{lemma}\label{lem:GKes}
   We have a long exact sequence
   \[ H^i_{\dR,c} (\cX^m_{\Kl}\langle -\cD\rangle,\cV)\to H^i_{\dR}(\cX_{\Kl}\langle -\cD\rangle,\cV)\to H^i_{\dR}(\cX_{\Kl}^e\langle -\cD\rangle,\cV) \to\, [+1]\]
  \end{lemma}
  \begin{proof}
   Recall that we have a decomposition of the special fibre
   \[ X_{\Kl,0}=X^m_{\Kl,0}\cup \overline{X^e_{\Kl,0}}.    \]
   By definition, we have an exact sequence
   \[ H^i_{\dR,c} (\cX^m_{\Kl}\langle -\cD\rangle,\cV)\to H^i_{\dR}(\cX_{\Kl}\langle -\cD\rangle,\cV)\to H^i_{\dR}(\overline{\cX_{\Kl}^e}\langle -\cD\rangle,\cV) \to\, [+1]. \]
   Now by \cite[Theorem C]{grossekloenne02}, we have
   \[ H^i_{\dR}(\overline{\cX_{\Kl}^e}\langle -\cD\rangle,\cV)\cong H^i_{\dR}(\cX_{\Kl}^e\langle -\cD\rangle,\cV),\]
   which finishes the proof.
  \end{proof}

  We shall prove the following identities:
  \begin{equation}\label{eq:BPidentities}
  \begin{aligned}
  \left(\varphi^2 - U'_{1, \Kl} \varphi + p^{r_2 + 1} U'_{2, \Kl}\right)^2 &= 0 \text{ on } H^i_{\dR,c} (\cX^m_{\Kl}\langle -\cD\rangle,\cV), \\
  \left(\operatorname{Ver}^2 - U'_{1, \Kl} \operatorname{Ver} + p^{r_2 + 1} U'_{2, \Kl}\right)^2 &= 0  \text{ on } H^i_{\dR} (\cX^e_{\Kl}\langle -\cD\rangle,\cV),
  \end{aligned}
  \end{equation}
  where $\varphi$ denotes the Frobenius map of rigid cohomology, and $\operatorname{Ver}$ its transpose (the Verschiebung), which acts on rigid cohomology as $p^{w} \varphi^{-1}$ (where $w = r_1 + r_2 + 3$ as usual).

 \subsection{Iwahori levels}

  We now introduce the Iwahori-level Shimura variety $Y_{\Iw}$. If $(A, C)$ is the universal abelian surface over $Y_{\Kl}$ and its level subgroup, the covering $Y_{\Iw} \to Y_{\Kl}$ classifies choices of $(p, p)$-subgroup schemes $D \subset A[p]$ with $C \subset D = D^\perp$. The map $Y_{\Iw, \Qp} \to Y_{\Kl, \Qp}$ is \'etale of degree $p + 1$ (although the map of $\Zp$-schemes $Y_{\Iw} \to Y_{\Kl}$ is not finite). Refining our boundary data if necessary, we obtain a smooth compactification $X_{\Iw}$, with a map $X_{\Iw} \to X_{\Kl}$ which is finite after inverting $p$.
  
  \begin{definition}
   Let $X_{\Iw,0}^m=X_{\Kl,0}^m\times_{X_{\Kl,0}} X_{\Iw,0}$ and $X_{\Iw,0}^e=X_{\Kl,0}^e\times_{X_{\Kl,0}} X_{\Iw,0}$.
  \end{definition}

  \begin{proposition}
   The natural pullback maps
   \[ H^i_{\dR,c} (\cX^m_{\Kl}\langle -\cD\rangle,\cV) \to H^i_{\dR,c} (\cX^m_{\Iw}\langle -\cD\rangle,\cV) \qquad\text{and} \qquad H^i_{\dR} (\cX^e_{\Kl}\langle -\cD\rangle,\cV) \to H^i_{\dR} (\cX^e_{\Iw}\langle -\cD\rangle,\cV) \]
   are injective.
  \end{proposition}

  \begin{proof}
   This follows from the finiteness of $X_{\Iw, \Qp}$ over $X_{\Kl, \Qp}$, allowing us to define a trace map which is a section of the above morphism.
  \end{proof}

  So it suffices to prove that the identities \eqref{eq:BPidentities} hold at Iwahori level, where the operator $U_{1, \Kl}'$ splits as $Z' + \Phi$ as above.

  \begin{notation}
   We let $X_{\Iw,0}^{m,m}$, resp.~ $X_{\Iw,0}^{m,e}$, denote the open subvarieties of $X_{\Iw,0}^{m, m}$ on which the quotient $D / C$ is multiplicative (resp.~\'etale); the union of these is exactly the preimage of $X_{\Kl}^{(2, m)}$.

   We define open subvarieties $X_{\Iw,0}^{e,m}$ and $X_{\Iw,0}^{e,e}$ of $X_{\Iw, 0}^e$ similarly.
  \end{notation}

  \begin{remark}
   Note that the ordinary locus of $X_{\Iw,0}$ is precisely the disjoint union of the $(m, m)$, $(m, e)$, $(e, m)$ and $(e, e)$ strata. These are exactly the top-dimensional strata in the EKOR stratification of the Iwahori-level Shimura variety, cf.~\cite[\S 6.3]{shenyuzhang21}.
  \end{remark}

  \begin{definition}
   We define the following partial-support de Rham cohomology groups:
   \begin{itemize}
   \item Let $H^i_{\dR,c,\varnothing}(\cX_{\Iw}^{m,e}\langle -\cD\rangle, -)$ denote cohomology with compact support towards $\cX_{\Iw} - \cX_{\Iw}^{m}$ (but non-compact support towards $\cX_{\Iw}^{m} - \cX_{\Iw}^{m,e}$).
   \item Let $H^i_{\dR,\emptyset, c}(\cX_{\Iw}^{e,m}\langle -\cD\rangle, -)$ denote cohomology with compact support towards $\cX_{\Iw}^e - \cX_{\Iw}^{e, m}$ (but non-compact support towards $\cX_{\Iw} - \cX_{\Iw}^{e}$).
   \end{itemize}
  \end{definition}

  \begin{lemma}
   \label{lem:BPexactseq}
   We have exact sequences
   \begin{align*}
    &H^i_{\dR,c} (\cX^{m,m}_{\Iw}\langle -\cD\rangle,\cV) \to
    H^i_{\dR, c}(\cX^m_{\Iw}\langle -\cD\rangle,\cV) \to H^i_{\dR,c,\emptyset}(\cX_{\Iw}^{m,e}\langle -\cD\rangle,\cV),\\
    &H^i_{\dR,\emptyset,c} (\cX^{e,m}_{\Iw}\langle -\cD\rangle,\cV)\to
    H^i_{\dR}(\cX^e_{\Iw}\langle -\cD\rangle,\cV)\to H^i_{\dR}(\cX_{\Iw}^{e,e}\langle -\cD\rangle,\cV)
   \end{align*}
   where the first arrow in each sequence is extension by $0$, and the second arrow is given by restriction.
  \end{lemma}

  \begin{proof}
   Analogous to the proof of Lemma \ref{lem:GKes}.
  \end{proof}

  \begin{proposition}
   On $\cX^{m,m}_{\Iw}$, the operator $\Phi = \left[ \Iw(p) \diag(1, 1, p, p) \Iw(p)\right]$ is a lifting of the Frobenius of the special fibre.
  \end{proposition}

  \begin{proof}
   The double coset $\left[ \Iw(p) \diag(1, 1, p, p) \Iw(p)\right]$ induces the morphism given by $(A, C, D) \mapsto \left(A/D, C^\perp/D, A[p] / D\right)$. Over the locus $X^{m,m}_{\Iw, 0}$, the subgroup $D$ is the unique connected subgroup of $A[p]$, which is exactly the kernel of Frobenius.
  \end{proof}

  Since the Frobenius of rigid cohomology can be computed using any overconvergent lifting, and $\cX^{(m,m)}_{\Iw}$ maps isomorphically to $\cX_{\Kl}^{(2, m)}$ (a section is given by the canonical subgroup), we deduce that
  \[ \left(\varphi^2 - U'_{1, \Kl} \varphi + p^{r_2 + 1} U'_{2, \Kl}\right) = 0 \quad\text{ on }\quad H^i_{\dR,c} (\cX^{2, m}_{\Kl}\langle -\cD\rangle, \cV).\]

  We now consider the diagram
  \[
   \begin{tikzcd}
    H^i_{\dR,c} (\cX^{(m,m)}_{\Iw}\langle -\cD\rangle,\cV)\rar &H^i_{\dR}(\cX^m_{\Iw}\langle -\cD\rangle,\cV)\rar["\res"] &H^i_{\dR,c,\emptyset}(\cX_{\Iw}^{(m,e)}\langle -\cD\rangle,\cV),\\
    H^i_{\dR,c} (\cX^{(2, m)}_{\Kl}\langle -\cD\rangle,\cV)  \rar \uar["\cong"] &H^i_{\dR,c}(\cX^m_{\Kl}\langle -\cD\rangle,\cV) \rar \uar["g"] &H^i_{\dR,c0}(\cX_{\Kl}^{(2, m)}\langle -\cD\rangle,\cV) \uar
   \end{tikzcd}
  \]
  The top row is part of the first exact sequence of \cref{lem:BPexactseq}. In the bottom row, the notation $H^i_{\dR,c0}(\cX_{\Kl}^{2, m}\langle -\cD\rangle,\cV)$ denotes cohomology with compact support towards $\cX_{\Kl}^e$ (but not towards $\cX_{\Kl}^{1, m})$); and the vertical arrows are the pullback maps. Since the middle vertical arrow $g$ is injective, we deduce that the bottom row is exact at the middle term.

  \begin{remark}
   Note we are not claiming that the bottom row extends to a long exact sequence.
  \end{remark}

  Exactly as before, the fact that $\Phi$ is a lifting of the Frobenius on the $(2, m)$ locus, and $U_{1, \Kl}' = Z' + \Phi$, implies that $\left(\varphi^2 - U'_{1, \Kl} \varphi + p^{r_2 + 1} U'_{2, \Kl}\right) = 0$ vanishes on $H^i_{\dR,c0} (\cX^{2, m}_{\Kl}\langle -\cD\rangle, \cV)$. Since this polynomial vanishes on the two end terms of the lower exact sequence, it follows that its square vanishes on the middle term. This proves the first of the identities \eqref{eq:BPidentities}.

  The proof of the second identity is similar: in this case, $\cX^{(e,m)}_{\Iw}$ maps isomorphically to $\cX^{(2, e)}_{\Kl}$, and on $\cX^{(e,m)}_{\Iw}$ the operator $Z'$ is a lift of the Verschiebung map, so the argument proceeds as before.

 \subsection{Consequences for $\eta$}

  Recall our running assumption that $\Pi$ be Klingen-ordinary, which implies $\{ \alpha, \beta\} \cap \{ \gamma, \delta\} = 0$.

  \begin{corollary}
   The class $\eta_{\dR,-D}$ maps to zero in $H^3_{\dR}(\cX_{\Kl}^e\langle -\cD\rangle,\cV)$.
  \end{corollary}

  \begin{proof}
   The image of $\eta_{\dR, -D}$ in this group is annilated by $\cQ(\varphi)$, since $\eta_{\dR, -D}$ itself is. However, since $\eta_{\dR, -D}$ lies in the $U_1' = \alpha + \beta$ and $p^{r_2 + 1} U_2' = \alpha \beta$ eigenspaces, it follows from the second identity of \cref{eq:BPidentities} that its image is also annihilated by $\cQ(\operatorname{Ver}) =\cQ(p^w \varphi^{-1})^2$.

   Since the roots of the quadratic polynomial $t^2\cQ(p^w t^{-1})$ are exactly $\gamma$ and $\delta$, it follows that $\cQ(p^w t^{-1})^2$ and $\cQ(t)$ are coprime. Hence the image of $\eta_{\dR, -D}$ must be zero.
  \end{proof}

  To fix a specific lifting, we use generalised eigenspaces. We consider the maximal submodule of each module in the exact sequence of \cref{lem:GKes} on which the two commuting operators
  \[
   U_{1, \Kl}' - (\alpha + \beta), \qquad p^{r_2 + 1} U_{2, \Kl}' - \alpha \beta
  \]
  act nilpotently. Passing to generalised eigenspaces is an exact functor (unlike ``normal'' eigenspaces), since it can be interpreted as localisation.

  From the identities \eqref{eq:BPidentities}, in this localised exact sequence, the operator $\varphi$ has generalised eigenvalues $\alpha, \beta$ on the $\cX_{\Kl}^m$ terms, and $\gamma, \delta$ on the $\cX_{\Kl}^e$ terms. Since these sets are disjoint, we conclude that the boundary maps
  \[ H^i_{\dR}(\cX_{\Kl}^e\langle -\cD\rangle,\cV) \to H^{i+1}_{\dR, c}(\cX_{\Kl}^m\langle -\cD\rangle,\cV) \]
  vanish on these generalised eigenspaces. Hence \cref{lem:GKes} splits into short exact sequences, one for each $i$; and thus there is a unique lifting of $\eta_{\dR, -D}$ to $H^{3}_{\dR, c}(\cX_{\Kl}^m\langle -\cD\rangle,\cV)$ which is a $U_1'$ and $U_2'$ generalised eigenvector. Since the lifting is unique, it must in fact be an eigenvector (not just generalised eigenvector) for both operators; and it is annihilated by both $\cQ(\varphi)^2$ and by $\cQ(\varphi) \cdot\cQ(p^w\varphi^{-1})$, and hence by their greatest common divisor, which is $\cQ(\varphi)$. This completes the proof of \cref{prop:BPrigclass}.


\mychapter{Step 3: Reduction to a pairing in coherent cohomology}



\section{Coherent cohomology and automorphic forms for \texorpdfstring{$G$}{G}}


 \subsection{Coefficient sheaves}

  We recall some definitions and results from \cite{LPSZ1}.
 
  \begin{definition}
   For $r_1,r_2,c\in\ZZ$ with $c\equiv r_1+r_2\pmod 2$, define $\lambda(r_1,r_2;c)$ to be the unique character of $T$ such that
   \[ \left(\begin{smallmatrix}st_1 &&&\\ & st_2 &&\\ && st_2^{-1} & \\ &&& st_1^{-1}\end{smallmatrix}\right)\mapsto t_1^{r_1}t_2^{r_2}s^c.\]
  \end{definition}

  \begin{remark}
   If $r_1\geq r_2$, then $\lambda(r_1,r_2;c)$ is dominant for $M_{\Sieg}$, and we write $W_G(r_1,r_2;c)$ for the irreducible representation of $M_{\Sieg}$ with highest weight $\lambda(r_1,r_2;c)$.
  \end{remark}

  \begin{definition}
   For $0\leq i\leq 3$, define $L_i$ to be the irreduible $M_{\Sieg}$-representation with the following highest weights:
   \begin{align*}
    L_0 &:  \lambda(r_1+3,r_2+3;r_1+r_2) &  L_1 &: \lambda(r_1+3,1-r_2;r_1+r_2)\\
    L_2 &:  \lambda(r_2+2,-r_1;r_1+r_2)  &  L_3 &: \lambda(-r_2,-r_1;r_1+r_2)
   \end{align*}
   For $K$ sufficiently small, write $\cL_i = [L_i]_{\can}$ for the associated vector bundle on $X_{K, \QQ}$ (the canonical extensions of the corresponding vector bundles over $Y_{K, \QQ}$).
  \end{definition}

  \begin{notation}
   For convenience, we re-number these vector bundles by setting $N^i = L_{3-i}$, and $\cN^i = \cL_{3-i}$ the corresponding vector bundles (so that $\cN^i$ is $\Omega^i\langle D \rangle$ if $r_1 = r_2 = 0$).
  \end{notation}

  \begin{note}
   The cohomology of these bundles, and their subcanonical analogues $[N^i]_{\mathrm{sub}} = \cN^i(-D)$, is canonically independent of the toroidal boundary data, and hence the direct limits
   \[ 
    \varinjlim_K H^\star(X_{K, \QQ}, \cN^i)\quad\text{and}
    \quad \varinjlim_K H^\star(X_{K, \QQ}, \cN^i(-D))
   \]
   are (left) $G(\Af)$-representations. Our normalisations are such that an element $\diag(x, \dots, x) \in Z_G(\Af)$ with $x \in \QQ_{> 0}$ acts on these as multiplication by $x^{r_1 + r_2}$.

   We know (see e.g. \cite[Theorem 5.2]{LPSZ1}) that for each $0 \le i \le 3$, the $\GSp_4(\Af)$-representation $\varinjlim_K H^{3-i}(X_{K, \QQ}, \cN^i) \otimes \QQbar_p$ and its cuspidal counterpart both contain a unique direct summand isomorphic to $\Pi_f'$; and if $j \ne 3-i$, then the $\Pi_f'$-generalised eigenspaces for the spherical Hecke operators in $H^j(X_{K, \QQ}, \cN^i)$ and $H^j(X_{K, \QQ}, \cN^i(-D))$ are zero.
  \end{note}

 \subsection{Classical Klingen-level Hecke operators}

  Taking the level at $p$ to be the Klingen parahoric $\Kl(p)$, we obtain an action of the local Hecke algebra $\ZZ[G(\Qp) \sslash \Kl(p) ]$ on the cohomology of the sheaves $\cN^i$.

  \begin{definition}
   We define the following operators:
   \begin{align*}
    U_{\Kl, 0} &= p^{-(r_1 + r_2)} \cdot \left[ \Kl(p) \diag(p, p, p, p) \Kl(p)\right] \\
    U_{\Kl, 1} &= \left[\Kl(p) \diag(p, p, 1, 1) \Kl(p)\right] &
    U'_{\Kl, 1} &= \left[ \Kl(p) \diag(1, 1, p, p) \Kl(p)\right] \\
    U_{\Kl, 2} &= p^{-r_2} \cdot \left[ \Kl(p) \diag(p^2, p, p, 1) \Kl(p)\right] &
    U'_{\Kl, 2} &= p^{-r_2} \cdot \left[\Kl(p)\diag(1, p, p, p^2)\Kl(p)\right].
   \end{align*}
  \end{definition}

  \begin{remark}
   The powers of $p$ are chosen so that these operators are \textbf{minimally integrally normalised}; that is, all their eigenvalues acting on $\Pi'_p$ are $p$-adically integral, because of the valuation estimates of \cref{eq:valuations}, and are units if $\Pi$ is ordinary at $p$. Of course, the eigenvalues of $U_{\Kl, 0}$ are roots of unity, and we shall generally use the more familiar alternative notation $\langle p \rangle$ for $U_{\Kl, 0}$.
  \end{remark}

  \begin{note}
   The operators $\{ \langle p \rangle, U_{\Kl, 1}, U_{\Kl, 2}\}$ generate a commutative subalgebra of the Hecke algebra, and $\{ \langle p \rangle, U'_{\Kl, 1}, U'_{\Kl, 2}\}$ generate another commutative subalgebra. Moreover, Serre duality interchanges these two subalgebras: more precisely, the transpose with respect to Serre duality of $U_{\Kl, 1}$ is $\langle p \rangle^{-1} U'_{\Kl, 1}$, and the transpose of $U_{\Kl, 2}$ is $\langle p \rangle^{-2} U'_{\Kl, 2}$.
  \end{note}

 \subsection{Restriction to the multiplicative locus, and classicity theorems}

  Recall that $X_{\Kl}$ parametrises pairs $(A, C)$, where $A$ is a semi-abelian surface with some prime-to-$p$ level structure and degeneration data at the cusps, and $C$ is a cyclic subgroup of order $p$. The \emph{Fargues degree} $\deg C$ is thus a function
  \[ \deg: X_{\Kl}(\CC_p)\to \left[0, 1\right], \]
   with degree 1 corresponding to the locus where $C$ is multiplicative.

  The images of $(A, C)$ under the correspondence $U_{\Kl, r}'$, for $r = 1, 2$, correspond to pairs $(A', C')$, where $\phi: A \to A'$ is an isogeny (of some specific type depending on $r$) whose kernel contains $C$, and $C'$ a cyclic subgroup of $A'[p]$ such that $\phi^\vee(C') = C$. This implies that $\deg C' \le \deg C$; so $U_{\Kl, r}'$ restricts to a correspondence $X_{G,\Kl}(\CC_p)_{[0, 1[} \,{\substack{\rightarrow \\ \rightarrow} }\,X_{G,\Kl}(\CC_p)_{[0, 1[}$. This implies that there is a well-defined action of $U_{\Kl, r}'$ on the compactly-supported cohomology $R\Gamma_c(\cX_{\Kl}^m, \cN^i)$ for any $i$, compatible with the extension-by-zero map to $R\Gamma(\cX_{G,\Kl}, \cN^i)$.

  \begin{proposition}
   \label{prop:slopeestimates}
   All slopes of $U_{\Kl, 2}'$ on $R\Gamma_c\left(\cX_{G,\Kl}^m, \cN^i\right)$ or $R\Gamma_c\left(\cX_{G,\Kl}^m, \cN^i(-D)\right)$ are:

   \begin{center}
   \begin{tabular}{c c}
      $ \ge -3$ & \quad if $i=0$,\\
      $ \ge -3$ & \quad if $i=1$,\\
      $ \ge (r_1 - r_2 - 2)$ & \quad if $i = 2$,\\
      $ \ge (r_1 + r_2)$ & \quad if $i = 3$.
    \end{tabular}
   \end{center}
   In particular, if $r_1 - r_2 \ge 3$, the $U_{\Kl, 2}'$-ordinary parts of these groups vanish for $i = 2$ or $i = 3$.
  \end{proposition}

  \begin{proof}
   We explain briefly how to deduce this from Pilloni's results in \cite{pilloni20}. We will first consider $U_{2, \Kl}$ acting on non-compactly-supported cohomology, and then dualize.
   
   The automorphic vector bundle of weight $(n_1, n_2; c)$, for $n_1 \ge n_2$, corresponds to $\Omega(n_1-n_2, n_2)$ in the notation of \emph{op.cit.}, but only up to a twist, since Pilloni always takes the central character of $\Omega(k, r)$ to be $k + 2r$, which is slightly different from our conventions for the $\cN^i$. In \S14.4.2 of \emph{op.cit.}, Pilloni defines an operator denoted ``$U$'' acting on the cohomology of $\Omega(k, r)$ (or its cuspidal variant); this is defined as $p^{-(3 + r)} (t_1)_\star (t_2)^\star$, where $t_i$ are appropriate degeneracy maps. Pilloni shows that this has slope $\ge -3$ (and conjectures that its slope should be $\ge 0$). Hence $(t_1)_\star (t_2)^\star$, which is the double-coset operator $\left[ \Kl(p) \diag(p^2, p, p, 1) \Kl(p)\right]$, has slope $\ge r$ (conjecturally $\ge r + 3$) on $\Omega(k, r)$. Correcting for the differences of central characters, we obtain that the slopes of our $U_{2, \Kl}$ acting on $R\Gamma\left(\cX_{G,\Kl}^m, \cN^i\right)$ or $R\Gamma\left(\cX_{G,\Kl}^m, \cN^i(-D)\right)$ are bounded below as follows:\medskip
   
   \noindent\begin{tabular}{r|c|c|c|c}
    sheaf & $\cN^3$ & $\cN^2$ & $\cN^1$ & $\cN^0$ \\
    \hline
    slope & $\ge -3$ & $\ge -3$ & $\ge r_1 - r_2 -2$ & $\ge r_1 + r_2$
   \end{tabular}
   \medskip
   
   A similar analysis can be carried out directly for $U'_{\Kl, 2}$ in place of $U_{\Kl, 2}$, but it is simpler to use the version of Serre duality for finite-slope parts developed in the proof of \cite[Theorem 5.7.2]{boxerpilloni20}. This gives a duality between the finite-slope part of $R\Gamma\left(\cX_{G,\Kl}^m, \cN^i\right)$ for $U_{\Kl, 2}$, and the finite-slope part of $R\Gamma\left(\cX_{G,\Kl}^m, \cN^{3-i}(-D)\right)$ for $U_{\Kl, 2}'$ (and similarly with $-D$ on the other side). Since the transpose of $U_{\Kl, 2}$ is $\langle p \rangle^{-2} U'_{\Kl, 2}$, and $\langle p \rangle$ acts via roots of unity, we obtain the stated slope bounds for $U'_{\Kl, 2}$.
  \end{proof}

  We also have the following complementary result. 
  
  \begin{theorem}[Classicity of ordinary eigenclasses]
   \label{thm:classicity}
   Suppose $i \in \{0, 1\}$. If $i = 0$, let $h = r_1 + r_2$. If $i = 1$, let $h = r_1 - r_2 - 2$. Then the extension-by-zero map
   \[ R\Gamma_c\left(\cX_{G,\Kl}^m, \cN^i(-D)\right) \longrightarrow R\Gamma\left(\cX_{G,\Kl}, \cN^i(-D)\right) \]
   is an isomorphism on the slope $< h$ generalised eigenspace (and similarly with $(-D)$).
  \end{theorem}
  
  \begin{proof}
   Again, this follows using Serre duality from the corresponding result for $U_{2, \Kl}$ on the non-compactly-supported cohomology, which is Theorem 14.7.1 of \cite{pilloni20}.
  \end{proof}
  
  We shall chiefly be interested in the ordinary parts of these modules, so we obtain a classicity result (for both $\cN^0$ and $\cN^1$) as long as $r_1 - r_2 \ge 3$.
  
  \begin{remark}
   The above two results are the only places in the proof of Theorem A where the assumption that $r_1 - r_2 \ge 3$ is needed.
  \end{remark}

 \subsection{The ordinary locus and the operator $Z'$}\label{ss:newHeckeop}

  Inside $\cX_{G,\Kl}^m$ we have the multiplicative-ordinary locus $\cX_{G,\Kl}^{(2,m)}$, parametrising $(A, C)$ where $A$ is ordinary and $C$ multiplicative. The correspondences $U_{\Kl, 1}'$ and $U_{\Kl, 2}'$ described above both act on $R\Gamma_{c0}(\cX_{G,\Kl}^{(2, m)}, \cN^i)$, since ordinarity is an isogeny invariant. However, over the multiplicative-ordinary locus there is an additional structure: we have a decomposition
  \[ \left. U_{\Kl, 1}' \right|_{\cX_{G,\Kl}^{(2,m)}} = Z' + \Phi \]
  as a sum of two simpler correspondences:
  \begin{itemize}
   \item The correspondence $\Phi$ is actually a morphism: it is the map $(A, C) \mapsto (A / \hat{A}[p], C' \bmod \hat{A}[p])$ where $\hat{A}$ is the formal group of $A$, and $C'$ is the unique subgroup of $\hat{A}[p^2]$ such that $pC' = C$. This is a lifting of the Frobenius map on the special fibre.
   \item The correspondence $Z'$ parametrises isogenies $(A, C) \mapsto (A/J, C')$, where $J \cap \hat{A}[p] = C$, and $C'$ is the unique multiplicative subgroup of $A'$ whose image under the dual isogeny is $C$.
  \end{itemize}

  These are related to classical correspondences at Iwahori level (since we can also see $\cX_{\Kl}^{(2, m)}$ as a dagger subvariety of the Iwahori-level Shimura variety, via the canonical-subgroup map): in the Iwahori-level Hecke algebra, $Z'$ corresponds to $\diag(1, p, 1, p)$, and $\Phi$ to $U'_{\Iw, 1} = \diag(1, 1, p, p)$.

  \begin{remark}
   For the sheaf $\cN^1$, this is the minimal integral normalisation of $Z'$ (but this is no longer the case on $\cN^i$ for $i \ne 1$). We have not attempted to give an optimal normalisation for the operator $\Phi$, since this will not play such a major role in our theory.
  \end{remark}

  \begin{note}
   The operator $U_{\Kl, 2}'$ (or more precisely its restriction to the ordinary locus) commutes with both $Z'$ and $\Phi$, and a simple double-coset computation shows that we have the following identity in the Iwahori level Hecke algebra:
   \begin{equation}
    \label{eq:commutationrelation}
    Z' \circ \Phi = p^{(r_2 + 1)} U'_{\Kl, 2}. \qedhere
   \end{equation}
  \end{note}

  \begin{convention}
   Since the operators $U'_{\operatorname{Kl}, 2}$ and $U'_{\operatorname{Iw}, 2}$ are compatible under pullback along the projection map induced by the inclusion of groups, there seems to be no harm in dropping the subscript and using the notation $U'_2$ for both.
  \end{convention}


 \subsection{Duality and vanishing for coherent cohomology}
  \label{sect:dualitycoherent}
  \begin{proposition}
   Both $X_{G,\Kl,0}^{(2, m)}$ and $X_{G,\Kl,0}^{(1, m)}$ are smooth, and their images in the minimal compactification $X_{G,\Kl}^{\mathrm{min}}$ are affine. 
  \end{proposition}

  \begin{proof}
   The smoothness of $X_{G,\Kl,0}^{(2, m)}$ is immediate from that of $X_{G,\Kl,0}^m$. It is easily seen that the space $X_{G,\Kl,0}^{(1, m)}$ maps isomorphically to the $p$-rank 1 locus in prime-to-$p$ level (cf.~proof of Lemma 10.5.2.2 in \cite{pilloni20}), and the smoothness of this image is established in the course of the proof of Lemma 6.4.2 of \emph{op.cit.}. For the second statement, see the proof of Theorem 11.2.1 of \cite{pilloni20}.
  \end{proof}

  \begin{notation}
   Let $\pi: X_{G,\Kl} \to X_{G,\Kl}^{\mathrm{min}}$ be the projection map. For the rest of this section, let $\cE = [W]$ be the canonical extension to $X_{G,\Kl}$ of an automorphic vector bundle attached to a $P_{\Sieg}$-representation $W$, and $\cE' = [W^\vee \otimes L(3, 3; 0)]$, so that the Serre dual of $\cE$ is $\cE'(-D)$ and vice versa.
  \end{notation}

  \begin{proposition}
   Let $U \subset \cX_{G, \Kl}^{\mathrm{min}}$ be an open dagger affinoid, and let $\tilde U = \pi^{-1}(U) \subseteq \cX_{G, \Kl}$. Then
   \begin{enumerate}[(a)]
    \item We have $H^i(\tilde U, \cE(-D)) = 0$ for $i \ne 0$.
    \item We have $H^i_c(\tilde U, \cE') = 0$ for $i \ne 3$.
    \item There is a perfect pairing of Hausdorff locally convex spaces
    \[ H^0(\tilde U, \cE(-D)) \times H^{3}_c(\tilde U, \cE') \to \Qp. \]
   \end{enumerate}
  \end{proposition}

  \begin{proof}
   Note that if $\tilde U$ is affinoid, then this is an instance of Grosse-Kl\"onne's Serre duality results for affinoid dagger spaces recalled in \S\ref{sect:serredual} above. So we shall aim to reduce to this case, using the fact that $R^i\pi_\star\left(\cE(-D)\right)= 0$ for all $i > 0$ by \cite[Theorem 8.6]{lan17}. Observe that $X_{G,\Kl,\Qp}^{\mathrm{min}}$ is normal, although not smooth.

   For part (a), we have $H^i(\tilde U, \cE(-D)) = \HH^i\big(U, R\pi_ \star\left(\cE(-D)\right)\big)$. By Lan's vanishing results, this is just $H^i(U, \pi_\star \cE(-D))$. As $U$ is affinoid, $H^i(U, -)$ vanishes for $i > 0$. So this group vanishes for all $i > 0$ as required.

   For parts (b) and (c), we argue as follows. For a quasicompact, normal dagger space $\cX$, equidimensional of dimension $d$, and a complex $C$ of coherent sheaves on $\cX$, we shall say $C$ is ``anti-concentrated in degree $0$'' if $\underline{\operatorname{RHom}}(C, \underline{\omega}_\cX)$ is concentrated in degree $-d$, where $\underline{\omega}_\cX$ is the dualizing complex.

   We now note that:
   \begin{itemize}
   \item If $\cX$ is smooth (or just Gorenstein), then $\underline{\omega}_\cX$ is a line bundle in degree $-d$, so a vector bundle (regarded as a complex concentrated in degree 0) will also be anti-concentrated in degree 0.

   \item If $\cX$ is affinoid, and $C$ is anti-concentrated in degree $0$, then the compactly-supported hypercohomology $\HH^i_c(\cX, C)$ vanishes for $i \ne d$. This follows straightforwardly from the Serre duality result recalled in \S\ref{sect:serredual} above: applying this with $\cF = \underline{\operatorname{Ext}}^{-d}(C, \underline{\omega}_\cX)$, so that $C = \underline{\operatorname{RHom}}(\cF, \underline{\omega}_{\cX})[-d]$, we have
   \begin{equation}\label{eq:serredual1}
   \HH^i_c(\cX, C) = H^{d-i}\left(\cX,\cF\right)^\vee,
   \end{equation}
   where $(-)^\vee$ denotes $\Hom_{\mathrm{cts}}(-, \Qp)$; this is 0 if $i \ne d$.

   \item If $\pi: \tilde{\cX} \to \cX$ is a proper map between normal dagger spaces, then we have relative Serre duality: if we define $D_{\cX} = \underline{\operatorname{RHom}}_{\cO(\cX)}(-, \underline{\omega}_{\cX})$ and similarly $D_{\tilde{\cX}}$, then
   \[ R\pi_\star\left(D_{\tilde{\cX}} \cF\right) = D_{\cX}\left(R\pi_\star \cF\right). \]
   It follows, in particular, that if $\tilde{\cX}$ is smooth (but $\cX$ may not be), and $\cF$ is a vector bundle such that $R\pi_\star(\cF)$ is concentrated in degree 0, then the Serre-dual vector bundle $\cF' = \underline{\Hom}(\cF, \omega_{\tilde{\cX}})$ is anti-concentrated in degree 0. (Actually we do not know a reference for the above assertion in full generality; but we shall only need this when the spaces and sheaves concerned are the analytifications of algebraic varieties and sheaves on them, in which case the result follows from classical Grothendieck duality for algebraic varieties.)
   \end{itemize}

   In our setting, Lan's vanishing results tell us that the higher direct images of $\cE(-D)$ vanish, so $R\pi_\star \cE(-D)$ is concentrated in degree 0. Hence the complex $R\pi_\star \cE'$ is anti-concentrated in degree 0, and we have
   \[ R\pi_\star \cE' = D_U(\pi_\star \cE(-D))[-3].\]
   Thus \eqref{eq:serredual1} becomes
   \[ \HH^i_c(U, R\pi_\star \cE') = H^{3-i}(U, \pi_\star \cE(-D))^\vee, \]
   or
   \[ H^i(\tilde{U}, \cE') = H^{3-i}(\tilde{U}, \cE(-D))^\vee. \]
   This clearly implies (c), and (b) follows from this together with (a).
  \end{proof}

  \begin{corollary}
   \label{cor:ge1vanish}
   For $\cE$ as above, we have $H^i(X_{G,\Kl}^m, \cE(-D)) = 0$ for $i \notin \{0, 1\}$, and $H^i_c(X_{G,\Kl}^m, \cE) = 0$ for $i \notin \{2, 3\}$.
  \end{corollary}

  \begin{proof}
   There are two affinoids $U_1, U_2$ in $X_{G,\Kl}^{\min}$ such that $\pi^{-1}(U_1)$ and $\pi^{-1}(U_2)$ cover $X_{G,\Kl}^m$ (see e.g. proof of \cite[Lemma 14.8.2]{pilloni20}). By the previous proposition, we see that $H^\bullet(X_{G,\Kl}^m, \cE(-D))$ is computed by a \u{C}ech complex concentrated in degrees 0 and 1. Similarly, the compactly-supported cohomology is supported by a ``homological'' \u{C}ech complex concentrated in degrees 2 and 3.
  \end{proof}

  \begin{corollary}
   \label{cor:c0vanish}
   For $\cE'$ as above, we have $H^i_{c0}(X_{G,\Kl}^{(2, m)}, \cE') = 0$ unless $i \in \{2, 3\}$.
  \end{corollary}

  \begin{proof}
   By definition, we have an exact triangle
   \[ R\Gamma_{c}(X_{G,\Kl}^{(1, m)}, \cE') \to R\Gamma_c(X_{G,\Kl}^m, \cE') \to R\Gamma_{c0}(X_{G,\Kl}^{(2, m)}, \cE') \to [+1]. \]

   We claim that $H^*_{c}(X_{G,\Kl}^{(1, m)}, \cE')$ is concentrated in degree 3. The image of $X_{G,\Kl}^{(1, m)}$ in the minimal compactification is the locus where the Hasse invariant has positive valuation. It is thus naturally covered by an increasing sequence of affinoids $U_i$ (given by requiring the valuation of a lift of Hasse to be $\ge r_i$, for some sequence of positive rationals $r_i \to 0$), and $R\Gamma_{c}(X_{G,\Kl}^{(1, m)}, \cE') = \varinjlim_r R\Gamma_c(\pi^{-1}(U_i), \cE')$, which vanishes outside degree 3 by the proposition above. It now follows from the mapping triangle that $R\Gamma_{c0}$ is supported in degrees $\{2, 3\}$.
  \end{proof}

  \begin{remark}
   It seems highly likely that $H^i_{c0}(X_{G,\Kl}^{(2, m)}, \cE')$ vanishes for $i = 3$ as well, but this is not easy to check. It is equivalent to showing that $H^3_c(X_{G,\Kl}^{(1, m)}, \cE') \to H^3_c(X_{G,\Kl}^m, \cE')$ is surjective. If we knew that Serre duality held for $X_{G,\Kl}^m$ this would be obvious, since the dual map $H^0(X_{G,\Kl}^m, \cE(-D)) \to H^0(X_{G,\Kl}^{(1, m)}, \cE(-D))$ is clearly injective; but we do not know this, since neither $X_{G,\Kl}^m$ nor its image in $X_{G,\Kl}^{\min}$ is affinoid.
  \end{remark}


   \subsection{Coherent $H^2$ eigenclasses from $\Pi$}\label{sect:coherentfromPi}

   The input we need from higher Coleman theory is the following. We fix an automorphic representation $\Pi$ which is cohomological with coefficients in $V(r_1, r_2; r_1 + r_2)$, and unramified and Klingen-ordinary at $p$, as before; and we choose a vector $\eta^{\mathrm{alg}}_{-D} \in H^2\left(X_{G, \Kl, \Qp}, \cN^1(-D)\right)[\Pif']$ (the vector denoted $\eta$ in \cref{ssec:testdataatp}) which is stable under $\Kl(p)$ and lies in the ordinary eigenspace for $U'_2$.

   \begin{remark}
    If $(\alpha, \beta, \gamma, \delta)$ are the Hecke parameters of $\Pi'_p$, ordered such that $v_p(\alpha) \le \dots \le v_p(\delta)$ and normalised such that $v_p(\alpha) \ge 0, v_p(\alpha \beta) \ge r_2 + 1$, then the Klingen-ordinarity condition is that $v_p(\alpha\beta)$ should be exactly $r_2 + 1$, and the ordinary $U'_2$ eigenvalue is the $p$-adic unit $\lambda = \frac{\alpha\beta}{p^{r_2 + 1}}$.
   \end{remark}

   \begin{note}
    The operator $U_{\Kl, 1}'$ acts on $\eta$ as multiplication by $\alpha + \beta$ (which may or may not be a $p$-adic unit).
   \end{note}

   \begin{proposition}
    \label{prop:etaproperties}
    Suppose $r_1 - r_2 \ge 3$. Then there exists a unique class \[\eta^m_{\coh,-D} \in H^2_c\left(\cX_{G,\Kl}^m, \cN^1(-D)\right)\] with the following two properties:
    \begin{enumerate}
     \item $U'_{\Kl, 2}$ acts on $\eta^m_{\coh,-D}$ as multiplication by  $\frac{\alpha\beta}{p^{r_2 + 1}}$.
     \item The image of $\eta^m_{\coh,-D}$ under the extension-by-zero map is $\eta^{\mathrm{alg}}_{-D}$.
    \end{enumerate}

    This class enjoys the following additional properties:
    \begin{enumerate}
     \item[(3)] The operator $U'_{\Kl, 1}$ acts on $\eta^m_{\coh,-D}$ as multiplication by $\alpha + \beta$.
     \item[(4)] The spherical Hecke algebra acts via the system of eigenvalues associated to $\Pi'$.
    \end{enumerate}
   \end{proposition}

   \begin{proof}
    The existence of a unique $\eta^m_{\coh,-D}$ satisfing (1) and (2) is an instance of the classicity result of Theorem \ref{thm:classicity}. Since the Hecke operator $U_{1, \Kl}'$ and the prime-to-$p$ Hecke operators commute with $U_{2, \Kl}'$, and with the extension-by-zero map, parts (3) and (4) follow from the corresponding properties of $\eta^{\mathrm{alg}}_{-D}$ and the uniqueness of the lifting.
   \end{proof}

   \begin{definition}
    Let $\eta^m_{\coh}$ be the image of $\eta^m_{\coh,-D}$ under the natural map
    \[ H^2_c\left(\cX_{G,\Kl}^m, \cN^1(-D)\right) \to H^2_c\left(\cX_{G,\Kl}^m, \cN^1\right).\]
   \end{definition}

   This enjoys analogues of properties (1)--(4) (\emph{mutatis mutandis}).

   \begin{definition}
    Let $\eta^{(2, m)}_{\coh, -D}$ be the image of $\eta^m_{\coh, -D}$ in $H^2_{c0}\left(\cX_{G,\Kl}^m,\cN^1(-D)\right)$ via the restriction map. For future use, we define $\eta^{(2, m)}_{\coh}$ to be the image of $\eta^{(2, m)}_{\coh, -D}$ in $H^2_{c0}\left(\cX_{G,\Kl}^m,\cN^1\right)$.
   \end{definition}


  \begin{proposition}\label{prop:kerZ} As above, let $\cQ(T)=(1 - \tfrac{T}{\alpha})( 1 - \tfrac{T}{\beta})$. Then the class $\cQ(\Phi)\cdot \eta^{(2, m)}_{\coh,-D}$ lies in the kernel of $Z'$.
  \end{proposition}

   \begin{proof}
   We know that $\eta^{(2, m)}_{\coh, -D}$ is an eigenvector for $U'_2$ with eigenvalue $\alpha\beta / p^{r_2 + 1}$, and for $Z' + \Phi$ with eigenvalue $\alpha + \beta$. Using the identity \eqref{eq:commutationrelation} the results follow formally.
   \end{proof}




\section{fp-cohomology and coherent fp-pairs for \texorpdfstring{$G$}{G}}\label{sec:fpcohomandfppairs}


 \subsection{The dual BGG complex}\label{ss:dualBGG}

  \begin{definition}\label{def:BGG}
   Define the \emph{dual BGG complex associated to $\cV$} to be
   \[
   \BGG(\cV): \cN^0\xrightarrow{\nabla^0} \cN^1 
    \xrightarrow{\nabla^1} \cN^2 \xrightarrow{\nabla^2} \cN^3,
   \]
   where the differentials are given by certain homogeneous differential operators of degrees $r_2+1$, $r_1-r_2+1$ and $r_2+1$, respectively (c.f. \cite[\S 7]{tilouine12}.

   We equip it with the following filtration:
   \[
   \sFil^n\BGG(\cV)=\begin{cases}
   \cN^0 \to \cN^1 \to \cN^2 \to \cN^3 \qquad &\text{if $n\leq 0$}\\
   0     \to \cN^1 \to \cN^2 \to \cN^3 & \text{if $1 \le n\leq r_2+1$}\\
   0     \to     0 \to \cN^2 \to \cN^3 & \text{if $r_2+2 \le n \leq r_1+2$}\\
   0     \to     0 \to     0 \to \cN^3 & \text{if $r_1+3 \le n\leq r_1+r_2+3$}\\
   0 & \text{if $r_1 + r_2 + 4 \le n$}.
   \end{cases}
   \]
   We define $\BGG_c(\cV)$ to be the subcomplex with $\cN^i$ replaced by $\cN^i(-D)$.
  \end{definition}

  \begin{proposition}\label{prop:dRdirectsummand}
   The dual BGG complex $\BGG(\cV)$ is a direct summand of the logarithmic de Rham complex $\DR(\cV) \coloneqq \cV\otimes \Omega^\bullet\langle D \rangle$ (in the category of abelian sheaves over $X_{G,\Kl, \Qp}$). The inclusion is Hecke equivariant,  and the projection map is a quasi-isomorphism of filtered complexes. The same holds for $\BGG_c(\cV)$ and $\DR_c(\cV)$.
  \end{proposition}

  \begin{proof}
   See \cite[\S 7]{tilouine12} for the statement for  $\BGG(\cV)$. For the version with compact support, see \cite[\S 5.4]{lanpolo}.
  \end{proof}

  \begin{definition}\label{def:cohBGG}
   For $i,j\geq 0$, $n\in\ZZ$ we define
   \begin{align*}
   \sC^{i,j}_c(\cX_{G,\Kl}^m,\BGG_c(\cV), n) &= H^j_c\!\left(\cX_{G,\Kl}^m, \sFil^{n} \cN^i(-\cD)\right)\!, & 
   \!\!\sC^{i,j}_c(\cX_{G,\Kl}^m,\BGG_c(\cV)) &= H_c^j\!\left(\cX_{G,\Kl}^m, \cN^i(-\cD)\right)\!,\\
   \sC^{i,j}_c(\cX_{G,\Kl}^m,\BGG(\cV), n) &= H^j_c\!\left(\cX_{G,\Kl}^m, \sFil^n \cN^i\right), & \sC_c^{i,j}(\cX_{G,\Kl}^m,\BGG(\cV)) &= H^j_c\left(\cX_{G,\Kl}^m, \cN^i\right),\\
   \sC^{i,j}_{c0}(\cX_{G,\Kl}^{2,m},\BGG(\cV), n) &= H^j_{c0}\left(\cX_{G,\Kl}^{2,m}, \sFil^n \cN^i\right), & \sC^{i,j}_{c0}(\cX_{G,\Kl}^{2,m},\BGG(\cV)) &= H^j_{c0}\left(\cX_{G,\Kl}^{2,m}, \cN^i\right).
   \end{align*}
  \end{definition}

  \begin{note}
   \label{note:newcomplex}
   If $n \ge 1$, which is the case which will interest us, we have
   \[ \sC^{i,j}_{c}(\cX_{G,\Kl}^m,\BGG_c(\cV), n) = \sC^{i,j}_{c0}(\cX_{G,\Kl}^{2,m},\BGG_c(\cV), n) = 0\]
    unless $i \in \{1, 2, 3\}$ and $j \in \{2, 3\}$, since $\sFil^n \cN^i$ is zero unless $1 \le i \le 3$, and the functors $H^j_c(\cX^m_{G,\Kl}, -)$ and $H^j_{c0}(\cX^{2,m}_{G,\Kl}, -)$ vanish on canonical vector bundles unless $j \in \{2, 3\}$ by \cref{cor:ge1vanish} and \cref{cor:c0vanish}.

   For $\sC^{i,j}_{c}(\cX_{G,\Kl}^m,\BGG_c(\cV)(-\cD), n)$ we have a slightly weaker result: the non-zero terms are in degrees $1 \le i \le 3$, $1 \le j \le 3$, since it is obvious that $H^0_c(\cX_{G,\Kl}^m, -)$ vanishes for any locally free sheaf.
  \end{note}

  \begin{note}\label{note:BGGvsDR1}
   Definition \ref{def:cohBGG} also makes sense when we replace $\BGG_\star(\cV)$ by $\DR_\star(V)$ and $\cN^i$ by $\cV\otimes \Omega_G^i$. By Proposition \ref{prop:dRdirectsummand}, we obtain natural maps from the `BGG-version' of the groups to the respective `de Rham' versions.
  \end{note}

  \begin{proposition}
   We have first-quadrant cohomological spectral sequences, starting at the $E_1$ page (with differentials on the $E_1$ page given by $\nabla$):
   \begin{align}
   \sC^{i,j}_{c}(\cX_{G,\Kl}^m,\BGG_c(\cV), n)  & \Rightarrow \wH^{i+j}_{\dR,c}(\cX_{G,\Kl}^m\langle -\cD\rangle,\cV,n),\label{eq:wssgeq1minusD}\\
   \sC^{i,j}_{c}(\cX_{G,\Kl}^m,\BGG(\cV), n)  & \Rightarrow \wH^{i+j}_{\dR,c}(\cX_{G,\Kl}^m,\cV,n),\label{eq:wssgeq1}\\
   \sC^{i,j}_{c0}(\cX_{G,\Kl}^{2,m},\BGG(\cV), n)  & \Rightarrow \wH^{i+j}_{\dR,c0}(\cX_{G,\Kl}^{2,m},\cV,n) \label{eq:wssord}
   \end{align}
   which are compatible under the restriction map $\res^{2,m}$. If $n \ge 1$, all three spectral sequences degenerate at $E_3$. Similarly, for the unfiltered complexes we have Fr\"olicher spectral sequences
   \begin{align}
   \sC^{i,j}_{c}(\cX_{G,\Kl}^m,\BGG_c(\cV))  & \Rightarrow H^{i+j}_{\rig,c}(X_{G,\Kl,0}^m\langle -D_0\rangle,\cV),\label{eq:ssgeq1minusD}\\
   \sC^{i,j}_{c}(\cX_{G,\Kl}^m,\BGG(\cV))  & \Rightarrow H^{i+j}_{\rig,c}(X_{G,\Kl,0}^m,\cV),\label{eq:ssgeq1}\\
   \sC^{i,j}_{c0}(\cX_{G,\Kl}^{2,m},\BGG(\cV))  & \Rightarrow H^{i+j}_{\dR,c0}(\cX_{G,\Kl}^{2,m},\cV). \label{eq:ssord}
   \end{align}
  \end{proposition}

  \begin{proof}
   In each case, the spectral sequence arises as one of the spectral sequences associated to a suitable double complex computing $\wH^\bullet$. The degeneration follows from the fact that the $E_1$ terms are zero outside $1 \le i \le 3$.
  \end{proof}

  \begin{notation}
   We denote the $E_2$ pages of these spectral sequences by $\cH^{ij}_{?}(\dots)$, so $\cH^{ij}_?(\dots)$ is the $i$-th cohomology of the complex $\sC^{\bullet j}_{?}(\dots)$.
  \end{notation}

  \begin{corollary}\label{cor:isototilderig}
   Let $0\leq q\leq r_2$. Then the edge maps at $(1, 2)$ of the spectral sequences \eqref{eq:wssgeq1} amd \eqref{eq:wssord} are isomorphisms
   \begin{subequations}
    \begin{align}
    \alpha_{G, \rig, c}: \wH^{3}_{\dR,c}(\cX_{G,\Kl}^m,\cV,1+q)
    &\xrightarrow{\,\cong\,} \cH^{1,2}_c(\cX_{G,\Kl}^m,\BGG(\cV), 1+q) = H_c^2\left(\cX_{G,\Kl}^m,\cN^1\right)^{\nabla=0},\label{eq:cohrigiso1}\\
    \alpha_{G, \rig, c0}: \wH^{3}_{\dR,c0}(\cX_{G,\Kl}^{2,m},\cV,1+q)
    &\xrightarrow{\,\cong\,} \cH^{1,2}_{c0}(\cX_{G,\Kl}^{2,m},\BGG(\cV), 1+q) = H_{c0}^2\left(\cX_{G,\Kl}^{2,m},\cN^1\right)^{\nabla=0}.\label{eq:cohrigiso2}
    \end{align}
    The spectral sequence $\sC^{i,j}_c(\cX_{G,\Kl}^m,\BGG_c(\cV)(-D), 1+q)$ gives an exact sequence
    \begin{equation}\label{eq:cohrigiso1D}
    \begin{tikzcd}[row sep=0, column sep=small]
    0 \rar &
    \dfrac{H_c^1\left(\cX_{G,\Kl}^m,\cN^2(-\cD)\right)^{\nabla=0}}
    {\nabla H_c^1\left(\cX_{G,\Kl}^m,\cN^1(-\cD)\right)} \rar &
    \wH^{3}_{\dR,c}(\cX_{G,\Kl}^m\langle -\cD\rangle,\cV,1+q) \ar[out=0, in=180, looseness=2.5, sloped, "\alpha_{G, \rig, c, -D}"]{d} \\
    & & H_c^2\left(\cX_{G,\Kl}^m,\cN^1(-\cD)\right)^{\nabla=0}
    \rar["\ \partial\ "] &
    \dfrac{H^1_c\left(\cX_{G,\Kl}^m,\cN^3(-\cD)\right)}{\nabla \cdot H^1_c\left(\cX_{G,\Kl}^m,\cN^2(-\cD)\right)}.
    \end{tikzcd}
    \end{equation}
   \end{subequations}
  \end{corollary}

  \begin{proof}
   For the first two formulae, we know that both of the relevant spectral sequences have $E_1^{ij} = 0$ unless $i \ge 1$ and $j \ge 2$, so $E^{i, 3-i}_\infty = 0$ for $i \ne 1$, and $E_{\infty}^{12} = E_2^{12} = \ker(E_1^{12} \to E_1^{22})$.
  \end{proof}

  \subsection{Ordinary parts}\label{sect:ordinaryparts}

   For the two spectral sequences over $\cX_{G,\Kl}^m$, the results of Corollary \ref{cor:isototilderig} can be sharpened enormously by taking into account the action of the Hecke operator $U_2'$. Recall that the coherent cohomology groups (both with and without $-D$) have slope decompositions for the action of $U_2'$, so the slope 0 subspace is finite-dimensional and there exists an idempotent projector $e(U_2')$ projecting onto this subspace. Moreover, the operator $U_2'$, and hence the slope 0 projector, are compatible with the morphisms in the spectral sequence.

  \begin{proposition}
   If $n \ge 1$, we have
   \[ e(U_2') \cdot \sC^{i,j}_c(\cX_{G,\Kl}^m,\BGG_c(\cV), n) = e(U_2') \cdot \sC^{i,j}_c(\cX_{G,\Kl}^m,\BGG(\cV), n) = 0\]
   for $i \ne 1$.
  \end{proposition}

  \begin{proof}
   This is a consequence of the slope estimates of \cref{prop:slopeestimates}.
  \end{proof}

  \begin{corollary}
   For any $0 \le q \le r_2$, \cref{eq:cohrigiso1,eq:cohrigiso1D} give isomorphisms
   \begin{subequations}
    \begin{align}
    e(U_2') \cdot \wH^{3}_{\dR,c}(\cX_{G,\Kl}^m\langle -\cD\rangle,\cV,1+q)
    &\xrightarrow{\,\cong\,}  e(U_2') \cdot H_c^2\left(\cX_{G,\Kl}^m,\cN^1(-\cD)\right),\label{eq:cohrigiso-ordD}\\
    e(U_2') \cdot \wH^{3}_{\dR,c}(\cX_{G,\Kl}^m,\cV,1+q)
    &\xrightarrow{\,\cong\,}  e(U_2') \cdot H_c^2\left(\cX_{G,\Kl}^m,\cN^1\right).
    \end{align}
   \end{subequations}
  \end{corollary}

  \begin{remark}
   These isomorphisms are clearly compatible under the ``forget $-D$'' maps on both sides.
  \end{remark}

  For the ``unfiltered'' spectral sequence (i.e.~taking $n = 0$), we have two nonzero columns on the $E_1$ page after applying $e(U_2')$; so for $i + j = 3$ we obtain
  \begin{multline*}
    0 \longrightarrow \frac{e(U_2') H^2_c(\cX^m_{G, \Kl}, \cN^1(-D))}{\nabla \cdot e(U_2') H^2_c(\cX^m_{G, \Kl}, \cN^0(-D))}
    \longrightarrow e(U_2')H^3_{\dR, c}(\cX^m_{G, \Kl}\langle-\cD\rangle, \cV) \\ 
    \longrightarrow e(U_2') H^3_c(\cX^m_{G, \Kl}, \cN^0(-D))^{\nabla = 0} \longrightarrow 0.
  \end{multline*}
  However, supposing $r_1 - r_2 > 2$, we can use the classicity results of \ref{thm:classicity} above to identify the coherent $H^2_c$ and $H^3_c$ terms with their analogues for the proper variety $X_{G, \Kl, \Qp}$. Since the Fr\"olicher spectral sequence (with log poles) degenerates at $E_1$, the morphisms $\nabla$ are zero, and we can write the above exact sequence as
  \begin{multline*}
   0 \to e(U_2') H^2_c(\cX^m_{G, \Kl}, \cN^1(-D))  \longrightarrow e(U_2')H^3_{\dR, c}(\cX^m_{G, \Kl}\langle-\cD\rangle, \cV) \\ \longrightarrow e(U_2') H^3_c(\cX^m_{G, \Kl}, \cN^0(-D))^{\nabla = 0} \to 0.
  \end{multline*}
  In particular, the natural inclusion of complexes $\sFil^{1+q} \BGG_c \into \BGG_c$ induces an injection
  \[ e(U_2')\wH^3_{\dR, c}(\cX^m_{G, \Kl}\langle-\cD\rangle, \cV, 1+q) \into e(U_2')H^3_{\dR, c}(\cX^m_{G, \Kl}\langle-\cD\rangle, \cV).\]

  \begin{definition}\label{def:tildeetageq1}
   Let
   \[ \tilde{\eta}^m_{\rig,-D} \in
    e(U_2') \wH^{3}_{\dR,c}(\cX_{G,\Kl}^m\langle -\cD\rangle,\cV,1+q)
   \]
   denote the preimage of the class $\eta_{\coh, -D}^m$ of \cref{prop:etaproperties} under the isomorphism of \eqref{eq:cohrigiso-ordD}.
  \end{definition}

  \begin{proposition}\label{prop:sameclass}
   Let $\iota$ denote the map on cohomology induced by the inclusion of complexes
   \[ \iota: \Fil^q \BGG_c(V) \into \BGG_c(V).\]
   Then we have
   \[ \iota\left(\tilde{\eta}^m_{\rig,-D}\right) = \eta^m_{\rig, -D} \in H^{3}_{\rig,c}(X_{G,\Kl,0}^m\langle -D_0\rangle,\cV), \]
   where $\eta^m_{\rig, -D}$ is the class constructed in Proposition \ref{prop:BPrigclass}.
  \end{proposition}

  \begin{proof}
   We deduce from Pilloni's control theorems \cite{pilloni20} and the slope estimates from Proposition \ref{prop:slopeestimates} that the map
   \[ e(U_2')\Gr^0 H^3_{\rig,c}(\cX^m_{G,\Kl}\langle -\cD\rangle,\cV)\longrightarrow e(U_2')\Gr^0 H^3_{\dR}(X_{G,\Kl,\Qp}\langle -D\rangle,\cV)\]
   is an isomorphism. Since $\eta_{\dR,-D}\in \Fil^{1+q} H^3_{\dR}(X_{G,\Kl}\langle -\cD\rangle,\cV)$, and the class $\eta^m_{\rig,-D}$ is $U_2'$-ordinary, we deduce that it is in the image of $\wH^3_{\dR,c}(\cX^m_{G,\Kl}\langle -\cD\rangle,\cV,1+q)$ in $H^3_{\dR,c}(\cX^m_{G,\Kl}\langle -\cD\rangle,\cV)$.

   On the other hand, we have a natural ``extension-by-zero'' map
   \[ \wH^3_{\dR,c}(\cX^m_{G,\Kl})\langle -\cD\rangle,\cV,1+q) \to \Fil^{1+q} H^3_{\dR}(X_{G,\Kl,\Qp}\langle -D\rangle,\cV), \]
   fitting into a commutative diagram
   \[
    \begin{tikzcd}[row sep=large]
     e(U_2')\wH^3_{\dR,c}(\cX^m_{G,\Kl}\langle -\cD\rangle,\cV,1+q)
       \arrow[r]
       \arrow[d, "\cong"']
     &
     e(U_2')\,\Fil^{1+q}H^3_{\dR}(X_{G,\Kl}\langle -D\rangle,\cV)
       \arrow[d]
     \\
     e(U_2')H^2_c(\cX^{m}_{G,\Kl},\cN^1(-\cD))
       \arrow[r, "\cong"]
     &
     e(U_2')H^2_c(X_{G,\Kl},\cN^1(-\cD))
    \end{tikzcd}
   \]
   Since $\eta^m_{\rig,-D}$ lifts $\eta_{\dR, -D}$, which in turn maps to $\eta^{\alg}_{-D}$ in the bottom right corner, we see that $\eta^m_{\rig,-D}$ must map to $\eta^m_{\coh, -D}$ at the bottom left.
  \end{proof}

 \subsection{Lifting to fp-cohomology}\label{sect:liftingfp}

  Let $P \in 1 + T \Qp[T]$ be a polynomial with constant term 1. Recall the definition of \emph{Gros fp-cohomology} given in \cref{def:grosfp} above. In our present context this becomes:

\begin{definition}
 \label{def:grosfp2}
 Define the \emph{Gros $\fp$-cohomology of $\cV$ over $\cX_{G,\Kl}^{2,m}$} with $c0$-support, twist $n$ and polynomial $P$, denoted $\wH^\bullet_{\rigfp,c0}(\cX_{G,\Kl}^{2,m},\cV,n;P)$, to be the cohomology of the mapping fibre of the diagram
 \begin{equation}
  \label{eq:SGfp0}
  \RGt_{\dR, c0}(\cX_{\Kl}^{2,m}, \cV, n) \xrightarrow{\vspace{1ex}P(\varphi / p^n) \circ \iota\vspace{1ex}}  R\Gamma_{\dR, c0}(\cX_{\Kl}^{2,m},\cV).
 \end{equation}
 where $\iota$ is the map on cohomology induced by the inclusion $\sFil^{n}\cN^\bullet \into \cN^\bullet$. We denote the analogous group formed using the sheaves $\cN^\bullet(-D)$ instead of $\cN^\bullet$ by $\wH^\bullet_{\rigfp,c0}(\cX_{G,\Kl}^{2,m},\cV(-D),n;P)$.

 We make similar definitions for $\cX_{G,\Kl}^{m}$ with compact support (and there is a natural restriction map between the two).
\end{definition}

\begin{lemma}\label{lem:rigsurjontokerP}
 For all $i\geq 0$, we have surjective maps
  \[ \wH^i_{\rigfp,c}(\cX_{G,\Kl}^{m},\cV, n;P)\twoheadrightarrow \wH^i_{\dR,c}(\cX_{G,\Kl}^{m},\cV, n)^{P(p^{-n}\varphi) \circ \iota =0} \]
  and similarly
  \[ \wH^i_{\rigfp,c0}(\cX_{G,\Kl}^{2,m},\cV, n;P)\twoheadrightarrow \wH^i_{\dR,c0}(\cX_{G,\Kl}^{2,m},\cV, n)^{P(p^{-n}\varphi) \circ \iota =0} \]
\end{lemma}

\begin{proof}
 Clear from the long exact sequence associated to the mapping fibre.
\end{proof}

\begin{proposition}
 We can choose a class (not necessarily unique!)
 \[  \tilde{\eta}_{\rigfp,-D}^{m} \in \wH^3_{\rigfp, c}(\cX_{G,\Kl}^m\langle -\cD\rangle, \cV, 1+q; \cQ_{1+q}) \]
 which lies in the $\Pif'$-generalised eigenspace for the spherical Hecke operators, and whose image in the group $\wH^3_{\dR, c}(\cX_{G,\Kl}^m\langle -\cD\rangle, \cV, 1+q)$ is $\tilde{\eta}_{\rig,-D}^m$.
\end{proposition}

\begin{proof}
 From Propositions \ref{prop:BPrigclass} and \ref{prop:sameclass} we know that
 \[
 \tilde{\eta}_{\rig, -D}^m \in \wH^3_{\dR, c}(\cX_{G,\Kl}^m\langle -\cD\rangle, \cV, 1+q)^{\cQ(\varphi) \circ \iota = 0}.
 \]
 It follows that $\tilde{\eta}_{\rig,-D}^m$ is in the image of the map
 \[ \wH^3_{\rigfp, c}(\cX_{G,\Kl}^m\langle -\cD\rangle, \cV, 1+q; \cQ_{1+q}) \to \wH^3_{\dR, c}(\cX_{G,\Kl}^m\langle -\cD\rangle, \cV, 1+q).\qedhere\]
\end{proof}

\begin{corollary}\label{cor:cosptildeetageq1}
 The class $\tilde{\eta}_{\rigfp,-D}^{m} $ is sent to $\eta_{\lrigfp,-D}$ under the cospecialisation map.
\end{corollary}
\begin{proof}
 Clear by uniqueness (c.f. Remark \ref{rem:Quniqueness}).
\end{proof}

 These two propositions show that, for any $\tilde{\eta}_{\rigfp,-D}^{m}$ satisfying the conditions of the proposition, its image in the (non-Gros) fp-cohomology $H^3_{\rigfp, c}(\cX_{G,\Kl}^m\langle -\cD\rangle, \cV, 1+q; \cQ_{1+q})$ is a valid choice for the class $\eta_{\rigfp,-D}^{m}$ of \cref{prop:eta-rigsyn}. So we can, and do, assume that these classes are chosen compatibly.

 \begin{remark}
  We can be a little more precise: both $\eta_{\rigfp,-D}^{m}$ and its tilde version are well-defined modulo elements lying in some quotient of $H^2_{\rig, c}(\cX_{\Kl}^m, \dots)\{ \Pif'\}$, where $\{ \Pif'\}$ denotes generalised eigenspace. So in fact \emph{any} $\eta_{\rigfp,-D}^{m}$ as in \cref{prop:eta-rigsyn} is the image of some $\tilde{\eta}_{\rigfp,-D}^{m}$.
 \end{remark}

 \begin{notation}
  Write $\widetilde\Eis^{[t_1,t_2],(m,m)}_{\rigsyn,\underline\Phi}$ for the image of $\Eis^{[t_1,t_2],(m,m)}_{\rigsyn,\underline\Phi}$ in $\wH^2_{\rigsyn}(\cX^{(m, m)}_{H,\Delta},2)$ under the specialisation map defined in Remark \ref{rem:Grosspcosp}.
 \end{notation}

By Remark \ref{rem:Grosspcosp}, we obtain the following result:
\vspace{1ex}

\begin{mdframed}
 \begin{corollary}\label{cor:redtoGros}
  Then
  \begin{multline*}
   \left\langle \Eis^{[t_1,t_2],(m,m)}_{\rigsyn,\underline\Phi},\, (\iota_\Delta^{[t_1,t_2]})^\star( \eta_{\rigfp,-D}^{m})\right\rangle_{\rigfp,X^{m,m}_{H,\Delta}}\\
   = \left\langle \widetilde\Eis^{[t_1,t_2],(m,m)}_{\rigsyn,\underline\Phi},\, (\iota_\Delta^{[t_1,t_2]})^\star( \tilde\eta_{\rigfp,-D}^{m})\right\rangle_{\widetilde\rigfp,X^{m,m}_{H,\Delta}}.
  \end{multline*}
 \end{corollary}
\end{mdframed}
\vspace{1ex}

Our aim is to express this pairing in terms of coherent cohomology. The main tool for relating syntomic, resp. Gros fp-cohomology with coherent cohomology is the \emph{ Poznan spectral sequence} (c.f. Proposition \ref{prop:Poznan}), which should be thought of as a syntomic (resp. $\fp$-) analogue of the Hodge-to-de Rham spectral sequence.


 \subsection{Gros fp-cohomology and the Pozna\'n spectral sequence}\label{section:Poznan}

  By \cref{prop:restricttoc0}, the pullback $(\iota_\Delta^{[t_1,t_2]})^\star( \tilde\eta_{\rigfp,-D}^{m})$ only depends on the restriction of $ \eta_{\rigfp,-D}^{m}$ to $X_G^{2,m}$. We shall show that this can be expressed in terms of coherent cohomology.


  \subsubsection*{The Pozna\'n spectral sequence}

   As we already saw in \cref{sect:BPproof} above, on $\cX^{2,m}_{G,\Kl}$ we have  a lifting of the Frobenius map to the cohomology of the individual sheaves $\cN^i$, given by the action of the Hecke operator $\diag(1, 1, p, p)$. This allows us to study Gros fp-cohomology via a spectral sequence, as follows.

   \begin{definition}
    \label{def:cij}
    Let $P(T)\in\Qp[T]$ have constant term $1$. Define $\sC^{\bullet, j}_{\fp,c0}(\cX_{G,\Kl}^{2,m}, \BGG(\cV), n; P)$ to be the mapping fibre of the morphism of complexes
    \[
     \sC^{\bullet, j}_{c0}(\cX_{G,\Kl}^{2,m},\BGG(\cV), n) 
     \xrightarrow{\ P(\varphi / p^n) \circ \iota\ } 
     \sC^{\bullet, j}_{c0}(\cX_{G,\Kl}^{2,m},\BGG(\cV)).
    \]
   \end{definition}

   Thus
   \[ 
    \sC^{i, j}_{\fp,c0}(\cX_{G,\Kl}^{2,m}, \BGG(\cV), n; P) = 
    H^j_{c0}\left(\cX_{G,\Kl}^{2,m}, \sFil^n \cN^i \right) \oplus H^j_{c0}\left(\cX_{G,\Kl}^{2,m},\cN^{i-1}\right) 
   \]
   with the differentials being $(x, y) \mapsto (\nabla x, P(\varphi / p^n) \iota(x) - \nabla y)$.

   \begin{remark}
    We shall only use this definition for $n \ge 1$, in which case one sees easily that this group is zero unless $j \in\{2,3\}$ and $1 \le i \le 4$, and the $i=0$ terms vanish if $n \ge 1$.
   \end{remark}

   \begin{proposition}\label{prop:Poznan}
    There is a first-quadrant spectral sequence, the \emph{Pozna\'n spectral sequence}, with
    \[ {}^{\Pz}E_1^{ij} = \sC^{i, j}_{\fp,c0}(\cX_{G,\Kl}^{2,m},\BGG(\cV), n; P). \]
    The spectral sequence degenerates at $E_3$, and its abutment is the Gros fp-cohomology \eqref{eq:SGfp0}.
   \end{proposition}

   \begin{proof}
    Choose double complexes computing $R\Gamma_{\dR, c0}(\cX_{G,\Kl}^{2,m}, \cV, n)$ and $R\Gamma_{\dR, c0}(\cX_{G,\Kl}^{2,m}, \cV, n)$ respectively, in such a way that $P(\varphi / p^n) \circ \iota$ extends to a map of double complexes. Then we can compute $\wH^{\bullet}_{\rigfp, c0}(Y_{G,\Kl,0}^{2,m},\cE, n, P)$ as the total complex of the associated mapping fibre, i.e. by the total complex of a triple complex.
    The Pozna\'n spectral sequence is one of the spectral sequences associated to this triple complex.
   \end{proof}


  \subsubsection*{Coherent fp-pairs}

   \begin{definition}
    \label{fppairs}
    \begin{enumerate}[(a)]

     \item We define a \emph{coherent fp-pair} of degree $(i, j)$, twist $n$ and $c0$-support to be an element of
     \[ \mathscr{Z}^{ij}_{\fp,c0}(\cX_{G,\Kl}^{2,m},\BGG(\cV), n;\!P) \coloneqq \ker\!\left(\! \sC^{ij}_{\fp,c0}(\cX_{G,\Kl}^{2,m},\BGG(\cV), n;\!P)\to \sC^{i+1,j}_{\fp,c0}(\cX_{G,\Kl}^{2,m},\BGG(\cV), n;\! P)\!\right)\!,\]
     i.e.~a pair of elements
     \[ \left(x, y\right) \in H_{c0}^{j}\left(\cX_{G,\Kl}^{2,m}, \sFil^n \cN^i\right)\oplus H_{c0}^j(\cX_{G,\Kl}^{2,m},\cN^{i-1})\]
     which satisfy $\nabla(x) = 0$, $\nabla(y)= P(p^{-n}\varphi)\iota(x)$, where $\iota$ is the map on cohomology induced by the inclusion $\sFil^{n}\cN^i \into \cN^i$.

     \item We define the \emph{group of coherent fp-classes} of degree $(i,j)$,  to be the $E^{ij}_2$-term of the Pozna\'n spectral sequence, so it is the quotient of the group of coherent fp-pairs by the subgroup of pairs of the form
     \[ (x, y) = \left( \nabla(u), P(p^{-n}\varphi) \iota(u) - \nabla(v)\right)\]
     for some $(u, v) \in \sC^{i-1,j}_{\fp,c0}(\cX_{\Kl}^{2,m},\BGG(\cV), n; P)$. We write $\cH^{i,j}_{\fp,c0}\left(\cX_{G,\Kl}^{2,m},\BGG(\cV), n; P\right)$ to denote this quotient. 

    \end{enumerate}
   \end{definition}

   \begin{lemma}\label{lem:fpsurjontokerP}
    For any $j$ and $n$ there is a long exact sequence (writing `$\dots$' for ``$\cX_{G,\Kl}^{2,m},\BGG(\cV)$''):
    \[
     \begin{tikzcd}[column sep=tiny]
     \dots
       \arrow[r]
     &
     \cH^{i,j}_{\fp,c0}\left(\dots, n;P\right)
       \arrow[r]
     &
     \cH^{i,j}_{c0}(\dots, n)
       \arrow[r, "P(p^{-n}\varphi)\circ \iota"]
       \arrow[d, equal]
     &
     \cH^{i,j}_{c0}(\dots)
       \arrow[r]
       \arrow[d, equal]
     &
     \cH^{i+1,j}_{\fp,c0}\left(\dots\right)
       \arrow[r]
     &
     \dots
     \\
     &&
     \dfrac{H^j_{c0}(\cX_{G,\Kl}^{2,m}, \sFil^n \cN^i)^{\nabla=0}}
          {\nabla H^j_{c0}(\cX_{G,\Kl}^{2,m}, \sFil^n \cN^{i-1})}
     &
     \dfrac{H^j_{c0}(\cX_{G,\Kl}^{2,m}, \cN^i)^{\nabla=0}}
          {\nabla H^j_{c0}(\cX_{G,\Kl}^{2,m}, \cN^{i-1})}
     \end{tikzcd}
    \]
   \end{lemma}

   \begin{proof}
    This is the long exact sequence associated to the mapping fibre \eqref{def:cij}.
   \end{proof}

   \begin{corollary}
    \label{cor:fppairmapsS}
    If $0\leq q\leq r_2$, then the spectral sequence gives rise to an isomorphism
    \begin{align*}
     \alpha_{G, \rigfp, c0}: \mathscr{Z}^{1, 2}_{\fp,c0}(\cX_{G,\Kl}^{2,m}, \BGG(\cV),1+q; P)
     & \xrightarrow{\,\cong\,} \cH^{1, 2}_{\fp,c0}(\cX_{G,\Kl}^{2,m}, \BGG(\cV),1+q;P)\\
     & \xrightarrow{\,\cong\,} \wH^{3}_{\rigfp, c0}(\cX_{G,\Kl}^{2,m}, \cV,1+q; P).
    \end{align*}
   \end{corollary}

   \begin{proof}
    Immediate from the fact that  ${}^{\Pz}E^{ij}_1$ is supported in the range $i \ge 1, j \ge 2$ by Note \ref{note:newcomplex}.
   \end{proof}

   \begin{note}\label{note:BGGvsDR2}
    Replacing $\BGG(\cV)$ by $\DR(\cV)$, we obtain an isomorphism
    \begin{equation}
     \label{eq:fpDR}
     \cH^{1, 2}_{\fp,c0}(\cX_{G,\Kl}^{2,m}, \DR(\cV),1+q;P)   
     \xrightarrow{\,\cong\,}    
     \wH^{3}_{\rigfp, c0}(\cX_{G,\Kl}^{2,m}, \cV,1+q; P)
    \end{equation}
   which is compatible with the natural map
   \[
     \cH^{1, 2}_{\fp,c0}(\cX_{G,\Kl}^{2,m}, \BGG(\cV),1+q;P) \longrightarrow \cH^{1, 2}_{\fp,c0}(\cX_{G,\Kl}^{2,m}, \DR(\cV),1+q;P)
   \]
   arising from Proposition \ref{prop:dRdirectsummand} (c.f. Note \ref{note:BGGvsDR1}).
    \end{note}

  \subsubsection*{Comparison of spectral sequences}

   There is a crucial compatibility between the edge maps of the Pozna\'n spectral sequence and the Fr\"olicher spectral sequence for (truncated) rigid cohomology:

   \begin{proposition}\label{prop:diagfpcoh}
    If $0\leq q\leq r_2$, then we have a commutative diagram
    \[ 
    \begin{tikzcd}[row sep=large]
     \wH^{3}_{\rigfp, c0}(\cX_{G,\Kl}^{2,m}, \cV, 1+q;P)
      \arrow[r, two heads]
      \arrow[d, equal, "\alpha_{G,\rigfp, c0}" left] &
     \wH^{3}_{\dR,c0}(\cX_{G,\Kl}^{2,m},\cV, 1+q)^{P(p^{-(1+q)}\varphi) \circ \iota =0}
      \arrow[d, equal, "\alpha_{G, \rig, c0}"]\\
     \cH^{1,2}_{\fp, c0}(\cX_{G,\Kl}^{2,m}, \BGG(\cV), 1+q;P) 
      \arrow[r, two heads] & 
     \cH^{1,2}_{c0}(X_{G,\Kl}^{2,m},\BGG(\cV), 1+q)^{P(p^{-(1+q)}\varphi) \circ \iota =0}
    \end{tikzcd}
    \]
%
    Here, the horizontal arrows are the surjections of \cref{lem:fpsurjontokerP,lem:rigsurjontokerP}, and the vertical isomorphisms are given by \cref{cor:isototilderig,cor:fppairmapsS}.
   \end{proposition}
   \begin{proof}
    Clear from the construction.
   \end{proof}


 \subsection{Coherent fp-pairs from $\eta$}\label{ssec:coherentfppaireta}

  \begin{definition}\label{def:etaord}
   Define
  \[
   \tilde{\eta}_{\rigfp}^{(2,m)} \in \wH^3_{\rigfp, c0}(\cX_{G,\Kl}^{2,m}, \cV, 1+q; \cQ_{1+q})
  \]
  to be the image of $\tilde{\eta}_{\rigfp,-D}^{m} $ under restriction to $\cX_{G,\Kl}^{2,m}$ and forgetting $-D$.
  \end{definition}

  We can now use \cref{cor:fppairmapsS} to represent $\tilde{\eta}_{\rigfp}^{(2,m)}$ by a pair of classes in coherent cohomology:

 \begin{proposition}
  There exists a unique  coherent $\fp$-pair $\left(\eta_{\coh}^{(2,m)}, \zeta\right)$ which maps to $\tilde{\eta}_{\rigfp}^{(2,m)}$ under the isomorphism (c.f. \cref{cor:fppairmapsS})
     \[ \wH^3_{\rigfp, c0}(\cX_{G,\Kl}^{2,m}, \cV, 1+q;\cQ_{1+q}) \cong  \mathscr{Z}^{1,2}_{\fp, c0}(\cX_{G,\Kl}^{2,m},\BGG(\cV),1+q; \cQ_{1+q}).\]
 \end{proposition}
\begin{proof}
 Clear.
\end{proof}

  \begin{note}\label{note:lackofuniqueness}
   By construction, the class $\zeta$ is a class in $ H^2_{c0}(\cX_{G,\Kl}^{2,m}, \cN^0) $ which satisfies
    \begin{equation}
     \label{eq:zetadef}
     \cQ_{1+q}(\Phi_{1+q}) \, \eta^{(2,m)}_{\coh} = \nabla \zeta.
    \end{equation}
   Observe that equation \eqref{eq:zetadef} does not determine $\zeta$ uniquely: it is only unique modulo
   \[ H^2_{c0}(\cX_{G,\Kl}^{2,m}, \cN^0)^{\nabla = 0} \cong H^2_{\dR, c0}(\cX_{G,\Kl}^{2,m}, \cV).\]
   In other words, if $\xi$ is any other element of $H^2_{c0}(\cX_{G,\Kl}^{2,m}, \cN^0) $ which satisfies
   \[\cQ_{1+q}(\Phi_{1+q}) \, \eta^{(2,m)}_{\coh} = \nabla \xi,\]
   then $\zeta-\xi\in H^2_{c0}(\cX_{G,\Kl}^{2,m}, \cN^0)^{\nabla = 0}$.
  \end{note}

%
%
%


 \subsection{Lifting to the de Rham sheaves}
  \label{sect:dRsheaves}

  \begin{definition} \
   \begin{itemize}
    \item Define $\breve{\zeta}$ to be the image of $\zeta$ in $H^2_{c0}(\cX_{G,\Kl}^{2,m}, \cV\otimes\Omega_G^0 \langle D \rangle)$.
    \item For $0\leq q\leq r_2$, define $\breve{\eta}^m_{\coh,-D}$ to be the image of $\eta^m_{\coh,-D}$ under the composition of maps
    \[  
     H^2_{c}(\cX_{G,\Kl}^m, \cN^1\langle -\cD\rangle) \to
     H^2_{c}(\cX_{G,\Kl}^m, \sFil^{r_2}\cV\otimes\Omega_G^1\langle -D \rangle) \to
     H^2_{c}(\cX_{G,\Kl}^m, \sFil^{q}\cV\otimes\Omega_G^1\langle -D \rangle),
    \]
    where the first map is given by the inclusion of complexes in Proposition \ref{prop:dRdirectsummand}, and the second map is induced from the natural inclusion of sheaves.
    \item Write $\breve{\eta}^{(2,m)}_{\coh}$ for the image of $\breve{\eta}^m_{\coh,-D}|_{\cX^{2,m}_{G,\Kl}}$ in $H^2_{c0}(\cX_{G,\Kl}^{2,m}, \sFil^{q}\cV\otimes\Omega_G^1)$.
    \end{itemize}
  \end{definition}

  \begin{lemma}
   The class $\breve{\eta}^m_{\coh,-D}$ maps to ${\eta}^m_{\coh,-D}$ under the natural map induced from the projection
   \[ \sFil^{r_2}\cV\otimes\Omega_G^1\langle -D \rangle) \to \cN^1.\]
  \end{lemma}
  \begin{proof}
   It is immediate from Proposition \ref{prop:dRdirectsummand} that the image of $\cN^1$ in the full de Rham complex is contained in $\sFil^{r_2}\cV\otimes\Omega_G^1$. In order to prove the result, it is hence sufficient to show that composition of the inclusion and the projection map is the identity on $\cN^1$.  But this follows from the results in \cite[Ch. VI, \S 6]{faltingschai}.
  \end{proof}

  The following lemma is direct consequence of the corresponding results for $\eta^m_{\coh,-D}$ (Proposition \ref{prop:etaproperties}), 
  using the the inclusion of the dual BGG complex into the de Rham complex is Hecke equivariant.

  \begin{lemma} \
   \begin{itemize}
    \item The operator $U'_{\Kl,2}$ acts on $\breve{\eta}^m_{\coh,-D}$ as multiplication by $\frac{\alpha\beta}{p^{r_2+1}}$.
    \item The operator $U'_{\Kl,1}$ acts on $\breve{\eta}^m_{\coh,-D}$ as multiplication by $\alpha+\beta$.
    \item The spherical Hecke algebra acts via the system of eigenvalues associated to $\Pi'$.
   \end{itemize}
  \end{lemma}
%

  \begin{proposition}\label{prop:etahornsrep}
The classes $\breve{\zeta}$ and $\breve{\eta}^{(2,m)}_{\coh,q}$ satisfy
   \[ \nabla \breve{\eta}^{(2,m)}_{\coh,q}=0\qquad\text{and}\qquad \cQ_{1+q}(\Phi_{1+q})\,\breve{\eta}^{(2,m)}_{\coh}=\nabla\breve{\zeta}\]
   and hence give rise to a class in $\cH^{1, 2}_{\fp,c0}(\cX_{G,\Kl}^{2,m}, \DR(\cV),1+q; \cQ_{1+q})$. Moreover, this class maps to $\tilde\eta^{(2,m)}_{\rigfp}$ under the isomorphism \eqref{eq:fpDR}.
  \end{proposition}
  \begin{proof}
   Immediate.
  \end{proof}


\section{fp-cohomology and coherent fp-pairs for \texorpdfstring{$H$}{H}}\label{section:PoznanforH}

  The results of the previous section show that we can express $\tilde\eta^{(2,m)}_{\rigfp}$ as a coherent $\fp$-pair. In this section, we develop the theory of coherent $\fp$-pairs for the syntomic cohomology of $\cY^{m,m}_{H,\Delta}$. We will apply it in  \cref{ssec:Eiscohfppair} to describe the class $\widetilde{\Eis}^{[t_1,t_2],(m,m)}_{\rigsyn,\underline{\Phi}}$ in terms of coherent cohomology.

 \subsection{The Pozna\'n spectral sequence for $H$}

  Let $W$ be an algebraic representation of $H$, and write $\cW$ for the corresponding coherent sheaf on $\cX_\Delta$. Let  $R\in\Qp[t]$ have constant coefficient $1$, and let $n\geq 0$.  We can then consider the Gros-fp cohomology
  \[  \wH^\bullet_{\rigfp,\star}(\cX^{m,m}_\Delta\langle \diamondsuit\rangle,\cW,n; R), \]
  where $\star\in\{ \varnothing,c\}$ and $\diamondsuit\in\{\varnothing,-\cD_\Delta\}$ (c.f. \cref{def:grosfp2}).

   Recall that if $R(p^{-1}) \ne 0$, we define the trace map
   \[ \wH^5_{\rigfp,c}(\cX^{m,m}_{H,\Delta},\Qp, 3; R) \to \Qp \]
   as $\tfrac{1}{R(p^{-1})}$ times the trace map on rigid cohomology.

  \begin{remark}
   As usual, the factor $\tfrac{1}{R(p^{-1})}$ serves to make the trace maps compatible with the natural maps of complexes $\RGt_{\rigfp,c}(-; R) \to \RGt_{\rigfp,c}(-; R')$ for polynomials $R \mid R'$. (This map acts as $(R'/R)(p^{-n} \varphi)$ on the rigid complex, with $n = 3$; but $\varphi = p^2$ on the top-degree cohomology, hence $R(p^{-1})$ is the correct normalising factor.)
  \end{remark}

  \begin{definition}[{cf.~\cref{def:cij}}]
   For $j,n\geq 0$ ,$\star\in\{\emptyset, c\}$ and $\diamondsuit\in\{\varnothing,-\cD_\Delta\}$, we define the complex
   \[ \sC^{\bullet,j}_{\star}(X_{H,\Delta}^{m,m}\langle \diamondsuit\rangle, \cW,n;R)\]
   with terms
   \[ \sC^{i,j}_{\star}(X_{H,\Delta}^{m,m}\langle \diamondsuit\rangle, \cW,n;R)= H^j_\star(\cX_{H,\Delta}^{m,m},\sFil^{n-i}\cW\otimes\Omega_\Delta^i\langle \diamondsuit\rangle)\oplus H^j_\star(\cX_{H,\Delta}^{m,m},\cW\otimes\Omega_\Delta^{i-1}\langle \diamondsuit\rangle)\]
   and differentials
   \[ (x,y)\mapsto \left(\nabla x,\, R(\varphi_H^\star/p^n)\iota(x)-\nabla y\right).  \]
  \end{definition}

  \begin{proposition}\label{prop:PoznanH}
   For $\star\in\{\varnothing,c\}$ and $\diamondsuit\in\{\varnothing,-\cD_\Delta\}$, we have the \emph{Pozna\'n spectral sequence}
   \[ {}^{\Pz}E_1^{ij}= \sC^{i,j}_{\star}(X_{H,\Delta}^{m,m}\langle \diamondsuit\rangle, \cW,n;R) \Rightarrow \wH^{i+j}_{\rigfp,\star}(\cX^{m,m}_{H,\Delta}\langle \diamondsuit\rangle,\cW,n;R).\]
  \end{proposition}
  \begin{proof}
   Analogous to the proof of Proposition \ref{prop:Poznan}
  \end{proof}

  We define the group of coherent fp-classes, denoted $H^{i,j}_{\star}(X_{H,\Delta}^{m,m}\langle \diamondsuit\rangle, \cW,n;R)$, analogously to \cref{fppairs}.

  \begin{corollary}\label{cor:HGrosisoms}
   The Pozna\'n spectral sequence gives rise to isomorphisms
   \begin{align*}
     \alpha_\Delta: &H^{i,0}(\cX^{m,m}_{H,\Delta}\langle \diamondsuit\rangle, \cW, n;R) \xrightarrow{\,\cong\,} \wH^i_{\rigfp}(\cX^{m,m}_{H,\Delta}\langle \diamondsuit\rangle, \cW, n;R) ,\\
     \alpha_{\Delta,c}: &H_c^{i,2}(\cX^{m,m}_{H,\Delta}\langle \diamondsuit\rangle, \cW, n;R) \xrightarrow{\,\cong\,} \wH^{i+2}_{\rigfp,c}(\cX^{m,m}_{H,\Delta}\langle \diamondsuit\rangle, \cW,n;R) .
   \end{align*}
  \end{corollary}
  
  \begin{proof}
   Easy computation, using that since $\cX^{m,m}_\Delta$ is affinoid, we have
   \begin{align*}
    H^i(\cX^{m,m}_{H,\Delta}\langle\diamondsuit\rangle,\cW)=0 & \qquad \text{for $i\neq 0$,}\\
    H^i_c(\cX^{m,m}_{H,\Delta}\langle\diamondsuit\rangle,\cW)=0 & \qquad \text{for $i\neq 2$}.\qedhere
   \end{align*}
  \end{proof}

  (Note that this holds for both $\diamondsuit=\varnothing$ \emph{and} $\diamondsuit=-\cD_{\GL_2}$, in contrast to the situation for $G$.)

  \begin{note}\label{note:highestfpclassphiaction}
   In particular, if $n\geq 3$ we have
   \[ H_c^{3,2}(\cX^{m,m}_{H,\Delta}\langle -\cD_\Delta\rangle, \Qp, n;R) \xrightarrow{\ \cong\ } \wH^{5}_{\rigfp,c}(\cX^{m,m}_{H,\Delta}\langle -\cD_\Delta\rangle, \Qp,n;R)
     \ \cong\ 
     \Qp.\]
   The Frobenius operator $\varphi_H^\star$ acts on $\wH^{5}_{\rigfp,c}(\cX^{m,m}_{H,\Delta}\langle -\cD_\Delta\rangle, \Qp,n;R)$ as multiplication by $p^2$.
  \end{note}

  \begin{note}\label{rem:PoznanisoGL2}
   Similarly, let $U$ be an algebraic representation of $\GL_2$, and write $\cU$ for the corresponding coherent sheaf on $\cX_{\GL_2, \Iw}$, the modular curve of Iwahori level at $p$. Let $\diamondsuit\in\{\varnothing,-\cD_{\GL_2}\}$. Then the Pozna\'n spectral sequence gives rise to an isomorphism
   \begin{equation}\label{eq:GL2iso}
    \alpha_{\GL_2}: H^{i,0}(\cX^{m}_{\GL_2,\Iw}\langle\diamondsuit\rangle, \cU, n;R) \xrightarrow{\,\cong\,} \wH^i_{\rigfp}(\cX^{m}_{\GL_2,\Iw}\langle \diamondsuit\rangle, \cU;R) .\qedhere
   \end{equation}
  \end{note}


 \subsection{Compatibility with cup products} \label{ssec:Bessercup}

  \begin{lemma}\label{lem:compcup}
   Let $P(T),\, Q(T)\in 1+T\Qp[T]$. Using the same formalism as \cite[\S 2]{besser12}, we can construct a cup product
   \[ H^{i,0}(\cX^{m,m}_{H,\Delta}\langle-\cD_\Delta\rangle, \cW, m;P)\times H_c^{j,2}(\cX^{m,m}_{H,\Delta}, \cW^\vee, n;Q)\xrightarrow{\  \cup \ } H_c^{i+j,2}(\cX^{m,m}_{H,\Delta}\langle-\cD_\Delta\rangle, \Qp, m+n;P \star Q)\]
   which is compatible under the isomorphisms from Corollary \ref{cor:HGrosisoms} with the cup product in Gros-fp cohomology.
  \end{lemma}
  \begin{proof}
   Standard check.
  \end{proof}

  \begin{note}
   If $m+n\geq 3$ and $i+j=3$, then we obtain a pairing
   \begin{equation}\label{eq:cohpairing}
    \langle\quad,\quad\rangle_{\coh-\fp, \cX^{m,m}_{H,\Delta}}:\, H^{i,0}(\cX^{m,m}_{H,\Delta}, \cW(-\cD_\Delta), m;P)\times H_c^{j,2}(\cX^{m,m}_{H,\Delta}, \cW^\vee, n;Q) \longrightarrow \Qp
   \end{equation}
   which is compatible with change of polynomial in $P$ and $Q$.
  \end{note}


 \section{Syntomic Eisenstein classes via coherent cohomology}

\subsection{Hecke operators for $\GL_2$}

Let $k \in \ZZ$. Then we define the space of modular forms for $\GL_2$ of weight $k$, denoted $M_{k}$, as a $\GL_2(\Af)$-module, normalised such that $\stbt A 0 0 A$, for $A \in \QQ^+$ acts as $A^{k-2}$. This means that the double-coset operator $\left[\stbt{\varpi_\ell}{0}{0}{1}\right]$ on the $\{ \stbt{*}{*}{0}{1} \bmod N\}$ invariants coincides with the classical $T_\ell$ (resp.~$U_\ell$) if $\ell \nmid N$ (resp.~$\ell \mid N$).

\begin{remark}
 These are the same normalisations as \cite{LPSZ1} \S7.1 and \S7.2.
\end{remark}

Let $\varpi_p$ be $p$ at the place $p$, and 1 elsewhere. Then we consider the operators on $M_k$ given by
\begin{itemize}
 \item $U_p = \sum_{i = 0}^{p-1}\stbt{\varpi_p}{i}{0}{1}$,\smallskip
 \item $\langle p \rangle = p^{2-k} \stbt{\varpi_p}{}{}{\varpi_p}$,\smallskip
 \item $\varphi = p^{1-k} \stbt{1}{0}{0}{\varpi_p}$.
\end{itemize}

\begin{note}
 \begin{enumerate}
  \item The first two operators preserve the space of forms of level $K_0(p^n)$ or $K_1(p^n)$, for any $n \ge 1$.
  \item The operator $\varphi$ does not preserve these forms, but sends level $p^n$ to level $p^{n+1}$.
  \item The operator $\langle p \rangle$ commutes with both $U_p$ and $\varphi$, and we have $U_p \circ \varphi = \langle p \rangle$.
 \end{enumerate}
\end{note}

\begin{remark}
 Calling this operator ``$\varphi$'' is a bit abusive since the action of $\GL_2(\Af)$ is linear (not semilinear). However, this operator agrees with the Frobenius on the forms that are defined over $\Qp$ with respect to our $\QQ$-model of the Shimura variety.
\end{remark}

We shall also need to consider $p$-adic modular forms of weight $k \in \ZZ$.

\begin{definition}
 Letting $X_0(p)$ denote the (compactified) modular curve of level $K^p K_0(p)$, for some prime-to-$p$ level $K^p$, we define
 \[ \cM_{k}(K^p) = H^0\left(\cX_0(p)^{m}, \omega(k; k-2)\right) \]
 where $\cX_0(p)^{m}$ is the multiplicative locus as a dagger space.
\end{definition}

\begin{note}
 For $k \ge 0$, the the differential operator
 \( \Theta: \cM_{-k}(K^p) \to \cM_{k+2}(K^p) \)
 twists the action of Hecke operators by the $(k+1)$-st power of the norm character. In particular, we have the relations
 \[ \cU_p \circ \Theta = p^{k+1} \Theta \circ \cU_p\qquad \text{and}\qquad \varphi \circ \Theta = p^{-1-k} \Theta \circ \varphi. \qedhere\]
\end{note}


\subsection{ Eisenstein series}
\label{sect:eisseries}
In \cite[\S 7.1]{LPSZ1} we defined real-analytic Eisenstein series $E^{(r, \Phi)}(-, s)$ for $r \ge 1$ and $\Phi \in \cS(\Af^2)$. We define $F^{k+2}_{\Phi}$ by setting $r = k+2$ and $s = -k/2$. This is a holomorphic modular form of weight $k+2$ if $k \ge 1$, or if $k = 0$ and $\Phi(0, 0) = 0$; its $q$-expansion is given by
\[ a_n\left(F^{k+2}_{\Phi}\right) = \sum_{\substack{u, v \in \QQ \\ uv = n}} u^{k+1} \sgn(u)\Phi'(u, v) \quad\text{for $n > 0$}, \]
where
\begin{equation}\label{eq:Fourier}
 \Phi'(u, v) = \int_{\Af}\Phi(u, x) e^{2\pi i xv}\, \mathrm{d}x.
\end{equation}

\begin{remark}
 This $F^{k+2}_{\Phi}$ is almost the same as the $F^{k+2}_{\phi}$ in \cite[Theorem 7.2.2]{LSZ17}; the difference is that we have changed our normalisations for the central characters.
\end{remark}

We will be particularly interested in the cases when $\Phi'_p$ is one of the following:
\begin{itemize}
 \item \emph{spherical}: $\Phi'_\mathrm{sph} = \ch(\Zp \times \Zp)$
 \item \emph{critical}: $\Phi'_{\crit} = \ch(\Zp \times \Zp^\times)$
 \item \emph{depleted}: $\Phi'_{\dep} = \ch(\Zp^\times \times \Zp^\times)$
\end{itemize}

\begin{note}
 \label{note:UpEis}
 If we transport the operators $U_p, \varphi, \langle p \rangle$ over to $\cS(\Qp^2)$ compatibly with $\Phi \mapsto F^{k+2}_{\Phi}$, we have $U_p \cdot \Phi_{\dep} = 0$. Moreover, if $\Phi'(x, y) = \ch(A)$ for some open compact $A \subseteq \Qp^2$, then we have
 \begin{align*}
  (\varphi \cdot \Phi)' &= \ch( (1, p) \cdot A), & (\langle p \rangle \cdot \Phi)' &= p^{k+1} \ch( (p^{-1}, p) \cdot A),\\
  (p^{k+1} \langle p \rangle^{-1} \varphi \cdot \Phi)' &= \ch( (p, 1) \cdot A).
 \end{align*}
 In particular, this shows that
 \[ (1 - \varphi) \Phi_{\mathrm{sph}} = \Phi_{\crit},
 \qquad  
 (1 - p^{k+1} \langle p \rangle^{-1} \varphi) \Phi_{\mathrm{crit}} = \Phi_{\dep}; \]
 and consequently that $F^{k+2}_{\Phi^p \Phi_{\crit}}$ is in the $U_p = p^{k+1}$ eigenspace and $F^{k+2}_{\Phi^p \Phi_{\dep}}$ in the $U_p = 0$ eigenspace, for any prime-to-$p$ Schwartz function $\Phi^p$ (hence the terminology). There is also a Schwartz function which gives rise to Eisenstein series in the ordinary $U_p$-eigenspace, but we shall not use this here.
\end{note}

\begin{note}\label{note:padicallycusp1}
 The Eisenstein series $F^{k+2}_{\Phi^p \Phi_{\crit}}$ is $p$-adically cuspidal, and hence so is $F^{k+2}_{\Phi^p\Phi_{\dep}}$ (since the operator $ (1 - p^{k+1} \langle p \rangle^{-1} \varphi)$ will preserve $p$-adic cuspforms).
\end{note}

As in \cite[\S 7.3]{LPSZ1}, if $\Phi_p = \Phi_{\dep}$ or $\Phi_{\mathrm{crit}}$, we can construct a $p$-adic modular form
\[ E^{-k}_{\Phi} \in H^0(\cX_0(p)^m, \omega^{-k})\]
of weight $-k$, such that $\theta^{k+1}\left(E^{-k}_{\Phi}\right) = F^{k+2}_{\Phi}$. Clearly the $q$-expansion of this form must be given by
\[ a_0 + \sum_{n > 0} \sum_{uv = n} v^{-1-k} \sgn(u) \Phi'(u, v); \]
and this form is $p$-adically cuspidal if $\Phi_p = \Phi_{\dep}$ (see Theorem 7.6 of \emph{op.cit.}).


\subsection{Eisenstein classes}  \label{sect:higherEis}

\begin{notation}
 Denote by $Y$ the infinite level modular curve.

 Write $\sH$ for the sheaf corresponding to the defining representation of $\GL_2$ on a modular curve.
\end{notation}

\begin{theorem}[Beilinson]
 Let $k \ge 1$. There is a $\GL_2(\Af)$-equivariant map
 \[ \cS(\Af^2, \QQ) \to H^1_{\mot}\left(Y, \Sym^k\sH,1+k\right),\qquad  \Phi \mapsto \Eis^{k+2}_{\mot,\Phi},\]
 the \emph{motivic Eisenstein symbol}, with the following property: the pullback of the de Rham realization $r_{\dR}\left( \Eis^{k+2, \Phi}_{\mot} \right)$ to the upper half-plane is the $\sH^k$-valued differential 1-form
 \[ -F^{(k+2)}_{\Phi}(\tau) (2\pi i d  z)^k (2\pi id \tau), \]
 where $F^{(k+2)}_{\phi}$ is the Eisenstein series defined by
 \[
 F^{(k+2)}_\phi(\tau) = \frac{(k+1)!}{(-2\pi i)^{k+2}}  \sum_{\substack{x, y \in \QQ \\ (x, y) \ne (0,0)}} \frac{\hat\phi(x, y)}{(x\tau + y)^{k+2}}.
 \]
\end{theorem}

\begin{proof}
 See \cite{beilinson86}.
\end{proof}

\begin{notation} Let $\Phi^{(p)} \in \cS( (\Af^{(p)})^2, \QQ)$, and let $\Phi=\Phi^{(p)}\Phi_{\crit}$.
 \begin{itemize}
  \item Write
  \[ \Eis_{\NNsyn,\Phi}^{k+2}\in H^1_{\NNsyn}(Y_0(p)_{\Qp},\Sym^k\sH,1+k)\]
  for the syntomic realisation of the class $\Eis^{k+2}_{\mot,\Phi}$, and denote by $\Eis_{\lrigsyn,\Phi}^{k+2}$ its image under the isomorphism in Theorem \ref{thm:comparison}.
  \item Write $\Eis_{\rigsyn,\Phi}^{k+2,m}$ for the restriction of $\Eis_{\lrigsyn,\Phi}^{k+2,m}$ to $\cY_0(p)^{m}$.
  \item Write $\widetilde{\Eis}_{\rigsyn,\Phi}^{k+2,m}$ for the image of $\Eis_{\rigsyn,\Phi}^{k+2,m}$ in Gros syntomic cohomology.
 \end{itemize}
\end{notation}

All of the above depend $\GL_2(\Af^{(p)})$-equivariantly on $\Phi^{(p)}$.

\begin{remark}
 Let $\underline\Phi=\left(\Phi_1,\, \Phi_2\right)$, where $\Phi_i=\Phi_i^{(p)}\Phi_{\crit}$. Then
 \[
 \widetilde\Eis_{\rigsyn,\underline\Phi}^{[t_1,t_2],(m,m)}=\widetilde{\Eis}_{\rigsyn,\Phi_1}^{t_1+2,m}\sqcup \widetilde{\Eis}_{\rigsyn,\Phi_2}^{t_2+2,m}.\qedhere
 \]
\end{remark}


\subsection{Reduction to a $p$-adically cuspidal Eisenstein class}\label{padicallycuspEis}


Let  $V_H$ be as defined in Section \ref{ss:algrep}.
Recall that by Remark \ref{rem:Grosspcosp} we have a pairing, denoted $\langle\quad,\quad\rangle_{\widetilde{\rigfp},\cX^{2,m}_\Delta}$,
\[
\wH^3_{\rigfp,c}\left(\cX_{H,\Delta}^{(m,m)}\langle -\cD_H\rangle,\cV_H,1+q;\cQ_{1+q}\right)\times  \wH^2_{\rigsyn}\left(\cY^{(m,m)}_{H,\Delta},\cV_H,2+t_1+t_2\right)  \longrightarrow \Qp.
\]
\smallskip

\noindent {\textbf{Aim.}} Recall from Corollary \ref{cor:redtoGros} that want to compute the quantity
\begin{equation}\label{eq:ordpairing}
 \left\langle  (\iota^{(t_1,t_2)}_\Delta)^\star\left(\tilde\eta^{(2,m)}_{\rigfp,-D}\right),\, \widetilde\Eis_{\rigsyn,\underline\Phi}^{[t_1,t_2],(m,m)}\right\rangle_{\widetilde\rigfp,\cX^{(m,m)}_\Delta},
\end{equation}
in terms of coherent cohomology.
\vspace{1ex}

\begin{note}
 The main tool for the evaluation is the Pozna\'n spectral sequence constructed in Propositions \ref{prop:Poznan}  and \ref{prop:PoznanH}. However, we only have explicit representatives (see \eqref{eq:zetadef} and Proposition \ref{prop:etahornsrep}) of $(\iota^{(t_1,t_2)}_\Delta)^\star\left(\tilde\eta^{(2,m)}_{\rigfp,-D}\right)$ after replacing $\tilde\eta^{(2,m)}_{\rigfp}$ by its image  $\tilde\eta^{(2,m)}_{\rigfp}\in \wH^3_{\rigfp,c0}(\cX^{2,m}_{\Kl},\cV,1+q;P_q)$. In order to be able to evaluate \eqref{eq:ordpairing}, we therefore need to replace the Eisenstein class by a version which is $p$-adically cuspidal.
\end{note}

Since rig-fp cohomology is compatible with change of polynomial,  we have a natural map
\[ \wH^1_{\rigfp}\left(\cY_0(p)^{m}, \Sym^k\sH,1+k;\operatorname{const} 1\right)\to \wH^1_{\rigsyn}\left(\cY_0(p)^{m},\Sym^k\sH,1+k\right).\]

\begin{lemma}
 The class $\widetilde{\Eis}_{\rigsyn,\Phi}^{k+2,m}$ is in the image of $\wH^1_{\rigfp}(\cY_0(p)^{m}; \Sym^k\sH,1+k;\operatorname{const} 1)$. In other words,  we can lift it to an element $\widetilde{\Eis}_{\rigfp,\operatorname{const} 1,\Phi}^{k+2,m}\in \wH^1_{\rigfp}(\cY_0(p)^{m},\Sym^k\sH,1+k; \operatorname{const} 1)$.
\end{lemma}
\begin{proof}
 This is just the statement that the critical-slope Eisenstein series is integrable over the ordinary locus.
\end{proof}

\begin{note}\label{note:padicallycusp}
 The class $\widetilde{\Eis}_{\rigfp,\operatorname{const} 1,\Phi}^{k+2,m}$  is \emph{not} in the image of $\wH^1_{\rigfp}(\cX_0(p)^{m}\langle -\cD\rangle,  \Sym^k\sH,1+k; \operatorname{const} 1)$ -- the ``degree 1 part'' of our fp-pair is cuspidal, but the ``degree 0 part'' is not -- but the constant term of the degree 0 part gets annihilated by $(1-p^{k+1}\langle p \rangle_{\GL_2}^{-1} \varphi)$, which corresponds to $1 - V_p$ on $q$-expansions in weight $-k$.
\end{note}

\begin{lemma}
 The image of $\widetilde{\Eis}_{\rigfp,\operatorname{const} 1,\Phi}^{k+2,m}$ under the natural map
 \[ \wH^1_{\rigfp}\left(\cY_0(p)^{m},  \Sym^k\sH,1+k;\operatorname{const} 1\right)\to \wH^1_{\rigfp}\left(\cY_0(p)^{m},  \Sym^k\sH,1+k;1-p^{k+1}\langle p \rangle^{-1}_{\GL_2}\, t\right)\]
 lifts to a class \[\widetilde{\Eis}_{\rigfp,\Psi}^{k+2,m}\in\wH^1_{\rigfp}(\cX_0(p)^{m}\langle -\cD_{\GL_2}\rangle, \Sym^k\sH,1+k; (1-p^{k+1}\langle p \rangle^{-1}_{\GL_2} t)),\]
    where $\Psi=\Phi^{(p)}\Phi_{\dep}$.
\end{lemma}
\begin{proof}
 Immediate from Note \ref{note:padicallycusp}.
\end{proof}

\begin{remark}\label{rem:herbchopper}
 These constructions are summarized by the following diagram (where we leave out the coefficients for reasons of space):
 \[
 \begin{tikzcd}[row sep=large, column sep=-2em]
 H^1_{\rigsyn}(\cY_0(p)^{m})
   \arrow[d]
 &
 \wH^1_{\rigfp}\big(\cY_0(p)^{m}; \operatorname{const} 1\big)
   \arrow[dl]
   \arrow[dr]
 &
 \wH^1_{\rigfp}\big(\cX_0(p)^{m}\langle -\cD_{\GL_2}\rangle; 1-p^{k+1}\langle p \rangle^{-1} t\big)
   \arrow[d]
 \\
 \wH^1_{\rigsyn}(\cY_0(p)^{m})
 \arrow[dr]
 &
 &
 \wH^1_{\rigfp}\big(\cY_0(p)^{m}; 1-p^{k+1}\langle p \rangle_{\GL_2}^{-1} t\big)
   \arrow[dl]
 \\
 &
 \wH^1_{\rigfp}\Big(\cY_0(p)^{m}; (1-t)\big(1-p^{k+1}\langle p\rangle_{\GL_2}^{-1} t\big)\Big)
 \end{tikzcd}
 \]
%
%
 Here, the diagonal arrows are given by the formalism for change of polynomial in fp-cohomology. We refer to this as the \emph{herb-chopper diagram}.
\end{remark}



\subsection{The $\GL_2$-Eisenstein class as a coherent fp-pair}

We want to find representatives of the image of the class $\widetilde{\Eis}_{\rigfp,\Psi}^{k+2,m}$ under the map
\begin{align*}
 \alpha_{\GL_2}^{-1}: \wH^1_{\rigfp}&\left(\cX_0(p)^{m}\langle -\cD_{\GL_2}\rangle, \Sym^k\sH,1+k;1-p^{k+1}\langle p \rangle^{-1}_{\GL_2}\, t\right) \\
 &\qquad \xrightarrow{\ \cong\ } H^{1,0}\left(\cX_0(p)^{m}\langle -\cD_{\GL_2}\rangle,\Sym^k\sH,1+k;1-p^{k+1}\langle p \rangle^{-1}_{\GL_2}\, t\right)
\end{align*}
constructed in Section \ref{section:PoznanforH} (c.f. Note \ref{rem:PoznanisoGL2}).

\begin{notation}
 Denote by $v$ and $w$ the  the basis of sections $\tilde{\omega}$ and $\tilde{u}$ of $\sH$ over the Igusa tower, as constructed in \cite[\S 4.5]{KLZ20}. 
\end{notation}

\begin{proposition}\label{prop:EisGL2cohpair}
 The class  $\widetilde{\Eis}_{\rigfp,\Psi}^{k+2,m}$ is represented by the pair $\left(\epsilon_0^{k,\Phi^{(p)}},\, \epsilon_1^{k,\Phi^{(p)}}\right)$, where
 \begin{align*}
  \epsilon_0^{k,\Phi^{(p)}} & = \sum_{j=0}^k \frac{(-1)^{j}k!}{(k-j)!} \theta^{k-j} E^{-k}_{\Phi^{(p)} \Phi_{\dep}}\cdot v^{k-j}w^{j},\\
  \epsilon_1^{k,\Phi^{(p)}} & =  F^{k+2}_{\Phi^{(p)} \Phi_{\crit}}\cdot v^k\otimes \xi\otimes e_1,
 \end{align*}
 where $\xi$ is as defined in \cite[\S 4.5]{KLZ20}.
\end{proposition}
\begin{proof}
 We argue as in \cite[Theorem 5.11]{bannaikings10}, who give an explicit formula for the coherent fp-pair representing the class $\Eis_{\rigsyn,{2,m}}^{k+2,\Phi^{(p)},\sph}$.

 The degree 1 part of  $\widetilde{\Eis}_{\rigfp,\Psi}^{k+2,m}$ is, by definition, the form $F^{k+2}_{\Phi^{(p)} \Phi_{\dep}}\cdot v^k\otimes \xi\otimes e_1$. By Note \ref{note:UpEis}, the image of this under $(1 - p^{k+1} \langle p \rangle^{-1} \varphi)$ is given by replacing $\Phi_{\crit}$ by $\Phi_{\mathrm{dep}}$; so we need to construct an overconvergent section of $\Sym^k \sH$ whose image under $\nabla$ is $F^{k+2}_{\Phi^{(p)} \Phi_{\dep}}\cdot v^k\otimes \xi\otimes e_1$. An elementary computation shows that the above class $\epsilon_0^{k,\Phi^{(p)}}$ does indeed have these properties; and, moreover, it vanishes at the ordinary cusps, so it defines a lifting of $F^{k+2}_{\Phi^{(p)} \Phi_{\crit}}\cdot v^k\otimes \xi\otimes e_1$ to $\wH^1_{\rigfp}(\cX_0(p)^{m}\langle -\cD_{\GL_2}\rangle, \Sym^k \sH, 1+k; 1-p^{k+1}\langle p \rangle^{-1} t)$, as required.
\end{proof}

\begin{lemma}\label{lem:UponEiscomp}
 We have
 \[ U_p\left(  \epsilon_0^{k,\Phi^{(p)}} \right)=0\qquad \text{and}\qquad U_p\left( \epsilon_1^{k,\Phi^{(p)}}\right)=p^{k-1} \epsilon_1^{k,\Phi^{(p)}}.\]
\end{lemma}
\begin{proof}
 Clear from Note \ref{note:UpEis} and from the fact that $\varphi^{-1}(w)=w$ and $\varphi^{-1}(\xi)=p^{-2}\xi$ (c.f. \cite[\S 5.4]{KLZ20}.
\end{proof}


\subsection{The Eisenstein class for $H$ as a coherent fp-pair}\label{ssec:Eiscohfppair}

\begin{lemma}\label{lem:HEisclassasfppair}
 For $i=1,2$, let $\Psi_i=\Phi^{(p)}_1\Phi_{\dep}$. Then the image of
 \[  \widetilde{\Eis}_{\rigfp,\Psi_1}^{k+2,m}\sqcup \widetilde{\Eis}_{\rigfp,\Psi_2}^{k+2,m}\]
 under the isomorphism $\alpha_\Delta^{-1}$ (c.f. Corollary \ref{cor:HGrosisoms}) is represented by the coherent fp-pair
 \[ \left( \alpha^{t_1,t_2, \Phi^{(p)}_1,\Phi^{(p)}_2}_1,\alpha^{t_1,t_2,\Phi^{(p)}_1,\Phi^{(p)}_2}_2\right)\in H^{2,1}\left(\cX^{(m,m)}_{H,\Delta},\cV_H,2+t_1+t_2;\cR\right),\]
 where $\cR(y)=1-p^{t_1+t_2+2}\langle p\rangle_H^{-1}y$ and
 \begin{align}
  \alpha^{t_1,t_2,\Phi^{(p)}_1,\Phi^{(p)}_2}_1 &= \epsilon^{t_1,\Phi^{(p)}_1}_0\sqcup \epsilon^{t_2,\Phi^{(p)}_2}_1 + p^{t_1+1}( \langle p\rangle^{-1}_{\GL_2} \varphi_{\GL_2}^\star\boxtimes 1)\, \left(\epsilon^{t_1,\Phi^{(p)}_1}_1\sqcup \epsilon^{t_2,\Phi^{(p)}_2}_0\right),\\
  \alpha^{t_1,t_2,\Phi^{(p)}_1,\Phi^{(p)}_2}_2 & =  \epsilon^{t_1,\Phi^{(p)}_1}_1\sqcup \epsilon^{t_2,\Phi^{(p)}_2}_1.
 \end{align}
 Here, $\epsilon_\ell^{t_m,\Phi^{(p)}_m}$ is as defined in Lemma \ref{prop:EisGL2cohpair}, and we write $\langle p\rangle_H$ for $\langle p\rangle\boxtimes \langle p\rangle$.
\end{lemma}\begin{proof}
 We use the explicit formulae for the cup product in fp-cohomology, as given in \cite[eq. (2.10)]{besser12}: the convolution of the polynomials $A_1(y)=1-p^{t_1+1}\langle p\rangle^{-1} y$ and $A_2(y)=1-p^{t_2+1}\langle p\rangle^{-1} y$ is given by $\cR(y)$. We then decompose
 \[ \cR(xy)=a(x,y)A_1(x)+b(x,y)A_2(y),\]
 where $a(x,y)=1$ and $b(x,y)=p^{t_1+1}(\langle p\rangle^{-1}\sqcup 1)\cdot x$. We then apply equation (2.11) in \emph{op. cit.} and  Lemma \ref{lem:compcup} to obtain the formulae for $ \alpha^{t_1,t_2,\Phi^{(p)}_1,\Phi^{(p)}_2}_1$ and $\alpha^{t_1,t_2,\Phi^{(p)}_1,\Phi^{(p)}_2}_2$.
\end{proof}

%
%
%
%

\section{Pairing in coherent cohomology}

 \subsection{Reduction of the pairing}

     We will now evaluate the pairing \eqref{eq:ordpairing}. By the herb--chopper diagram and the compatibility of the pairings under change of polynomial, \eqref{eq:ordpairing}   is equal to
 \begin{equation}\label{eq:pairingtoeval}
  \left\langle  (\iota^{[t_1,t_2]}_\Delta)^\star\left(\tilde\eta^{(2,m)}_{\rigfp}\right),\,\widetilde\Eis^{t_1+2,m}_{\rigfp,\Psi_1}\sqcup \widetilde\Eis^{t_2+2,m}_{\rigfp,\Psi_2}\right\rangle_{\widetilde\rigfp,\cX^{(m,m)}_{H,\Delta}}.
 \end{equation}

 \begin{lemma}\label{lem:reductiontocoherent}
  The pairing \eqref{eq:pairingtoeval} is equal to
  \begin{equation}\label{reduction1}
   \left\langle \left(\iota^{[t_1,t_2]}_\Delta\right)^\star(\breve\zeta,\breve\eta^{(2,m)}_{\coh}),\, \left( \alpha^{t_1,t_2,\Phi^{(p)}_1,\Phi^{(p)}_2}_1 , \alpha^{t_1,t_2,\Phi_1^{(p)},\Phi_2^{(p)}}_2\right)\right\rangle_{\coh-\fp,\cX^{(m,m)}_{H,\Delta}},
  \end{equation}
  where $(\breve\zeta,\breve\eta^{(2,m)}_{\coh})$ is as defined in Proposition \ref{prop:etahornsrep}.
 \end{lemma}
 \begin{proof}
   By Proposition \ref{prop:etahornsrep}, $\tilde\eta^{(2,m)}_{\rigfp}$ is represented by the coherent $\fp$-pair $( \breve\zeta,\breve\eta^{(2,m)}_{\coh})$.  Similarly, Lemma \ref{lem:HEisclassasfppair} expresses the class
    \[ \widetilde\Eis_{\rigfp,\Psi_1}^{t_1+2,m}\sqcup \widetilde\Eis_{\rigfp,\Psi_2}^{t_2+2,m}\]
    as a coherent $\fp$-pair. By  Lemma \ref{lem:compcup}, these representations are compatible with cup products, which implies the result.
 \end{proof}

 \subsection{Independence of the lift of $\eta_{\coh}^{2,m}$}

  The following proposition shows that the value of the pairing \eqref{reduction1} is independent of the lift of $\breve\eta^{2,m}_{\coh,q}$ to a coherent $\fp$-pair.

\begin{proposition}\label{prop:indepoflift}
 Let $\xi\in H^2_{c0}(\cX^{2,m}_{G,\Kl},\cN^0)$ be \textbf{any} element which lies in the $\Pif'$-generalised eigenspace for the spherical Hecke algebra, and satisfies
 \[ \cQ_{1+q}(\Phi_{1+q}) \eta^{(2,m)}_{\coh} = \nabla \xi,    \]
 and write $\breve\xi$ for its image in $H^2_{c0}(\cX^{2,m}_{G,\Kl},\cV\otimes\Omega^0)$. Then
 \begin{multline*}
  \left\langle \left(\iota^{[t_1,t_2]}_\Delta\right)^\star(\breve\zeta,\breve\eta^{(2,m)}_{\coh}),\, \left( \alpha^{t_1,t_2,\Phi^{(p)}_1,\Phi^{(p)}_2}_1 , \alpha^{t_1,t_2,\Phi_1^{(p)},\Phi_2^{(p)}}_2\right)\right\rangle_{\coh-\fp,\cX^{m,m}_{H,\Delta}}\\=  \left\langle \left(\iota^{[t_1,t_2]}_\Delta\right)^\star(\breve\xi,\breve\eta^{(2,m)}_{\coh}),\, \left( \alpha^{t_1,t_2,\Phi^{(p)}_1,\Phi^{(p)}_2}_1 , \alpha^{t_1,t_2,\Phi_1^{(p)},\Phi_2^{(p)}}_2\right)\right\rangle_{\coh-\fp,\cX^{(m,m)}_{H,\Delta}}.
 \end{multline*}
\end{proposition}

\begin{remark}
 We will choose a suitable $\xi$ in Proposition \ref{prop:goodxi} below.
\end{remark}

 As shown in Note \ref{note:lackofuniqueness}, we have
 \[ \zeta-\xi\in H^2_{c0}(\cX^{2,m}_{G,\Kl},\cN^0)^{\nabla=0}\cong H^2_{\dR,c0}(\cX^{2,m}_{G,\Kl},\cV).\]
 Proposition \ref{prop:indepoflift} will hence follow from the following result:

 \begin{proposition}\label{prop:kernablapairsto0}
  Let $\Omega\in H^2_{c0}(\cX^{2,m}_{G,\Kl},\cN^0)^{\nabla=0}[\Pif']$, and regard it as the coherent $\fp$-pair $(\Omega,0)$. Then
  \[  \left\langle \left(\iota^{[t_1,t_2]}_\Delta\right)^\star(\Omega, 0),\, \left( \alpha^{t_1,t_2,\Phi^{(p)}_1,\Phi^{(p)}_2}_1 , \alpha^{t_1,t_2,\Phi_1^{(p)},\Phi_2^{(p)}}_2\right)\right\rangle_{\coh-\fp,\cX^{(m,m)}_{H,\Delta}}=0.\]
 \end{proposition}

  We can consider $\Omega$ as a class in the $\Pif'$-eigenspace (for the spherical Hecke algebra) acting on $H^2_{\dR, c0}(\cX^{2,m}_{G, \Kl}, \cV)$. To prove \cref{prop:kernablapairsto0} we will use the following fact, which will be proved in the appendix to this paper:
  
  \begin{proposition}
   The natural map
   \[ H^i_{\dR, c0}(\cX^{2,m}_{G, \Kl}\langle -\cD\rangle, \cV) \to H^i_{\dR, c0}(\cX^{2,m}_{G, \Kl}, \cV), \]
   arising from the inclusion of complexes $\BGG_c(\cV) \into \BGG(V)$, is an isomorphism on the $\Pif'$ generalized eigenspace in all degrees $i$.
  \end{proposition}

  \begin{remark}
   We only need this statement for $i = 2$; and we suspect that, in fact, for $i \ne 3$ the $\Pif'$-generalized eigenspaces in both $H^2_{\dR, c0}(\cX^{2,m}_{G, \Kl}\langle -\cD\rangle, \cV)$ and $H^2_{\dR, c0}(\cX^{2,m}_{G, \Kl}, \cV)$ are actually zero. We have not been able to prove this stronger statement.
  \end{remark}

  \begin{proof}[Proof of \cref{prop:kernablapairsto0}]
   Since the coherent and rigid cup-products are compatible, it is enough to show that
   \[
    \left\langle \left(\iota^{[t_1,t_2]}_\Delta\right)^\star(\Omega), \left( \epsilon^{t_1,\Phi^{(p)}_1}_1\sqcup \epsilon^{t_2,\Phi^{(p)}_2}_1\right) \right\rangle_{\rig, \cX_{H, \Delta}^{(m, m)}} = 0.
   \]
   Here the $\epsilon_1$'s are considered as classes in rigid cohomology with compact support towards the cusps (and non-compact support towards the supersingular locus). These are in the kernel of the map to cohomology with non-compact supports at the cusps.
   
   However, by the result of the appendix (see \S\ref{sect:forgetsupports}) shows that $\Omega$ is in the image of a class with compact support towards the toroidal boundary of $\cX_{G, \Kl}$. Hence its restriction pairs to 0 with the Eisenstein classes.
  \end{proof}

 \subsection{Choice of a good lift of $\eta_{\coh}^{2,m}$}


  \begin{proposition}\label{prop:goodxi}
 There exists $\xi\in  H^2_{c0}(\cX^{2,m}_{G,\Kl},\cN^0)$ with the following properties:
 \begin{enumerate}
  \item $\nabla\, \xi=\cQ_{1+q}(\Phi_{1+q})\eta_{\coh}^{(2,m)}$;
  \item $(U_2'-\lambda)\xi$ lies in the $U_2$-generalized eigen-subspace of $H^2_{\dR,c0}(\cX^{2,m}_{G,\Kl},\cV)$ with generalized eigenvalue $\lambda$;
  \item we have $Z'\cdot\xi=0$.
 \end{enumerate}
\end{proposition}
\begin{proof}

 \textbf{Step 1.}  We first show that there exists some $\xi$ such that $\nabla\,\xi=\cQ_{1+q}(\Phi_{1+q})\cdot \eta_{\coh}^{(2,m)}$.  By Proposition \ref{prop:kerZ}, we know that $Z'\circ \cQ(\Phi)\cdot \eta_{\coh,-D}^{(2,m)}=0$, which implies that
 \[ Z'\circ \cQ(\Phi)\cdot \eta_{\coh}^{2,m}=0. \]
 Recall that $\tilde\eta^{(2,m)}_{\rig}\in \wH^3_{\dR,c0}(\cX^{2,m}_{G,\Kl},\cV,1+q)$ is the preimage of $\eta_{\coh}^{(2,m)}$ under the isomorphism \eqref{eq:cohrigiso2}, so we deduce that
 \[ Z'\circ \cQ_{1+q}(\Phi_{1+q})\circ\iota ( \tilde\eta^{(2,m)}_{\rig})=0.\]
 Now recall that both $Z'$ and $\Phi$ commute with $U_2'$. Hence $\cQ_{1+q}(\Phi_{q})\circ\iota(\tilde\eta^{(2,m)}_{\rig})$ lies in the $(U_2'=\lambda)$ eigenspace of $ \wH^3_{\dR,c0}(\cX^{2,m}_{G,\Kl},\cV)$, and since $Z'\circ\Phi=p^{r_2+1}\, U_2'$, the restriction of $Z'$ to this eigenspace is a bijection. We deduce that
 \[  \cQ_{1+q}(\Phi_{1+q}) \circ \iota (\tilde\eta^{(2,m)}_{\rig})=0.\]
 We can hence lift $\tilde\eta^{(2,m)}_{\rig}$ to a class
 \[ \tilde\eta^{(2,m)}_{\rigfp} \in \wH^3_{\rigfp,c0}(\cX^{2,m}_{G,\Kl},\cV,1+q;\cQ_{1+q})   \]
 which lies in the $\Pi'_f$-eigenspace for the spherical Hecke operators. By Corollary \ref{cor:fppairmapsS}, this class corresponds to a coherent fp-pair, which has the required form.
 \vspace{1ex}

 \textbf{Step 2.} Note that
 \begin{itemize}
  \item we have  $(U_2'-\lambda)\xi\in H^2_{c0}(\cX^{2,m}_{G,\Kl},\cN^0)^{\nabla=0}$, since $U_2'\,\eta_{\coh}^{(2,m)}=\lambda\eta_{\coh}^{(2,m)}$,;
  \item we have $Z'\cdot\xi\in H^2_{c0}(\cX^{2,m}_{G,\Kl},\cN^0)^{\nabla=0}$, by Proposition \ref{prop:kerZ} (1).
 \end{itemize}
 Now $H^2_{c0}(\cX^{2,m}_{G,\Kl},\cN^0)^{\nabla=0}\cong H^2_{\dR,c0}(\cX^{2,m}_{G,\Kl},\cV)$ is finite-dimensional, so by applying a suitable projector we can assume without loss of generality that both $(U_2'-\lambda)\xi$ amd $Z'\cdot \xi$ lie in the $U_2$-generalized eigen-subspace of $H^2_{\dR,c0}(\cX^{2,m}_{G,\Kl},\cV)$ with generalized eigenvalue $\lambda$ (we use here that $U_2'$ commutes with $Z'$); denote this subspace by
 \[ H^2_{\dR,c0}(\cX^{2,m}_{G,\Kl},\cV)[U_2'=\lambda]^{\gen}.\]
 Now since $Z'\circ\Phi=p^{r_2+1}\, U_2'$, the restriction of $Z'$ to $H^2_{\dR,c0}(\cX^{2,m}_{G,\Kl},\cV)[U_2'=\lambda]^{\gen}$ is a bijection, so there exists $\nu\in H^2_{\dR,c0}(\cX^{2,m}_{G,\Kl},\cV)[U_2'=\lambda]^{\gen}$ such that $Z'\cdot \nu=Z'\cdot \xi$. Replacing $\xi$ by $\xi-\nu$ proof the result.
\end{proof}

Write $\breve\xi$ for the image of $\xi$ in $H^2_{c0}(\cX^{2,m}_{G,\Kl},\cV\otimes\Omega^0)$.

\begin{corollary}\label{cor:goodxiprop}
 The class $\breve\xi$ satisfies
 \begin{enumerate}
  \item $\nabla\, \breve\xi=\cQ_{1+q}(\Phi_{1+q})\breve\eta_{\coh}^{(2,m)}$;
  \item $(U_2'-\lambda)\breve\xi$ lies in the $U_2$-generalized eigen-subspace of $H^2_{\dR,c0}(\cX^{2,m}_{G,\Kl},\cV)$ with generalized eigenvalue $\lambda$;
  \item we have $Z'\cdot\breve\xi=0$.
 \end{enumerate}
\end{corollary}


   We will evaluate this pairing in Section \ref{sect:evaluation}, and we will see that properties (2) and (3) in Corollary \ref{cor:goodxiprop} are crucial for the evaluation.


 \subsection{A Hecke operator identity}
 \label{sect:weirdidentity}

  The reason why we care about Corollary \ref{cor:goodxiprop} (2) is the following result, comparing constructions on $G$ and on $H$. Recall the embedding
  \[ \iota_\Delta: \cX_{H, \Delta}^{2,m} \to \cX_{G, \Kl}^{2,m} \]
  constructed in Section \ref{ss:iotaproperties}.

  \begin{proposition}\label{prop:weirdcorresp}
    We have the following identity of correspondences $\cX_{H, \Delta}^{2,m} \rightrightarrows \cX_{G, \Kl}^{2,m}$:
    \begin{equation}\label{eq:weirdeqn}
      U_2' \circ \iota_\Delta \circ (U_p \boxtimes U_p) = p\langle p \rangle Z' \circ \iota_\Delta.
     \end{equation}
   \end{proposition}

  \begin{note}
    Correspondences act contravariantly on cohomology, so this means that
    \[ (U_p \boxtimes U_p) \circ \iota_\Delta^\star \circ U_2' = \iota_\Delta^\star \circ p\langle p\rangle Z'\]
    as maps $H^*(\cX_{G, \Kl}^{2,m}) \to H^*(\cX_{H, \Delta}^{(m,m)})$.
   \end{note}

  \begin{proof}
   Since $\cY_{H, \Delta}^{(m,m)}$ is open in $\cX_{H, \Delta}^{2,m}$, it suffices to prove the identity over this open subset.

   We recall the moduli-space description of the varieties and correspondences involved. A point of $\cY_{H, \Delta}^{(m,m)}$ (over some $p$-adic field $L$) corresponds to a triple $(E_1, E_2, \alpha)$, where $E_i$ are elliptic curves over $L$ with good ordinary reduction, and $\alpha$ is an isomorphism $\hat{E}_1[p] \xrightarrow{\,\cong\,} \hat{E}_2[p]$. The operator $U_p \boxtimes U_p$ maps $(E_1, E_2, \alpha)$ to the formal sum $\sum_{J_1, J_2} (E_1/J_1, E_2/J_2, \bar{\alpha})$ where $J_i$ vary over cyclic $p$-subgroups of $E_i$ distinct from $\hat{E}_i[p]$, and $\bar\alpha$ is the ensuing isomorphism
   \[ \widehat{E_1/J_1}[p] \xleftarrow[\, \cong\,]{} \hat{E}_1[p] \xrightarrow[\,\cong\,]{\alpha} \hat{E}_2[p] \xrightarrow[\,\cong\,]{} \widehat{E_2/J_2}[p].\]

   Concretely, if $e_1, f_1$ denotes a choice of basis of $T_p E_1$, and $e_2, f_2$ of $T_p E_2$, giving isomorphisms $E_i[p^\infty] \cong (\Qp/\Zp)^2$, and we assume that $e_1$ and $e_2$ span the Tate modules of the formal groups $T_p \hat{E}_i$, then $J_1$ has to be one of the groups $\langle \tfrac{f_1 + a_1 e_1}{p} \rangle$ for $0 \le a_1 \le p-1$, and similarly $J_2$. We can and do assume that $\alpha(e_1) = e_2$.

   Meanwhile, points of $\cX_{G, \Kl}^{2,m}$ correspond to pairs $(A, C)$ where $A$ is an abelian surface and $C \subset \hat{A}[p]$ is a cyclic $p$-subgroup (again with some prime-to-$p$ level structure being ignored). The map $\iota_\Delta$ maps $(E_1, E_2, \alpha)$ to $(E_1 \oplus E_2, C)$ where $C \subset (\hat{E}_1 \oplus \hat{E}_2)[p]$ is the subgroup of points of the form $(x, \alpha(x))$.

   Finally, the Hecke correspondences $Z'$, $U_2'$ and $\langle p \rangle$ are given as follows. Let $P = (A, C)$ be a point of $\cX_{G, \Kl}^{2,m}$.
   \begin{itemize}
    \item The correspondence $Z'$ is given by
    \[ (A, C) \mapsto \sum_J \sum_{\tilde C} (A/J, \tilde C \bmod J),\]
    where $J$ varies over isotropic $(p, p)$-subgroups such that $J \cap \hat{A}[p] = C$, and $\tilde C$ varies over cyclic $p^2$-subgroups of $\widehat{A/J}[p]$ such that $p\tilde C = C$. (Note that there are $p$ choices of $J$, and $p$ choices of the subgroup $\tilde C$, so this is a correspondence of degree $p^2$.)

    \item For the correspondence $U_2'$, let $J_0$ be the subgroup $(p^{-1} C \cap \hat{A}) + C^\perp$; this has invariants $(p^2, p, p)$ and is isotropic in the sense that $p J_0$ and $J_0[p]$ are orthogonal complements inside $A[p]$. Then $U_2'$ is given by
    \[ (A, C) \mapsto \sum_{\tilde C} (A/J_0, (p^{-1} \tilde C \cap \hat A) \bmod J_0), \]
    where $\tilde C$ again varies over liftings of $C$ to a cyclic $p^2$-subgroup of $\hat{A}$.

    \item The correspondence $\langle p \rangle$ sends $(A, C)$ to itself, but acts on the prime-to-$p$ level structure by multiplying it by $p$.
   \end{itemize}

   We now consider composing these operations. We choose a point $P = (E_1, E_2, \alpha)$ and fix coordinates on the $E_i$, as above. Let $(A, C) = \iota_1(P) = (E_1 \oplus E_2, \langle \tfrac{e_1 + e_2}{p}\rangle)$; and let $(A', C') = \iota_\Delta(P')$ where $P'$ is one of the points in the 0-cycle $(U_p, U_p) \cdot P$, corresponding to a choice of $a_1, a_2 \in \ZZ / p$; thus we have $A' = A / \langle f_1', f_2'\rangle$, where $f_i' = \tfrac{f_i + a_i e_i}{p}$. Thus $(e_1, e_2, f_1', f_2')$ form a basis of $T_p A'$, regarded as a lattice in $V = T_p A \otimes \Qp$ (strictly containing $T_p A$ itself), and $C'$ is the image of $C$, generated by $\tfrac{e_1 + e_2}{p}$ as a subgroup of $V / T_pA'$.

   We now compute the canonical $(p^2, p, p)$-subgroup $J_0$ of $A'$: it is uniquely determined by $pJ_0 = C = \langle \tfrac{e_1 + e_2}{p}\rangle$ and $J_0[p] = C^\perp = \langle \tfrac{e_1}{p}, \tfrac{e_2}{p}, \tfrac{f_1' - f_2'}{p}\rangle$, from which we easily compute that $J_0$ is generated by $\langle \tfrac{e_1 + e_2}{p^2}, \tfrac{e_1}{p}, \tfrac{f_1' - f_2'}{p}\rangle$ as a subgroup of $A'[p^\infty] = V / T_pA'$. Note that this subgroup contains the image of $A[p]$. Thus the isogeny $A \to A' \to A' / J_0$ is the composite of multiplication by $p$ on $A$ (which gives the factor $\langle p \rangle$) and quotient by the subgroup $K = \langle \tfrac{e_1 + e_2}{p}, \tfrac{f_1 - f_2 + a_1 e_1 - a_2 e_2}{p}\rangle$. So the image of $(A', C')$ under $U_2'$ is given by $\sum_{C'} (A/K, C' \bmod K)$, where $C'$ varies over multiplicative $p^2$-subgroups of $A$ lifting $C$. Note that this is the same as the inner sum of $Z' \cdot (A, C)$ when we take the subgroup $J$ to be our $K$.

   To conclude the proof, it suffices to note that as $(a_1, a_2)$ vary, the subgroup $K$ hits every one of the groups $J$ in the outer sum defining $Z' \cdot (A, C)$, and each such $J$ occurs $p$ times (since $K$ only depends on $a_1 - a_2 \bmod p$).
  \end{proof}

  Proposition \ref{prop:weirdcorresp} has the following immediate consequence, which will be crucial in the regulator evaluation (c.f. Section \ref{sect:evaluation}):

     \begin{corollary}\label{pullbackinkerUp}
    Let $\xi\in H^2_{c0}(X_{G,\Kl}^{2,m}, \cV\otimes\Omega_G^0)$ be as in Corollary \ref{cor:goodxiprop}. Then
    \[ \iota_\Delta^\star(\xi) \in \ker( U_p \boxtimes U_p).\]
   \end{corollary}
     \begin{proof}
      Consider the vector space
      \[ W=\Qp\cdot \xi + H^2_{\dR,c0}(\cX^{2,m}_{G,\Kl},\cV)[U_2'=\lambda]^{\gen}.\]
      Then $W$ is equipped with an action of $U_2'$, which is invertible since $\lambda\neq 0$, and with an action of $Z'$ (which clearly isn't invertible, but that does not matter). Since $Z'$ commutes with $U_2'$, it also commutes with $(U_2')^{-1}$. It follows from the correspondence \eqref{eq:weirdeqn} that when restricted to $W$, we have
      \begin{align*}
        (U_p\boxtimes U_p)\circ\iota_\Delta^*&=\iota^*\circ p\langle p\rangle\circ Z'\circ (U_2')^{-1}\\
        & = \iota^*\circ p\langle p\rangle\circ (U_2')^{-1}\circ Z'.
        \end{align*}
      Since $Z'\cdot \xi =0$, the result follows.
     \end{proof}

  \begin{note}\label{note:pullbackofxiprop}
   By Corollary \ref{cor:goodxiprop}, it hence follows that $ \iota_\Delta^\star(\breve{\xi}) \in \ker( U_p \boxtimes U_p)$.
  \end{note}


 \section{Coherent versus de Rham pullbacks}

  \subsection{Algebraic representations of $G$ and $H$}\label{ss:algrep}

   We can identify the representation $\Sym^k$ of $\GL_2$ with the space of polynomial functions on $\GL_2$ satisfying
   \[ f\left(\stbt{a}{0}{\star}{\star} \cdot g\right) = a^k f(g). \]
   If $v$ and $w$ are the functions $\stbt x y \star \star \mapsto x$ and $\stbt x y \star \star \mapsto y$, then $\{ v^{k-i} w^i : 0 \le i \le k\}$ is the standard basis of $\Sym^k$, with $v^k$ being the highest-weight vector. Note that if $X_{21}$ denotes the generator $\stbt{0}{0}{1}{0}$ of the Lie algebra, then we have
   \[ (X_{21})^i \cdot v^k = \tfrac{k!}{(k-i)!} v^{k-i} w^i. \]

   Now let us return to the setting where $V_G = V_G(r_1, r_2; r_1 + r_2)$ for some $r_1 \ge r_2 \ge 0$, and $V_H = V_H(t_1, t_2; t_1 + t_2)$, where $(t_1, t_2) = (r_1 -q-r, r_2 - q + r)$ for some $0 \le q \le r_2$, $0 \le r \le r_1 - r_2$ as per our running conventions.

   Since the representation $V_H(t_1, t_2; t_1 + t_2)$ of $H$ is the exterior product $\Sym^{t_1} \boxtimes \Sym^{t_2}$, we thus have a weight-vector basis $\{ v^{t_1-i_1} w^{i_1} \boxtimes v^{t_2-i_2} w^{i_2} : 0 \le i_n \le t_n\}$ of this representation, realised as a space of $\overline{N}_H$-invariant functions on $H$.

   We can similarly model $V_G(r_1, r_2)$ as the space of $f \in \cO(\overline{N}_B \backslash G)$ which transform via the character $\lambda(r_1, r_2; r_1 + r_2)$ under left-translation by $T$. The standard basis vectors $v_1, \dots, v_4$ of the 4-dimensional representation $V(1, 0)$ thus correspond to the functions sending $g \in G$ to the four entries of its first row. A choice of highest-weight vector $w$ of $V(1, 1)$ is given by $g \mapsto \left| \begin{smallmatrix} g_{11} & g_{12} \\ g_{21} & g_{22} \end{smallmatrix} \right|$, and the vector $w' = Z \cdot w$ is $g \mapsto \left| \begin{smallmatrix} g_{13} & g_{14} \\ g_{23} & g_{24} \end{smallmatrix} \right|$.

   In \cite{LSZ-asai} \S 4.3 we described a specific choice of morphism
   \[ \br^{[q,r]}: V_H \otimes \sideset{}{^q}\det \to V_G \]
   given by mapping the highest-weight vector $v^{t_1} \boxtimes v^{t_2}$ of $V_H$ to the vector $v^{[q,r]} \in V_G$ (denoted $v^{[a,b,q,r]}$ in \emph{op.cit.}) defined by
   \[ w^{r_2-q} \cdot (w')^q \cdot v_1^{r_1-r_2-r}\cdot v_2^r \]
   where the products are taken in $\cO(\overline{N}_B \backslash G)$ (the ``Cartan product'' construction).

   \begin{note}
    It is important to note that the Lie algebra $\mathfrak{g}$ acts on $\cO(G)$ by derivations, so for $X \in \mathfrak{g}$ we have the Leibniz rule
    \[ X^n \cdot (f_1 \times \dots \times f_m) = \sum_{u_1 + \dots + u_m = n} \binom{n}{u_1, \dots, u_m} (X^{u_1} \cdot f_1) \dots (X^{u_m} \cdot f_m).\]
    In particular, $X^n \cdot f^n = n! (X \cdot f)^n$.
   \end{note}

   \begin{lemma}
    Consider the vector $v^{t_1 -t} w^t \boxtimes v^{t_2} \in V_H$, where $t = r_2 - q$. The image of this vector under $\br^{[q,r]}$ is in $\ker(X_{12}^{n+1}) - \ker(X_{12}^{n})$, where $X_{12} = \begin{smatrix} 0 & 1 \\ 0 & 0 \\ &&0 & -1 \\ &&0 & 0 \end{smatrix} \in \mathfrak{g}$ and $n = 2r_2 - q + r$. We have
    \[ \br^{[q,r]}\left(v^{t_1 -t} w^t \boxtimes v^{t_2}\right) = \frac{1}{\binom{t_1}{t}} (w'')^t (w')^q v_1^{r_1-r_2-r} v_2^r \pmod{\ker(X_{12}^{n-1})}. \]
    where $w'' = X_{41} \cdot w = \left(g \mapsto \left| \begin{smallmatrix} g_{14} & g_{12} \\ g_{24} & g_{22} \end{smallmatrix} \right|\right)$ spans the $(-1, 1)$ weight space of $V_G(1, 1)$.
   \end{lemma}

   \begin{proof}
    We have $v^{t_1 -t} w^t \boxtimes v^{t_2} = \tfrac{(t_1 - t)!}{t_1!} X_{41}^t \cdot v^{t_1} \boxtimes v^{t_2}$ (identifying $X_{41}$ with an element of $\mathfrak{h} \subseteq \mathfrak{g}$). So we have
    \[ \br^{[q,r]}\left(v^{t_1 -t} w^t \boxtimes v^{t_2}\right) = \tfrac{(t_1 - t)!}{t_1!} X_{41}^t \cdot v^{[q,r]}.\]
    We now compute how $X_{41}$ acts on the four vectors used in the definition of $v^{[q,r]}$: it maps $v_1$ to $v_4$ and kills the other $v_i$; it sends $w$ to $w''$, and it kills $w'$ and $w''$. So $X_{41}^t \cdot v^{[q,r]}$ is a sum of terms of the form
    \[ w^{t-\alpha}\ (w'')^\alpha\ (w')^q\ v_1^{b-r - \beta}\ v_2^r\ v_4^\beta \]
    where $\alpha + \beta = t$; and the term for $\alpha = t, \beta = 0$ has coefficient $t!$.

    We now consider how $X_{12}$ acts on this element. One checks that $X_{12}$ acts on $V_G(1, 0)$ by $v_2 \mapsto v_1 \mapsto 0$, $v_4 \mapsto -v_3 \mapsto 0$; and on $V_G(1, 1)$ by $w'' \mapsto -w'$, , $w' \mapsto -2w_-$, $w_- \mapsto 0$, $w \mapsto 0$ where $w_- = X_{32} \cdot w = \left(g \mapsto \left| \begin{smallmatrix} g_{11} & g_{13} \\ g_{21} & g_{23} \end{smallmatrix} \right|\right)$ spans the $(1, -1)$ weight space of $V(1, 1)$. It follows that the number of applications of $X_{12}$ needed to kill the above element is exactly $2\alpha + q + \beta + r + 1$. Since $\alpha + \beta$ is fixed, the last term to be annihilated is the one for $\alpha = t, \beta = 0$.
   \end{proof}

   We now consider the image of $\br^{[q,r]}\left(v^{t_1 -t} w^t \boxtimes v^{t_2}\right)$ in the graded pieces of the $P_{\Sieg}$-stable filtration on $V_G$ given by eigenspaces for $Z(M_{\Sieg})$ as in \cite[Definition 6.1]{LPSZ1}. Note that we have $\br^{[q,r]}\left(v^{t_1 -t} w^t \boxtimes v^{t_2}\right) \in \Fil^{r_1} V_G$. Moreover, although the representation $\Gr^{r_1} V_G$ is far from being irreducible, it is semi-simple and has a unique direct summand of highest highest weight, isomorphic to $W_G(r_1, -r_2; r_1 + r_2)$.

   Since $M_{\Sieg} \cap \operatorname{Sp}_4$ is isomorphic to $\GL_2$, via $\stbt{A}{}{}{\star} \mapsto A$, we can identify $W_G(r_1, -r_2; r_1 + r_2)$ with the representation $\Sym^{r_1 + r_2} \otimes \det^{-r_2}$ of $\GL_2$, so it has a canonical basis $v^{(r_1 +r_2 - i)} w^i$ for $0 \le i \le r_1 + r_2$. We normalise the projection $\Gr^{r_1} V_G \onto W_G(r_1, -r_2; r_1 + r_2)$ to send $v_1^{r_1 - r_2} (w_-)^{r_2}$ to $v^{r_1 + r_2}$.

   \begin{proposition}
    \label{prop:branchingcomparison}
    The image of $\br^{[q,r]}\left(v^{t_1-t} w^t \boxtimes v^{t_2}\right) \in \Fil^{r_1} V_G$ under projection to the summand $W_G(r_1, -r_2; r_1 + r_2)$ is given by
    \[ \frac{(-2)^q}{\binom{t_1}{t}} \cdot v^{r_1 + r_2-n}w^n,\quad n = 2r_2 - q + r.\]
    Similarly, the image of $\br^{[q,r]}\left(v^{t_1} \boxtimes v^{t_2-t} w^t\right)$ is given by
    \[ \frac{(-2)^q}{\binom{t_2}{t}} v^{r_1 + r_2-m}w^m, m = q + r.\]
   \end{proposition}

   \begin{proof}
    Letting $X = X_{12}$ for brevity, and recalling that $n = q + r + 2t$, we have
    \[ X^{n} \cdot (w'')^t (w')^q v_1^{b-r} v_2^r = \frac{n!}{(2t)! q!r!} (X^{2t} \cdot (w'')^t) \times (X^q\cdot  (w')^q) \times v_1^{r_1-r_2-r}\times (X^r \cdot v_2^r) \]
    by the Leibniz rule, with all other terms being 0. Since $X^2 w' = 0$, we have $X^q\cdot  (w')^q = q! (X \cdot w')^q = (-2)^q q! (w_-)^q$, and similarly $X^r \cdot v_2^r = r! v_1^r$. The term $X^{2t} \cdot (w'')^t$ is a little more fiddly to evaluate; we conclude that
    \[ X^{2t} \cdot (w'')^t = \binom{2t}{2, \dots, 2} (X^2 \cdot w'')^t = \frac{(2t)!}{2^t} (2w_-)^t,\]
    so the conclusion is that
    \[ X^{2t+q+r} \cdot (w'')^t (w')^q v_1^{b-r} v_2^r = (-2)^q n!  v_1^{r_1 - r_2} (w_-)^{r_2}.\]
    On the other hand, the unique vector in the standard basis of $W_G(r_1, -r_2; r_1 + r_2)$ having the same weight as $\br^{[q,r]}\left(v^{t_1 -t} w^t \boxtimes v^{t_2}\right)$ is $v^{r_1 + r_2 - n} w^n$, whose image under $X_{12}^n$ is $n! v^{r_1 + r_2}$. Hence the factor $(-2)^q$.
   \end{proof}

   Let us now perform a similar computation for $v^{t_1} \boxtimes v^{t_2-t} w^t \in V_H$. The image of this in $V_G$ is clearly $\tfrac{(t_2 - t)!}{t_2!}  X_{32}^t \cdot v^{[q,r]}$ and we compute that this is equal to
   \[ \tfrac{1}{\binom{t_2}{t}} (w_-)^t v_1^{b-r} (w')^q v_2^r \]
   plus other terms killed by lower powers of $X_{12}$. Acting on this by $X_{12}^{q + r}$ gives $(-2)^q (q + r)! v_1^{r_1-r_2} (w_-)^{r_2}$, so its image in $W_G(r_1, -r_2; r_1 + r_2)$ has to be
   \[ \frac{(-2)^q}{\binom{t_2}{t}} v^{r_1 + r_2-m}w^m, m = q + r.\]

   \begin{remark}
    Compare \cite[Theorem 9.6.4]{LSZ17}. With the benefit of hindsight, one can observe that it would have been better to define $v^{[q,r]}$ to be $\tfrac{1}{(-2)^q}$ times its present definition; this would simultaneously kill the error terms $(-2)^q$ both here and in \emph{op.cit.}.
   \end{remark}

  \subsection{Unit-root splittings}

   Now let us consider the following construction. Our choice of embedding $V_H \otimes \det^q \into V_G$ is strictly compatible with the filtrations, and hence gives rise to a pushforward map
   \[
    H^0\left(\cX_{H, \Delta}^{(m,m)}, \frac{\Fil^m V_H}{\Fil^n V_H} \otimes \Omega^1_H(-D)\right)
    \to H^1\left(\cX_{G, \Kl}^m, \frac{\Fil^{m + q} V_G}{\Fil^{n+q} V_G} \otimes \Omega^2_G(-D)\right)
   \]
   for any $m \le n$, and dually a pullback map
   \[
    H^2_c\left(\cX_{G, \Kl}^m, \frac{\Fil^{m + q} V_G}{\Fil^{n+q} V_G} \otimes \Omega^1_G\right) \to H^2_c\left(\cX_{H, \Delta}^{(m,m)}, \frac{\Fil^m V_H}{\Fil^n V_H} \otimes \Omega^1_H\right).
   \]
   (where we have identified $V_H$ and $V_G$ with their own duals, up to twisting).

   \begin{remark}
    More precisely, \emph{a priori} we have two slightly different versions of the pushforward map. One such map (the one which is ``natural'' from the point of view of de Rham cohomology) arises from tensoring the short exact sequence of sheaves on $X_G$
    \[ 0 \to \Omega^2_{X_G} \to \Omega^2_{X_G}(\log X_H) \to \iota_\star(\Omega^1_{X_H}) \to 0\]
    with $\Fil^{n+q} V_G(-D)$. However, from the point of view of coherent sheaves it is natural to consider instead the sequence of line bundles
    \[ 0 \to \Omega^3_{X_G} \to \Omega^3_{X_G}(\log X_H) \to \iota_\star(\Omega^2_{X_H}) \to 0 \]
    and tensor with $\Fil^{n+q} V_G(-D) \otimes \Omega^2_{X_G} \otimes (\Omega^3_{X_G})^\vee$. 
    The two constructions are compatible via a map
    \[ \iota^\star\left(\Omega^2_{X_G} \otimes (\Omega^3_{X_G})^\vee\right)\otimes \Omega^2_{X_H} \longrightarrow  \Omega^1_{X_H}\]
    defined by dualising the natural map
    \( \iota^\star(\Omega^1_{X_G}) \to \Omega^1_{X_H}.\)
   \end{remark}

   We shall be interested in the pushforward map in the form
   \[
    H^0\left(\cX_{H, \Delta}^{(m,m)}, \frac{V_H}{\Fil^{r_1 - q + 1} V_H} \otimes \Omega^1_H(-D)\right)
   \to H^1\left(\cX_{G, \Kl}^m, \frac{V_G}{\Fil^{r_1 + 1} V_G} \otimes \Omega^2_G(-D)\right).
   \]
   The sheaf on the right-hand side was denoted by $[\tilde{L}_1]$ in \S 6 of \cite{LPSZ1}, and its cohomology was termed ``automorphic nearly-coherent cohomology''. We can expand this to the following diagram:
   \[
    \begin{tikzcd}
     H^0\left(\cX_{H, \Delta}^{(m,m)}, \frac{V_H}{\Fil^{r_1 - q + 1} V_H} \otimes \Omega^1_H(-D)\right)
     \rar & 
     H^1\left(\cX_{G, \Kl}^m, \frac{V_G}{\Fil^{r_1 + 1} V_G} \otimes \Omega^2_G(-D)\right)
     \\
     H^0\left(\cX_{H, \Delta}^{(m,m)}, \Gr^{r_1-q} V_H \otimes \Omega^1_H(-D)\right)
     \rar \uar \arrow{rd} &
     H^1\left(\cX_{G, \Kl}^m, \Gr^{r_1} V_G \otimes \Omega^2_G(-D)\right)
     \uar \dar \\
     &
     H^1\left(\cX_{G, \Kl}^m, \cN^2(-D)\right)
    \end{tikzcd}
   \]
   The content of \cref{prop:branchingcomparison} is to express the lower diagonal arrow on the direct summands $\omega_H^{(t_1 - t, t_2)}$ and $\omega_H^{(t_1, t_2 - t)}$ of $\Gr^{r_1 - q} V_H$ as a multiple of the ``standard'' pushforward maps from these spaces to $\cN^1$ considered in \cite[\S 4.6]{LPSZ1}.

   We now pass to the $p$-adic completion (i.e.~we replace the dagger spaces $\cX_{H, \Delta}^{(m,m)}$ and $\cX_{G, \Kl}^{2,m}$ with their underlying rigid-analytic spaces, which amounts to forgetting overconvergence).

   \begin{notation}
    We denote these spaces by $\mathbb{X}_{G, \Kl}^m$ and $\mathbb{X}_{H, \Delta}^{(m,m)}$.
   \end{notation}

   Then we have the following diagram:
   \[
    \begin{tikzcd}
    H^0\left(\mathbb{X}_{H, \Delta}^{(m,m)}, \frac{V_H}{\Fil^{r_1 - q + 1} V_H} \otimes \Omega^1_H(-D)\right)
    \rar  \arrow[dashed, bend right=90]{d}&[-1.5em]
    H^1\left(\mathbb{X}_{G, \Kl}^m, \frac{V_G}{\Fil^{r_1 + 1} V_G} \otimes \Omega^2_G(-D)\right)
    \arrow[dashed, bend left=30]{ddr}
    \\
    H^0\left(\mathbb{X}_{H, \Delta}^{(m,m)}, \Gr^{r_1-q} V_H \otimes \Omega^1_H(-D)\right)
    \rar \uar \arrow{rd} &
    H^1\left(\mathbb{X}_{G, \Kl}^m, \Gr^{r_1} V_G \otimes \Omega^2_G(-D)\right)
    \uar \dar \\
    &
    H^1\left(\mathbb{X}_{G, \Kl}^m, \cN^2(-D)\right)
    \rar &[-4em]
    H^1\left(\mathbb{X}_{G, \Kl}^m, \mathfrak{F}(-D)\right).
    \end{tikzcd}
   \]
   Here $\mathfrak{F} = \mathfrak{F}_G(3+r_1, 1-r_2)$ is the Banach sheaf introduced in \cite[\S 9]{pilloni20} (see also \cite[\S 3.2]{LPSZ1}). The dashed arrow on the right is given by \cite[Corollary 6.15]{LPSZ1}, while the dashed arrow on the left is the unit-root splitting of the Hodge filtration.

   \begin{proposition}
    The two maps
    \[ H^0\left(\mathbb{X}_{H, \Delta}^{(m,m)}, \frac{V_H}{\Fil^{r_1 - q + 1} V_H} \otimes \Omega^1_H(-D)\right)
    \longrightarrow
    H^1\left(\mathbb{X}_{G, \Kl}^m, \mathfrak{F}(-D)\right),\]
    given by composing either of the two dashed arrows with the remaining maps in the diagram, coincide.
   \end{proposition}

   \begin{proof}
    This follows from the argument of Theorem 6.16 of \cite{LPSZ1}; see Remark 6.18 of \emph{op.cit.}. (In \emph{op.cit.} the cotangent sheaf $\Omega^1_H \cong \omega^{(2, 0)} \oplus \omega^{(0, 2)}$ was replaced by the conormal sheaf $\ker(\iota^\star \Omega^1_G \to \Omega^1_H) \cong \omega^{(1, 1)}$, but this makes no difference to the argument.)
   \end{proof}

   We can now summarize the computations of this section in the following corollary:

   \begin{corollary}\label{cor:unitroot}
    Let $\eta \in H^2(\cX_{G, \Kl}^m, \cN^1)$ be a class which is ordinary for the $U_{2, \Kl}'$ operator, and let $\breve{\eta}$ be its unique ordinary lifting to $H^2(\cX_{G, \Kl}^m,\Fil^{r_2}\cV_G\otimes\Omega^1_G)$.

    Then the linear functional on $H^0\left(\cX_{H, \Delta}^{(m,m)}, \frac{V_H}{\Fil^{r_1 - q + 1} V_H} \otimes \Omega^1_H(-D)\right)$ given by pairing with the class $(\iota_{\Delta}^{[t_1,t_2]})^\star(\breve\eta)$ factors through the composite of restriction to $\mathbb{X}_{H, \Delta}^{(m,m)}$ (forgetting overconvergence) and the unit-root splitting into $\Gr^{r_1 - q} V_H$; and it is given on $H^0\left(\mathbb{X}_{H, \Delta}^{(m,m)}, \omega^{(t_1-2t, t_2)} \otimes \Omega^{1}_H(-D)\right)$, where $t = r_2 - q$, by the formula
    \[ \frac{(-2)^q}{\binom{t_1}{t}}\langle \iota_{\star}^{\text{$p$-adic}}(-), \eta \rangle\]
    where $\iota_{\star}^{\text{$p$-adic}}$ denotes the pushforward map for $p$-adic modular forms defined in \cite[ \S 4]{LPSZ1}. There is an analogous formula on $\omega^{(t_1, t_2-2t)}$ with the factor $\frac{(-2)^q}{\binom{t_2}{t}}$.
   \end{corollary}

   \begin{remark}
    We will apply Corollary \ref{cor:unitroot} later to the element $\eta^{(2,m)}_{\coh,q}$.
   \end{remark}

\mychapter{Step 4: Computation of the regulator}




  \section{Evaluation of the pairing}\label{sect:evaluation}


   \subsection{Expansion in coherent cohomology}

   We now want to evaluate the pairing
   \[ 	\left\langle \left(\iota^{[t_1,t_2]}_\Delta\right)^\star(\breve\xi,\breve\eta^{(2,m)}_{\coh}),\, \left( \alpha^{t_1,t_2,\Phi^{(p)}_1,\Phi^{(p)}_2}_1 , \alpha^{t_1,t_2,\Phi_1^{(p)},\Phi_2^{(p)}}_2\right)\right\rangle_{\coh-\fp,\cX^{(m,m)}_{H,\Delta}},\]
   where $\breve\xi$ is as defined in Proposition \ref{prop:goodxi}.

    We expand the pairing using Besser's formalism for computing the cup product, as explained in Section \ref{ssec:Bessercup}.  Let
   \[ a(x,y)=p^{t_1+t_2+2}\langle p\rangle^{-1}_Hy\]
   and
   \[ b(x,y)=\frac{\cQ_q\left(p^{t_1+t_2+2}\langle p\rangle^{-1}_Hxy\right)-p^{t_1+t_2+2}\langle p\rangle^{-1}_H\,y\cQ_q(x)}{1-p^{t_1+t_2+2}\langle p\rangle^{-1}_H\,y},\]
   so we have
   \[ \cQ_q\star \cR(xy)=a(x,y)\cQ_q(x)+b(x,y)\left(1-p^{t_1+t_2+2}\langle p\rangle^{-1}_H\,y\right),\]
   where $\cR$ is as defined in Lemma \ref{lem:HEisclassasfppair}.
   Then \eqref{reduction1} is equal to
   \begin{equation*}
     \ a(\varphi_{H,1}^\star\otimes 1,1\otimes \varphi_{H}^\star)\! \left[(\iota^{(t_1,t_2)}_\Delta)^\star(\zeta)\cup \alpha^{t_1,t_2\Phi^{(p)}_1,\Phi^{(p)}_2}_2\right] + b(\varphi_{H,1}^\star\otimes 1,1\otimes \varphi_{H}^\star)\! \left[(\iota^{(t_1,t_2)}_\Delta)^\star(\eta^{(2,m)}_{\coh})\cup \alpha^{t_1,t_2,\Phi^{(p)}_1,\Phi^{(p)}_2}_1\right].
    \end{equation*}

    \begin{proposition}\label{prop:evilterm0}
     We have
     \[ \text{\eqref{reduction1}}= b(\varphi_{H,1}^\star\otimes 1,1\otimes \varphi_H^\star)\, \left[(\iota^{(t_1,t_2)}_\Delta)^\star(\breve\eta^{(2,m)}_{\coh,q})\cup \alpha^{t_1,t_2,\Phi^{(p)}_1,\Phi^{(p)}_2}_1\right].\]
    \end{proposition}
    \begin{proof}
     We need to show that
     \[ (\iota^{(t_1,t_2)}_\Delta)^\star(\breve\xi)\cup \varphi_H^\star\alpha^{(t_1,t_2,\Phi_1,\Phi_2)}_2=0.\]
     Now $U_p\boxtimes U_p\circ\varphi_H^\star=\langle p\rangle_H$, 
     so
     \begin{equation}\label{eq:equal0}
       (U_p\boxtimes U_p)\left[(\iota^{(t_1,t_2)}_\Delta)^\star(\breve\xi)\cup \varphi_H^\star\,\alpha^{t_1,t_2,\Phi^{(p)}_1,\Phi^{(p)}_2}_2\right]=(U_p\boxtimes U_p)(\iota^{(t_1,t_2)}_\Delta)^\star(\breve\xi)\cup \langle p\rangle_H\alpha^{t_1,t_2,\Phi^{(p)}_1,\Phi^{(p)}_2}_2.
      \end{equation}
     But $(U_p\boxtimes U_p)(\iota^{(t_1,t_2)}_\Delta)^\star(\breve\xi)=0$ by Note \ref{note:pullbackofxiprop}, and hence \eqref{eq:equal0} is zero. Now by Note \ref{note:highestfpclassphiaction}, the operator  $U_p\boxtimes U_p$ acts as multiplication by $p^{-2}$ on $H^{3,2}_c(\cX_\Delta^{2,m}\langle -\cD_\Delta\rangle,\Qp,3;\cQ_q\star \cR)$ and hence is invertible. This finishes the proof.
    \end{proof}

    Write $P(x)=1+c_1x+c_2x^2$; by definition, we have $c_2=(\alpha\beta)^{-1}$ and $c_1=-\frac{\alpha+\beta}{\alpha\beta}$. Then
    \[ b(x,y)=1-c_2\,p^{t_1+t_2+2}\langle p\rangle_H^{-1}\, x^2y.\]
    We  now identify $\varphi_{H,1}^\star$ with $p^{-1}\varphi_H^\star$.

    \begin{corollary}
     We have
     \begin{align}
      \text{\eqref{reduction1}} =\quad  & (\iota^{(t_1,t_2)}_\Delta)^\star(\breve\eta^{(2,m)}_{\coh})\cup \alpha^{(t_1,t_2,\Phi_1,\Phi_2)}_1-c_2p^{t_1+t_2}\langle p\rangle_H^{-1}\cdot \varphi_H^\star\left[(\iota^{(t_1,t_2)}_\Delta)^\star\varphi_H^\star(\breve\eta^{(2,m)}_{\coh})\cup \alpha^{t_1,t_2,\Phi^{(p)}_1,\Phi^{(p)}_2}_1\right]\notag\\
      =\quad&   (\iota^{(t_1,t_2)}_\Delta)^\star(\breve\eta^{(2,m)}_{\coh})\cup \left(\epsilon^{t_1,\Phi^{(p)}_1}_0\sqcup \epsilon^{t_2,\Phi^{(p)}_2}_1\right) \label{1stterm}\\
      &  -c_2\,p^{t_1+t_2}\langle p\rangle_H^{-1}\, \varphi_H^\star\left[(\iota^{(t_1,t_2)}_\Delta)^\star\varphi_H^\star(\breve\eta^{(2,m)}_{\coh})\cup \left(\epsilon^{t_1,\Phi^{(p)}_1}_0\sqcup \epsilon^{t_2,\Phi^{(p)}_2}_1\right)\right]\label{2ndterm}\\
      & +p^{t_1}\, (\iota^{(t_1,t_2)}_\Delta)^\star(\breve\eta^{(2,m)}_{\coh})\cup  \left( \langle p\rangle_{\GL_2}^{-1}\, \varphi_{\GL_2}^\star\epsilon^{t_1,\Phi^{(p)}_1}_1\sqcup\epsilon^{t_2,\Phi^{(p)}_2}_0\right)\label{3rdterm}\\
      & -c_2\,p^{2t_1+t_2+1}\langle p\rangle_H^{-1}\,\varphi_H^\star\left[ (\iota^{(t_1,t_2)}_\Delta)^\star\varphi_H^\star(\breve\eta^{(2,m)}_{\coh})\cup  \left( \langle p\rangle_{\GL_2}^{-1}\,\varphi_{\GL_2}^\star\epsilon^{t_1,\Phi^{(p)}_1}_1\sqcup \epsilon^{t_2,\Phi^{(p)}_2}_0\right)\right].\label{4thterm}
     \end{align}
    \end{corollary}

    We will see that this expression simplifies.

    \begin{lemma}
     We have
     \begin{align*}
       \varphi_H^\star\left[(\iota^{(t_1,t_2)}_\Delta)^\star\varphi_H^\star(\breve\eta^{(2,m)}_{\coh})\cup \left(\epsilon^{t_1,\Phi^{(p)}_1}_0\sqcup \epsilon^{t_2,\Phi^{(p)}_2}_1\right)\right]&=0\\
       \varphi_H^\star\left[ (\iota^{(t_1,t_2)}_\Delta)^\star\varphi_G^\star(\breve\eta^{(2,m)}_{\coh})\cup  \left(\langle p\rangle_{\GL_2}^{-1}\varphi_{\GL_2}^\star\epsilon^{t_1,\Phi^{(p)}_1}_1\sqcup \epsilon^{t_2,\Phi^{(p)}_2}_0\right)\right]&=0
      \end{align*}
    \end{lemma}
    \begin{proof}
     Apply $(U_p\boxtimes U_p)^2$ and use Note \ref{note:highestfpclassphiaction} and the fact that $U_p\left(\epsilon^{t_1,\Phi^{(p)}_1}_0\right)=0$ by Lemma \ref{lem:UponEiscomp}.
    \end{proof}

    We hence deduce the following formula for the pairing:

    \begin{proposition}\label{prop:red2}
     We have
     \begin{align*}
      \text{\eqref{reduction1}}= & (\iota^{(t_1,t_2)}_\Delta)^\star(\breve\eta^{(2,m)}_{\coh})\cup \left(\epsilon^{t_1,\Phi^{(p)}_1}_0\sqcup \epsilon^{t_2,\Phi^{(p)}_2}_1\right)\\
     & +p^{t_1+t_2}\, (\iota^{(t_1,t_2)}_\Delta)^\star(\breve\eta^{(2,m)}_{\coh})\cup  \left( \langle p\rangle_{\GL_2}^{-1}\, \varphi_{\GL_2}^\star\epsilon^{t_1,\Phi^{(p)}_1}_1\sqcup\epsilon^{t_2,\Phi^{(p)}_2}_0\right)
     \end{align*}
    \end{proposition}

    We now apply Corollary \ref{cor:unitroot}:

  \begin{note}\label{lem:Grbasis}
   For $0\leq \ell\leq t_1+t_2$. A basis of $\Gr^\ell V_H$ is given by
   \[\{ v^{t_1-i_1}w^{i_1}\boxtimes v^{t_2-i_2}w^{i_2}: 0\leq i_n\leq t_n,\ i_1+i_2=t_1+t_2-\ell\}.\]
  \end{note}

   \begin{lemma}\label{lem:EsynimageinGr1}
    The image of $\epsilon^{t_1,\Phi^{(p)}_1}_{0}\sqcup \epsilon^{t_2,\Phi^{(p)}_2}_{1}$ under projection to $\Gr^{r_1-q}\cV_H$ is given by
    \[ (-1)^{r_2-q}\frac{t_1!}{(r_1-r_2-r)!}\times \theta^{(r_1-r_2-r)}E^{-t_1}_{\Phi_1^{(p)}\Phi_{\dep}}\cdot  \,v^{r_1-r_2-r}w^{r_2-q}\boxtimes F^{t_2+2}_{\Phi^{(p)}_2\Phi_{\crit}}\cdot (v^{t_2}\otimes \xi\otimes e_1).\]
   \end{lemma}

   \begin{proof}
    The basis vectors with non-trivial coefficients of $\epsilon^{t_1+2,\Phi_1}_{0}\sqcup \epsilon^{t_2+2,\Phi_2}_{1}$ are of the form
    \[ v^{t_1-i_1}w^{i_1}\boxtimes w^{t_2}\qquad 0\leq i\leq t_1.\]
    By Note \ref{lem:Grbasis}, this will project non-trivially to $\Gr^{r_1-q}\cV_H$ if and only if
    \[i_1=t_1-(r_1-q)=r.\]
   \end{proof}

   We analogously prove the following result:

   \begin{lemma}
    \label{lem:EsynimageinGr2}
    The image of $\langle p\rangle_{\GL_2}^{-1}\, \varphi_{\GL_2}^\star\epsilon^{t_1,\Phi^{(p)}_1}_1\sqcup \epsilon^{t_2,\Phi^{(p)}_2}_0$ in $\Gr^{r_1-q}\cV_H$ is given by
    \[
     (-1)^{r_2-q}\frac{t_2!}{r!}\times \langle p\rangle_{\GL_2}^{-1}\varphi_{\GL_2}^\star\left(F^{t_1+2}_{\Phi^{(p)}_1\Phi_{\crit}}\cdot\,(v^{t_1}\otimes \xi\otimes e_1)\right)\boxtimes \left(\theta^{r}E^{-t_2}_{\Phi^{(p)}_2\Phi_{\dep}}\cdot v^rw^{r_2-q}\right).
    \]
   \end{lemma}

   \begin{proposition}
    \label{prop:linktoLPSZ}
    We have
    \begin{align*}
       (\iota^{(t_1,t_2)}_\Delta)^\star&(\breve\eta^{(2,m)}_{\coh})\cup \left(\epsilon^{t_1,\Phi^{(p)}_1}_0\sqcup \epsilon^{t_2,\Phi^{(p)}_2}_1\right)\\
       & =\frac{(-1)^{r_2-q}\,t_1!}{(r_1-r_2-q)!}\times \frac{(-2)^q}{\tbinom{t_1}{r_2-q}}\times \left\langle \eta^{(2,m)}_{\coh},\,\iota^{p-adic}_\star\left( \theta^{(r_1-r_2-r)}E^{-t_1}_{\Phi_1^{(p)}\Phi_{\dep}}\sqcup F^{t_2+2}_{\Phi^{(p)}_2\Phi_{\crit}}\right)\right\rangle\\
       & = (-1)^{r_2-q}(-2)^q\, (r_2-q)!\times \left\langle \eta^{(2,m)}_{\coh},\,\iota^{p-adic}_\star\left( \theta^{(r_1-r_2-r)}E^{-t_1}_{\Phi_1^{(p)}\Phi_{\dep}}\sqcup F^{t_2+2}_{\Phi^{(p)}_2\Phi_{\crit}}\right)\right\rangle,
      \end{align*}
    and
    \begin{align*}
    &(\iota^{(t_1,t_2)}_\Delta)^\star(\breve\eta^{(2,m)}_{\coh})\cup \left( \langle p\rangle_{\GL_2}^{-1}\, \varphi_{\GL_2}^\star\epsilon^{t_1,\Phi^{(p)}_1}_1\sqcup \epsilon^{t_2,\Phi^{(p)}_2}_0\right)\\
    &\quad =\frac{(-1)^{r_2-q}\, t_2!}{r!}\times \frac{(-2)^q}{\tbinom{t_2}{r_2-q}}\times \left\langle \eta^{(2,m)}_{\coh},\,\iota^{p-adic}_\star\left( \langle p\rangle_{\GL_2}^{-1}\varphi_{\GL_2}^\star\left(F^{t_1+2}_{\Phi^{(p)}_1\Phi_{\crit}}\right)\boxtimes \theta^{r}E^{-t_2}_{\Phi^{(p)}_2\Phi_{\dep}}\right)\right\rangle\\
    & \quad = (-1)^{r_2-q}(-2)^q(r_2-q)!\times \left\langle \eta^{(2,m)}_{\coh},\,\iota^{p-adic}_\star\left( \langle p\rangle_{\GL_2}^{-1}\varphi_{\GL_2}^\star\left(F^{t_1+2}_{\Phi^{(p)}_1\Phi_{\crit}}\right)\boxtimes \theta^{r}E^{-t_2}_{\Phi^{(p)}_2\Phi_{\dep}}\right)\right\rangle.
   \end{align*}
   \end{proposition}


 \subsection{Families of Eisenstein series}

  We now interpret the cup-products of \cref{prop:linktoLPSZ} in terms of the 2-parameter $p$-adic family of Eisenstein series studied in \cite{LPSZ1}, and a 1-parameter ``critical'' variant.

  \begin{proposition}
   If $\Phi^{(p)} \in \cS(\Af^p, (\chi^{(p)})^{-1})$, then the $p$-adic Eisenstein series $E^{-k}_{\Phi^{(p)}\Phi_{\dep}}$ is the specialisation at $(\kappa_1,\kappa_2)=(0,-1-k)$ of a $2$-parameter family of Eisenstein series $\cE^{\Phi^{(p)}}(\kappa_1,\kappa_2)$.  This family has $q$-expansion
   \[ \sum_{u,v\in (\ZZ_{(p)})^2,\, uv>0} \mathrm{sgn}(u) u^{\kappa_1}v^{\kappa_2}(\Phi^{(p)})'(u,v)q^{uv},\]
   and its specialisation at $(a+\mu,b+\nu)$, for integers $a,b\geq 0$ and finite-order characters $\mu,\nu$ of $\Zp^\times$, is the $p$-adic modular form associated to the algebraic nearly-holomorphic modular form
   \[ (g,\tau)\mapsto \nu(\det g)\cdot E^{(a+b+1,\Phi_{\mu,\nu})}\left(g,\tau;\chi^{(p)}\mu^{-1} \nu, \frac{b-a+1}{2}\right).\]
   Here, $\Phi_{\mu,\nu}$ is defined as in \cite[\S 7.3]{LPSZ1}, and $(\Phi^{(p)})'(u,v)$ is as defined in \eqref{eq:Fourier}.
  \end{proposition}
  \begin{proof}
   See \cite[Theorem 7.6]{LPSZ1}.
  \end{proof}

  We can also put critical-slope Eisenstein series into $1$-parameter $p$-adic families:

  \begin{proposition}
   Let $\ell\geq 0$. Then there exists a $1$-parameter family $\cE^{\Phi^{(p)}}(\underline{\ell},\kappa)$ of Eisenstein series  with $q$-expansion
   \[  \sum_{u\in \Zp,\, v\in (\ZZ_{(p)})^2,\, uv>0} \mathrm{sgn}(u) u^{\ell}v^{\kappa}\grave{\Phi}^{(p)}(u,v)q^{uv}.\]
   (Here, we underline the parameter which does \textbf{not} vary in a $p$-adic family.)
   Its specialisation at $a+\nu$, for an integer $a\geq 0$ and a finite-order character $\nu$ of $\Zp^\times$, is the $p$-adic modular form associated to the algebraic nearly-holomorphic modular form
   \[ (g,\tau)\mapsto \nu(\det g)\cdot E^{(\ell + b +1,\Phi^p \Phi_{\crit, \nu})}\left(g,\tau;\chi^{(p)}\nu, \frac{b-\ell+1}{2}\right),\]
   where $\Phi_{\crit, \nu}'(x,y) = \ch(\Zp \times \Zp^\times)(x,y) \cdot \nu(y)$.
  \end{proposition}
%
%
  We can now restate \cref{prop:linktoLPSZ} in the following form:

  \begin{proposition}
   \label{def:sLi}
   Let us define
   \begin{align}
    \sL_1 &\coloneqq \left\langle \eta^{(2,m)}_{\coh},\,\iota_\star^{p-adic}\left(\cE^{\Phi_1^{(p)}}(r', -1-q')\boxtimes \cE^{\Phi_2^{(p)}}(\underline{q' + r + 1},0)\right)\right\rangle,\\
    \sL_2 &\coloneqq \left\langle\eta^{(2,m)}_{\coh},\, \iota_\star^{p-adic}\left(\langle p\rangle^{-1}\varphi_{\GL_2}^\star \cE^{\Phi_1^{(p)}}(\underline{q'+r'+1}, 0)\boxtimes \cE^{\Phi_2^{(p)}}(r,-1-q')\right)\right\rangle,
   \end{align}
   where $q' = r_2 - q$ and $r' = r_1-r_2 - r$ (so $q', r' \ge 0$). Then the cup-product \eqref{eq:ordpairing} is equal to
   \[ \frac{(-1)^{r_2-q}(-2)^q\, (r_2-q)!}{\left(1-\tfrac{\gamma}{p^{1+q}}\right)\left(1-\tfrac{\delta}{p^{1+q}}\right)} \times \left(\sL_1 + p^{r_1 + r_2 - 2q} \sL_2\right).\]
  \end{proposition}

  \begin{proof}
   In the above notation, the two terms appearing in \ref{prop:linktoLPSZ} are
   \begin{gather*}
    (\iota^{(t_1,t_2)}_\Delta)^\star(\breve\eta^{(2,m)}_{\coh})\cup \left(\epsilon^{t_1,\Phi^{(p)}_1}_0\sqcup \epsilon^{t_2,\Phi^{(p)}_2}_1\right) = (-1)^{r_2-q}(-2)^q\, (r_2-q)!\times \sL_1,\\
    (\iota^{(t_1,t_2)}_\Delta)^\star(\breve\eta^{(2,m)}_{\coh})\cup \left( \langle p\rangle_{\GL_2}^{-1}\varphi_{\GL_2}^\star\left(F^{t_1+2}_{\Phi^{(p)}_1\Phi_{\crit}}\right)\sqcup\theta^{r_2-q}E^{-t_2}_{\Phi^{(p)}_2\Phi_{\dep}}\right) = (-1)^{r_2-q}(-2)^q(r_2-q)!\times\sL_2.
   \end{gather*}
   The normalisation of the trace map on FP-cohomology gives rise to a factor $P_q\left(p^{t_1+t_2+1}\chi_\Pi(p)\right)$, and using the relation $\alpha\delta=\beta\gamma=p^{r_1+r_2+3}\chi_\Pi(p)$, we deduce that
   \[
    P_q\left(p^{t_1+t_2+1}\chi_\Pi(p)\right) = \left(1-\frac{\gamma}{p^{1+q}}\right)\left(1-\frac{\delta}{p^{1+q}}\right).\qedhere
   \]
  \end{proof}

  We will see shortly that $\sL_2$ is in fact zero, and that $\sL_1$ coincides with a non-critical $p$-adic $L$-value. We first make a preliminary reduction.

  \begin{proposition}
   \label{prop:thetainvar}
   We have
   \[\sL_1 = -\left\langle\eta^{(2,m)}_{\coh}, \iota^{p-adic}_\star\left[\cE^{\Phi_1^{(p)}}(r_1-q+1, r)\boxtimes  \cE^{\Phi_2^{(p)}}(\underline{0},-1-q'-r)\right]\right\rangle.\label{xi1stterm}
   \]
   and similarly
   \[
    \sL_2 = (-1)^{r_1 - r_2 + 1} \left\langle\eta^{(2,m)}_{\coh},\, \iota_\star^{p-adic}\left(\langle p\rangle^{-1}\varphi_{\GL_2}^\star \cE^{\Phi_1^{(p)}}(\underline{0}, -1-q'-r')\boxtimes \cE^{\Phi_2^{(p)}}(r_1 - q+1, r')\right)\right\rangle\]
  \end{proposition}

  \begin{proof}
   Both of these statements follow from the general fact that
   \begin{equation}
    \label{eq:thetatwist}
    \left\langle\eta^{(2,m)}_{\coh}, \iota^{p-adic}_\star\left[\cF\boxtimes \theta(\cG) +\theta(\cF) \boxtimes \cG \right]\right\rangle = 0
   \end{equation}
   for any nearly-overconvergent $p$-adic modular forms $\cF$, $\cG$ whose weights sum to $r_1 - r_2$. (This, in turn, follows from the fact that $\cF\boxtimes \theta(\cG) +\theta(\cF) \boxtimes \cG$ is the projection to a graded piece of the Hodge filtration of an overconvergent vector-valued form in the image of $\nabla$, which must pair to 0 with $\eta^{(2,m)}_{\coh}$, since $\nabla(\eta^{(2,m)}_{\coh}) = 0$).
  \end{proof}


 \subsection{Evaluation of $\sL_1$}
  \label{sect:evalL1}
  We shall now evaluate $\sL_1$. We shall do this by interpreting this value as the specialisation at the trivial character of a measure on $\Zp^\times$, whose values at certain other characters (corresponding to \emph{critical} $L$-values) can be compared with the $p$-adic $L$-function of \cite{LPSZ1}.

  \begin{definition}
   Define an element of $\Lambda(\Zp^\times \times \Zp^\times)$ by
   \[ \sL_1(\bfj_1, \bfj_2) \coloneqq \left\langle\eta^{(2,m)}_{\coh}, \iota^{p-adic}_\star\left[\cE^{\Phi_1^{(p)}}(r_1-r_2 - \bfj_1, \bfj_2)\boxtimes  \cE^{\Phi_2^{(p)}}(\underline{0},\bfj_1 - \bfj_2)\right]\right\rangle.\]
  \end{definition}

  \begin{proposition}
   We have
   \[ \sL_1(\bfj_1, \bfj_2) = \left\langle\eta^{(2,m)}_{\coh}, \iota^{p-adic}_\star\left[\cE^{\Phi_1^{(p)}}(r_1-r_2 - \bfj_1, \bfj_2)\boxtimes  \cE^{\Phi_2^{(p)}}(0,\bfj_1 - \bfj_2)\right]\right\rangle \]
   (without the underline).
  \end{proposition}

  Note that the Eisenstein series in the second formula is exactly the $\cE(\uPhi^p)$ appearing in Proposition \ref{prop:weprove}.

  \begin{proof}
   It suffices to prove that these two measures agree after specialising at $(\bfj_1, \bfj_2) = (a_1 + \rho_1, a_2 + \rho_2)$ with $\rho_i$ finite-order characters and $r_1 -r_2 \ge a_1 \ge a_2 \ge 0$. In this range, both sides of the claimed formula reduce to cup-products in classical coherent cohomology; and as in \cite{LPSZ1}, they can be written as Euler products of local integrals at each place, with the factors at all primes except possibly $p$ being identical. The computation of \S 4.4 of \cite{LZ-equivar} shows that the factors at $p$ are also equal (despite the slightly different choice of test data). Thus the two measures are equal.
  \end{proof}

  Specialising the above proposition at $(\bfj_1, \bfj_2) =  (-1-r_2 + q, r)$, we conclude that
  \[ \sL_1 = -\frac{\wZ^p(w^p, \uPhi^p)}{\vol(V)} \cL_{p, \nu}(\Pi, -1-r_2 + q, r) .\]

  \subsection{Vanishing of $\sL_2$}
   \label{sect:sL2vanishing}

   In order to show that $\sL_2$ is identically zero, we will use a similar deformation argument. Let us write
   \[
    \sL_2(\bfj_1, \bfj_2) = \left\langle\eta^{(2,m)}_{\coh},\, \iota_\star^{p-adic}\left(\langle p\rangle^{-1}\varphi_{\GL_2}^\star \cE^{\Phi_1^{(p)}}(\underline{0}, \bfj_1 - \bfj_2)\boxtimes \cE^{\Phi_2^{(p)}}(r_1-r_2-\bfj_1, \bfj_2)\right)\right\rangle.
   \]
   Again, if we let $\bfj_1, \bfj_2 = (a_1 + \rho_1, a_2 + \rho_2)$ with $r_1 -r_2 \ge a_1 \ge a_2 \ge 0$ and $\rho_i$ of finite order, we obtain a cup-product in classical coherent cohomology; and the  value $\sL_2$ above corresponds (up to a sign) to specialising at $(\bfj_1, \bfj_2) = (-1-r_2 + q, r')$. However, for all of the specialisations corresponding to critical values, the term at $p$ in the resulting product is 0, again by the computations in \S 4.4 of \cite{LZ-equivar}. So the measure $\sL_2(\bfj_1, \bfj_2)$ is identically 0, and hence so is its special value $\sL_2$.

  \subsection{Conclusion of the proof}

   \begin{proof}[Proof (of Theorem \ref{thm:mainthm})]
    The computation in this chapter shows that
    \begin{align*}
      &\left\langle  \left( \alpha^{t_1,t_2,\Phi^{(p)}_1,\Phi^{(p)}_2}_1 , \alpha^{t_1,t_2,\Phi_1^{(p)},\Phi_2^{(p)}}_2\right),\,  \left(\iota^{[t_1,t_2]}_\Delta\right)^\star(\breve\xi,\breve\eta^{(2,m)}_{\coh})\right\rangle_{\coh-\fp,\cX^{2,m}_{H,\Delta}}\\
     &\qquad = \frac{(-1)^{r_2 - q+1}(-2)^q (r_2 - q)!}{\left(1 - \tfrac{ \gamma}{p^{q+1}}\right)\left(1 - \tfrac{\delta}{p^{q+1}}\right)}\times \langle \eta, \iota_{\Delta, \star}(\cE(\uPhi^p))\rangle_{X_{G, \Kl}^m}
    \end{align*}
    in the notation of \cref{prop:weprove}. However, we have
    
    \begin{tabular}{ll}
       $ \left\langle  \left( \alpha^{t_1,t_2,\Phi^{(p)}_1,\Phi^{(p)}_2}_1 , \alpha^{t_1,t_2,\Phi_1^{(p)},\Phi_2^{(p)}}_2\right),\,  \left(\iota^{[t_1,t_2]}_\Delta\right)^\star(\breve\xi,\breve\eta^{(2,m)}_{\coh})\right\rangle_{\coh-\fp,\cX^{2,m}_{H,\Delta}}$&\\
      \quad $=\left\langle \left( \alpha^{t_1,t_2,\Phi^{(p)}_1,\Phi^{(p)}_2}_1 , \alpha^{t_1,t_2,\Phi_1^{(p)},\Phi_2^{(p)}}_2,\,  \left(\iota^{[t_1,t_2]}_\Delta\right)^\star(\breve\zeta,\breve\eta^{(2,m)}_{\coh})\right)\right\rangle_{\coh-\fp,\cX^{2,m}_{H,\Delta}}$ & by Prop. \ref{prop:indepoflift}\\
       $\quad =  \left\langle  \widetilde\Eis^{t_1+2,m}_{\rigfp,\Psi_1}\sqcup \widetilde\Eis^{t_2+2,m}_{\rigfp,\Psi_2},\, (\iota^{[t_1,t_2]}_\Delta)^\star\left(\eta^{(2,m)}_{\rigfp}\right)\right\rangle_{\widetilde\rigfp,\cX^{2,m}_{H,\Delta}}$ & by \eqref{eq:pairingtoeval}\\
       \quad $= \left\langle \Eis^{[t_1,t_2],(m,n)}_{\rigsyn,\underline\Phi},\, (\iota_\Delta^{[t_1,t_2]})^\star( \eta_{\rigfp,-D}^m|_{Y^{2,m}_{\Kl}})\right\rangle_{\rigfp,Y^{2,m}_{H,\Delta}}$ & by Cor. \ref{cor:redtoGros}\\
       \quad $= \left\langle \Eis^{[t_1,t_2]}_{\lrigsyn,\underline\Phi},\, (\iota_\Delta^{[t_1,t_2]})^\star(\eta_{\lrigfp,-D})\right\rangle_{\lrigfp,Y_{\Delta}}$ & by Thm. \ref{thm:redtoord}\\
       \quad $ =  \left\langle \iota_\star^{[t_1,t_2]}( \Eis^{[t_1,t_2]}_{\syn,\underline\Phi}),\,  \eta_{\NNfp,-D}\right\rangle_{\NNfp,Y_{G,\Kl}}$ & by Prop. \ref{prop:lrigreduction}\\
       \quad $ = \left\langle  \left(\log \circ \pr_{\Pif^{\prime \vee}} \circ \iota^{[t_1,t_2]}_{\Delta, \star}\right)\left(\Eis^{[t_1, t_2]}_{\et, \underline\Phi}\right),\,\eta_{\dR}\right\rangle_{Y_{G, \Kl}}$ & by \eqref{eq:step1}.
   \end{tabular}

   So we have proved the equality of the two sides of \eqref{eq:goal2}; and Proposition \ref{prop:weprove} shows that this assertion is equivalent to \cref{thm:mainthm}.
  \end{proof}


\mychapter{Step 5: Deformation to critical values}


\section{Hida families}
 \label{sect:families}

 We will now change our focus slightly: rather than working with a single, fixed automorphic representation $\Pi$, we shall consider $p$-adic families of these objects. In order to avoid fiddly issues involving choices of test vectors at ramified primes, we shall suppose for simplicity that $\Pi$ has level 1 from here onwards (i.e.~that $\Pi_\ell$ is unramified for all finite primes $\ell$). Note that this implies that $r_1 - r_2$ is even, and that the central character $\chi_{\Pi}$ is trivial.

 \subsection{Families of Galois representations}

  \begin{notation}
   Let $\cW$ denote the $p$-adic weight space $\Hom(\Zp^\times, \mathbf{G}_{m, L}^{\mathrm{rig}})$ (the analytification of the formal scheme $\operatorname{Spf} \cO_L[[\Zp^\times]]$). For $\epsilon \in \{ \pm 1\}$ we write $\cW^\epsilon$ for the union of components classifying characters with $\kappa(-1) = \epsilon$, so $\cW = \cW^{+1} \sqcup \cW^{-1}$.
  \end{notation}

  \begin{definition}
   \label{def:family}
   Let $U$ be an affinoid disc in $\cW$ containing 0. By a \emph{Siegel-type Hida family} $\uPi$ over $U$ of tame level $1$ (passing through weight $(r_1, r_2)$), we shall mean the following data:
   \begin{itemize}
    \item for each $n \in U \cap \ZZ_{\ge 0}$, a cuspidal automorphic representation $\Pi(n)$ of $\GSp_4$ which is globally generic, cohomological at $\infty$ with coefficients in $V(r_1 + n, r_2 + n)$, and has level $1$;
    \item for each such $n$, an embedding of the coefficient field of $\Pi(n)$ into $L$, with respect to which $\Pi(n)$ is Siegel-ordinary at $p$;
    \item a collection of rigid-analytic functions $\mathbf{t}_{i, \ell} \in \cO(U)$, for $i = 1, 2$ and $\ell \ne p$, such that for each $n \in U \cap \ZZ_{\ge 0}$, the values of $\mathbf{t}_{1, \ell}$ and $\mathbf{t}_{2, \ell}$ at $n$ are the eigenvalues of the spherical Hecke operators $\diag(\ell, \ell, 1, 1)$ and $p^{-(r_2+n)} \diag(\ell^2, \ell, \ell, 1)$ on the arithmetic twist $\Pi'(n)$;
    \item rigid-analytic functions $\mathbf{u}_{i, p} \in \cO(U)$ for $i = 1, 2$, with $\mathbf{u}_{1, p}$ taking $p$-adic unit values, such that for all $n \in U \cap \ZZ_{\ge 0}$, we can write the Hecke parameters of $\Pi_p'(n)$ as $(\alpha_n, \beta_n, \gamma_n, \delta_n)$ with
    \[ \mathbf{u}_{1, p}(n) = \alpha_n, \qquad \mathbf{u}_{2, p}(n) = \frac{\beta_n + \gamma_n}{p^{(r_2 + 1+n)}}.\]
   \end{itemize}
  \end{definition}

  The following theorem, which is an instance of the main result of \cite{tilouineurban99}, is fundamental for our arguments:

  \begin{theorem}[Tilouine--Urban]
   \label{thm:tilouineurban}
   For any $\Pi$ which satisfies the conditions of \S \ref{sect:heckeparams} and is unramified and Siegel-ordinary at $p$, there exists a disc $U \subset \cW$ around 0, and an ordinary family of eigensystems $\uPi$ over $U$, such that $\Pi(0) = \Pi$.
  \end{theorem}
  
  \begin{remark}
   Note that Klingen-ordinarity is not needed for this theorem, nor for the constructions below, until Corollary \ref{cor:motivicequalsanalytic}. However, Siegel-ordinarity is fundamental here (whereas it plays no role in the main body of the paper).
  \end{remark}

  The computations of \emph{op.cit.}~also give rise to a natural $\cO(U)$-module with an action of Hecke operators, which interpolates the $\Pi'(n)$-eigenspace in Betti cohomology of level $G(\widehat{\ZZ})$ (with coefficients varying with $n$). One can equally work with \'etale cohomology, to obtain the following:

  \begin{theorem}
   In the situation of \cref{thm:tilouineurban}, after possibly shrinking $U$, there exists a free rank 4 $\cO(U)$-module $W_{\uPi}$, whose fibre at $n \in U \cap \ZZ_{\ge 0}$ is canonically isomorphic to the Galois representation $W_{\Pi(n)}$.
  \end{theorem}

  \begin{note}
   More precisely, the fibre at $n$ of $W_{\uPi}$ is canonically identified with the subspace of \'etale cohomology of level $G(\widehat{\ZZ}^{(p)}) \times \Sieg(p)$ on which the prime-to-$p$ Hecke operators act via the eigensystem of $\Pif'(n)$ and $U_{1, \Sieg}$ acts as $\alpha_n = \mathbf{u}_{1, p}(n)$. This, in turn, is canonically identified with the $\Pif'(n)$-eigenspace at prime-to-$p$ level via the map
   \[\Pr_{\alpha_n}^{\star}: H^3_{\mathrm{et}}\big(G(\widehat{\ZZ})\big)[\Pif']
   \xhookrightarrow{\ \ \ } H^3_{\mathrm{et}}\big(G(\widehat{\ZZ}^{(p)}) \times \Sieg(p)\big)[\Pif'] \xtwoheadrightarrow{p_{\alpha_n}} H^3_{\mathrm{et}}\big(G(\widehat{\ZZ}^{(p)}) \times \Sieg(p)\big)[\Pif', U_{1,\Sieg} = \alpha_n]
   \]
   where $p_{\alpha_n}$ denotes the Hecke operator $(1 - \frac{\beta_n}{U_{1, \Sieg}})(1 - \frac{\gamma_n}{U_{1, \Sieg}})(1 - \frac{\delta_n}{U_{1, \Sieg}})$.
   (Here we have written $H^3_{\et}(K)$ as a shorthand for $H^3_{\et}(Y_G(K)_{\QQbar}, \cV)$ where $\cV$ is the appropriate \'etale coefficient sheaf).
  \end{note}

 \subsection{Two-variable Euler system classes}
  \label{sect:twovarES}
  We now construct families of Euler system classes taking values in $W_{\uPi}^\star$. 
  \begin{notation}
   Write
   \[
    \cE_{\Sieg}(\Pi, q, r) \coloneqq  \left(1 - \tfrac{p^q}{\alpha}\right)\left(1 - \tfrac{\beta}{p^{1+q}}\right)\left(1 - \tfrac{\gamma}{p^{1 + q}}\right)\left(1 - \tfrac{\delta}{p^{1 + q}}\right) \left(1 - \tfrac{p^{(r_2 + r + 1)} }{\alpha}\right)\left(1 - \tfrac{\delta}{ p^{(r_2 + r + 2)}}\right),
   \]
   and similarly $\cE_{\Sieg}(\Pi(n), q, r)$ for $n \ge 0$ (with $r_2$ in the last two factors replaced by $r_2 + n$).
  \end{notation}

  \begin{theorem}
   \label{thm:Iweltinterp}
   Let $0 \le r \le r_1 - r_2$ be a given integer, and let $c_1, c_2 > 1$ be integers coprime to $6pN$. Then there exists a class
   \[ {}_{c_1, c_2} z^{[\uPi, r]}_{\Iw} \in H^1_{\Iw}(\QQ(\mu_{p^\infty}), W_{\uPi}^{\star}) \]
   with the following property: for each $(n, q)$ with $n \in U \cap \ZZ_{\ge 0}$ and $0 \le q \le r_2 + n$, we have
   \[  \mom_{n, q}\left( {}_{c_1, c_2} z^{[\uPi, r]}_{\Iw}\right) = C_{n, q} \cdot z^{[\Pi(n), q, r]}_{\can}, \]
   where $C_{n, q}$ denotes the quantity
   \[ \left(c_1^2 - c^{-(t_1+n)}\right)\left(c_2^2 - c_2^{-(t_2 + n)}\right) \frac{\cE_{\Sieg}(\Pi(n), q, r)}{(-2)^{q}} .\]
  \end{theorem}

  \begin{proof}
   It follows from the results of \cite{LSZ17} that there exists a cohomology class interpolating the projections of Lemma--Flach elements $\mathcal{LE}_{\et}(\uPhi \otimes \xi)$ to the $U_{1, \Sieg}'$-ordinary part of cohomology at level $K_G^p \times\Sieg(p)$, for any prime-to-$p$ level $K_G^p$. Here $\uPhi$ and $\xi$ are products of arbitrary test data away from $p$ with certain specific test data at $p$ determined by the construction.

   If we choose the prime-to-$p$ parts of $\xi$ and $\uPhi$ to be the spherical test data, then this interpolating class is invariant under the group $G(\widehat{\ZZ}^{(p)}) \times \Sieg(p)$, and its image under $\mom_{n, q}$ is given by
   \[ (\text{$c$-factor}) \cdot (-2)^{-q}\left(1 - \tfrac{p^q}{\alpha_n}\right) z^{[\Pi(n), q, r]}\left(w_{\sph}^{(p)} \times w_{p, \Sieg}, \uPhi_{\sph}^{(p)}\times \uPhi_{p, \Sieg}\right), \]
   where $w_{p, \Sieg}$ is the image of the spherical Whittaker vector of $\Pi_p(n)$ under $\Pr_{\alpha_n}^{\star}$, and $\uPhi_{p, \Sieg} = \ch( (p\Zp \times \Zp^\times)^2)$. The cohomology class in the above formula is the product of $z^{[\Pi(n), q, r]}_{\can}$ and a local zeta-integral $\widetilde{Z}_p\left(w_{p, \Sieg}, \uPhi_{p, \Sieg}\right)$, which is evaluated in \cite[Prop. 4.3.2]{LZ-equivar}. After rescaling the test data to remove the harmless factor of $\frac{1}{(p+1)^2}$, we obtain the formula stated.
  \end{proof}

 \subsection{Two-variable motivic p-adic $L$-functions}

  We recall the following description of the Galois representation $W_{\uPi}$. Let $\kappa_U : \Zp^\times \to \cO(U)^\times$ be the canonical character over $U$ (specialising to $x \mapsto x^n$ at each $n \in U \cap \ZZ$).

  \begin{theorem}[Urban]
   After possibly shrinking $U$, the module $W_{\uPi}$ has a 3-step increasing filtration stable under $G_{\Qp}$, with graded pieces of ranks $(1, 2, 1)$: we can write
   \[ 0 = \sF_0 W_{\uPi} \subset \sF_1 W_{\uPi} \subset \sF_3 W_{\uPi}\subset \sF_4 W_{\uPi} =W_{\uPi}\]
   in which $\sF_n$ is free of rank $n$ as an $\cO(U)$-module and is a direct summand of $W_{\uPi}$, and the subquotients
   \[ \sF_1 W_{\uPi}, \quad
   \frac{\sF_3 W_{\uPi}}{\sF_1 W_{\uPi}} \otimes \chi_{\mathrm{cyc}}^{\kappa_U}, \quad
   \frac{W_{\uPi}}{\sF_3 W_{\uPi}} \otimes \chi_{\mathrm{cyc}}^{2\kappa_U}\]
   are all crystalline as $\cO(U)$-linear representations.

   More precisely, the graded pieces have the following description:
   \begin{itemize}
    \item $\sF_1$ is unramified, with geometric Frobenius acting as multiplication by $\mathbf{u}_{1, p} \in \cO(U)^\times$.
    \item $(\sF_3 / \sF_1)(\chi_{\mathrm{cyc}}^{(\kappa_U + r_2 + 1})$ has constant Hodge--Tate weights $(0, -r_1 + r_2-1)$, and the trace of Frobenius on $\Dcris\left((\sF_3 / \sF_1)(\chi_{\mathrm{cyc}}^{(\kappa_U + r_2 + 1)})\right)$ is $\mathbf{u}_{2, p}$.
    \item $(W_{\uPi} / \sF_3)(\chi_{\mathrm{cyc}}^{(2\kappa_U + r_1 + r_2 + 3)})$ is unramified with geometric Frobenius acting as $\chi(p) \mathbf{u}_{1, p}^{-1}$.
   \end{itemize}
  \end{theorem}

  \begin{proof}
   The fact that such filtrations exist ``pointwise'', on the fibre at $n$  for each $n \in U \cap \ZZ_{\ge 0}$, is due to Urban \cite{urban05}. Since we know that the Galois representations interpolate over $U$, the existence of an $\cO(U)$-linear filtration follows from the finite generation of local Galois cohomology groups for $\cO(U)$-linear representations.
  \end{proof}

  Dually, we obtain a filtration on $W_{\uPi}^{\star}$ by setting $\sF^i$ to be the orthogonal complement of $\sF_i$.

  \begin{proposition}
   After possibly shrinking $U$, the projection of the Iwasawa cohomology class ${}_{c_1, c_2} z^{[\uPi, r]}_{\Iw}$ to $W_{\uPi}^{\star} /\sF^1 W_{\uPi}^{\star}$ is zero.
  \end{proposition}

  \begin{proof}
   This follows from the corresponding vanishing result in the fibre at a given $n \in U \cap \ZZ_{\ge 0}$, which is \cite[Proposition 11.2.2]{LPSZ1}.
  \end{proof}

  We can thus regard $\loc_p\left({}_{c_1, c_2} z^{[\uPi, r]}_{\Iw}\right)$ as an element of the module
  \[
   H^1_{\Iw}\left(\Qp(\mu_{p^\infty}), {\sF^1 W_{\uPi}^{\star}} / {\sF^3 W_{\uPi}^{\star}}\right)
   \cong
   H^1_{\Iw}\left(\Qp(\mu_{p^\infty}), \frac{\sF^1 W_{\uPi}^{\star}} {\sF^3 W_{\uPi}^{\star}}\otimes \chi_{\mathrm{cyc}}^{-(\kappa_U + r_2 + 1)}\right)
  \]
  where the isomorphism comes from the canonical twisting map (the twist is convenient because we land in a representation with constant Hodge--Tate weights, and also matches up better with our normalisation for analytic $p$-adic $L$-functions).
  Perrin-Riou's regulator $\cL^{\mathrm{PR}}$ gives a canonical map from this module to $\cH(\Zp^\times) \mathop{\hat\otimes} \mathcal{D}^\star = \cO(\cW) \mathop{\hat\otimes} \mathcal{D}^\star$, where
   \[
   \mathcal{D} \coloneqq \Dcris\left( (\sF_3 W_{\uPi} / \sF_1 W_{\uPi}) \otimes \chi_{\mathrm{cyc}}^{(\kappa_U + r_2 + 1)}\right).
  \]
  Let us now assume that the Hecke parameters of $\Pi = \Pi(0)$ satisfy $\beta \ne \gamma$. After possibly shrinking $U$ even further, we can arrange that $\beta_n \ne \gamma_n$ for every $n \in U \cap \ZZ_{\ge 0}$, and that there is a rank 1 direct summand $\cD_\beta$ of $\cD$, stable under $\varphi$, whose specialisation at any $n$ is canonically identified with the $\varphi = \beta_n$ eigenspace of $\Dcris(W_{\Pi(n)} / \sF_1)$.

  \begin{definition}
   \label{def:motivicL}
   Let $\unu$ be a basis of the free rank 1 $\cO(U)$-module $\cD_\beta$. We shall set
   \[ {}_{c_1, c_2} \cL^{\mot, r}_{p, \unu}(\uPi) \coloneqq \left\langle\unu_\beta, \cL^{\mathrm{PR}}\left({}_{c_1, c_2} z^{[\uPi, r]}_{\Iw}\right) \right\rangle \in \cO(U \times \cW),
   \]
   which we consider as a ``two-variable motivic $p$-adic $L$-function''.
  \end{definition}

  The dependence on $(c_1, c_2)$ is mild: the element of $\operatorname{Frac}\cO(U\times\cW)$ given by
  \[ \cL^{\mot, r}_{p,\unu}(\uPi) \coloneqq \frac{{}_{c_1, c_2} \cL^{\mot, r}_{p, \unu}(\uPi)}{\left(c_1^2 - c_1^{(\bfj + 1 - r')}\middle)\middle(c_2^2 - c_2^{(\bfj+ 1 - r)}\right)}\]
  is independent of $c_1, c_2$, where $\bfj$ is the canonical character $\Zp^\times \to \cO(W)^\times$ (which we think of as a ``coordinate'' on $\cW$) and $r' = r_1 - r_2 - r$. This can be seen as a meromorphic function on $U \times \cW$, with poles along the lines $\bfj= r + 1$ and $\bfj= r' + 1$.

  \begin{proposition}
   For $n \in U \cap \ZZ_{\ge 0}$, there exists a unique vector $\nu_\beta(n) \in \Fil^1 \Dcris(W_{\Pi(n)})$ whose image in $\Dcris(W_{\Pi(n)} / \sF_1)$ coincides with the specialisation of $\unu_\beta$ at $n$. This vector is annihilated by $(1 - \tfrac{\varphi}{\alpha_n})(1 - \tfrac{\varphi}{\beta_n})$.
  \end{proposition}

  \begin{proof}
   Since $\sF_1 W_{\Pi(n)}$ has Hodge--Tate weight 0, the subspace $\Dcris(\sF_1 W_{\Pi(n)})$ of $\Dcris(W_{\Pi(n)})$ (which is simply the $\varphi = \alpha_n$ eigenspace) has zero intersection with  $\Fil^1$. Since $\Fil^1$ is 3-dimensional, we conclude that it maps isomorphically to $\Dcris(W_{\Pi(n)} / \sF_1)$; so the image of $\unu_\beta$ in $\Dcris(W_{\Pi(n)} / \sF_1)$ has a unique lifting to $\Fil^1$. On the other hand, since the specialisation of $\unu_\beta$ is annihilated by $(1 - \tfrac{\varphi}{\beta_n})$, and $\sF^1$ is annihilated by $(1 - \tfrac{\varphi}{\alpha_n})$, we see that this lifting must be annihilated by the given quadratic polynomial.
  \end{proof}

  \begin{notation}
   We let $\Sigma_{\mathrm{crit}}$ and $\Sigma_{\mathrm{geom}}$ denote the subsets of $U \times \cW$ given by
   \[ \Sigma_{\mathrm{crit}} = \{ (n, j) : n \in U \cap \ZZ_{\ge 0}, j \in \ZZ, 0 \le j \le r_1 - r_2\}\]
   and
   \[ \Sigma_{\mathrm{geom}} = \{ (n, j) : n \in U \cap \ZZ_{\ge 0}, j \in \ZZ, -1-r_2 \le j \le -1\}.\]
  \end{notation}

  \begin{proposition}
   For any $(n, j) \in \Sigma_{\mathrm{geom}}$, the value of $\cL^{\mot, r}_{p,\unu}(\uPi)$ at $(n, j)$ is given by
   \[
    \cL^{\mot, r}_{p,\unu}(\uPi, n, j) = \frac{(-1)^{r_2 + n - q}}{(-2)^q (r_2 +n - q)!} \cdot \frac{\cE(\Pi(n), q) \cE(\Pi(n), 1+r_2 + r)}{\left(1 - \tfrac{p^{r_2 + r + n + 1}}{\beta_n}\middle)
     \middle(1 - \tfrac{\gamma_n}{p^{r_2 + r + n + 2}}\right)} \cdot
    \left\langle \nu_\beta(n), \log_{\mathrm{BK}}\left(z^{[\Pi(n), q, r]}_{\can}\right) \right\rangle,
   \]
   where $q = j + 1 + r_2 + n$.
  \end{proposition}

  \begin{proof}
   This follows from the interpolation formulae relating the Perrin-Riou regulator to the Bloch--Kato logarithm. These formulae include a twist by $(1 - p^j \varphi)(1 - p^{-1-j} \varphi^{-1})^{-1}$, with $\varphi$ acting as $p^{r_2+1+n} \beta_n^{-1}$; so we have
   \[
    \cL^{\mot, r}_{p, \unu}(\uPi, n, j) = \frac{(-1)^{r_2 +n - q}}{(r_2 +n - q)!} \cdot \frac{\left(1 - \tfrac{p^q}{\beta_n}\right)}{\left(1 - \frac{\beta_n}{p^{1+q}}\right)} \cdot \left\langle \nu_\beta(n), \log_{\mathrm{BK}}\left(\mom_{n, q} z^{[\uPi, r]}_{\Iw}(\Phi)\right) \right\rangle.
   \]
   Combining this with \cref{thm:Iweltinterp} gives the result.
  \end{proof}

  \begin{note}
   The parity constraint of \cref{eq:parity} implies $\cL^{\mot, r}_{p, \unu}(\uPi)$ is supported on $U \times \cW^{(-1)^{r + 1}}$.
  \end{note}

  \begin{proposition}
   There exists an element $\mathcal{E}_r(\uPi) \in \cO(U)$ whose value at $n \in U \cap \ZZ_{\ge 0}$ is
   \[ \left(1 - \tfrac{p^{r_2 + r + n + 1}}{\beta_n}\middle)
   \middle(1 - \tfrac{\gamma_n}{p^{r_2 + r + n + 2}}\right).\]
  \end{proposition}

  \begin{proof}
   Clear from the fact that $p^{-n} \beta_n$ and $p^{-n}\gamma_n$ are analytic functions on $U$.
  \end{proof}

  \begin{notation}
   \label{not:Lpboxr}
   We define $\cL^{\mot, [r]}_{p, \unu}(\uPi) \coloneqq \cE_r(\uPi)\cdot \cL^{\mot, r}_{p, \unu}(\uPi) \in \cO(U \times \cW^{(-1)^{r + 1}})$.
  \end{notation}

  \begin{remark}
   This is an cowardly definition. We should really have defined a 3-parameter or even 4-parameter family of zeta elements with both $q$ and $r$ varying, and shown directly that it recovered the above element after specialisation, with the Euler factor $\cE_r$ arising naturally from a comparison between elements at Siegel and Iwahori level. (See \cite[\S 7.1]{LRZ} for a construction along these lines.)
  \end{remark}

  We now reimpose the assumption that $\Pi$ be Klingen-ordinary, and we suppose that $\beta$ is the unique Hecke parameter of minimal possible valuation $r_2 + 1$, so that the conditions of \cref{thm:main} are satisfied. With the present notations, we can state \cref{thm:main} as follows:

  \begin{corollary}
   \label{cor:motivicequalsanalytic}
   For all $(n, j) \in \Sigma_{\mathrm{geom}}$, we have
   \[ \cL^{\mot, [r]}_{p, \unu}(\uPi, n, j) = \cL_{p, \unu(n)}(\Pi(n), j, r).\qed\]
  \end{corollary}

 \subsection{Conjectures on Eichler--Shimura isomorphisms}
  \label{sect:ESconj}

  \begin{conjecture}
   \label{conj:ESiso}
   Let $\uPi$ be a Siegel-type Hida family over $\cO(U)$ through $(r_1, r_2)$, which is also Borel-ordinary. Then:
   \begin{enumerate}[(A)]
    \item There exists a rank 1 free $\cO(U)$-module $H^1(\uPi)$, whose fibre at $n \in U \cap \ZZ_{\ge 0}$ is canonically identified with the direct summand of
    $H^1\left(Y_G(K_1(M, N) \cap \Iw(p)), \cN_n^2(-D)\right)[\Pif'(n)]$ which is ordinary for the Hecke operators
    \[ U_{2, \Iw} \coloneqq \left[\Iw(p) \diag(p^2, p, p, 1)\Iw(p)\right] \quad\text{and}\quad
    Z \coloneqq \left[\Iw(p) \diag(p, 1, p, 1) \Iw(p)\right].
    \]
    \item There exist a pushforward map sending families of $p$-adic modular forms for $H$ to elements of $H^1(\uPi)$, compatible via specialisation with the pushforward maps on classical modular forms.

    \item There is an isomorphism of $\cO(U)$-modules $\cD_{\beta}^\star \cong H^1(\uPi)$, interpolating the comparison isomorphisms of $p$-adic Hodge theory.
   \end{enumerate}
  \end{conjecture}

  A proof of part (A) of this conjecture has already been announced by Pilloni, and will appear in forthcoming work. Part (B), which is an analogue for Siegel-type families of the pushforwards constructed for Klingen-type families in \cite{LPSZ1}, should also be accessible.

  These two parts of the conjecture would suffice to define a 3-variable analytic $p$-adic $L$-function $\cL_{\underline{\mu}}(\uPi)$ on $U \times \cW \times \cW$, where $\underline{\mu}$ is any basis of $H^1(\uPi)^\star$, whose restriction to $\{n\} \times \cW \times \cW$ coincides with $\cL_{\underline{\mu}(n)}(\Pi(n))$ for each $n \in U \cap \ZZ_{\ge 0}$.

  If part (C) holds, then we can arrange that $\underline{\mu}$ is the image of $\unu$. Then Corollary \ref{cor:motivicequalsanalytic} would assert the equality of two analytic functions on $U \times \cW$ at every point $(n, q)$ in a Zariski-dense set; hence these functions would agree everywhere. Specialising to $n = 0$, we would then obtain the strongest possible form of an explicit reciprocity law, namely the following:

  \begin{conjecture}
   \label{conj:ERLstrong}
   We have the following equality of rigid-analytic functions of $\bfj \in \cW^{(-1)^{r+1}}$:
   \[ \cL^{\mot, [r]}_{p, \nu}(\Pi, \bfj) = \cL_{p, \nu}(\Pi, \bfj, r).\]
  \end{conjecture}

  \begin{remark}
   In the forthcoming work \cite{LZ21-erl}, we prove the analogues of parts (A), (B) of Conjecture \ref{conj:ESiso} for Coleman families rather than Hida families, and a partial result towards part (C), using the ``leading term argument'' introduced in the remaining sections of the present paper. However, since the families of Eisenstein series used in the definition of the $p$-adic $L$-function are not overconvergent, this does not immediately give a proof of Conjecture \ref{conj:ERLstrong}.
  \end{remark}

 \subsection{Comparison with a $\GL_4$ p-adic $L$-function}

  In order to work around our ignorance of Conjecture \ref{conj:ESiso}, we shall make use of the functorial transfer to $\GL_4$. This allows one to make use of a somewhat different toolset (based on Betti rather than coherent cohomology).

  \begin{notation}
   We write $\Theta$ for the functorial transfer of $\Pi \otimes \|\cdot\|^{-(r_1 - r_2 - 1)/2}$ to $\GL_4(\AA)$, so that $\Theta$ is a isobaric automorphic representation of $\GL_4$ satisfying
   \[ L(\Theta, s) = L(\Pi, \tfrac{1-r_1 + r_2}{2} + s).\]
  \end{notation}

  The choice of twist implies that the critical values of $L(\Theta, s)$ are at the integers $0 \le s \le r_1 - r_2$, matching our normalisation for $p$-adic $L$-functions. Note that since $\Pi$ is assumed to be non-CAP and non-endoscopic, the representation $\Theta$ is in fact cuspidal. The compatibility of local and global transfers at $\infty$ implies that $\Theta$ is cohomological (with infinity-type determined by $(r_1, r_2)$); and the compatibility at finite places implies that $\Theta$ has level 1, and is ordinary at $p$.

  \begin{definition}
   For each sign $\epsilon \in \{\pm 1\}$, we write $H^5_{B, c}(\Theta)_F^\epsilon$ for the eigenspace inside the compactly-supported Betti cohomology of the infinite-level symmetric space for $\GL_4$ (with coefficients in the local system of $E$-vector spaces determined by $(r_1, r_2)$) which is $\Theta_{\mathrm{f}}$-isotypical for the $\GL_4(\Af)$ action, and on which complex conjugation acts as $\epsilon$.
  \end{definition}

  It follows from the Eichler--Shimura--Matsushima isomorphism, together with strong multiplicity one for $\GL_4$, that each of the two spaces $H^5_{B, c}(\Theta)_E^\epsilon$ is isomorphic to a single copy of $\Theta_{\mathrm{f}}$. In particular, for each choice of $\epsilon$, the $\GL_4(\widehat{\ZZ})$-invariants of $H^5_{B, c}(\Theta)^\epsilon_E$ are one-dimensional. We denote this space of invariants by $W^\epsilon(\Theta)_E$, and its base-extension to $L$ by $W^\epsilon(\Theta)_L$.

  \begin{definition}
   We let $\tau = (\tau^+, \tau^-)$ be a pair of $L$-bases of the spaces $W^\epsilon(\Theta)_L$, for each choice of sign.
  \end{definition}

  Having chosen $\tau$, the construction of \cite{DJR20} shows that for each sign $\epsilon$ we can find constants $\Omega_p(\Theta, \tau^{\epsilon}) \in L^\times / E^\times$, and $\Omega_\infty(\Theta, \tau^{\epsilon}) \in \CC^\times / E^\times$, such that the following proposition holds:

  \begin{proposition}
   \label{def:GL4Lp}
   There exists a measure $\cL_{p, \tau}(\Theta) \in \Lambda_L(\Zp^\times)$ such that for all $0 \le a \le r_1-r_2$ we have
   \[ \frac{\cL_{p, \tau}(\Theta, a + \rho)}{\Omega_p(\Theta, \tau^\epsilon)} = R_p(\Theta, \rho, a) \cdot \frac{\Lambda(\Theta \otimes \rho, a)}{\Omega_\infty(\Theta, \tau^\epsilon)}\]
   where $\epsilon = (-1)^a \rho(-1)$, and $R_p(\Theta, \rho, a)$ is a product of Euler factors and Gauss sums at $p$.
  \end{proposition}

  \begin{remark}
   As with the $\GSp_4$ $p$-adic $L$-function defined above, the quantity $\Omega_p(\Theta, \tau^\epsilon)^{-1} \otimes \Omega_\infty(\Theta, \tau^\epsilon) \in L\otimes_E \CC$ is uniquely determined by $\tau$, although the individual factors are only determined modulo $E^\times$, so the measure $\cL_{p, \tau}(\Theta)$ depends only on $\tau$.
  \end{remark}

  By comparing the interpolating properties of the $p$-adic $L$-functions, we obtain the following:

  \begin{corollary}
   Suppose that $\cL_{p, \nu}(\Pi)$ is not identically 0 (which is automatic if $r_1 - r_2 > 0$). Then there is an isomorphism of $L$-vector spaces
   \[ t_{\Theta}: W^+(\Theta)_L \otimes W^-(\Theta)_L \cong \operatorname{Gr}^1 \DdR(V_{\Pi})\]
   with the following property: if $\nu$ is the image of $\tau^+ \otimes \tau^-$, then we have
   \[ \cL_{p, \nu}(\Pi)(\bfj_1, \bfj_2) = \cL_{p, \tau}(\Theta)(\bfj_1) \cdot \cL_{p, \tau}(\Theta)(\bfj_2)\]
   for all $(\bfj_1, \bfj_2) \in \cW \times \cW$ with $\bfj_1 + \bfj_2$ odd.
  \end{corollary}

  Note that this isomorphism matches up the $E$-structure $W^+(\Theta)_E \otimes W^-(\Theta)_E$ with the $E$-rational structure on the right-hand side determined by de Rham cohomology, although we shall not use this fact.

 \subsection{Variation in families for $\GL_4$}

  This discussion applies identically with $\Pi$ replaced by any of the other specialisations $\Pi(n)$ of Siegel-type family through $\Pi$ discussed above; and we have the following two statements, proved in the paper \cite{BDGJW25}:

  \begin{proposition}
   After possibly shrinking $U$, we can find free rank 1 $\cO(U)$ modules $W(\uTh)^{\epsilon}$ for each sign $\epsilon$, whose specialisation at $n \in U \cap \ZZ_{\ge 0}$ is canonically identified with $W(\Theta(n))_L^{\epsilon}$.
  \end{proposition}

  \begin{proof}
   Since $\Theta$ arises by base-change from an ordinary $\GSp_4$ representation, it has a unique ``spin $p$-refinement'' in the notation of \emph{op.cit.}~(which corresponds to the ordering $(\alpha, \dots, \delta)$ of the Hecke parameters). Theorem 13.6 of \emph{op.cit.}~shows that there exists a neighbourhood of $(r_1, r_2)$ in $\cW \times \cW$, and a 2-parameter family of $\GL_4$ automorphic representations over this neighbourhood deforming $\Theta$, all of which admit Shalika models and thus are functorial transfers from $\GSp_4$. Restricting to the line $r_1 - r_2 = \text{constant}$, we obtain a uniquely-determined one-parameter family $\underline{\Theta}$ of Shalika-type $\GL(4)$ representations deforming $\Theta$; and from the uniqueness of $\uTh$, it follows that the weight $n$ specialization of $\uTh$ coincides with $\Theta(n)$ (which we defined as the functorial lift of $\Pi(n)$). Theorem 13.6(ii) of \emph{op.cit.} then gives us rank 1 modules $W(\uTh)^{\epsilon}$ interpolating $W(\Theta(n))_L^{\epsilon}$. 
  \end{proof}
  
  \begin{theorem}
   \label{thm:gl4famL}
   Let $\utau = (\utau^+, \utau^-)$ be $\cO(U)$-bases of the modules $W(\uTh)^{\epsilon}$. Then there exists a bounded rigid-analytic function
   \[ \cL_{p, \utau}(\uTh): U \times \cW \to L \]
   with the following property: for every $n \in U \cap \ZZ_{\ge 0}$, the restriction of $\cL_{p, \utau}(\uTh)$ to $\{n \} \times \cW$ is $\cL_{p, \tau(n)}(\Theta(n))$, where $\tau(n)$ is the specialisation of $\utau$ at $n$.
  \end{theorem}
  
  \begin{proof}
   This is Theorem 13.8 of \emph{op.cit.} (restricted to the line $r_1 - r_2 = \text{const}$).
  \end{proof}

 \subsection{The reciprocity law}

  We now carry out a rather delicate comparison argument. We choose a $\utau$, giving us a 2-variable analytic $p$-adic $L$-function; and we choose a $\unu$ and a value of $r$, giving a 2-variable motivic one. For technical reasons we shall suppose that $r_1 -r_2 > 0$, and take $r \in \{0, \dots, r_1 - r_2\}$ such that
  \( r \ne \frac{r_1 - r_2}{2}. \)

  \begin{notation}
   Define $\cL^{[r]}_{p, \utau}(\uTh) \in \cO(U \times \cW)$ by
   \[ \cL^{[r]}_{p, \utau}(\uTh,\mathbf{u}, \bfj) =
   \begin{cases}
    \cL_{p, \utau}(\uTh, \mathbf{u}, r) \cdot\cL_{p, \utau}(\uTh, \mathbf{u}, \bfj) & \bfj \in \cW^{(-1)^{r+1}}, \\
    0 & \bfj \in \cW^{(-1)^r}.
   \end{cases} \]
  \end{notation}

  So Corollary \ref{cor:motivicequalsanalytic} tells us that for all $(n, j) \in \Sigma_{\mathrm{geom}}$, we have
  \begin{equation}
  \label{eq:defBn}
  \cL^{\mathrm{mot}, [r]}_{p, \unu}(\uPi, n, j) = B(n) \cdot \cL^{[r]}_{p, \utau}(\uTh, n, j),
  \end{equation}
  where $B(n) \in L^\times$ is the constant such that
  \[
  B(n) t_{\Theta(n)}\left( \utau(n)^+ \otimes \utau(n)^-\right) =   \unu(n).
  \]

  \begin{lemma}
   The function on $U \times \cW\times\cW$ defined by
   \[ C(\mathbf{u}, \bfj, \bfj') \coloneqq
   \cL^{[r]}_{p, \utau}(\uTh, \mathbf{u}, \bfj) \cdot \cL^{\mathrm{mot}, [r]}_{p, \unu}(\uPi, \mathbf{u}, \bfj')- \cL^{[r]}_{p, \utau}(\uTh, \mathbf{u}, \bfj')\cdot \cL^{\mathrm{mot}, [r]}_{p, \unu}(\uPi, \mathbf{u}, \bfj).
   \]
   is identically zero.
  \end{lemma}

  \begin{proof}
   From Corollary \ref{cor:motivicequalsanalytic}, we know that $C(\mathbf{u}, \bfj, \bfj')$ vanishes at all triples $(n, j, j')$ such that both $(n, j)$ and $(n, j')$ are in $\Sigma_{\mathrm{geom}}$. Such triples are clearly Zariski-dense, so the result follows.
  \end{proof}

  \begin{proposition}
   \label{prop:existD}
   There is a non-zero meromorphic function $D \in \operatorname{Frac} \cO(U)$ (independent of the $\cW$ variable) such that we have
   \[ \cL^{\mathrm{mot}, [r]}_{p, \unu}(\uPi, \mathbf{u}, \bfj) = D(\mathbf{u}) \cdot \cL^{[r]}_{p, \utau}(\uTh,\mathbf{u}, \bfj). \]
   Moreover, $D$ has no pole at any $n \in U \cap \ZZ_{\ge 0}$.
  \end{proposition}

  \begin{proof}
   Let $s \in \{0, \dots, r_1 - r_2\}$ with $s \ne \tfrac{r_1 - r_2}{2}$, and let $\rho$ be a finite-order character of $\Zp^\times$, such that $(-1)^s \rho(-1) \ne (-1)^r$. (If $r_1 -r_2$ is $\ge 4$ then we can assume $\rho$ is trivial.) We shall substitute $\bfj' =  s + \rho$ into the identity $C(\mathbf{u}, \bfj, \bfj') = 0$. Unravelling the notations, we find that
   \[\cL^{[r]}_{p, \utau}(\uTh,\mathbf{u}, s + \rho) =
   \cL_{p, \utau}(\uTh, \mathbf{u}, r) \cL_{p, \utau}(\uTh, \mathbf{u}, s + \rho).\]
   Both factors on the right-hand side are non-vanishing at $\mathbf{u} = n$ for any $n \in U \cap \ZZ_{\ge 0}$, since they correspond to non-central critical values of the complex $L$-function, which are non-zero by the convergence of the Euler product. So this function is a non-zero-divisor in $\cO(U)$; and dividing the identity $C(\mathbf{u}, \bfj, s+\rho) = 0$ by this function, we obtain
   \[
     \cL^{\mathrm{mot}, [r]}_{p, \unu}(\uPi, \mathbf{u}, \bfj) = D(\mathbf{u})
      \cdot \cL^{[r]}_{p, \utau}(\uTh, \mathbf{u}, \bfj),\qquad D(\mathbf{u})\coloneqq \frac{\cL^{\mathrm{mot}, [r]}_{p, \unu}(\uPi, \mathbf{u}, s + \rho)}{\cL_{p, \utau}(\uTh, \mathbf{u}, r) \cL_{p, \utau}(\uTh, \mathbf{u}, s + \rho)}.\qedhere
   \]
  \end{proof}

  \begin{proposition}
   \label{prop:goodpair}
   For all but finitely many integers $n \in U \cap \ZZ_{\ge 0}$, the following holds: there exists an integer $j$ with $j = r+1 \bmod 2$ such that $(n, j) \in \Sigma_{\mathrm{geom}}$ and $\cL_{p, \tau(n)}(\Theta(n),j) \ne 0$.
  \end{proposition}

  \begin{proof}
   Assume the contrary. Then there exists an infinite sequence of integers $n_k \in U \cap \ZZ_{\ge 0}$ such that the function $\cL_{p,\utau}(\uTh)$ vanishes at $(n_k, j)$ for all $j$ such that $j = r + 1 \bmod 2$ and $(n, j) \in \Sigma_{\mathrm{geom}}$. In particular, if we fix a $j \le -1$ congruent to $r+1 \bmod 2$, then $\cL_{p,\utau}(\uTh)$ vanishes at $(n_k, j)$ for all sufficiently large $k$, and since the sequence $(n_k)$ is Zariski-dense in $U$, it follows that $\cL_{p,\utau}(\uTh, u, j)$ vanishes for all $u \in U$. Since this holds for all $j \le -1$ of the appropriate parity, we conclude that $\cL_{p,\utau}(\uTh)$ has to be identically 0 on $U \times \cW^{(-1)^{1+r}}$. This is a contradiction, since its values at $(n, j+\rho)$ with $0 \le j \le r_1 - r_2$ and $\rho$ a finite-order character are critical values of the complex $L$-function multiplied by explicit non-zero factors, and if $j \ne \tfrac{r_1-r_2}{2}$ these values are not central or near-central, so they are non-zero by the convergence of the Euler product.\footnote{Note that since $\Pi$ has tame level 1, $r_1 -r_2$ must be even, and since we have assumed it is not zero, it is $\ge 2$. If we allow general tame levels, then this argument becomes more delicate in the case $r_1 - r_2 = 1$: we need to invoke the non-vanishing of $\GL_4$ $L$-functions along the abcissa of convergence (the ``prime number theorem'' for $\GL_4$ $L$-functions) due to Jacquet and Shalika.}
  \end{proof}

  \begin{corollary}
   For any $n \in U \cap \ZZ_{\ge 0}$, one of the following two possibilities occurs:
   \begin{itemize}
    \item $\cL^{\mot, [r]}_{p, \unu(n)}(\Pi(n))$ is a non-zero scalar multiple of the analytic $p$-adic $L$-function $\cL_{\unu(n)}(\Pi(n),-, r)$.
    \item $\cL^{\mot, [r]}_{p, \unu(n)}(\Pi(n))$ is identically 0.
   \end{itemize}
   Moreover, for all but finitely many $n$, the first possibility occurs and the scalar multiple is the constant $B(n)$ of \cref{eq:defBn}, so we have
   \[ \cL_{p, \nu(n)}^{\mot, [r]}(\Pi(n), \bfj) = \cL_{p, \nu(n)}(\Pi(n), \bfj, r) \]
   as an identity of rigid-analytic functions of $\bfj \in \cW^{(-1)^{r+1}}$.
  \end{corollary}

  \begin{proof}
   Since the function $D$ of \cref{prop:existD} is finite at any positive integer $n$, it must either be zero there, or an element of $L^\times$, and the result of the proposition gives the two cases stated. However, if $n$ satisfies the condition of \cref{prop:goodpair}, then \cref{eq:defBn} shows that $D(n)$ must equal the constant $B(n)$, and in particular is non-zero; and by that proposition we know that this case occurs for all but finitely many $n$.
  \end{proof}

  \begin{remark}
   There are two ``bad'' cases which could possibly occur for some $n$: either $D(n) = 0$, in which case the motivic $p$-adic $L$-function of $\Pi(n)$ vanishes identically; or $D(n) \ne 0$ but $B(n)\ne D(n)$, in which case the motivic $p$-adic $L$-function is still a non-zero multiple of the analytic one, but the ``wrong'' multiple. The first case is disastrous for applications, while the second is only a minor irritant. However, since both cases occur for only finitely many $n$, we can shrink $U$ to assume that neither case occurs \textbf{except possibly for $n = 0$}.
  \end{remark}

  We have so far been quite agnostic about the value of $r$; we assumed only that it was non-central. We now consider varying $r$. Note that the meromorphic function $D(\mathbf{u})$ must be independent of $r$, since the constants $B(n)$ \emph{are} independent of $r$. So we may conclude that the function
  \[
   \frac{\cL_{p,\unu}^{\mot, [r]}(\uPi)(\mathbf{u}, \bfj)}{\cL_{p, \utau}(\uTh, \mathbf{u},r)}
  \]
  is also independent of $r$, being equal to $D(\mathbf{u})\cdot  \cL_{p, \utau}(\uTh, \mathbf{u}, \bfj)$.

 \subsection{Proof of Theorem B}

  We note the following theorem:

  \begin{theorem}
   There exists a collection of classes
   \[
   {}_{c_1, c_2} z^{[\uPi, r]}_{\Iw, M} \in H^1(\QQ(\mu_{Mp^\infty}), W_{\uPi}^\star)
   \]
   for every $M \ge 1$ coprime to $pc_1 c_2$, satisfying the Euler system norm compatibility relations as $M$ varies, with the $M = 1$ case being the class ${}_{c_1, c_2} z^{[\uPi, r]}_{\Iw}$ above.
  \end{theorem}

  \begin{proof}
   This follows from the results of \cite{LSZ17} in the same way as the $M = 1$ case covered in \cref{thm:Iweltinterp}.
  \end{proof}

  \begin{notation}
   We let $c_M$ be the image of ${}_{c_1, c_2} z^{[\uPi, r]}_{\Iw, M}$ under the Soul\'e twist map
   \[ H^1(\QQ(\mu_{Mp^\infty}), W_{\uPi}^\star) \to H^1(\QQ(\mu_{Mp^\infty}), W_{\uPi}^\star(-1-r_2-\kappa_U)).\]
  \end{notation}

  The following result follows easily from the integrality of the original Lemma--Flach classes:

  \begin{lemma}
   If $\cO^+(U)$ is the subring of functions of supremum norm $\le 1$ in $\cO(U)$, then there exists a $G_{\QQ}$-stable $\cO^+(U)$-lattice $\cT \subseteq W_{\uPi}^\star(-1-r_2-\kappa_U)$ independent of $M$ such that all these classes take values in $H^1(\QQ(\mu_{Mp^\infty}),\cT)$.
  \end{lemma}

  If $D(0) \ne 0$, then it is a small step from here to Theorem B. The chief difficulty is that we cannot rule out the possibility of $D(0)$ vanishing, so we shall perform a delicate argument with ``leading terms''.

  \begin{notation}
   Let $u$ denote a generator of the principal ideal of $\cO^+(U)$ corresponding to the point $0 \in U$.
  \end{notation}

  \begin{definition}
   For $M \ge 0$, let $h(M)$ be the largest integer $n$ such that
   \[  c_M \in u^n \cdot H^1(\QQ(\mu_{Mp^\infty}), \cT), \]
   and let $h = \inf_{M} h(M)$, where the infimum is over $M \ge 1$ coprime to $pc_1c_2$.
  \end{definition}

  The Euler system norm-compatibilities imply that $h(M) \le h(1)$ for all $M$, and $h(1)$ is finite, since ${}_{c_1, c_2} z^{[\uPi, r]}_{\Iw, M}$ is not zero. From \cref{prop:existD}, we have
  \[ h \le h(1) \le v_u\left(D\right) \]
  where $v_u$ denotes the $u$-adic valuation on $\cO(U)$.

  \begin{proposition}
   There exists a collection of classes $c_M^{(h)} \in H^1(\QQ(\mu_{Mp^\infty}), \cT)$ satisfying the Euler-system norm relations, such that we have
   \[ c_M = u^h \cdot c_M^{(h)} \]
   for all $M$. Moreover, there is some $M$ such that $c_M$ has non-zero image in $H^1(\QQ(\mu_{Mp^\infty}), T)$, where $T$ denotes the lattice $\cT/u\cT \subset V_{\Pi}^\star$.
  \end{proposition}

  \begin{proof}
   Let us write temporarily $\cM = H^1_{\Iw}(\QQ(\mu_{Mp^\infty}), \cT)$ for some $M$. We note that $\cM/ u^h \cM$ injects into $H^1_{\Iw}(\QQ(\mu_{Mp^\infty}), \cT/u^h \cT)$, which is the Iwasawa cohomology of a finite-rank free $\Zp$-linear representation and is therefore $p$-torsion-free. Thus the fact that ${}_{c_1, c_2} z^{[\uPi, r]}_{\Iw, M}$ is divisible by $u^h$ in $\cM[1/p]$ implies that it is in fact divisible by $u^h$ in $\cM$. Moreover, it is even \emph{uniquely} divisible by $u^h$, since the $u^h$-torsion of $\cM$ is a subquotient of $H^0_{\Iw}(\QQ(\mu_{Mp^\infty}), \cT/u^h\cT)$ which is zero by standard properties of Iwasawa cohomology. Hence $c_M^{(h)}$ is well-defined. Since multiplication by $u^h$ is injective, and the $c_M$ for varying $M$ satisfy the Euler-system norm relations, so do the $c_M^{(h)}$.

   This argument also shows that $c_M^{(h)}$ has non-zero image in $H^1(\QQ(\mu_{Mp^\infty}), T)$ if and only if $h(M) = h$. Since this does occur for some $M$ by the definition of $h$, the final claim follows.
  \end{proof}

  \begin{proposition}
   Assume that $h < v_u(D)$. Then we have
   \[ \loc_p(c_M^{(h)} \bmod u) \in H^1_{\Iw}(\QQ(\mu_{M p^\infty}) \otimes \Qp, \Fil^2 T)\]
   for all $M$.
  \end{proposition}

  \begin{proof}
   It suffices to show that for every $M$, the class $c_M^{(h)} \bmod u$ lies in the kernel of the Perrin-Riou regulator map for $\Fil^1 T / \Fil^2$, since the kernel of this map is zero by \cref{lem:trivzero}.

   Repeating the construction of the previous sections with the additional tame level $M$, we obtain an ``equivariant'' motivic $p$-adic $L$-function $\cL^{\mathrm{mot}, [r]}_{p, \unu}(\uPi, M)$ over $U \times \cW$, taking values in the group ring of $(\ZZ / M\ZZ)^\times$. For each character $\chi$ of $(\ZZ / M\ZZ)^\times$, the $\chi$-isotypical projection of this object interpolates values of the $L$-function of the twisted representation $\Pi(n) \otimes \chi$ in the geometric range $\Sigma_{\mathrm{geom}}$.

   On the other hand, the $\GL_4$ construction extends straightforwardly to an equivariant version of the analytic $p$-adic $L$-function, $\cL^{[r]}_{p, \utau}(\uTh, M)$. Both of these objects depend on the same choices of periods $\unu$, $\utau$ as the non-equivariant $L$-functions of the previous section.

   Hence we can run the argument of Proposition \ref{prop:existD} to obtain a relation between the motivic and analytic equivariant $p$-adic $L$-functions; and the function $D(\mathbf{u})$ that appears must be the same for all $M$, since it is characterised by agreeing with the numbers $B(n)$ of \eqref{eq:defBn} for almost all $n$, and these numbers are independent of $M$.

   From this and the definition of $c_M^{(h)}$, we have
   \[
    \langle \nu_\beta, \cL^{\mathrm{PR}}(c_M^{(h)} \bmod u)\rangle = \left( (u^{-h} D)(0)\right)\cdot \cL^{[r]}_{p, \utau}(\uTh, M).
   \]
   So if $(u^{-h} D)(0) = 0$, we can conclude that $c_M^{(h)} \bmod u$ lies in the kernel of the regulator for all $M$ as required.
  \end{proof}

  \begin{corollary}
   If the ``big image'' assumption $\operatorname{Hyp}(\QQ(\mu_{p^\infty}), -)$ of \cite{rubin00} is satisfied for every Dirichlet-character twist of $T$, then we have $h = v_u(D)$.
  \end{corollary}

  \begin{proof}
   In the book \cite{mazurrubin04}, the authors define a notion of \emph{Euler characteristic} associated to a Galois representation and a collection of local conditions, and show that if the Euler characteristic is 0 (and the big-image condition holds), then no nonzero Kolyvagin systems exist.

   In our setting, one computes easily that the Euler characteristic of the Greenberg-type local condition at $p$ defined by $\Fil^2 T$ (with the usual unramified local conditions at all other primes) is 0. However, since $c_M^{(h)}$ is non-zero for some $M$, its projection to some character component $\chi$ must also be non-zero, so it gives a non-zero Kolyvagin system for $T(\chi)$, contradicting Mazur and Rubin's result.
  \end{proof}

  \begin{theorem}[{\cref{thmB}}]
   \label{thm:explicitrecip}
   Let $\Pi$ be an automorphic representation which satisfies our running hypotheses, and has tame level 1, is Borel-ordinary at $p$, and satisfies the ``big image'' condition of \cite[Assumption 11.1.2]{LSZ17}. Suppose also that $r_1 - r_2 \ge 6$.

   Then for any choice of basis $\tau = (\tau^+, \tau^-)$ as above, there exists an Euler system for $V_{\Pi}^\star(-1-r_2)$ with the following property: for all $M$, the localisation of the class at $p$ lands in $\Fil^1$; and the image of the bottom class in this Euler system under the Perrin-Riou regulator is $\cL_{p, \tau}(\Theta)$.
  \end{theorem}

  \begin{proof}
   The above argument shows that for each $r$ we can construct an Euler system whose regulator is $(c_1^2 - c_1^{\bfj + 1-r'})(c_2^2 - c_2^{\bfj+1-r}) \cL_{p,\tau}(\Theta, r) \cL_{p, \tau}(\Theta, \bfj)$ on $\cW^{(-1)^{r+1}}$, and 0 on $\cW^{(-1)^{r}}$.

   Over the $-1$ component of weight space, we note that the factors
   \[ (c_1^2 - c_1^{\bfj + 1-r'})(c_2^2 - c_2^{\bfj+1-r}) \cL_{p,\tau}(\Theta, r)\]
   for $r = 0$ and $r = 2$ between them generate the unit ideal of $\cO(\cW^{(-1)})$, so we can take a suitable linear combination to obtain an Euler system with the desired regulator $\cL_{p, \tau}(\Theta, \bfj)$. Similarly, over the other sign component, we use $r=1$ and $r = 3$, unless $r_1 - r_2 = 6$, in which case we can use $r=1$ and $r=5$.
  \end{proof}

\section{Applications}

 Throughout this section, we let $\Pi$ be a non-endoscopic, non-CAP automorphic representation of $G(\Af)$ of weights $(r_1+3,r_2+3)$ with $r_2 \ge 1$ and $r_1-r_2\ge 6$. Assume that $\Pi$ has tame level $1$, and that it is Borel ordinary at $p$.

 \subsection{Selmer groups over $\QQ_{\infty}$}

  Let $\QQ_\infty = \QQ(\mu_{p^\infty})$. For simplicity we write $V = V_{\Pi}^\star(-1-r_2)$ in this section. (Note that this conflicts with our earlier use of $V$ for an algebraic $G$-representation, but that usage will not recur here.)

  \begin{definition}
   Let $\RGt_{\Iw}(\QQ_\infty, V)$ denote the \Nek Selmer complex, with the unramified local conditions at $\ell \ne p$, and at $p$ the Greenberg-type local condition determined by $\Fil^2 V_{\Pi}^\star$.
  \end{definition}

  This is a perfect complex of $\Lambda_L(\Zp^\times)$-modules. Its cohomology groups are zero for $i \notin \{1, 2\}$, and we have
  \[
   \wH^1_{\Iw}(\QQ_\infty, V) = \ker\Big( H^1_{\Iw}(\QQ_\infty, V) \to H^1_{\Iw}\left(\QQ_{p, \infty}, V / \Fil^2\right)\Big).
  \]
  The degree 2 cohomology is related to classical $p$-torsion Selmer groups via Pontryagin duality:

  \begin{proposition}
   \label{prop:nekovarduality}
   If $T$ denotes a choice of lattice in $V$, and $(-)^\vee$ denotes Pontryagin dual, then we have a canonical isomorphism of $\Lambda_L(\Zp^\times)$-modules
   \[ \wH^2_{\Iw}(\QQ_\infty, V) = \left(\varprojlim_n H^1_{\mathrm{f}}(\QQ(\mu_{p^n}), T^\vee(1+j))\right)^\vee(j) \otimes L,\]
   for any integer $0 \le j \le r_1 - r_2$.
  \end{proposition}

  We can now state our main theorem in Iwasawa-theoretic form:

  \begin{theorem}
   The module $\wH^2_{\Iw}(\QQ_\infty, V)$ is torsion over $\Lambda_L(\Zp^\times)$ and its characteristic ideal divides the $p$-adic $L$-function $\cL_{p, \tau}(\Theta)$. Moreover, we have $ \wH^1_{\Iw}(\QQ_\infty, V) = 0$.
  \end{theorem}

  \begin{proof}
   This is proved in Theorem 11.3.2 of \cite{LSZ17} with the motivic $p$-adic $L$-function (for some specific choice of $r$) in place of $\cL_{p, \tau}(\Theta)$. Applying the same argument with the Euler system emerging from \cref{thm:explicitrecip} we obtain the result stated.
  \end{proof}

 \subsection{Selmer groups over $\QQ$}

  By a standard descent argument (using the fact that no exceptional-zero phenomena arise because of \cref{lem:trivzero}), we deduce the following:

  \begin{theorem}
   Let $0 \le j \le r_1 - r_2$, and let $\rho$ be a finite-order character of $\Zp^\times$. If $L(\Pi \otimes \rho, \tfrac{1-r_1+r_2}{2} + j) \ne 0$, then $H^1_{\mathrm{f}}(\QQ, V(-j-\rho)) = 0$.
  \end{theorem}

  This establishes the analytic rank 0 case of the Bloch--Kato conjecture for all critical values of the $L$-function of $\Pi$.

  \begin{note}
   The hypothesis $L(\Pi \otimes \rho, \tfrac{1-r_1+r_2}{2} + j) \ne 0$ is automatic if $j \ne \tfrac{r_1 - r_2}{2}$.
  \end{note}

\mychapter{Index of notation}

\subsection*{Variants of \texorpdfstring{$\eta$}{eta}}
\label{sect:variantsofeta}
 \vspace{1ex}

 \begin{center}
\tabulinesep=1mm

  $\begin{tabu}{c|c|c}
   \text{Notation} & \text{Cohomology group} & \text{Definition} \\
   \tabucline\\
   \eta_{\dR} & \Fil^1\mathbf{D}_{\dR}(V_\Pi) &  \text{\S \ref{sect:reduction1}}\\
   \eta_{\NNfp,-D} &\qquad  H^3_{\NNfp,c}(Y_{\Kl},\cV,1+q,\cQ)\qquad \qquad& \text{\S \ref{sect:reduction1}}\\
   \eta_{\lrig,-D} & \qquad H^3_{\dR}(\cX_{\Kl}\langle -\cD\rangle,\cV) & \text{\S \ref{sect:rigidfromPi}}\\
   \eta^m_{\rig,-D}&\qquad H^3_{\dR,c}(\cX^m_{\Kl}\langle -\cD\rangle, \cV) &  \text{\S \ref{sect:rigidfromPi}}\\
   \eta_{\lrigfp,-D} & H^3_{\lrigfp}(X_{\Kl}\langle -D\rangle,\cV,1+q;\cQ)& \text{\S \ref{sect:logrigfromPi}}\\
   \breve\eta_{\lrigfp,-D} & H^3_{\lrigfp}(\breve{X}_{\Kl}\langle -D\rangle,\cV,1+q;\cQ)& \text{\S \ref{sect:logrigfromPi}}\\
   \eta^{m}_{\rigfp,-D} & H^3_{\rigfp}(X^m_{\Kl}\langle -D\rangle,\cV,1+q;\cQ) & \text{\S \ref{sect:logrigfromPi}}\\
   \breve\eta^{m}_{\rigfp,-D} & H^3_{\rigfp}(\breve{X}^m_{\Kl}\langle -D\rangle,\cV,1+q,\cQ) & \text{\S \ref{sect:logrigfromPi}}\\
   \eta^{\alg}_{-D} & H^2(X_{\Kl},\cN^1(-D))& \text{\S \ref{sect:coherentfromPi}}\\
   \eta_{\coh,-D}^{m} & H^2_c(\cX_{\Kl}^m,\cN^1(-D)) & \text{\S \ref{sect:coherentfromPi}}\\
   \eta_{\coh,-D}^{(2,m)} & H^2_{c0}(\cX_{\Kl}^{(2,m)},\cN^1(-D)) & \text{\S \ref{sect:coherentfromPi}}\\
   \eta_{\coh}^{(2,m)} & H^2_{c0}(\cX_{\Kl}^{(2,m)},\cN^1) & \text{\S \ref{sect:coherentfromPi}}\\
   \tilde\eta^{m}_{\rig,-D} & \wH^3_{\dR,c}(\cX^{m}_{\Kl}\langle -D\rangle,\cV,1+q) & \text{\S \ref{sect:ordinaryparts}}\\
   \tilde\eta^{m}_{\rigfp,-D} & \wH^3_{\rigfp,c}(\cX^{m}_{\Kl}\langle -D\rangle,\cV,1+q;\cQ) & \text{\S \ref{sect:liftingfp}}\\
   \tilde\eta^{(2,m)}_{\rigfp} & \wH^3_{\rigfp,c0}(\cX^{(2,m)}_{\Kl},\cV,1+q;\cQ) & \text{\S \ref{ssec:coherentfppaireta}}\\
   \breve\eta_{\coh,-D}^{m} & H^2_c(\cX_{\Kl}^{m},\sFil^q\cV\otimes \Omega^1_G(-D)) & \text{\S \ref{sect:dRsheaves}}\\
   \breve\eta_{\coh}^{(2,m)} & H^2_c(\cX_{\Kl}^{(2,m)},\sFil^q\cV\otimes \Omega^1_G) & \text{\S \ref{sect:dRsheaves}}\\
  \end{tabu}$
\end{center}

\subsection*{P-adic L-functions}
\vspace{1ex}

\begin{center}
 \tabulinesep=1mm

 $\begin{tabu}{c|c|c}
 \text{Function} & \text{Domain} & \text{Defined in} \\
 \tabucline\\
 \cL_{p, \nu}(\Pi) & \cW \times \cW & \text{\cref{thm:padicLfcn}} \\
 \sL_1, \sL_2 & \cW & \text{\cref{def:sLi}}\\
 {}_{c_1, c_2} \cL^{\mot, r}_{p, \underline{\nu}}(\uPi), \cL^{\mot, r}_{p, \underline{\nu}}(\uPi)
 & U \times \cW & \text{\cref{def:motivicL}} \\
 \cL^{\mot, [r]}_{p, \underline{\nu}}(\uPi)& U \times \cW & \text{\cref{not:Lpboxr}}\\
 \cL_{\underline{\mu}}(\uPi) \text{ (conjectural)}& U \times \cW \times \cW  & \text{\cref{sect:ESconj}}\\
 \cL_{p, \tau}(\Theta) & \cW & \text{\cref{def:GL4Lp}}\\
 \cL_{p, \utau}(\uTh) & U \times \cW & \text{\cref{thm:gl4famL}}
 \end{tabu}$
\end{center}

\mychapter{Appendix: Rigid cohomology of EKOR strata}

\label{chap:appendix}

\section{Introduction}

 The purpose of this appendix is to study the rigid cohomology of certain EKOR strata in a $\GSp_4$ Shimura variety (of Klingen-parahoric level at $p$), in order to supply a technical result which is an input in the study of reciprocity laws for the $\GSp_4$ Euler system in the main text. We will state our result more precisely below, but a rough outline is as follows.

 \subsection{Setting: de Rham cohomology}

  Let $Y_K$ be the $\GSp_4$ Shimura variety of some (sufficiently small) level $K$. Then, to each algebraic representation $V$ of $\GSp_4$, we can associate a vector bundle with connection $(\cV, \nabla)$ on $Y_K$. If $X_K$ denotes a smooth projective toroidal compactification of $Y_K$, then $(\cV, \nabla)$ extends to a connection with logarithmic singularities along the boundary divisor $D_K = X_K - Y_K$; and hence we can define two natural complexes of sheaves on $X_K$: the logarithmic de Rham complex $\DR^\bullet(V) = \mathcal{V} \otimes \Omega^\bullet_{X_K}(\log D_K)$, and its ``cuspidal'' variant $\DR^\bullet_c(V) = \DR^\bullet(V) \otimes \cO_{X_K}(-D_K)$. The hypercohomology of these complexes computes the de Rham cohomology, with and without compact supports, of $Y_K$ with coefficients in $V$.

  There is a natural map of complexes $\DR^\bullet_c(V) \into \DR^\bullet(V)$ and hence a natural map
  \[ H^*(X_K, \DR^\bullet_c(V)) \to H^*(X_K, \DR^\bullet(V)). \tag{\dag} \]
  Moreover, these cohomology groups have a natural action of Hecke operators, and the map is Hecke-equivariant.

  We are interested in the \emph{localisation} of this map at a Hecke eigenvalue system (for the Hecke algebra at unramified primes) associated to a cuspidal automorphic representation $\pi$ of $G$; we claim that if $\pi$ is not of CAP type, then the localisation of $(\dag)$ is a isomorphism (in all degrees). This can be shown by first base-extending to $\CC$, and then comparing the de Rham cohomology with Betti cohomology of $Y_K(\CC)$ (with coefficients in the local system corresonding to $V$). This can be computed using the Borel--Serre compactification of $Y_K(\CC)$, which is only a manifold-with-corners rather than an algebraic variety, but has better Hecke-equivariance properties than the toroidal. Using the stratification of the Borel--Serre boundary in terms of parabolic subgroups of $G$, one can show that all of the Hecke eigenvalue systems appearing in the boundary are parabolically induced from automorphic representations of proper Levi subgroups of $G$; so the localisation of the boundary cohomology at a cuspidal, non-CAP representation is 0.

 \subsection{Goal: a rigid-analytic variant}

  We now consider the de Rham cohomology of the rigid-analytic dagger space $X_K^{\mathrm{an}}$ over $\Qp$ associated to $X_K$, for a prime $p$. By the GAGA theorem, the rigid-analytic de Rham cohomology of $X_K^{\mathrm{an}}$ is simply the base-extension to $\Qp$ of the algebraic de Rham cohomology of $X_K$. However, the rigid-analytic description brings up some new phenomena. Assuming $K$ to be of parahoric type at $p$, there is a natural model of $Y_K$ over $\ZZ_{(p)}$, whose special fibre has a natural stratification -- the EKOR (Ekedahl--Kottwitz--Oort--Rapoport) stratification. Moreover, this can be extended to the compactification $X_K$; and we can consider the \emph{tubes} of these mod $p$ strata, which are subspaces of the dagger space $X_K^{\mathrm{an}}$. The de Rham cohomology of these tubes can be interpreted as rigid cohomology of the mod $p$ strata, hence the title of this paper.

  Our goal is to compute the analogue of $(\dag)$ for the tubes in $X_K^{\mathrm{an}}$ of certain strata (or unions of strata) in the special fibre. These strata are invariant under the action of prime-to-$p$ Hecke correspondences, so we can ask about the Hecke eigensystems appearing in their cohomology. More precisely, we want to show that, for Klingen-parahoric levels at $p$ and two particular locally-closed subspaces $T$ in $X_{K, \mathbb{F}_p}$, the analogue of $(\dag)$,
  \[ H^*_c\big(\tb{T}, \DR^\bullet_c(V)\big) \to H^*_c\big(\tb{T}, \DR^\bullet(V)\big) \tag{$\dag_p$}, \]
  is an isomorphism after localising at a non-CAP, cuspidal eigensystem. This is the result we need for our computations in the main text of this paper.

  \begin{remark}
   Note that in $(\dag_p)$, both cohomology groups are compactly-supported towards the complement of $\tb{T}$ in $X_K^{\mathrm{an}}$; only the support condition towards the toroidal boundary $T \cap D_K$ is changing. In particular, we should intepret $H^*_c(\tb{T}, \DR^\bullet(V))$ as a sort of ``partially compactly supported'' cohomology group, with compact support towards the complement of $T$ but non-compact support towards $T \cap D_K$. Cohomology groups with this sort of ``mixed support condition'' have appeared in many recent works on $p$-adic geometry of Shimura varieties, such as \cite{HLTT16} and \cite{boxerpilloni20}.
  \end{remark}

 \subsection{Outline of the argument}

  In order to analyse the map $(\dag_p)$, we proceed in two main steps.

  We first carry out a geometric computation, describing the intersections of EKOR strata and boundary strata inside the mod $p$ special fibre of a Klingen-level $\GSp_4$ Shimura variety. These intersectinos turn out to be either trivial, or preimages of EKOR strata in modular-curve boundary components (again of parahoric level at $p$).

  The second step is to consider the coefficient sheaves $(\cV, \nabla)$; we recall a result of Burgos and Wildeshaus, showing that the image of $(\cV, \nabla)$ under (derived) pushforward to a boundary stratum in the \emph{minimal} compactification of $Y_K$ can be expressed in terms of automorphic vector bundles on Shimura varieties of smaller dimension.

  Combining these two results, we obtain a description of the mapping fibre of $(\dag_p)$ in terms of parabolic inductions from the cohomology of EKOR strata in $\GL_1$ and $\GL_2$ Shimura varieties, allowing us to conclude that $(\dag_p)$ localises to an isomorphism at a non-CAP cuspidal representation.

  \begin{remark}
   We have not attempted to specify precisely \emph{which} parabolically-induced eigensystems appear in the kernel and cokernel of $(\dag_p)$. This is clearly possible, but it is not necessary for our intended applications so we shall not pursue it further here.
  \end{remark}

%
%
%


\section{General theory of compactifications}

 We first recall how to compactify Siegel modular varieties (over $\QQ$). We suppose $G = \GSp_{2n}$ and $K \subset G(\Af)$ is a neat open compact. There are $g$ maximal parabolic subgroups $P_1, \dots, P_r$ in $G$, with $P_r$ being the stabiliser of $\langle e_1, \dots, e_r\rangle$ in the standard representation, so the Levi $M_r$ of $P_r$ is $\GL_r \times \GSp_{2n-2r}$. We put $P_0 = G$. For each $r$ we let $M_{r, h}$ be the Hermitian part of $M_r$, isomorphic to $\GSp_{2n-2r}$ embedded into $G$ via
 \[ A \mapsto \begin{smatrix} \nu(A) \\ & A \\ && 1 \end{smatrix},\]
 where $\nu$ is the symplectic multiplier; and $P_{r, h} \subset P_r$ the preimage of $M_{r, h}$. 

 \subsection{Cusp labels and minimal compactification}

  \begin{definition}
   For $0 \le r \le n$, let $\mathfrak{C}(r, K)$ denote the double quotient
   \[ P_r(\QQ) P_{r, h}(\Af) \backslash G(\Af) / K; \]
   and let $\mathfrak{C}(K)$ be the set of pairs $(r, [g])$ with $r \in \{0, \dots, n\}$ and $[g] \in \mathfrak{C}(r, K)$. We call these \emph{cusp labels} at level $K$. We give $\mathfrak{C}(K)$ a poset structure by defining $(r, [g] ) \preccurlyeq (r', [g])$ for each $g \in G(\Af)$ and $r \ge r'$ (sic).
  \end{definition}

  Given a pair $(r, [g]) \in \mathfrak{C}(K)$, we let $K_{r, g} \subset \GSp_{2n-2r}(\Af)$ be the image of $P_{r, h}(\Af) \cap g K g^{-1}$ under the natural projection map. This is a neat open compact subgroup. We let $Z_{r, g}$ be the corresponding Shimura variety.

  \begin{note}
   For $r = n$, we need to understand $Z_{r, g}$ as $\QQ^{\times}_{> 0} \backslash \Af^\times / \det \pi(K_g)$; this is actually a double covering of the ``usual'' Shimura variety for $\mathbf{G}_m$, which is $\QQ^{\times} \backslash \Af^\times / \det \pi(K_g)$. Cf.~Definition 2.1 in \cite{pink90}. This issue does not arise for $r < n$, since the $\GSp_{2n-2r}$ Shimura datum has the expected number of components.
  \end{note}

  \begin{proposition}
   The Shimura variety $Y_K$ has a canonical compactification, the \emph{minimal (Baily--Borel) compactification} $j^{\min}: Y_K \into X^{\min}_K$, where $X^{\min}_K$ is a projective normal variety over $\QQ$.

   The variety $X^{\min}_K$ has a stratification\footnote{Recall that a stratification of a scheme is a decomposition as a set-theoretic disjoint union of locally-closed subschemes (the strata), with the property that the closure of any stratum is a union of strata.} by smooth strata $Z_{r, g}$ indexed by pairs $(r, g) \in \mathfrak{C}(K)$, with $Y_K$ corresponding to the stratum $(0, \id)$, and the closure relation given by the above poset structure. The stratum $Z_{r, g}$ is canonically identified with the quotient of the Shimura variety for $\GSp_{2n-2r}$ of level $K_{r, g}$ by a free action of a finite group of algebraic automorphisms $\Delta$.
  \end{proposition}

  \begin{proof}
   See e.g.~\cite[\S 3]{pink92} or \cite[\S 1]{burgoswildeshaus04}.
  \end{proof}

  \begin{remark}
   We shall restrict to level groups of the form $K = K_p K^p$ where $K_p$ is a standard parahoric subgroup at some prime $p$, and $K^p$ is a principal congruence subgroup in $G(\Af^p)$ of large enough level. In these cases the finite groups $\Delta$ appearing in Pink's construction are all trivial, so the boundary components $Z_{r, g}$ of $X_K^{\min}$ are themselves Shimura varieties (see \cite{stroh10-minimal}).
  \end{remark}

  Associated to each boundary stratum, we also have a discrete group $\bar{H}_C$ (in the notations of \cite{burgoswildeshaus04}); it is a neat arithmetic subgroup of the ``linear part'' $M_{r, \ell}(\QQ) \cong \GL_r(\QQ)$. For $r = 1$, this group is just $\QQ^\times$, which has no nontrivial neat subgroups, so  $\bar{H}_C$ is trivial.
%

 \subsection{Toroidal compactification}

  Since $X_K^{\min}$ is in general non-smooth, it is convenient to work with \emph{toroidal compactifications}. We briefly recall how these are defined.

  For each cusp-label $(r, g)$, we consider the space $\cP_r$ of positive-semidefinite bilinear forms on $\RR^r$ with rational radicals, and its interior $\cP_r^+$ consisting of positive-definite forms. The discrete group $\bar{H}_C \subset \GL_r(\QQ)$ associated with $(r, g)$ as above acts on $\cP_r$ and $\cP_r^+$. We choose, for each $(r, g) \in \mathfrak{C}(K)$, a collection $\Sigma_{r, g}$ of rational polyhedral cones in $\RR^{r(r+1)/2}$ forming a cone decomposition of $\cP_r$ (i.e. the cones are disjoint, their union is $\cP_r$, and each face of a cone in $\Sigma_{r, g}$ is also in $\Sigma_{r, g}$). These are required to satisfy the following properties, for each $(r, g)$:
  \begin{itemize}
   \item The action of $\bar{H}_C$ preserves $\Sigma_{r, g}$, and the set of orbits for this action is finite.
   \item There is a subset $\Sigma^+_{r, g} \subseteq \Sigma_{r, g}$ forming a cone decomposition of $P_r^+$.
  \end{itemize}
  The cones in $\Sigma_{r, g} - \Sigma_{r, g}^+$ are required to satisfy a compatibility condition with the $\Sigma_{r', g'}$ for $(r', g') \succcurlyeq (r, g)$, which we shall not specify here (see e.g.~\cite{faltingschai}).

  For any collection $\Sigma = (\Sigma_{r, g})$ satisfying these conditions, we can define a toroidal compactification $X_K^\Sigma$ of $Y_K$. We frequently omit the decoration $\Sigma$ once a choice of cone decomposition $\Sigma$ has been fixed.

  In general, $X_K^{\Sigma}$ is only an algebraic space; but if $\Sigma$ is chosen suitably, it is a smooth projective algebraic variety, and the complement $X_K - Y_K$ is a smooth normal-crossing divisor. It is a standard fact that cone decompositions $\Sigma$ with these properties do exist.

  The strata of the toroidal compactification are indexed by triples $(r, g, [\sigma])$, for $(r, g) \in \mathfrak{C}(K)$, and $[\sigma]$ a $\bar{H}_C$-orbit in $\Sigma_{r, g}^+$. Geometrically, each cusp-label $(r, g)$ determines a chain of maps
  \[ \Xi \to C \to Z \]
  where $Z = Z_{r, g}$, $C$ is an abelian scheme over $Z$ (of relative dimension $r(n-r)$), and $\Xi$ is a torus bundle over $C$ (of relative dimension $\tfrac{r(r+1)}{2}$). Each cone $\sigma \in \Sigma_{r, g}^+$ determines a torus embedding $\Xi \into \Xi(\sigma) = \bigsqcup_{\tau} \Xi_{\tau}$, where $\tau$ varies over the faces of $\sigma$, and $\Xi_{\sigma}$ is the unique closed fibre; the stratum $Z_{r, g, [\sigma]}$ is isomorphic to $\Xi_{\sigma}$ (and the formal completion of $X_K^{\Sigma}$ along $Z_{r, g, [\sigma]}$ is isomorphic to the completion of $\Xi(\sigma)$ along $\Xi_{\sigma}$). In particular, the codimension of $Z_{r, g, [\sigma]}$ in $X_K^{\Sigma}$ is the dimension of the cone $\sigma$.
  
  \begin{notation} For a cusp-label $(r, g)$, we write $Z_{r, g}^{\Sigma}$ for the preimage in $X_{K}^{\Sigma}$ of the stratum $Z_{r, g} \subseteq X_K^{\min}$, which is the union of the $Z_{r, g, [\sigma]}$ as $\sigma$ varies over $\bar{H}_C$-orbits in $\Sigma_{r, g}^+$.
  \end{notation}

  \begin{remark}
   The space $\Xi$ can be interpreted as a moduli space for polarised 1-motives (with toric parts of dimension $r$, and abelian part of dimension $n - r$) with level structures. The projection from $Z_{r, g, [\sigma]}$ to the underlying Shimura variety $Z_{r, g}$ correponds to forgetting the toric and linear part of the 1-motive.
  \end{remark}

 \subsection{Local cusp-labels at \texorpdfstring{$p$}{p}}

  Let $p$ be a prime, and suppose $K$ has the form $K^p K_p$ for $K_p \subset G(\Qp)$ and $K^p \subset G(\Af^p)$ open compacts.

  \begin{definition}
   A \emph{local cusp-label} is a pair $(r, g)$ with $0 \le r \le n$ and $g \in P_r(\Qp) \backslash G(\Qp) / K_p$. We write $\mathfrak{C}_p(K_p)$ for the set of these.
  \end{definition}

  There is a obvious map $\mathfrak{C}(K) \to \mathfrak{C}_p(K_p)$ given by $(r, g) \mapsto (r, g_p)$ (and this is compatible with the poset structure). Since $P_r(\QQ) P_{r, h}(\Qp)$ is dense in $P_r(\Qp)$, the fibre of this map over $(r, g_p)$ can be identified with the away-from-$p$ double quotient $P_r(\QQ) P_{r, h}(\Af^p) \backslash G(\Af^p) / K^p$; in particular, if we take the limit over prime-to-$p$ levels, the action of $G(\Af^p)$ is transitive on the fibres. Clearly, the subgroup $K_{r, g}$ is itself a product of groups at $p$ and away from $p$, and the factor at $p$ is determined by $g_p$; we write $K_{r, g_p}$ for this factor. Then the strata $Z_{r, g}$, for all $(r, g) \in \mathfrak{C}(K)$ mapping to a given $(r, g_p) \in \mathfrak{C}_p(K_p)$, are all Shimura varieties for $\GSp_{2n-2r}$ with the same $p$-level structure, namely $K_{r, g_p}$.

  \begin{definition}
   We define $Z_{(p, r, g_p)}$, for $g \in \mathfrak{C}_p(K_p)$, to be the disjoint union of the $Z_{r, g}$ for all $(r, g) \in \mathfrak{C}(K)$ mapping to $(r, g_p) \in \mathfrak{C}_p(K)$.
  \end{definition}

  So the set $\{ Z_{p, r, g_p} : (r, g_p) \in \mathfrak{C}_p(K_p)\}$ is a stratification of $X_K^{\min}$, somewhat coarser than the one described above, which is stable under the prime-to-$p$ Hecke action (its strata are exactly the prime-to-$p$ Hecke orbits on the set of strata $Z_{r, g}$). Similarly, we write $Z_{p, r, g_p}^{\Sigma}$ for the corresponding stratification of the toroidal compactification $X_K^{\Sigma}$.

\section{Compactifications at parahoric level}

 Recall that a \emph{parahoric subgroup} of $\GSp_{2n}(\Qp)$ is an open compact subgroup containing an Iwahori subgroup; since all Iwahori subgroups are conjugate, it suffices to consider the standard Iwahori subgroup (the preimage in $G(\Zp)$ of the upper-triangular Borel of $G(\Fp)$). We shall restrict attention to parahorics contained in $G(\Zp)$, which biject with subsets $J \subseteq I = \{1, \dots, n\}$, with $K_J$ denoting the subgroup given by the preimage of the mod $p$ parabolic $P_J(\Fp) = \bigcap_{i \in J} P_i(\Fp)$. Thus $K_{\varnothing} = G(\Zp)$, and $K_I$ is the standard Iwahori.

 \subsection{Weyl groups}

  We let $s_1, \dots, s_n$ be the usual generators of the Weyl group $W \subset \mathfrak{S}_{2n}$; explicitly $s_i = (i, i+1) (2n+1-i, 2n-i)$ for $1 \le i \le n-1$ and $s_n = (n, n+1)$. We choose arbitrary lifts of these to elements of $G(\Zp)$. With this numbering, $s_i$ fixes the lattice $W_j$ for all $i \ne j$; in particular, for any $J \subset I$, the subgroup $W_{J^c}$ generated by the $s_i$ for $i \notin J$ is a finite group contained in $K_J$ (which we can identify with the Weyl group of the Levi $M_J$). Moreover, we have $K_J = \bigsqcup_{w \in W_{J^c}} K_I w K_I$.

  \begin{lemma}
   For any subsets $J_1, J_2 \subseteq I$, we have
   \[ P_{J_1}(\Qp) \backslash G(\Qp) / K_{J_2} = K_{J_1} \backslash G(\Zp) / K_{J_2} = W_{J_1^c} \backslash W / W_{J_2^c}. \]
  \end{lemma}

  \begin{proof}
   This follows readily from the Bruhat decomposition of $G(\Fp)$.
  \end{proof}

 \subsection{Boundary components}

  Taking $J_1 = \{r\}$ for some $r$, it follows that every element of $\mathfrak{C}_p(K_J)$ has a representative in $W$.

  \begin{proposition}
   For $r \in \{1, \dots, n\}$ and $w \in W$, the image of $P_{r, h} \cap w P_J w^{-1}$ in $M_{r, h} \cong \GSp_{2n-2r}$ is a Weyl-group conjugate of a standard parabolic in $M_{r, h}$.
  \end{proposition}

  \begin{proof}
   The image of $P_{r, h} \cap w K_Jw^{-1}$ is clearly contained in $M_{r, h}(\Zp)$, and it contains the preimage of $M_{r, h} \cap w B(\Fp) w^{-1}$, which is a Borel subgroup of $M_{r, h}$.
  \end{proof}

  Thus the boundary strata $Z_{r, g}$ in the minimal compactification of a Siegel Shimura variety of standard-parahoric level at $p$ are themselves Siegel Shimura varieties (of smaller genus) of standard-parahoric level at $p$.

  \begin{remark}
   If $J = I$, so $K_J$ is the Iwahori subgroup, then the boundary components are themselves Iwahori-level Shimura varieties. This is not necessarily the case for general $J$, as we shall see below in the genus 2 setting.
  \end{remark}

 \subsection{Integral models}

  We recall that Siegel Shimura varieties of parahoric level at $p$ have canonical $\ZZ_{(p)}$-models. These are moduli spaces for abelian varieties of dimension $n$ over $\ZZ_{(p)}$-algebras, endowed with a prime-to-$p$ polarisation and level structure determined by $K^p$, and with a partial flag of isotropic subgroup-schemes $C_i \subset A[p]$ for $i \in J$, with $C_i$ of degree $p^i$ and $C_i \subset C_j$ for $i \le j$.

  Moreover, this extends to the compactification in the natural fashion:

  \begin{proposition}[{Stroh, see \cite[Theoreme principal] {stroh10-minimal}}]
   The compactification $X_K^{\min}$ has a $\ZZ_{(p)}$-model $j^{\min} : \cY_K \into \cX_K^{\min}$, with $\cX_K^{\min}$ a projective $\ZZ_{(p)}$-scheme; and $\cX_K^{\min}$ has a stratification indexed by $\mathfrak{C}(K)$, whose strata are the canonical models of the parahoric-level Shimura varieties for $\GSp_{2n-2r}$.
  \end{proposition}

  The construction of $\cX_K^{\min}$ involves, as an intermediate step, the construction of a $\ZZ_{(p)}$-model $\cX_K^{\Sigma}$ of $X_K^\Sigma$ (for suitable cone decompositions $\Sigma$), which maps naturally to $\cX_K^{\min}$; and the description of the boundary strata in terms of cusp-labels, and the formal coordinate charts along these boundary strata, applies also with $\ZZ_{(p)}$-coefficients.

\section{Boundary strata for the \texorpdfstring{$\GSp_4$}{GSp4} Klingen}

 \subsection{Group-theoretic description} We now specialise to $G = \GSp_4$ and $K_p = \Kl(p)$, the \emph{Klingen parahoric} (the preimage of $P_1(\Fp)$ in $G(\Zp)$). Computing the sets $\mathfrak{C}_p(K_p)$ as double-quotients of Weyl groups, $P_r(\Qp) \backslash G(\Qp) / K_p \cong W_{M_r} \backslash W_G / W_{M_{\Kl}}$, we obtain the following picture:
  \begin{itemize}

   \item There are exactly three local cusp-labels with $r = 1$, represented by the classes of $g_p = \id$, $g_p = s_1$, and $g_p = s_2$. For the first and third, we have $K_{r, g_p} = \GL_2(\Zp)$; for the second, $K_{r, g_p}$ is the standard Iwahori subgroup.

   \item There are two cusp-labels with $r = 2$, represented by $g_p = \id$ and $g_p = s_2$.

   \item We have $(2, \id) \preccurlyeq (1, \id)$ and $(2, s_2) \preccurlyeq (1, s_2)$, while both of the $r=2$ cusp labels precede $(1, s_1)$.
  \end{itemize}

  Thus, for $K = K^p \Kl(p)$, the stratification of the boundary of $X_K^{\min}$ by $p$-cusp labels has three 1-dimensional and two 0-dimensional strata. Two of the 1-dimensional strata are unions of modular curves of prime-to-$p$ levels, each of which has a single $G(\Af^p)$-orbit of cusps; while the third is a union of modular curves of $\Gamma_0(p)$ level, which have two $G(\Af^p)$-orbits of cusps ($0$ and $\infty$).

 \subsection{Preimages in the toroidal compactification}
 
  For the Klingen ($r = 1$) boundary strata, the space $\cP_r$ is simply $\RR_{\ge 0}$, so it has a unique polyhedral cone decomposition with a single non-trivial cone. Thus, for each $g \in \mathfrak{C}(1, K)$, there is a single toroidal boundary stratum above $Z_{r, g}$ and its fibres over $Z_{r, g}$ are elliptic curves.
  
  For $r = 2$, in contrast, the ``geometric'' parts of the stratification are trivial ($Z_{2, g}$ is 0-dimensional, and the abelian family $C$ over it is trivial), but the polyhedral cone decomposition is combinatorially rich -- the choice of $\Sigma_{2, g}$ amounts to choosing a triangulation of a fundamental domain for the action of $\bar{H}_C$ on the upper half-plane. So the preimage $Z^{\Sigma}_{r, g}$ is a 2-dimensional toric variety embedded in $X_K^{\Sigma}$.
  
 \subsection{Moduli-space interpretation}

  We can interpret the above stratification in terms of 1-motives with parahoric level structures. (This simply amounts to specialising the general statements of \S 1.2.6 of \cite{stroh10-siegel} to the case $n = 2$ and $J = \{1\}$).

  Along the boundary strata of $\cX_K^{\Sigma}$ with $r = 2$, the universal abelian surface $A$ degenerates into a 1-motive $M$ with no abelian part, i.e.~of the form $[Y \to T]$, where $Y \cong \ZZ^2$, and $T$ is a rank 2 torus. Its $p$-torsion $M[p]$ is therefore an extension of the \'etale group scheme $Y / pY$ by the multiplicative group scheme $T[p]$. The canonical level subgroup $C \subset A[p]$ extends to a subgroup of $M[p]$; and the two local cusp-labels at $p$ correspond to the two possibilities for its position relative to the filtration: either $C$ lies inside $T[p]$ (the case $g_p = 1$), or it maps isomorphically to its image in $Y / pY$ (the case $g_p = s_2$).

  For the $r = 1$ boundary strata we have a more complicated picture: the 1-motive has all three of its graded pieces non-trivial (a toric part, an elliptic curve, and a lattice part); the group $M[p]$ thus also has 3 graded pieces, with the ``middle'' graded piece given by the $p$-torsion of the elliptic curve, and the two outer pieces arising from the toric and lattice parts. There are thus 3 possibilities for where an order $p$ subgroup can land: it is either fully contained in the torus part (corresponding to $g_p = \id$); maps isomorphically to $E[p]$ (corresponding to $g_p = s_1$); or maps isomorphically to $Y / pY$ (corresponding to $g_p = s_2$). So the implied $p$-level structure on the elliptic curve $E$ is trivial in the first and last cases, and a $\Gamma_0(p)$-level structure in the second case.

\section{EKOR strata}

 \subsection{EKOR strata}

  Let $Y_K$ be a Shimura variety for $\GSp_{2n}$, whose level has the form $K^p K_J$ for some standard parahoric $K_J$ as above, and $\cY_K$ its $\ZZ_{(p)}$-model. Then the special fibre of $\cY_K$ has a canonical stratification, the \emph{Ekedahl--Kottwitz--Oort--Rapoport} (EKOR) stratification
  \[ \cY_{K, \Fp} = \bigsqcup_{x \in\, {}^J\!\operatorname{Adm}(\mu)} \cY_{K, \Fp}^x, \]
  where ${}^J\!\! \operatorname{Adm}(\mu)$ is a certain (finite) subset of the Iwahori Weyl group $\widetilde{W}$ depending on $J$ and the cocharacter $\mu$ defining the Shimura datum. We shall not recall the exact details here, but refer to \cite{herapoport17} and \cite{shenyuzhang21}.

  Note that for each $x \in {}^J\!\! \operatorname{Adm}(\mu)$, the EKOR stratum $\cY_{K, \Fp}^x$ is preserved by the action of prime-to-$p$ Hecke correspondences.

  \begin{remark}
   Note that the special cases of the EKOR stratification when the level group is either hyperspecial, or Iwahori, at $p$ are respectively the Ekedahl--Oort and the Kottwitz--Rapoport stratifications. These have a longer history; see \cite{herapoport17} and the references therein. However, we are principally interested in the case of the Klingen parahoric in $\GSp_4$, which does not fit into either of these extreme cases.
  \end{remark}

 \subsection{Well-positioned subschemes}

  We recall from \cite{lanstroh} the notion of a \emph{well-positioned subscheme} of $Y_{K, T}$, where $T$ is a $\ZZ_{(p)}$-scheme (such as $T = \Spec \mathbf{F}_p$). Associated to each boundary stratum $Z = Z_{r, g}$ of $X_K^{\min}$, there is a chain of morphisms
  \[ \Xi \to C \to Z\]
  with $C \to Z$ an abelian scheme, and $\Xi \to C$ a torsor under a split torus. The space $\Xi$ gives a formal coordinate chart for $Y_K$ in a neighbourhood of $Z$. With this notation, a locally-closed subscheme $S \subset Y_{K, T}$ is well-positioned if, for each $Z$, the pullback of $S$ to $\Xi$ coincides with the preimage in $\Xi$ of a locally-closed subscheme $S_{r, g}^\natural \subseteq (Z_{r, g})_T$.

  In \cite{lanstroh} it is shown that, for any stratification of $Y_{K, T}$ by well-positioned subschemes, we can define a stratification of $X_{K, T}^{\min}$ by setting
  \[ S^{\min} = \overline{j^{\min}(S)} - \bigcup_{\mathclap{T \subseteq \overline{S}, T \ne S}} \overline{j^{\min}(T)} \]
  for each stratum $S$ where $j^{\min}$ is the inclusion of $\cY_K$ into $\cX_K^{\min}$; and these satisfy $S^{\min} \cap Z_{r, g} = S_{r, g}^\natural$, for each cusp-label $(r, g)$, where $S_{r, g}^\natural$ is the subscheme appearing in the definition of ``well-positioned''. This evidently also defines an extension of the stratification of $Y_{K, T}$ to $X_{K, T}^\Sigma$, for any choice of toroidal boundary data $\Sigma$, via pullback along the natural map $\pi : X_{K, T}^\Sigma \to X_{K, T}^{\min}$.

  \begin{remark}
   Note that \cite{lanstroh} considers a more general setting where $\Xi$ is a torus torsor over a scheme $C \to Z$ which is not necessarily an abelian scheme, but some more complicated map. This leads to a complication, which is that the collection $(S_Z^\natural)_Z$ is not necessarily uniquely determined by $S$, and although there is always a map $S_Z^\natural \to S^{\min} \cap Z$ which is a bijection on underlying sets, this may not be an isomorphism of schemes. However, in our case $C$ is genuinely an abelian scheme over $Z$ (see \cite{pilloni20}), so in particular it is faithfully flat and the issue does not arise.
  \end{remark}

  \begin{proposition}[S.~Mao]
   \label{prop:lanstroh}
   The EKOR strata are well-positioned subschemes. Moreover, for each EKOR stratum $S$ and each cusp-label $(r, g)$, the intersection $S^{\natural}_{r, g}$ is either empty, or is an EKOR stratum in $(Z_{r, g})_{\Fp}$.
  \end{proposition}
  
  \begin{proof}
   This is a special case of a very general result for Hodge-type Shimura varieties of parahoric level due to Shengkai Mao \cite{mao-preprint}; see Theorem 1.1 of \emph{op.cit.}.
  \end{proof}
  
  \begin{remark}
   Note that the $p$-rank of the torsion $A[p]$ is constant on each EKOR stratum in $Y_{K, \Fp}$; and the union of the $p$-rank 0 EKOR strata is proper, so its image in $X_{K, \Fp}^{\min}$ is closed. Hence the $p$-rank 0 EKOR strata have empty intersection with every boundary stratum.
  \end{remark}

\section{EKOR strata for \texorpdfstring{$n = 1, 2$}{n = 1 and 2}}
 \label{sect:EKORexplicit}

 We recall what the EKOR stratification looks like in simple cases, following the account in \cite{shenyuzhang21}. (However, we shall label the EKOR strata by symbols denoting the properties of the $p$-divisible group, rather than by the more general but less concrete labelling by affine Weyl groups used in \emph{op.cit.}.)

 Let $G$ be either $\GL_2$ or $\GSp_4$, and let $K \subset G(\Af)$ be a level structure of the form $K = K_p K^p$, with $K_p \subset G(\Qp)$ a standard parahoric, and $K^p \subset G(\Af^p)$ any neat subgroup.

 \subsection{$\GL_2$ case}

  For $G = \GL_2$ there are exactly two parahoric level groups (up to conjugacy), namely $\GL_2(\Zp)$ and the Iwahori $\Iw$.

  If $K_p = \GL_2(\Zp)$, then the Shimura variety $Y_K$ has a canonical smooth model over $\ZZ_{(p)}$, which is a moduli space for elliptic curves $E$ with prime-to-$p$ level structure. Over $\Fp$ any elliptic curve is either ordinary or supersingular, and this gives a stratification with two smooth strata
  \[ Y_{K, \Fp} = Y_{K, \Fp}^{\ord} \sqcup Y_{K, \Fp}^{\mathrm{ss}}, \]
  of dimensions 1 and 0, which is the EKOR stratification (or the EO stratification, which is the same thing in this case).

  In the Iwahori-level case we have a marginally more complicated picture: $Y_K$ is semistable, but not smooth; and it parametrises pairs $(E, C)$, where $E$ is as before and $C \subset E[p]$ is a finite flat subgroup of order $p$. We have a decomposition
  \[ Y_{K, \Fp} = Y_{K, \Fp}^{m} \sqcup Y_{K, \Fp}^{\et} \sqcup Y_{K, \Fp}^{\alpha}\]
  into loci where $C$ is \'etale-locally isomorphic to $\mu_p$, $\ZZ/p$, or $\alpha_p$ respectively (forcing $E$ to be ordinary in the first two cases, and supersingular in the third). These strata have dimensions 1, 1, and 0 respectively (the same as their $p$-ranks), and the closure relation is given by the diagram
  \[
   \begin{tikzcd}[column sep=small, row sep=0pt]
             & m \\
     \alpha\ar[ru]\ar[rd] & \\
             & \et
   \end{tikzcd}
  \]
  where an arrow denotes that the source stratum is contained in the closure of the target stratum.

 \subsection{$\GSp_4$ spherical-level case}

  The $\GSp_4$ Shimura variety of prime-to-$p$ level, i.e.~for $K_p = \GSp_4(\Zp)$, parametrises abelian surfaces $A$ with some prime-to-$p$ polarisation and level structure (depending on $K^p$). We can decompose its special fibre according to the $p$-rank (the dimension of the multiplicative part of $A[p]$), which can be 0, 1 or 2. This can be refined by decomposing the $p$-rank 0, i.e.~supersingular, locus as the union of a ``superspecial'' locus (where $A$ is isomorphic over $\overline{\mathbf{F}}_p$ to a product of supersingular elliptic curves) and a ``supergeneral'' locus (where $A$ is isogenous, but not isomorphic, to such a product). This gives a stratification of $Y_{K, \Fp}$ with 4 strata, one of each dimension, with closure relation
  \[ (0, \mathrm{ss}) \longrightarrow (0, \mathrm{sg}) \longrightarrow (1) \longrightarrow (2) \]
  (where $(r)$ denotes $p$-rank equal to $r$). Note that the closed subvariety $Y_{K, \Fp}^{(0)} = Y_{K, \Fp}^{(0, ss)} \cup Y_{K, \Fp}^{(0, sg)}$ is proper but non-smooth: it is a union of projective lines intersecting at the superspecial points.

 \subsection{$\GSp_4$ Klingen-parahoric level}

  The EKOR stratification at Klingen parahoric level is described in \S 6.3 of \cite{shenyuzhang21}. There are 8 smooth strata, which we denote by the symbols
  \[ \{ (2, m), (2, \et),\ (1, m), (1, \et), (1, \alpha),\ (0, sg), (0, ss_1), (0, ss_2) \} \]
  where the integer $r$ in $(r, *)$ denotes the $p$-rank. The closure relations among the strata are given by the poset
  \[
   \begin{tikzcd}
             & (0, sg) \rar\ar[rd]\ar[rdd, end anchor=north west]& (1, \et) \rar                & (2, \et) \\
     (0, ss_1)\ar[ru]\ar[rd] & & (1, \alpha)\ar[rd] \ar[ru] & \\
             & (0, ss_2)\rar\ar[ru]\ar[ruu, end anchor=south west] & (1, m)  \rar               & (2, m)\\
     (\dim = 0) & (1) & (2) & (3)
   \end{tikzcd}
  \]
  We now explain the labelling. The space $Y_K$ is a moduli space for abelian surfaces $A$ as before with the additional data of a cyclic subgroup-scheme $C \subset A[p]$. For the non-supersingular strata, $(r, ?)$ signifies that $A$ has $p$-rank $r$, and $C$ has type $?$ (i.e.~is multiplicative, \'etale, or $\alpha_p$). For the $p$-rank 0 strata, $sg$ denotes that $A$ is supergeneral, while $ss_1$ and $ss_2$ denote two types of superspecial strata. (Note that the natural map to prime-to-$p$ level contracts the one-dimensional $ss_2$ strata to points, while the map to paramodular level contracts the $(0, sg)$ strata to points.)

  \begin{remark}
   In \cite{shenyuzhang21} the strata are labelled by certain elements of an affine Weyl group. Comparing with the description of the KR strata at Iwahori level in \cite{yu10}, one sees that the stratum labelled $s_{010}\tau$ in \emph{op.cit.} corresponds to $C$ \'etale, and $s_{120}\tau$ has $C$ multiplicative.
  \end{remark}

  See Figure \ref{fig:strata} in the main text for a diagram (in which we have amalgamated together the three rank 0 strata).

 \subsection{Intersections with boundary strata}
  \label{sect:explcitintersection}
  We now make explicit the result of Proposition \ref{prop:lanstroh}, by identifying the intersections $S_{r, g}^\natural$, for each EKOR stratum $S$ and each cusp-label $(r, g)$, as EKOR strata of $Z_{r, g}$. Since the EKOR stratifications are independent of the prime-to-$p$ level structure, the intersection depends only on the local cusp-label $(r, g_p) \in \mathfrak{C}_p(\Kl(p))$. See \cref{fig:EKORboundary} for a diagram.

  \paragraph*{$r = 2$ boundary components}

   Since the EKOR stratification on a Shimura variety for $\mathbf{G}_m$ is trivial, the each $r = 2$ boundary component lies wholly within a single EKOR stratum. Explicitly, $Z_{p, 2, \id}$ is contained in the $(2, m)$ EKOR stratum, while $Z_{p, 2, s_2}$ is contained in the $(2, \et)$ stratum. Geometrically, this corresponds to the fact that the $p$-torsion of the 1-motive along $Z_{p, 2, g}$ is an extension of an \'etale subgroup by a multiplicative one, so its $p$-rank must be 2, and a $p$-subgroup is multiplicative if it is contained in the toric part and \'etale otherwise.
   
   (We have accordingly drawn the $r = 2$ boundary strata as points in the diagram; geometrically this is slightly misleading, since their codimension may be 1 or 2 depending on whether the cone decompositions $\Sigma_{2, g}^+$ contain any one-dimensional cones.)

  \paragraph*{$r = 1$ boundary components}

   For $r = 1$, we have seen that the boundary components are modular curves.

   \begin{itemize}
    \item Along the boundary stratum $Z_{1, \id}$ (or, more precisely, its preimage in $X_{K, \Fp}^{\Sigma}$), the level group is always multiplicative (since it is contained in the torus part); and the $p$-rank of $M[p]$ is $r_p(M[p]) = 1 + r_p(E[p])$ where $E$ is the elliptic curve part of $M$.

    Hence the only EKOR strata which can intersect $Z_{1, \id}$ are the $(1, m)$ and $(2, m)$ strata; the intersection with the $(1, m)$ stratum is the $p$-rank 0 locus of $Z_{1, \id}$, i.e. the supersingular locus, and the intersection with the $(2, m)$ stratum is the ordinary locus. All other EKOR strata have empty intersection with $Z_{1, \id}$.

    \item The picture for $Z_{1, s_2}$ is similar, except the level group is always \'etale. So the supersingular locus is the intersection with the $(1, \et)$ stratum, and the ordinary locus the $(2, \et)$ stratum, and all other EKOR strata are disjoint from $Z_{1, \id}$.

    \item For $Z_{1, s_1}$, the $p$-rank the 1-motive is one greater than the $p$-rank of the elliptic curve part $E$, while the $p$-level subgroup of the 1-motive is isomorphic to its image in $E$. Hence the intersections with the $(2, m)$, $(2, \et)$, and $(1, \alpha)$ strata correspond to the $m$, $\et$, and $\alpha$ strata of the modular curve.
   \end{itemize}

\begin{figure}
\caption{Intersections of boundary strata with EKOR strata.}
\label{fig:EKORboundary}

\tikzset{every picture/.style={line width=0.75pt}} 

\begin{tikzpicture}[x=0.75pt,y=0.75pt,yscale=-1,xscale=1]

\draw   (380,116) -- (380,270) -- (200,204) -- (200,50) -- cycle ;
\draw   (20,116) -- (20,270) -- (200,204) -- (200,50) -- cycle ;
\draw    (140,30) -- (200,50) ;
\draw    (260,30) -- (200,50) ;
\draw    (140,30) -- (140,70) ;
\draw    (260,30) -- (260,70) ;
\draw    (20,170) .. controls (102,105.5) and (122,195.5) .. (200,120) ;
\draw    (200,120) .. controls (254,183) and (351,131) .. (380,180) ;
\draw  [fill={rgb, 255:red, 0; green, 0; blue, 0 }  ,fill opacity=1 ] (196,120) .. controls (196,117.79) and (197.79,116) .. (200,116) .. controls (202.21,116) and (204,117.79) .. (204,120) .. controls (204,122.21) and (202.21,124) .. (200,124) .. controls (197.79,124) and (196,122.21) .. (196,120) -- cycle ;
\draw [color={rgb, 255:red, 74; green, 144; blue, 226 }  ,draw opacity=1 ][line width=1.5]    (110,85) -- (110,235) ;
\draw  [fill={rgb, 255:red, 74; green, 144; blue, 226 }  ,fill opacity=1 ] (106,115) .. controls (106,112.24) and (108.24,110) .. (111,110) .. controls (113.76,110) and (116,112.24) .. (116,115) .. controls (116,117.76) and (113.76,120) .. (111,120) .. controls (108.24,120) and (106,117.76) .. (106,115) -- cycle ;
\draw [color={rgb, 255:red, 74; green, 144; blue, 226 }  ,draw opacity=1 ][line width=1.5]    (284,80) -- (284,236) ;
\draw  [fill={rgb, 255:red, 126; green, 211; blue, 33 }  ,fill opacity=1 ] (279,155) .. controls (279,152.24) and (281.24,150) .. (284,150) .. controls (286.76,150) and (289,152.24) .. (289,155) .. controls (289,157.76) and (286.76,160) .. (284,160) .. controls (281.24,160) and (279,157.76) .. (279,155) -- cycle ;
\draw [color={rgb, 255:red, 74; green, 144; blue, 226 }  ,draw opacity=1 ][line width=1.5]    (20,150) -- (200,80) ;
\draw [color={rgb, 255:red, 74; green, 144; blue, 226 }  ,draw opacity=1 ][line width=1.5]    (200,80) -- (380,140) ;
\draw  [fill={rgb, 255:red, 74; green, 144; blue, 226 }  ,fill opacity=1 ] (280,109) .. controls (280,106.24) and (282.24,104) .. (285,104) .. controls (287.76,104) and (290,106.24) .. (290,109) .. controls (290,111.76) and (287.76,114) .. (285,114) .. controls (282.24,114) and (280,111.76) .. (280,109) -- cycle ;
\draw [color={rgb, 255:red, 74; green, 144; blue, 226 }  ,draw opacity=1 ][line width=1.5]  [dash pattern={on 1.69pt off 2.76pt}]  (170,70) -- (200,80) ;
\draw [color={rgb, 255:red, 74; green, 144; blue, 226 }  ,draw opacity=1 ][line width=1.5]  [dash pattern={on 1.69pt off 2.76pt}]  (200,80) -- (230,70) ;
\draw  [fill={rgb, 255:red, 126; green, 211; blue, 33 }  ,fill opacity=1 ] (106,149) .. controls (106,146.24) and (108.24,144) .. (111,144) .. controls (113.76,144) and (116,146.24) .. (116,149) .. controls (116,151.76) and (113.76,154) .. (111,154) .. controls (108.24,154) and (106,151.76) .. (106,149) -- cycle ;
\draw  [fill={rgb, 255:red, 126; green, 211; blue, 33 }  ,fill opacity=1 ] (195,81) .. controls (195,78.24) and (197.24,76) .. (200,76) .. controls (202.76,76) and (205,78.24) .. (205,81) .. controls (205,83.76) and (202.76,86) .. (200,86) .. controls (197.24,86) and (195,83.76) .. (195,81) -- cycle ;
\draw  [dash pattern={on 4.5pt off 4.5pt}]  (370,60) -- (360.28,128.02) ;
\draw [shift={(360,130)}, rotate = 278.13] [color={rgb, 255:red, 0; green, 0; blue, 0 }  ][line width=0.75]    (10.93,-3.29) .. controls (6.95,-1.4) and (3.31,-0.3) .. (0,0) .. controls (3.31,0.3) and (6.95,1.4) .. (10.93,3.29)   ;
\draw  [dash pattern={on 4.5pt off 4.5pt}]  (60,60) -- (60,128) ;
\draw [shift={(60,130)}, rotate = 270] [color={rgb, 255:red, 0; green, 0; blue, 0 }  ][line width=0.75]    (10.93,-3.29) .. controls (6.95,-1.4) and (3.31,-0.3) .. (0,0) .. controls (3.31,0.3) and (6.95,1.4) .. (10.93,3.29)   ;
\draw  [dash pattern={on 4.5pt off 4.5pt}]  (160,270) -- (111.28,211.54) ;
\draw [shift={(110,210)}, rotate = 50.19] [color={rgb, 255:red, 0; green, 0; blue, 0 }  ][line width=0.75]    (10.93,-3.29) .. controls (6.95,-1.4) and (3.31,-0.3) .. (0,0) .. controls (3.31,0.3) and (6.95,1.4) .. (10.93,3.29)   ;
\draw  [dash pattern={on 4.5pt off 4.5pt}]  (250,270) .. controls (234.16,249.21) and (239.88,240.18) .. (278.81,210.89) ;
\draw [shift={(280,210)}, rotate = 143.13] [color={rgb, 255:red, 0; green, 0; blue, 0 }  ][line width=0.75]    (10.93,-3.29) .. controls (6.95,-1.4) and (3.31,-0.3) .. (0,0) .. controls (3.31,0.3) and (6.95,1.4) .. (10.93,3.29)   ;
\draw  [dash pattern={on 4.5pt off 4.5pt}]  (420,210) .. controls (403.09,219.95) and (326.77,224.95) .. (284.63,156.04) ;
\draw [shift={(284,155)}, rotate = 59.04] [color={rgb, 255:red, 0; green, 0; blue, 0 }  ][line width=0.75]    (10.93,-3.29) .. controls (6.95,-1.4) and (3.31,-0.3) .. (0,0) .. controls (3.31,0.3) and (6.95,1.4) .. (10.93,3.29)   ;
\draw  [dash pattern={on 4.5pt off 4.5pt}]  (300,30) .. controls (308.82,52.54) and (310.92,77.96) .. (295,103.44) ;
\draw [shift={(294,105)}, rotate = 303.18] [color={rgb, 255:red, 0; green, 0; blue, 0 }  ][line width=0.75]    (10.93,-3.29) .. controls (6.95,-1.4) and (3.31,-0.3) .. (0,0) .. controls (3.31,0.3) and (6.95,1.4) .. (10.93,3.29)   ;
\draw  [dash pattern={on 4.5pt off 4.5pt}]  (210,30) .. controls (219.65,42.55) and (208.81,58.81) .. (200.85,74.32) ;
\draw [shift={(200,76)}, rotate = 296.57] [color={rgb, 255:red, 0; green, 0; blue, 0 }  ][line width=0.75]    (10.93,-3.29) .. controls (6.95,-1.4) and (3.31,-0.3) .. (0,0) .. controls (3.31,0.3) and (6.95,1.4) .. (10.93,3.29)   ;
\draw  [dash pattern={on 4.5pt off 4.5pt}]  (50,270) .. controls (33.08,279.95) and (47.85,228.52) .. (99.22,161.02) ;
\draw [shift={(100,160)}, rotate = 127.41] [color={rgb, 255:red, 0; green, 0; blue, 0 }  ][line width=0.75]    (10.93,-3.29) .. controls (6.95,-1.4) and (3.31,-0.3) .. (0,0) .. controls (3.31,0.3) and (6.95,1.4) .. (10.93,3.29)   ;
\draw  [dash pattern={on 4.5pt off 4.5pt}]  (120,30) .. controls (128.82,52.54) and (130.92,77.96) .. (115,103.44) ;
\draw [shift={(114,105)}, rotate = 303.18] [color={rgb, 255:red, 0; green, 0; blue, 0 }  ][line width=0.75]    (10.93,-3.29) .. controls (6.95,-1.4) and (3.31,-0.3) .. (0,0) .. controls (3.31,0.3) and (6.95,1.4) .. (10.93,3.29)   ;

\draw (349,32) node [anchor=north west][inner sep=0.75pt]   [align=left] {$\displaystyle ( Z_{p, 1, s_{1}})^{\et}$};
\draw (31,28) node [anchor=north west][inner sep=0.75pt]   [align=left] {$\displaystyle ( Z_{p, 1, s_{1}})^{m}$};
\draw (134,272) node [anchor=north west][inner sep=0.75pt]   [align=left] {$\displaystyle ( Z_{p, 1, id})^{ord}$};
\draw (251,269) node [anchor=north west][inner sep=0.75pt]   [align=left] {$\displaystyle ( Z_{p, 1, s_{2}})^{ord}$};
\draw (424,202) node [anchor=north west][inner sep=0.75pt]   [align=left] {$\displaystyle ( Z_{p, 1, s_{2}})^{ss}$};
\draw (282,2) node [anchor=north west][inner sep=0.75pt]   [align=left] {$\displaystyle Z_{p, 2, s_{2}}$};
\draw (181,2) node [anchor=north west][inner sep=0.75pt]   [align=left] {$\displaystyle ( Z_{p, 1, s_{1}})^{\alpha }$};
\draw (21,278) node [anchor=north west][inner sep=0.75pt]   [align=left] {$\displaystyle ( Z_{p, 1, id})^{ss}$};
\draw (92,2) node [anchor=north west][inner sep=0.75pt]   [align=left] {$\displaystyle Z_{p, 2, id}$};

\end{tikzpicture}
\end{figure}

 \begin{remark}
  A by-product of this analysis is the following observation:

  \begin{quotation}
   ``For each EKOR stratum $S$, there is a unique smallest $G(\Af^p)$-orbit of boundary strata with which it intersects; and the intersection with this smallest boundary stratum is a $p$-rank 0 EKOR stratum in the corresponding boundary component.''
  \end{quotation}

  So there is a bijection between EKOR strata for $\GSp_4$ of level $\Kl(p)$, and pairs consisting of a local cusp-label $(r, g_p)$, and a $p$-rank 0 EKOR stratum for $\GSp_{4-2r}$ at level $K_{r, g_p}$. (This generalises the fact that in a modular curve of $\Gamma_0(p)$ level, each of the two open EKOR strata contains a unique prime-to-$p$ Hecke orbit of cusps.)

  We expect that the above statement should hold more generally (for any $n$ and any standard parahoric $K_J$ in $\GSp_{2n}(\Zp)$). This appears to be related to the `shuffle' construction described by C.-F.~Yu in \cite{yu10}.
 \end{remark}

\section{Coefficient sheaves}

 Since the EKOR strata in Siegel Shimura varieties are smooth varieties over $\Fp$, we may form their rigid cohomology (in the sense of Berthelot). These cohomology groups are finite-dimensional $\Qp$-vector spaces; and they have an action of the Hecke algebra of $\GSp_4(\Af^p)$ of level $K^p$, and hence \emph{a fortiori} of the spherical Hecke algebra away from $S$, where $S$ consists of $p$ and the ramified primes of $K^p$. One can therefore ask which systems of eigenvalues for the spherical Hecke algebra appear in these rigid cohomology groups.

 More generally, one can also consider cohomology with compact support; or cohomology with \emph{partial} compact support, in which we impose varying support conditions towards the other EKOR strata and the toroidal boundary.

 \begin{remark}
  For an example of these partial-support cohomology groups, see \cite{HLTT16}, where the authors consider the rigid cohomology of the ordinary locus in a unitary Shimura variety (with hyperspecial level at $p$) taking compact support towards the toroidal boundary, but non-compact support toward the non-ordinary locus.
 \end{remark}

 The aim of this section is to examine the ``boundary contributions'' to these cohomology groups.

 \subsection{Logarithmic de Rham complexes}

  Let $V$ be an algebraic representation of $\GSp_4$ and $(\cV, \nabla) = \mu_{K, \dR}(V)$ the corresponding vector bundle on $Y_K$, where $\mu_{K, \dR}$ is the canonical-construction functor. This extends canonically to a vector bundle on $X_K^\Sigma$, equipped with a connection with logarithmic poles along the boundary divisor $D_K$; so we have a logarithmic de Rham complex
  \[ \DR^\bullet(V) = \cV \otimes_{\cO_{X^\Sigma_K}} \Omega^\bullet_{X^\Sigma_K}(\log D_K), \]
  and its subcomplex
  \[ \DR^\bullet_c(V) = \DR^\bullet(V)(-D_K),\]
  which are finite complexes whose terms are finite free $\cO_{X^\Sigma_K}$-modules (although the differentials are not $\cO_{X^\Sigma_K}$-linear). Since $\DR^\bullet_c(V)$ is a subcomplex of $\DR^\bullet(V)$, we can consider the quotient complex $\frac{\DR^\bullet(V)}{\DR^\bullet_c(V)}$, which is supported on the boundary $X_K - Y_K$.

  \begin{remark}
   Note that this complex can be defined for the algebraic varieties $Y_K \into X_K$, or for the rigid-analytic spaces $Y_K^{\mathrm{an}} \into X_K^{\mathrm{an}}$, and these correspond under the analytification functor for coherent sheaves.
  \end{remark}

  Our next goal will be to describe this complex. More precisely, if $Z_{r, g}$ is a boundary stratum of $X_K^{\min}$, and $\pi$ denotes the map $X_K \to X_K^{\min}$, we shall compute the restriction of $R\pi_*\left( \frac{\DR^\bullet(V)}{\DR^\bullet_c(V)} \right)$ to $Z_{r, g}$, in terms of de Rham complexes associated to algebraic representations of the Levi subgroup $M_{r, g}$.

 \subsection{D-modules}

  We shall need the concept of a \textbf{D-module} over an algebraic variety; for a systematic account see e.g.~\cite{hotta08}. For our purposes it suffices to note the following properties, for a smooth variety $Y$ over a field $k$ of characteristic 0:
  \begin{itemize}
   \item a $D$-module on $Y$ is a quasicoherent sheaf on $Y$ with an action of a certain (noncommutative) algebra sheaf on $Y$ (the sheaf of differential operators);
   \item a vector bundle with integrable connection on $Y$ is naturally a D-module on $Y$;
   \item for any $D$-module $M$ we can define a de Rham complex $\DR^\bullet(M)$, extending the usual definition when $M$ is a vector bundle.
  \end{itemize}
  This can be extended to non-smooth varieties, by choosing an embedding $Y \into X$ with $X$ a smooth variety and considering $D$-modules on $X$ supported in $Y$.

  Following \cite[\S 6]{hotta08}, we may define a subcategory $D^b_{\mathrm{rh}}(D_Y)$ of the derived category of complexes of $D$-modules consisting of bounded complexes of $D$-modules whose cohomology groups are \emph{regular holonomic}. This category satisfies the formalism of Grothendieck's ``six operations''; in particular, we have functors $f^!, f_!, f^*, f_*$ for any morphism $f$.

  \begin{note}
   In good geometric situations, the Grothendieck operations on D-modules can be understood in terms of pushforward functors for coherent sheaves. In particular, we have the following:
   \begin{itemize}
    \item If $\pi: X \to Y$ is a smooth morphism, the pushforward $\pi_*(M)$ for a $D$-module $M$ is the coherent-sheaf pushforward of the relative de Rham complex $M \otimes_{\cO_X} \Omega^\bullet_{X / Y}$. In particular, we can understand de Rham cohomology with coefficients in $M$ as pushforward along the structure map to $\Spec k$.
    \item If $D$ is a normal-crossing divisor in $X$, $Y \into^j X$ the complementary open subset, and $(M, \nabla)$ a connection on $Y$, which extends to a logarithmic connection on $X$ with unipotent monodromy along $D$, then $M$ is regular holonomic on $Y$, and the de Rham complex of $j_* M \in D^b_{\mathrm{rh}}(D_X)$ is the logarithmic de Rham complex $M \otimes_{\cO_X} \Omega^\bullet_{X}(\log Z)$, and similarly for $j_! M$.\qedhere
   \end{itemize}
  \end{note}

 \subsection{Restriction to the boundary}

  We first consider the following special case: let $X$ be a variety, $Z \subset X$ a closed subvariety, and $Y = X - Z$ the complementary open subvariety, with $j : U \into X$ and $i : Z \into X$ the inclusion maps. Then, for any $V \in D^b_{\mathrm{rh}}(D_Y)$, we have an exact triangle in $D^b_h(D_X)$,
  \begin{equation}
   \label{eq:exacttri}
   j_! V \to j_* V \to i_* i^* j_* V \to [+1].
  \end{equation}
  Thus the $D$-module complex $i^* j_* V$ on $Z$ is the ``boundary contribution'' to the cohomology of $Y$ (with coefficients in $V$).

  We shall need the following computation:

  \begin{proposition}
   Suppose $Z$ is a normal crossing divisor, and $V = (M, \nabla)$ is a connection on $Y$, extending to a logarithmic connection on $X$ with unipotent monodromy along $Z$.

   Then the de Rham complex of the $D$-module $i^* j_* V$ on $Z$, considered as a complex of sheaves on $X$ supported on $Z$, is equal to the quotient complex $\frac{\DR^\bullet(V)}{\DR^\bullet_c(V)}$.
  \end{proposition}

  \begin{proof}
   There is an explicit formula for $i^* j_* M$, as a complex of $D$-modules on $Z$ (which is well-known to the experts, although hard to find written down explicitly; we thank Claude Sabbah for explaining this formula to us). We equip $D_X$ with an increasing filtration by letting $F_n D_X$ be the operators $P$ satisfying $P \cdot I_Z^k \subseteq I_Z^{k-n}$ for all $k$ (with $I_Z$ the reduced ideal sheaf of the divisor $Z$). Then $F_0 D_X$, the sheaf of logarithmic differential operators, is a subring, and $F_1 D_X$ a module over $F_0 D_X$; moreover, $F_0 D_X$ preserves the canonical extension $\overline{V}$ of $V$. Then we may form the complex
   \[ F_1 D_X \otimes_{F_0 D_X} \overline{V} \longrightarrow \overline{V}(Z), \]
   and this gives a canonical representative of $i^* j_* V$.

   Locally around smooth points of $Z$ (i.e.~away from the intersections of the components), this complex is quasi-isomorphic to the much simpler complex $\overline{V}|_{Z} \longrightarrow \overline{V}|_{Z}$, where $\overline{V}|_{Z}$ is the pullback of $\overline{V}$ as an $\cO$-module, and the map is the monodromy operator (the residue along $Z$ of the logarithmic connection on $V$). One checks easily that the total complex of the double complex $\DR^\bullet\left(\overline{V}|_{Z} \longrightarrow \overline{V}|_{Z}\right)$ coincides with the $\cO$-module pullback to $Z$ of the logarithmic de Rham complex of $\overline{V}$.
  \end{proof}

 \subsection{Sheaves attached to algebraic representations}

  We are interested in taking $Y = Y_G(K)$ a Siegel Shimura variety. Then there is a natural \emph{canonical construction} functor
  \[ \mu_{K, \dR} : \operatorname{Rep}_{\QQ}(G) \to \operatorname{VB}(Y), \]
  where $\operatorname{VB}(-)$ denotes the category of vector bundles with integrable connection (and this takes values in the subcategory of connections with regular singularities along $X_K - Y_K$). This functor maps the defining representation of $G$ to the de Rham homology of the universal abelian variety over $Y$, with its Gauss--Manin connection. This extends naturally to a functor from complexes of algebraic representations to $D^b_{\mathrm{rh}}(Y)$.

  We want to compute the composite
  \[ i_{Z}^* \circ j_* \circ \mu_{K, \dR}\ :\ \operatorname{Rep}_{\QQ}(G) \to D^b_{\mathrm{rh}}(Z) \]
  when $X = X_K^{\min}$ and $Z = Z_{r, g}$ is a boundary stratum. This coincides, by proper base-change, with the pushforward to $Z_{r, g}$ of the corresponding D-module complex on $Z_{r, g}^{\Sigma} = \pi^{-1}(Z_{r, g})$, which is a union of boundary strata of $X_K^{\Sigma}$, for any choice of toroidal boundary data $\Sigma$. Since the proper pushforward functor is compatible with the formation of de Rham complexes, it follows that the de Rham complex of $(i_{Z_{r, g}}^* \circ j_* \circ \mu_{K, \dR})(V)$ coincides, as a complex of sheaves on $Z_{r, g}$, with the complex $R\pi_*\left( \frac{\DR^\bullet(V)}{\DR^\bullet_c(V)} \right)$ considered above.

  \begin{theorem}[Burgos--Wildeshaus]\label{thm:BWh}
   For each boundary component $Z = Z_{r, g}$ of $X_K^{\min}$ (identified with a Shimura variety $Y_{M_{r, h}}(K_{r, g})$), and each $V \in \operatorname{Rep}_{\QQ}(G)$, we have the formula
   \[ i_Z^*\, j_*\,  \mu_{K, \dR}(V) = \mu_{K_{r, g}} R\Gamma(\bar{H}_C, R\Gamma(\mathfrak{n}_r, V))[-c] \]
   after base-extension to $\QQbar$, where $c$ denotes the codimension of $Z_{r, g}$ in $X_K^{\min}$. More concretely, the cohomology sheaves of this complex are given by
   \[ \mathcal{H}^n\left(i_Z^*j_* \mu_{K, \dR}(V)\right) =
    \bigoplus_{i + j = n + c}
    \mu_{K_{r, g}} \left(H^i(\bar{H}_C, H^j(\mathfrak{n}_r, V)) \right). \]
  \end{theorem}

  \begin{note}
   Observe that $H^j(\mathfrak{n}_r, V)$ is a finite-dimensional algebraic representation of $M_r = M_{r, h} \times M_{r, \ell}$. It is therefore a direct sum of simple subrepresentations, each of which is a tensor product of a representation of $M_{r, h}$ (on which $\bar{H}_C$ acts trivially) and a representation of $M_{r, \ell}$ (on which $K_{r, g}$ acts trivially). In particular, if $M_{r, \ell}$ is a product of copies of $\mathbf{G}_m$, so that $\bar{H}_C$ (being a neat subgroup of $M_{r, \ell}(\QQ)$) is trivial, then we simply obtain
   \[ \mathcal{H}^n\left(i_Z^*\, j_*\, \mu_{K, \dR}(V)\right) =
       \mu_{K_{r, g}} \left(H^{n + c}(\mathfrak{n}_r, V)\right) .\]
  \end{note}

 \begin{proof}[Proof of \cref*{thm:BWh}]
  The main theorem of \cite{burgoswildeshaus04} is exactly the analogue of the above statement in the category of mixed Hodge modules (in the sense of Saito) on the Shimura variety over $\CC$. Recall that a mixed Hodge module is a pair consisting of a D-module and a perverse sheaf related by appropriate comparison maps; in particular, there is a forgetful functor to D-modules, and applying the forgetful functor to their result gives the theorem after base-extension to $\CC$. However, since the statement to be proved is a purely algebraic one, we can descend it to $\QQbar$ by the Lefschetz principle.
 \end{proof}

 \begin{remark} \
  \begin{enumerate}[(i)]
   \item The isomorphism of \cref{thm:BWh} satisfies a compatibility with the action of $G(\Af)$, analogous to the statements for \'etale sheaves in \cite[\S 4.8]{pink92}.

   \item We expect that these canonical isomorphisms over $\QQbar$ should descend to $\QQ$, but this cannot be deduced directly from the main result of \cite{burgoswildeshaus04}, since mixed Hodge modules only make sense over $\CC$. We have been informed by Kai--Wen Lan (pers.~comm.) that forthcoming work of his and his collaborators will establish a compatibility result for the $p$-adic Riemann--Hilbert correspondence with the six-operations functors; using this, one can derive a version of the theorem with $\Qp$-coefficients, by applying the $p$-adic Riemann--Hilbert functor to Pink's analogous result for \'etale sheaves.

   \item A related result in the case $V = \mathbf{1}_G$ is proved in \cite{wildeshaus07} using Voevodsky's triangulated category of geometrical motives. However, although there exists a ``de Rham realisation'' functor for Voevodsky's category, this still does not imply the result above: it gives an isomorphism between the de Rham cohomology groups of these two objects of $D^b_{\mathrm{rh}}(Z)$, but we do not know if this arises from an isomorphism between the underlying complexes of D-modules, so it is not meaningful to ``restrict'' to an open subspace.

   Ideally, one would like to have an isomorphism analogous to that of \cref{thm:BWh} in the category of relative motives over $Z$, from which the statement of the theorem given above would follow by taking de Rham realisations, and the result of \cite{wildeshaus07} would follow by pushing forward along the structure map to $\Spec \QQ$. However, we shall not attempt this here.\qedhere
  \end{enumerate}
 \end{remark}

\section{Application to rigid cohomology}
 \label{sect:forgetsupports}
 \subsection{Setup}
 
  We now apply the above results to prove \cref{prop:kernablapairsto0} of the main text. We recall the statement: 
  
  \begin{proposition}
   \label{prop:forgetsupports-appendix}
   The map
   \[ \tag{\dag} H^i_{c0}\left(\tb{X_{K, \Fp}^{(2, m)}}, \DR^\bullet_c(V)\right) \to  H^i_{c0}\left(\tb{X_{K, \Fp}^{(2, m)}}, \DR^\bullet(V)\right) \]
   induces an isomorphism on the $\Pif'$-generalized eigenspace for the spherical Hecke algebra, for all degrees $i$, where $\Pif$ is an automorphic representation of $G$ which is cohomological of weight $V$ and not CAP.
  \end{proposition}
  
  Here the subscript ``c0'' denotes cohomology with compact support towards the tube of the closed subvariety $X_{K, \Fp}^{\alpha}$, as in \cref{notation:defc0support}.
  
  \begin{proposition}
   We have an exact triangle
   \[ 
    R\Gamma_{c}\left(\tb{X_{K, \Fp}^{(1, m)}}, \DR^\bullet(V)\right) \to R\Gamma_{c}\left(\tb{X_{K, \Fp}^{m}}, \DR^\bullet(V)\right) \to R\Gamma_{c0}\left(\tb{X_{K, \Fp}^{(2, m)}},  \DR^\bullet(V)\right) \to [+1] 
   \]
   and similarly for $\DR_c$.
  \end{proposition}
  
  \begin{proof} 
   This is an instance of \cref{prop:exactseqs}. 
  \end{proof}
  
  \begin{proposition}
   \label{prop:interior-reduction}
   If the maps 
   \[ H^i_{c}\left(\tb{X_{K, \Fp}^{(1, m)}}, \DR^\bullet_c(V)\right) \to  H^i_{c}\left(\tb{X_{K, \Fp}^{(1, m)}}, \DR^\bullet(V)\right)\]
   and
   \[ H^i_{c}\left(\tb{X_{K, \Fp}^{m}}, \DR^\bullet_c(V)\right) \to  H^i_{c}\left(\tb{X_{K, \Fp}^{m}}, \DR^\bullet(V)\right)\]
   induce isomorphisms on the $\Pif'$-generalized eigenspaces for all degrees $i$, then the same is true for the map $(\dag)$.
  \end{proposition}
  
  \begin{proof}
   Since taking generalized eigenspaces is exact (a particular case of the exactness of localization), this follows from the five-lemma and the cohomology long exact sequences associated to the two triangles in the previous proposition.
  \end{proof}

 \subsection{Rigid cohomology of the $(1, m)$ locus}

  To analyse the first map in \cref{prop:interior-reduction}, we consider the short exact sequence of complexes of coherent sheaves on $X_K$,
  \[
   \tag{\dag}
   0 \to \DR^\bullet_c(V) \to \DR^\bullet(V) \to \frac{\DR^\bullet(V)}{\DR^\bullet_c(V)} \to 0.
  \]
  This makes sense either on the algebraic variety $X_K$, or on its dagger analytification $\cX_K$ (since analytification of coherent sheaves is exact). Clearly, the third term is supported on the boundary $Z_K^{\Sigma} = X_K^{\Sigma} - Y_K^{\Sigma}$ of the toroidal compactification (resp.~on its analytification $\mathcal{Z}_K^{\Sigma}$.) To lighten the notation we shall drop the $K$'s for the remainder of this proof.
  
  \begin{proposition}
   To prove that the first map of \cref{prop:interior-reduction} is an isomorphism in all degrees, it suffices to prove that the $\Pif'$-generalized eigenspace of $H^*_{c}\left(\tb{X_{K, \Fp}^{(1, m)}} \cap \mathcal{Z}^{\Sigma}, \frac{\DR^\bullet(V)}{\DR^\bullet_c(V)}\right)$ vanishes in all degrees.
  \end{proposition} 
  
  \begin{proof}
   We can interpret the map concerned as part of the cohomology long exact sequence associated to the exact triangle
   \[ 
    R\Gamma_{c}\left(\tb{X_{K, \Fp}^{(1, m)}}, \DR^\bullet_c(V)\right) \to  R\Gamma_{c}\left(\tb{X_{K, \Fp}^{(1, m)}}, \DR^\bullet(V)\right) \to
    R\Gamma_{c}\left(\tb{X_{K, \Fp}^{(1, m)}}\, \cap \mathcal{Z}_K^{\Sigma}, \frac{\DR^\bullet(V)}{\DR^\bullet_c(V)}\right)\to [+1], 
   \]
   given by $(\dag)$ and the identification of $\frac{\DR^\bullet(V)}{\DR^\bullet_c(V)}$ with a complex of sheaves on $\mathcal{Z}^\Sigma$. Since passing to $\Pif'$-generalized eigenspaces is exact (an instance of the exactness of localization at a maximal ideal), we have an analogous long exact sequence of localisations; and if the terms arising from $\frac{\DR^\bullet(V)}{\DR^\bullet_c(V)}$ localize to 0, the maps between remaining terms must be isomorphisms in every degree.
  \end{proof}
  
  \begin{proposition}
   We have
   \[ \tb{X_{\Fp}^{(1, m)}} \cap \mathcal{Z}^{\Sigma} = \tb{X_{\Fp}^{(1, m)}} \cap Z_{p, 1,\mathrm{id}}^{\Sigma} = \pi^{-1}\left(\tb{Z_{p, 1, \mathrm{id}, \Fp}^{\mathrm{ss}}}\right), \]
   where $Z_{p, 1, \mathrm{id}, \Fp}^{\mathrm{ss}}$ denotes the supersingular locus in the special fibre of $Z_{p, 1, \mathrm{id}}$, which is a finite union of prime-to-$p$ level modular curves.
  \end{proposition}
  
  \begin{proof}
   From the functoriality of the tube construction we have
   \[ \tb{X_{K, \Fp}^{(1, m)}}\ \cap\ \mathcal{Z}^\Sigma =\ 
     \Big] X_{K, \Fp}^{(1, m)}  \cap Z^{\Sigma}_{\Fp} \Big[\ .
   \]
   We saw in \cref{sect:explcitintersection} above that the intersection $X_{K, \Fp}^{(1, m)}  \cap Z^{\Sigma}_{\Fp}$ is contained in the stratum $Z^{\Sigma}_{p, 1, \mathrm{id}, \Fp}$, and the intersection is exactly the preimage of the supersingular locus $Z_{p, 1, \mathrm{id}, \Fp}^{\mathrm{ss}}$.
  \end{proof}
  
  These two propositions combine to give a canonical isomorphism
  \[ R\Gamma_{c}\left(\tb{X_{K, \Fp}^{(1, m)}}\, \cap \mathcal{Z}_K^{\Sigma}, \frac{\DR^\bullet(V)}{\DR^\bullet_c(V)}\right) = R\Gamma_{c}\left(\tb{Z_{p, 1, \mathrm{id}, \Fp}^{\mathrm{ss}}}, R\pi_* \left(\frac{\DR^\bullet(V)}{\DR^\bullet_c(V)}\right) \right), \]
  where the variety $Z_{p, 1, \mathrm{id}}$ is a disjoint union of modular curves $Z_{1, g}$ indexed by the global cusp-labels $(1, g) \in \mathfrak{C}(K)$ mapping to $(1, \mathrm{id}) \in \mathfrak{C}(\Kl(p))$.

  \begin{proposition}
   For each cusp-label $(1, g)$ mapping to $(1, \mathrm{id}) \in \mathfrak{C}(\Kl(p))$, the base-extension to $\QQbar_p$ of the complex of sheaves $R\pi_*\left( \frac{\DR^\bullet(V)}{\DR^\bullet_c(V)}\right)$ on the dagger space $\tb{Z_{1, g, \Fp}^{\mathrm{ss}}}$ has a filtration whose $n$-th graded piece is the de Rham complex of the vector bundle $\mu_{K_{1, g}, \dR}\left(W_n\right)$, where $W_n$ denotes the algebraic representation $H^{n+2}(\mathfrak{n}_{\Kl}, V)$ of $M_{\Kl}$.
  \end{proposition}
  
  \begin{proof}
   Since the map $\pi$ is proper, $R\pi_*$ commutes with rigid-analytification of coherent sheaves; so for each cusp-label $(1, g) \in \mathfrak{C}(K)$ mapping to $(1, \mathrm{id}) \in \mathfrak{C}(\Kl(p))$, the sheaf $R\pi_*\left( \frac{\DR^\bullet(V)}{\DR^\bullet_c(V)}\right)$ on the dagger space $\tb{Z_{1, g, \Fp}^{\mathrm{ss}}}$ is restriction to the open dagger subvariety $\tb{Z_{1, \mathrm{id}, \Fp}^{\mathrm{ss}}}$ of the analytification of the corresponding algebraic coherent sheaf $R\pi_*\left( \frac{\DR^\bullet(V)}{\DR^\bullet_c(V)}\right)$ on $Z_{1, \mathrm{id}}$ as a $\Qp$-variety.
   
   This we have computed above, using the formalism of D-modules, in \cref{thm:BWh}. In this case the codimension $c$ is 2, and the group $\overline{H}_C$ is trivial (as we observed above, since it is a neat arithmetic subgroup of $\mathbf{G}_m / \QQ$); so the direct sum over $i + j = n + c$ collapses to a single term with $j = n + 2$.
  \end{proof}
  
  We now consider passage to the limit over prime-to-$p$ level groups $K^{p}$. Recall that $Z_{p, 1, \mathrm{id}}$ denotes the disjoint union of the boundary strata $Z_{1, g}$ over global cusp-labels mapping to the local cusp-label $(1, \id)$.
  
  \begin{remark}
   This step is identical to the corresponding step in the computation of boundary computations to Betti cohomology of locally symmetric spaces, a classical result due to Harder (see e.g.~\cite[\S 4.2.1]{harderraghuram20} for an account). We do not need to use any properties of the stratification at $p$ other than its invariance under the $G(\Af^p)$-action.
  \end{remark}
  
  \begin{proposition}
   We have an isomorphism of $G(\Af^p)$-representations
   \[ 
    \varinjlim_{K^p} H^*_{c}\left(\tb{Z_{p, 1, \mathrm{id}, \Fp}^{\mathrm{ss}}}, - \right) = \operatorname{Ind}_{P_{\Kl}(\QQ) P_{\Kl, h}(\Af^p)}^{G(\Af^p)}\left( H^*_c \left(\tb{Z_{1, \id, \Fp}^{\mathrm{ss}}}, -\right)\right),
   \]
   where $-$ denotes the complex $ R\pi_*\tfrac{\DR^\bullet(V)}{\DR^\bullet_c(V)}$.
  \end{proposition}
  
  \begin{proof}
   For each $K^p$ we have
   \[ R\Gamma_{c}\left(\tb{Z_{p, 1, \mathrm{id}, \Fp}^{\mathrm{ss}}}, - \right) = 
      \bigoplus_{g} 
       R\Gamma_{c}\left(\tb{Z_{1, g, \Fp}^{\mathrm{ss}}}, -\right).
   \]
   The $g$ appearing in this sum are parametrized by the quotient $P_{\Kl}(\QQ) P_{\Kl, h}(\Af^p) \backslash G(\Af^p) / K^p$; so the above is exactly the standard formula given by Mackey theory for the $K^p$-invariants of the induced representation. Passing to the limit over $K^p$ we obtain the result.
  \end{proof}

  It remains to identify the action of the group $P_{\Kl}(\QQ) P_{\Kl, h}(\Af^p)$ on $R\Gamma_c \left(\tb{Z_{1, \id, \Fp}^{\mathrm{ss}}}, -\right)$. Again, this step is identical to the corresponding computation for Betti cohomology. The action of the unipotent radical $N_{\Kl}(\Af)$ is trivial (since $N_{\Kl}(\Af)$ acts trivially on the boundary components of the minimal compactification, and on the Lie algebra cohomology $R\Gamma(\mathfrak{n}_{\Kl}, -)$); so the action factors through $M_{\Kl, h}(\Af^p) M_{\Kl}(\QQ) \cong \GL_2(\Af^p) \times \QQ^\times$. The $\QQ^\times$ factor also acts trivially on the boundary components, but acts on the coefficients via its natural action on the complex of algebraic $M_{\Kl}$-representations $R\Gamma(\mathfrak{n}_{\Kl}, V)$. The action of $\GL_2(\Af^p)$ is given by the following theorem:
  
  \begin{proposition}[Deuring, Serre]
   If $W$ denotes the algebraic representation $\Sym^k$ of $\GL_2$, and $L_p = \GL_2(\Zp)$, then the cohomology
   \[ \varinjlim_{L^p \subset \GL_2(\Af^p)} H^*_c\big( \tb{Y_{\GL_2}(L^p L_p)^{ss}}, W\big)\]
   is concentrated in degree 1, and is isomorphic as a $\GL_2(\Af^p)$-representation to the direct sum of the prime-to-$p\infty$ parts of all cuspidal automorphic representations associated to modular forms of weight $k + 2$ whose local factor at $p$ is an unramified twist of Steinberg.
  \end{proposition}
  
  \begin{proof}
   This follows from the fact that the supersingular locus in the $\GL_2$ Shimura variety can be identified (compatibly with the $G(\Af^p)$-action) with the 0-dimensional Shimura variety for the unique quaternion algebra $D$ of discriminant $p\infty$, with the local level at $p$ given by a maximal order in $D \otimes \Qp$. We refer to introduction of \cite{tianxiao16} for a modern account of this theory (including the compatibility with the passage from algebraic representations to coefficient sheaves). Hence the $G(\Af^p)$-representations which contribute are precisely those which are in the image of Jacquet--Langlands transfer from automorphic representations of $D^\times$ which have maximal level at $p$, and have Archimedean component $W^\vee$. These correspond to the modular forms stated.
  \end{proof}
  
  \begin{corollary}
   Every representation of $G(\Af^p)$ appearing as a Jordan-H\"older factor of the representation
   \[ \varinjlim_{K^p} H^*_{c}\left(\tb{Z_{p, 1, \mathrm{id}, \Fp}^{\mathrm{ss}}}, R\pi_*\left( \frac{\DR^\bullet(V)}{\DR^\bullet_c(V)} \right)\right)\]
   is a subquotient of a representation of the form $\operatorname{Ind}_{P(\Af^p)}^{G(\Af^p)}\left(\pi_{\mathrm{f}}^p \times \chi\right)$, where $\pi$ is the representation associated to a modular form of some weight $\ge 2$, and $\chi$ is an algebraic Gr\"ossencharacter.
  \end{corollary}
  
  \begin{proof}
   We can write the induction $\operatorname{Ind}_{P_{\Kl}(\QQ) P_{\Kl, h}(\Af^p)}^{G(\Af^p)}$ as the composite of two steps: induction from $\GL_2(\Af^p) \times \QQ^\times$ to $(\GL_2 \times \GL_1)(\Af^p)$, and parabolic induction from $P_{\Kl}(\Af^p)$ to $G$.
   
   We have seen that the representation $R\Gamma_c \left(\tb{Z_{1, \id, \Fp}^{\mathrm{ss}}}, -\right)$ has a filtration whose gradeds are direct sums of the representations $\pi_{\mathrm{f}}^p$ associated to modular forms, considered as representations of $\GL_2(\Af^p) \times \QQ^\times$ with $\QQ^\times$ acting via some algebraic character. Induction from $\GL_2(\Af^p) \times \QQ^\times$ to $\GL_2(\Af^p) \times \GL_1(\Af^p)$ maps each $\pi_{\mathrm{f}}^p$ to the direct sum of the $\pi_{\mathrm{f}}^p \times \chi$, as $\chi$ varies over Gr\"ossencharacters of conductor coprime to $p$ and infinity-type determined by the $\QQ^\times$-action. So the possible Jordan-H\"older factors of $\varinjlim_{K^p} H^*_{c}\left(\tb{Z_{p, 1, \mathrm{id}, \Fp}^{\mathrm{ss}}}, R\pi_*\left( \frac{\DR^\bullet(V)}{\DR^\bullet_c(V)} \right)\right)$ are precisely the parabolic inductions of these representations. 
  \end{proof}
  
  Passing to $K_p$-invariants, for an arbitrary $K_p$, we conclude that the spherical Hecke eigensystems appearing in $H^*_{c}\left(\tb{X_{K, \Fp}^{(1, m)}} \cap \mathcal{Z}^{\Sigma}, \frac{\DR^\bullet(V)}{\DR^\bullet_c(V)}\right)$ are all associated to pairs $(\pi, \chi)$ as above; in terms of Galois representations, they correspond to 4-dimensional representations which are the sum of two 2-dimensional blocks. Hence this module must localize to 0 at a non-endoscopic, non-CAP representation $\Pi$, and we have shown that the first map in \cref{prop:interior-reduction} is an isomorphism on the $\Pif'$-generalized eigenspace after base-extension to $\QQbar_p$. Since base-extension to $\QQbar_p$ is compatible with forming kernels and cokernels, the map is an isomorphism over $\Qp$.

 \subsection{Rigid cohomology of the multiplicative locus}

  We now consider the second statement of \cref{prop:interior-reduction}, using full multiplicative locus $X_{K, \Fp}^m = X_{K, \Fp}^{(1, m)} \cup X_{K, \Fp}^{(2, m)}$. Exactly as in the previous section, it suffices to show that the $\Pif'$-generalized eigenspace is zero in the cohomology
  \[ H^*_{c}\left(\tb{X_{K, \Fp}^{m}} \cap \mathcal{Z}^{\Sigma}, \frac{\DR^\bullet(V)}{\DR^\bullet_c(V)}\right),\]
  where $\mathcal{Z}^{\Sigma}$ is the toroidal boundary.
  
  We now recall that we can decompose $Z^{\Sigma}$ (and hence $\mathcal{Z}^{\Sigma}$) as a union $Z^{\Sigma}_1 \cup Z_2^{\Sigma}$, where $Z^{\Sigma}_r = \bigsqcup_{g \in \mathfrak{C}(r, K)} Z^{\Sigma}_{r, g}$. Moreover, $Z_1^{\Sigma}$ is open in $Z^{\Sigma}$, and $Z_2^{\Sigma}$ is closed. This gives an exact triangle
  \[ 
   R\Gamma_c\left(\tb{X_{K, \Fp}^{m}} \cap \mathcal{Z}^{\Sigma}_1, -\right) \to 
   R\Gamma_c\left(\tb{X_{K, \Fp}^{m}} \cap \mathcal{Z}^{\Sigma}, -\right) \to 
   R\Gamma\left(\tb{X_{K, \Fp}^{m}} \cap \mathcal{Z}_2^{\Sigma}, -\right)  \to [+1]. \]
  
  \begin{remark}
   Here we do not need to specify a support condition for the third term, since $\tb{X_{K, \Fp}^{m}} \cap \mathcal{Z}_2^{\Sigma}$ is a union of connected components of $\mathcal{Z}_2^{\Sigma}$ and hence proper.
   
   The exact triangle above can be obtained purely sheaf-theoretically using the description of $i^* j_* V$ given above: there is a natural map from the de Rham complex of $(i^{\Sigma})^* (j^\Sigma)_*(V)$, as a complex of coherent sheaves on $Z^{\Sigma}$, to the corresponding complex for the restriction to $Z_2^{\Sigma}$. A somewhat tedious computation in local coordinates shows that the kernel of this map can be identified with the image of the natural map from the logarithmic de Rham complex along $Z_2^{\Sigma}$ with compact support towards the intersection points.
  \end{remark}
  
  For the first term (with $\mathcal{Z}^{\Sigma}_1$), we have seen in \cref{sect:explcitintersection} that  
  \[ \tb{X_{K, \Fp}^{m}} \cap \cZ^{\Sigma}_1 = \pi^{-1} \left(\tb{Z^m_{p, 1, s_1, \Fp}}\right)\, \sqcup\, \pi^{-1}\left(\tb{Z_{p, 1, \id, \Fp}}\right),\]
  giving a corresponding direct-sum decomposition of the cohomology. Exactly as in the previous section, after passing to the limit over prime-to-$p$ levels $K^p$, we obtain the induction formula
  \[ \varinjlim_{K^p} H^*_c\left(\tb{Z^m_{p, 1, s_1, \Fp}}, -\right) = \operatorname{Ind}_{P_{\Kl}(\Af^p)P_{\Kl}(\QQ)}^{G(\Af^p)}
  H^*_c\left(\tb{Z^m_{p, 1, s_1, \Fp}}, - \right),\]
  and similarly for $Z_{p, 1, \id, \Fp}$. So we are reduced to computing the $\GL_2(\Af^p)$-representations appearing in the cohomology of a modular curve of prime-to-$p$ level, resp.~of the multiplicative locus in a modular curve of $\Gamma_0(p)$ level. The cohomology of either can be computed (over $\QQbar_p$) in terms of modular forms of either unramified or Steinberg type at $p$, and the proof concludes exactly as before.
  
  One can also make a similar argument in the case of the Siegel boundary strata, but this is somewhat more intricate (since the group $\bar{H}_C$ appearing in the computation varies with $K^p$). So it is much simpler just to note that as $\tb{X_{K, \Fp}^{m}} \cap \mathcal{Z}_2^{\Sigma} = \mathcal{Z}_{p, 2, \mathrm{id}}^{\Sigma}$ is a union of components of the analytification of $Z^{\Sigma}_{p, 2, \mathrm{id}}$, its cohomology is a $G(\Af^p)$-invariant direct summand of the cohomology of the full Siegel boundary stratum $\cZ_2^{\Sigma}$. The latter is the analytification of a projective algebraic variety over $\QQ$; so, by the rigid-analytic GAGA theorem, we can compute its cohomology using algebraic de Rham cohomology, and hence compare with Betti cohomology of the complex points. (That is, we may directly quote Harder's computations rather than recapitulating their proofs.) We conclude, again, that all the eigensystems for the spherical Hecke algebra appearing in $R\Gamma\left(\tb{X_{K, \Fp}^{m}} \cap \mathcal{Z}_2^{\Sigma}, \DR^\bullet(V)/\DR^\bullet_c(V)\right)$ correspond to parabolic inductions (this time via the Siegel parabolic), of cohomological automorphic representations of $\GL_2 \times \GL_1$. Hence these localize to 0 at $\Pif'$, completing the proof of \cref{prop:forgetsupports-appendix}.
  

%
\newcommand{\noopsort}[1]{\relax}
 \bibliographystyle{../../amsalphaurl}
\newcommand{\etalchar}[1]{$^{#1}$}
\providecommand{\bysame}{\leavevmode\hbox to3em{\hrulefill}\thinspace}
\providecommand{\MR}[1]{%
 MR \href{http://www.ams.org/mathscinet-getitem?mr=#1}{#1}.
}
\providecommand{\href}[2]{#2}
\newcommand{\articlehref}[2]{\href{#1}{#2}}


\end{document}